\documentclass[a4paper,oneside,11pt]{article}

\usepackage{amsmath,amsfonts,amscd,amssymb}
\usepackage{longtable,geometry}
\usepackage[english]{babel}
\usepackage[utf8]{inputenc}
\usepackage[active]{srcltx}
\usepackage[T1]{fontenc}
\usepackage{graphicx}
\usepackage{pstricks}
\usepackage{bbm}
\usepackage{mathtools}

\usepackage{MnSymbol}
\usepackage{stmaryrd}
\usepackage{nicefrac}
\usepackage{calrsfs}
\usepackage{enumitem}

\usepackage{tocloft}

\usepackage{rotating}

\usepackage{xcolor}
\usepackage{framed}
\usepackage{hyperref}
\usepackage{cleveref}

\colorlet{shadecolor}{blue!15}

\newtheorem{theorem}{Theorem}[section]
\newtheorem{conj}{Conjecture}

\newtheorem{corollary}[theorem]{Corollary}
\newtheorem{lemma}[theorem]{Lemma}
\newtheorem{claim}[theorem]{Claim}
\newtheorem{proposition}[theorem]{Proposition}
\newtheorem{definition}[theorem]{Definition}

\newtheorem{remark}[theorem]{Remark}

\newtheorem{question}[conj]{Question}
  
\newcommand{\be}[1]{\begin{equation}\label{#1}}
\newcommand{\ee}{\end{equation}}
\numberwithin{equation}{section}

\newcommand{\ba}[1]{\begin{align}\label{#1}}
\newcommand{\ea}{\end{align}}
\numberwithin{equation}{section}

\newcommand{\ben}{\begin{equation*}}
\newcommand{\een}{\end{equation*}}
\numberwithin{equation}{section}

\newenvironment{proof}[1][\relax]%s
  {\paragraph{Proof\ifx#1\relax\else~of #1\fi}}%
  {~\hfill$\square$\par\bigskip}

\newcommand{\calA}{\mathcal{A}}

\newcommand{\calC}{\mathcal{C}}
\newcommand{\calD}{\mathcal{D}}
\newcommand{\calE}{\mathcal{E}}
\newcommand{\calF}{\mathcal{F}}
\newcommand{\calG}{\mathcal{G}}
\newcommand{\calH}{\mathcal{H}}

\newcommand{\calP}{\mathcal{P}}

\newcommand{\bbE}{\mathbb{E}}

\newcommand{\bbN}{\mathbb{N}}

\newcommand{\bbP}{\mathbb{P}}

\newcommand{\bbR}{\mathbb{R}}

\newcommand{\bbU}{\mathbb{U}}

\newcommand{\bbZ}{\mathbb{Z}}
\newcommand{\sfC}{{\sf C}}

\newcommand{\bfN}{\mathbf N}
\newcommand{\bfP}{\mathbf P}
\newcommand{\bfT}{\mathbf T}

\newcommand{\ep}{\varepsilon}
\newcommand{\eps}{\varepsilon}

\newcommand{\la}{\lambda}
\newcommand{\La}{\Lambda}

\newcommand{\Ann}{{\rm Ann}}
\newcommand{\ind}{\mathbbm 1}

\setlist[itemize]{itemsep=0pt, topsep=2pt}
\setlist[enumerate]{itemsep=1pt, topsep=4pt}

\newcommand{\rk}[1]{\bgroup\color{red}%
  \par\medskip\hrule\smallskip%
  \noindent\textbf{#1}%
  \par\smallskip\hrule\medskip\egroup}

\newcommand{\ghost}{\mathfrak g}

\newcommand{\lra}{\leftrightarrow}
\newcommand{\llra}{\longleftrightarrow}
\newcommand{\xlra}{\xleftrightarrow}
\newcommand\concel[2]{\ooalign{$\hfil#1\mkern0mu/\hfil$\crcr$#1#2$}}  % used in above
\newcommand\nxlra[1]{\mathrel{\mathpalette\concel{\xlra{#1}}}}

\renewcommand{\ep}{\varepsilon}

\newcommand{\Cov}{\mathrm{Cov}}

\usepackage{microtype}
\def\les{\!{\begin{array}{c}<\\[-10pt]\scriptstyle{\frown}\end{array}}\!}
\def\ges{\!{\begin{array}{c}>\\[-10pt]\scriptstyle{\frown}\end{array}}\!}

\title{Planar random-cluster model: scaling relations}
\author{Hugo Duminil-Copin\thanks{Universit\'e de Gen\`eve}\ \thanks{Institut des Hautes \'Etudes Scientifiques} and Ioan Manolescu\addtocounter{footnote}{1}\thanks{University of Fribourg}}
\date{\today}

\begin{document}

\maketitle

\begin{abstract}
    This paper studies the critical and near-critical regimes of the planar random-cluster model on $\mathbb Z^2$ with cluster-weight $q\in[1,4]$ using novel coupling techniques. More precisely, we derive the {\em scaling relations} between the critical exponents $\beta$, $\gamma$, $\delta$, $\eta$, $\nu$, $\zeta$ as well as $\alpha$ (when $\alpha\ge0$). As a key input, we show the stability of crossing probabilities in the near-critical regime using new interpretations of the notion of influence of an edge in terms of the rate of mixing. As a byproduct, we derive a generalization of Kesten's classical scaling relation for Bernoulli percolation involving the ``mixing rate'' critical exponent $\iota$ replacing the four-arm event exponent $\xi_4$.
  \end{abstract}

\setcounter{tocdepth}{2}
\tableofcontents

\section{Introduction}
 
\subsection{Motivation}

Understanding the behaviour of physical systems undergoing a continuous phase transition at and near their critical point is one of the major challenges of modern statistical physics, both on the physics and the mathematical sides. In the first half of the twentieth century, the understanding relied essentially on exact computations, as exemplified by the analysis of mean-field systems and Onsager's revolutionary solution of the 2D Ising model~\cite{Ons44}. In the sixties, the arrival of the renormalization group (RG) formalism (see~\cite{Fis98} for a historical
exposition) led to a (non-rigorous) deep physical and geometrical understanding of continuous phase transitions. The RG formalism suggests that ``coarse-graining'' renormalization transformations correspond to appropriately changing the scale and the parameters of the model under study. The large scale limit of the critical regime then arises as the fixed
point of the renormalization transformations. 

A striking consequence of the RG formalism is that the critical fixed point being usually unique, the scaling limit at the critical point must satisfy translation, rotation, scale and even conformal invariance, see e.g.~\cite{BPZ84b,BPZ84a}.
In two dimensions, this prediction allowed for the computation of {\em critical exponents} ruling the behaviour of thermodynamical quantities and the classification of models into {\em universality classes}, meaning classes of models undergoing the same critical behaviour. 

Another observation related to the previous developments is that the critical exponents are related to each other: if the behaviours of the specific heat, the order parameter, the susceptibility, the source-field, the two-point function and the correlation length are governed by the exponents $\alpha$, $\beta$, $\gamma$, $\delta$, $\eta$ and $\nu$ respectively, then the following {\em scaling relations} should be satisfied (below,  the dimension $d$ of the lattice is assumed to be equal to 2, but we state the relations in this generality as they are predicted to extend to any dimension below the so-called upper critical dimension of the system):
\begin{align}
&\frac{2-\alpha}d=\nu=\frac{2\beta}{d-2+\eta},\label{eq:a1}\\
&2-\eta=d\frac{\delta-1}{\delta+1}=\frac\gamma\nu.\label{eq:a2}\end{align}
 A striking feature of these relations is that they hold {\em for different universality classes}, meaning that the critical exponents may be different for different models, yet they are always related via~\eqref{eq:a1} and~\eqref{eq:a2}.

The aim of this paper is to provide rigorous proofs of these scaling relations for a family {\em of planar percolation models}. Percolation models are models of random subgraphs of a given lattice. Bernoulli percolation is maybe the most studied such model, and breakthroughs in the understanding of its phase transition have often served as milestones in the exciting history of statistical physics.  The random-cluster model (also called Fortuin-Kasteleyn percolation) is another example of a percolation model. It was introduced by Fortuin and Kasteleyn around 1970~\cite{For70,ForKas72} as a generalization of Bernoulli percolation. 
It was found to be related to many other models of statistical mechanics, including the Ising and Potts models, and to exhibit a very rich critical behaviour. Of particular importance from the point of view of physics and for the relevance of our paper is the fact that the scaling limits of the random-cluster models at criticality are expected to fall into different universality classes and to be related to various 2D conformal field theories.

Let us conclude this section by reminding the reader that the theory of Bernoulli percolation is now well developed, with a decent understanding of the properties of the scaling limit~\cite{AizBur99}, of crossing probabilities~\cite{Rus78,SeyWel78}, universal critical exponents~\cite{KesSidZha98}, scaling relations~\cite{Kes87,Nol08}, noise sensitivity and near-critical window~\cite{GPS10b,GPS}, etc. For a variant of the model (site percolation on the triangular lattice), the existence of the scaling limit and its conformal invariance was proved~\cite{Smi01} and critical exponents have been computed~\cite{SW01}, see~\cite{BefDum13} and references therein for an overview of two-dimensional Bernoulli percolation. Deriving all these properties for Bernoulli percolation relies on specific features, such as independence of the states of different edges and geometric interpretations of differential formulae using so-called pivotal events. These features are not satisfied for more general random-cluster models.
Another more prosaic goal of this paper is therefore to develop robust tools enabling one to bypass these special characteristics of Bernoulli percolation to extend the results mentioned in the abstract to the whole regime of critical random-cluster models undergoing a continuous phase transition. As such, these tools may have a number of implications that are not mentioned in the present paper, in particular for the study of other planar dependent percolation models.

\subsection{Definition of the random-cluster model}\label{sec:1.2}

As mentioned in the previous section, the model of interest in this paper is the random-cluster model, which we now define. 
For background, we direct the reader to the monograph~\cite{Gri06} and to the lecture notes~\cite{Dum17a} for an exposition of the most recent results.

Consider the square lattice $(\mathbb Z^2,\mathbb E)$, that is the graph with vertex-set $\mathbb Z^2=\{(n,m):n,m\in\mathbb Z\}$ and edges between nearest neighbours. In a slight abuse of notation, we will write $\bbZ^2$ for the graph itself. 
Consider a finite subgraph $G = (V,E)$ of the square lattice ($V$ denotes the vertex-set and $E$ the edge-set) and let $\partial G$ be the set of vertices in $V$ incident to at most three edges in $E$.
Write $\La_n$ for the subgraph of $\bbZ^2$ spanned by the vertex-set $\{-n,\dots, n\}^2$. For  $1 \le r < R$, write $\Ann(r,R)$ for the annulus $\Lambda_R\setminus\Lambda_{r-1}$. 
We also write $\Lambda_n(x)$ and $\Lambda_n(e)$ for the boxes of size $n$ recentred around $x$ and the bottom left endpoint of the edge $e$, respectively.

In order to define the model, consider first a finite graph $G$. A percolation configuration $\omega$ on $G$ is an element of $\{0,1\}^{E}$. An edge $e$ is said to be {\em open} (in $\omega$) if $\omega_e=1$, otherwise it is {\em closed}. A configuration  $\omega$ can be seen as a subgraph of $G$ with vertex-set $V$ and edge-set $\{e\in E:\omega_e=1\}$. When speaking of connections in $\omega$, we view $\omega$ as a graph. For sets of vertices $A$ and $B$, we say that $A$ is connected to $B$ if there exists a path of edges of $\omega$ with endpoints that connect a vertex of $A$ to a vertex of $B$. This event is denoted by $A{\longleftrightarrow}B$. We also speak of connections in a set of vertices~$C$ if the endpoints of the edges of the path are all in~$C$.

A {\em cluster} is a connected component of $\omega$. The {\em boundary conditions} $\xi$ on $G$ are given by a partition of $\partial G$. We say that two vertices of $G$ are {\em wired together} if they belong to the same element of the partition $\xi$.
\begin{definition} The random-cluster measure on $G$ with edge-weight $p$, cluster-weight $q>0$ and boundary conditions $\xi$  is given by
\begin{equation}\label{eq:RCM_def1}
	\phi_{G,p,q}^\xi[\omega]=\frac{1}{Z^\xi(G,p,q)} (\tfrac{p}{1-p})^{|\omega|}q^{k(\omega^\xi)},
\end{equation}
where $|\omega|=\sum_{e\in E}\omega_e$, $k(\omega)$ is the number of connected components of the graph, $\omega^\xi$ is the graph obtained from $\omega$ by identifying wired vertices together, and  $Z^\xi(G,p,q)$ is a normalising constant called the {\em partition function} chosen in such a way that $\phi_{G,p,q}^\xi$ is a probability measure. 
\end{definition}

Two specific families of boundary conditions will be of special interest to us. On the one hand, the {\em free} boundary conditions, denoted 0, correspond to no wirings between boundary vertices. On the other hand, the {\em wired} boundary conditions, denoted 1, correspond to all boundary vertices being wired together. 

The random-cluster model may be modified to accommodate an external magnetic field as follows. 
Add to the lattice $\bbZ^2$ a vertex $\ghost$ called the {\em ghost vertex} and connect it to each vertex $v$ of $\bbZ^2$ by an edge $v\ghost$. The random-cluster measure $\phi_{G,p,q,h}^i$ (for $i=0$ or $1$ and $h \geq 0$) is defined exactly as the random-cluster model on $G$, except that the boundary is now $\partial G\cup\{\mathfrak g\}$, and the edge-weight is $p$ for the edges of $G$ and $1-e^{-h}$ for edges having $\ghost$ as an endpoint, i.e.~that
\begin{equation}\label{eq:RCM_def1}
	\phi_{G,p,q,h}^i[\omega]=\frac{1}{Z^\xi(G,p,q,h)} (\tfrac p{1-p})^{|\omega|}(e^h-1)^{\Delta(\omega)}q^{k(\omega^i)},
\end{equation}
where $\Delta(\omega):=\sum_{v\in V}\omega_{v{\mathfrak g}}$. The probability that $\ghost$ is in the cluster of 0 has an interpretation in terms of spin models with a magnetic field: for $q=2$, this probability is equal to the {\em spontaneous magnetization} with an external field $h$ for the Ising model on the square lattice. 
A similar interpretation holds for the 3-state and 4-state Potts models. For more details on this topic, see~\cite{BiskBorChaKot00}. 

For $q\ge1$ and $i=0,1$, the family of measures $\phi_{G,p,q,h}^i$ converges weakly as $G$ tends to the whole square lattice. The limiting measures on $\{0,1\}^{\bbE}$ are denoted by $\phi_{p,q,h}^i$ and are called {\em infinite-volume} random-cluster measures with free and wired boundary conditions. They are invariant under translations and ergodic. When $h=0$, we simply drop it from the notation.

The random-cluster model undergoes a phase transition at $h=0$ and a critical parameter $p_c=p_c(q)$  in the following sense: if $p>p_c(q)$, the probability 
$$\theta(p):=\phi_{p,q}^1[0\text{ is in an infinite cluster}]$$
is strictly positive, while for $p<p_c(q)$, it is equal to 0.
In the past ten years, considerable progress has been made in the understanding of this phase transition: the critical point was proved in~\cite{BefDum12} (see also~\cite{DumMan14,DCRT16}) to be equal to 
$$p_c(q)=\frac{\sqrt q}{1+\sqrt q}.$$
It was also shown in these papers that the {\em correlation length} \begin{align}\label{eq:correlation}
\xi(p)&:=\lim_{n\rightarrow\infty}-n/\log [\pi_1(p,n)-\theta(p)]\in[0,\infty]
\end{align}
is finite as soon as $p\ne p_c$, where
\begin{equation}
\pi_1(p,n):=\phi_{p,q}^1[0\longleftrightarrow\partial\Lambda_n]
\end{equation}
(when $p=p_c$ we drop $p_c$ from this notation).

For $q\ge1$, it was  proved in~\cite{DumSidTas16, DGHMT16} (see also~\cite{RaySpi19} when $q>4$) that the correlation length at $p_c$ is infinite if and only if $q\le 4$. As the divergence of the correlation length is one of the characterizations of a continuous phase transition, and as we are interested in this type of phase transition only,  in the whole paper we will restrict our attention to the range $q\in[1,4]$. Also, since the $q=1$ case was already treated by Kesten in~\cite{Kes87}, and later solved in numerous other places (see references below), we will often assume that $q>1$.

\paragraph{Two notational conventions} Since $q\in[1,4]$ will always be fixed, we drop it from notation. For $q\in[1,4]$, there is a unique infinite-volume random-cluster measure, so we omit the superscript corresponding to the boundary condition and denote it simply  by $\phi_{p}$.

For two families  $(f_i)_{i\in I}$ and $(g_i)_{i\in I}$, introduce $ f\asymp g$ (resp.~$ f\les g$ and $f\ges g$) to refer to the existence of constants $c,C\in(0,\infty)$ such that for every $i\in I$, $cg_i\le f_i\le Cg_i$ (resp.~$f_i\le Cg_i$ and $f_i\ge cg_i$). In most cases, the family $I$ will be obvious from context and omitted. In the special case where $I$ contains (implicitly or explicitly) the edge-parameter $p\in(0,1)$, we in fact further require that $p$ is not close to 0 or 1 (which is justified for every application that we have in mind since we are interested in properties for $p$ close to $p_c$).

\subsection{Stability below the characteristic length}\label{sec:stab}\label{sec:1.5}

When studying a non-critical system, a natural length-scale is provided by the {\em characteristic length}, which appeared in a simplified context of Bernoulli percolation in the work of Kesten~\cite{Kes87} (see also~\cite{BCKS99})  and was explicited for the random-cluster model in~\cite{DumGarPet14}. In order to define the characteristic length, we first introduce the notion of crossing probability.

A {\em quad} $(\calD; a,b,c,d)$ is a finite subgraph of $\bbZ^2$ whose boundary $\partial\calD$ is a simple path of edges of $\mathbb Z^2$,  along with four points $a,b,c,d$ found on $\partial \calD$ in counterclockwise order. These points define four arcs $(ab)$, $(bc)$, $(cd)$, and $(da)$ corresponding to the parts of the boundary between them. We also see the quad as a domain of $\mathbb R^2$ with marked points on its boundary by taking the union of the faces enclosed by $\partial\calD$. The typical example is the case of rectangles $[0,n]\times[0,m]$ or $\La_n$ with $a,b,c,d$ being the corners of the rectangle, oriented in counterclockwise order, starting from the bottom-right one. In this case, we omit $a,b,c,d$ from the notation.
We say that the quad $(\calD,a,b,c,d)$ is {\em crossed} if $(ab)$ is connected to $(cd)$ in $\calD$. The event is denoted by $\calC(\calD)$. 

We say that a quad $(\calD; a,b,c,d)$ is {\em $\eta$-regular at scale $R$} for some $\eta > 0$ if $\calD$
is contained in  $\Lambda_R$, 
is the union of a finite number of translates of $\Lambda_{\eta R}$ by points of $\eta R\mathbb Z^2$, 
and $a,b,c,d\in\eta R\bbZ^2$.

Now, consider $\delta>0$ small enough. How small $\delta$ should be is dictated by the proof of Theorem~\ref{thm:RSWnear} and we simply wish to mention here that $\delta$ can be taken independent of $q\in[1,4]$, and can easily be estimated (even though the value is irrelevant for our study).

\begin{definition}[Characteristic length] For each $q\ge1$ and $p\in(0,1)$, let
\begin{align}\label{eq:L_def}
	L(p) =L(p,q):= 	\inf\{ R \geq 1 : \phi_{p} [\calC(\Lambda_R)] \notin[\delta, 1-\delta]\}\in[1,\infty].
\end{align}
\end{definition}

Note that $L(p)<+\infty$ for every $p\ne p_c$ by~\cite{BefDum12}; 
by duality,  $L(p_c)=+\infty$ as long as $\delta < 1/2$, which we will always assume. The interest of the characteristic length lies in its connection with the {\em scaling window}, i.e.~the regime of parameters $(R,p)$ for which one expects typical properties of the random-cluster model in $\Lambda_R$ with parameters $p$ to be similar to the critical ones. 
 In physics, the statement that the system looks critical is usually related to another length-scale, namely the correlation length $\xi(p)$ defined in \eqref{eq:correlation}.
The correlation length encodes the rate of exponential decay of the probability of being connected to distance $n$ but not to infinity as $n$ tends to infinity; it is not a priori directly related to $L(p)$. Nevertheless, the following result reunite the two notions of correlation and characteristic lengths, thus affirming that the  characteristic length is simply the correlation length in disguise. 

\begin{theorem}[Equivalence  correlation/characteristic lengths]\label{thm:L_equiv_xi}
    Fix $1\le q\le 4$, we have that for $p\in(0,1)$,   
    \begin{equation}
    	L(p)\asymp \xi(p).
    \end{equation}
\end{theorem}

The proof is based on a coarse-grained procedure.  
We wish to highlight that the result is new for every $1<q\le 4$, even for $q=2$ for which
\cite[Theorem~1.2]{DumGarPet14} proves almost the same statement, 
but with a logarithmic control over the ratio of $L(p)/\xi(p)$ rather than a constant one.

One of the main results of~\cite{Kes87} is that the scaling window is simply the set of $(p,R)$ such that $R=O(L(p))$. 
This result is sometimes referred to as {\em stability below the characteristic length}; it is the subject of the following theorem in the context of the random-cluster model.
Together with Theorem~\ref{thm:L_equiv_xi} the stability result legitimates the two physical interpretations of the correlation length: in terms of rate of decay and in terms of scaling window. We state the result for $q\ne1$ as the case $q=1$ was already treated in~\cite{Kes87}.

\begin{theorem}[Stability below the characteristic length]\label{thm:stability}
    Fix $1< q \le 4$. \medbreak\noindent
  {\em (Stability of crossing probabilities)} There exists $\eps > 0$ such that, for every  $\eta$-regular discrete quad $(\calD,a,b,c,d)$ at scale $R\ge1$ and every $p \in (0,1)$ (the constant in $\les$ depends on $\eta$ but $\eps$ does not), 
    \begin{align}\label{eq:stability crossings}
		|\phi_{p} [\calC(\calD)] - \phi_{p_c} [\calC(\calD)]|
		\les(\tfrac R{L(p)})^\eps.
    \end{align}
 \medbreak\noindent
  {\em (Stability of the one-arm event)}  For every $p \in (0,1)$ and every $R\le L(p)$, 
	\begin{align}\label{eq:stability arm exponent}
    	\pi_1(R)\asymp \pi_1(p,R).
	\end{align}
\end{theorem}
The stability of arm event probabilities~\eqref{eq:stability arm exponent} extends to more general arm events. Moreover, an improved version may be formulated; see Remark~\ref{rem:improved_stability} for details.

The strategy for proving Theorem~\ref{thm:stability} is related to Kesten's original one, in that it uses Theorem~\ref{thm:delta} to study the behaviour of derivatives of crossing events. 
Nevertheless, several additional difficulties occur, mainly due to the replacement of pivotality by influence in the differential formulas for probabilities of events: recall~\cite[Thm.~(2.46)]{Gri06} the general formula, valid for every $q>0$,
\begin{equation}\label{eq:Russ}
	\tfrac{{\rm d}}{{\rm d}p}\phi_{p}[\calC(\calD)]
	=\tfrac1{p(1-p)} \sum_{e \in \bbE} \mathrm{Cov}_p(\omega_e, \calC(\calD)), 
\end{equation}
where $\Cov_p$ denotes the covariance under $\phi_p$. 
For $q=1$, the sum of covariances gets nicely rewritten in terms of {\em pivotal edges}, i.e.~edges which, when switched from close to open, change the occurrence of the event. In particular, it is possible to prove that for crossing events of a rectangle of size $R$ and edges that are far from the boundary of the rectangle, the probability of being pivotal is of the order of the probability $\pi_4(p,R)$ that the two extremities of a given edge $e$ belong to different clusters of radius at least $R$ (when $p=p_c$, we simply write $\pi_4(R)$).
This property was used crucially in~\cite{Kes87} and ultimately leads to Kesten's scaling relation $L(p)^2 \pi_4(L(p)) \asymp (p-p_c)^{-1}$.
The description in terms of pivotal edges is wrong for random-cluster models with $q>1$ as the covariance between an edge and crossing events at scale $R$ is no longer of the order of $\pi_4(p,R)$. 

Driven by this different phenomenology, in this paper we introduce a {\em new interpretation} of the covariance valid for every $q>1$ encoding how much an edge is influenced by boundary conditions at a distance $R$, or equivalently, how fast the model mixes.
\begin{definition}[Mixing rate]
    For $1< q\le4$, $1 \le r < R$, $p\in(0,1)$ and $e$ an edge incident to the origin, write 
    \begin{align}
    	\Delta_p(R) &:= \phi_{\La_R,p}^1[\omega_e] -  \phi_{\La_R,p}^0[\omega_e],\\
    	\Delta_p(r,R) &:= \phi_{\La_R,p}^1[\calC(\La_r)] -  \phi_{\La_R,p}^0[\calC(\La_r)].
    \end{align}   
\end{definition}

The quantity $\Delta_p(R)$, to which we now refer as the {\em mixing rate}, will be crucial in our study, as it will replace the amplitude $\pi_4(p,R)$ of standard pivotal events in the study of Bernoulli percolation. As such, it is very important to derive some of its properties. 

\begin{theorem}[Properties of the mixing rate]\label{thm:delta}
    Fix $1<q\le 4$. 
    \medbreak\noindent
 	{\em (i)  (Mixing rate/covariance connection)} 
	For every $\eta>0$, every $p\in(0,1)$, every $\eta$-regular quad $(\calD,a,b,c,d)$ at scale $R\le L(p)$ containing $\Lambda_{\eta R}(e)$ for some edge $e$ (below the constants in $\asymp$ depend on $\eta$), 
    \begin{equation}\label{eq:delta_thm1}
    \Cov_{p}[\omega_e;\calC(\calD)]\asymp\Delta_p(R).
    \end{equation}
    \medbreak\noindent
 {\em (ii)   (Quasi-multiplicativity)} For every $p\in(0,1)$ and $1\le r\le R\le L(p)$,
    \begin{equation}\label{eq:delta_thm3}
    	 \Delta_p(r)\Delta_p(r,R)\asymp\Delta_p(R).
    \end{equation}
      \medbreak\noindent
{\em (iii)  (Stability below the characteristic length)} For every $p\in(0,1)$ and $1\le R\le L(p)$,
    \begin{equation}\label{eq:delta_thm5}
    	\Delta_p(R)\asymp\Delta_{p_c}(R).
    \end{equation}
 \medbreak\noindent
{\em (iv)   (Comparison to pivotality)} There exists $\ep>0$ such that for every $1\le R\le L(p)$,
    \begin{equation}\label{eq:delta_pi4_improvement}
    	  \Delta_{p}(r,R) \ges (R/r)^\ep\pi_4(R)/\pi_4(r).
    \end{equation}
 \medbreak\noindent
{\em (v)   (Mixing interpretation)} For every $1\le 2r\le R\le L(p)$,
    \begin{equation}\label{eq:delta_mixing_interpretation}
    	   \Delta_p(r,R)\asymp\max\Big\{\Big|\frac{\phi_p[A\cap B]}{\phi_p[A]\phi_p[B]}-1\Big|\ :\ A\in\mathcal F(\Lambda_r)\text{ and }B\in\mathcal F(\mathbb Z^2\setminus\Lambda_R)\Big\},
	 \end{equation}
	where $\mathcal F(S)$ is the $\sigma$-algebra generated by the edges with both endpoints in $S$.
\end{theorem}
 
The proof of this theorem is the main innovation of the paper. It is based on new increasing couplings between random-cluster models.
While coupling Bernoulli percolation at different parameters is fairly straightforward, coupling different random-cluster models can be quite elaborate. In this paper, we develop several increasing couplings between random-cluster models (typically one at $p_c$ and one at $p$, or one at $p_c$ and another one at $p_c$, but conditioned on an event) that satisfy various properties.

In the previous theorem, Properties (i)--(v) have crucial interpretations. 
Property (i) will have the following important application. It states that the covariance between an edge which is deep inside an $\eta$-regular quad and the crossing event of said quad is of the order of $\Delta_p(R)$.  
Property (ii) is an analog of the quasi-multiplicativity of probabilities of arm events and will be popping up everywhere in the applications of $\Delta_p(R)$, in particular when trying to estimate the covariance between a crossing event and an edge close to the boundary of the quad. Property (iii) states the stability below the characteristic length for the mixing rate, analogously to that proved by Kesten for the four-arm event probability. Property (iv) shows that replacing $\pi_4(p,R)$ by $\Delta_p(R)$ is really necessary, as the trivial bound stating that the covariance is larger than or equal to pivotality is polynomially far from being sharp for any $q>1$. Finally, (v) justifies the reference to mixing in the name of $\Delta_p$, as it links this quantity to the error term in the ratio-weak mixing.

We finish comments on this theorem by a crucial observation. When trying to compute asymptotics for the covariance between an edge and a crossing  event, (i) and (ii) imply that it suffices to understand for every $\ep>0$ the limit of $\Delta_{p_c}(\ep R,R)$ as $R$ tends to infinity. Indeed, these limits allow to estimate the covariance up to arbitrarily small polynomial terms and therefore to estimate the critical exponent. 
This is very useful as the covariance itself is not easily expressed in terms of properties of large interfaces of the critical system, while $\Delta_{p_c}(\ep R,R)$ (which is equal by duality to $1-2\phi_{\Lambda_R,p_c}^0[\calC(\Lambda_{\ep R})]$) is a quantity that can be derived from the scaling limit of the critical model, for instance using the conjectural convergence to the Conformal Loop Ensemble.

It is tempting to deduce from the previous theorem that when $q>1$ the derivative of crossing probabilities of $\eta$-regular quads at scale $R\le L(p)$ is of order $R^2\Delta_p(R)$ (exactly like it is of order $R^2\pi_4(R)$ for Bernoulli percolation). {\em This statement is actually wrong and illustrates the subtle but deep difference with Bernoulli percolation}. Indeed, there is a competition on the right-hand side of~\eqref{eq:Russ} between two possible scenarios:
\begin{itemize}
\item The collective contribution of edges in $\calD$ is the main part of the right-hand side. In such case, we expect the derivative at $p_c$ to exist and to be equal to $R^2\Delta_p(R)$. Moreover, it may be proved in this case that the derivative is stable within the critical window.
\item The collective contribution of edges far from $\calD$ is the main part of the right-hand side. In such case, the derivative at $p_c$ is infinite. For $p\ne p_c$, the contribution comes mostly from edges at distance $L(p)$, and the derivative is of order $L(p)^2\Delta_p(L(p))^2/\Delta_p(R)$.
\end{itemize}
An accurate estimate of the derivative, valid in all scenarios, is therefore given by the following statement.

\begin{corollary}\label{cor:sta}
	Fix $1<q\le 4$ and $\eta>0$.
	Every $p\in(0,1)$ and every $\eta$-regular quad $(\calD,a,b,c,d)$ at scale $R\le L(p)$,
	\begin{equation}\label{eq:delta_cor:sta}
    	\tfrac{{\rm d}}{{\rm d}p}\phi_{p}[\calC(\calD)]\asymp 
		 R^2\Delta_p(R)+\sum_{\ell=R}^{L(p)}\ell \Delta_p(\ell)\Delta_p(R,\ell),
    \end{equation}
    where the constants in $\asymp$ depend on $\eta$.
\end{corollary}
Looking at this sum formula for the derivative, one sees that whether the derivative is governed by the collective contribution of edges in or close to $\calD$ or by that of edges far from $\calD$ is related to whether $\ell\Delta_p(\ell)$ decays or not as $\ell$ tends to infinity.
 This can also be related to whether the specific heat blows up or not at $p_c$, as will be seen in the next section. Note that this up-to-constant formula unraveled a third possible scenario where each scale contributes the same amount. This scenario happens for the random-cluster model with $q=2$, in which the derivative blows up logarithmically in $L(p)$.

In order to derive this corollary, one will need an important result which is reminiscent of the classical claim that the four-arm exponent is strictly smaller than 2 for Bernoulli percolation.  
\begin{proposition}[Lower bound on $\Delta_{p}(r,R)$] \label{prop:lower bound Delta}
	There exists $\delta>0$ such that for every $1<q\le 4$ and $1\le r\le R\le L(p)$,
    \begin{equation}\label{eq:delta_thm4}
    	\Delta_{p}(r,R)\ge \delta(r/R)^{2-\delta}.
    \end{equation}
\end{proposition}
While the rest of the paper relies on fairly generic assumptions of the percolation model at hand, the previous proposition harvests a much more specific property of the random-cluster model on $\mathbb Z^2$, namely the parafermionic observable. For $1\le q<4$, the result will follow from  crossing estimates that were 
 recently obtained in~\cite{DumManTas20} using this observable. These crossing estimates are uniform in boundary conditions and in (possibly fractal) domains. The byproduct of the analysis in~\cite{DumManTas20} is that $\pi_4(R)/\pi_4(r)$ is bounded from below by $\delta(r/R)^{2-\delta}$, and therefore by (iv) so is $\Delta_p(r,R)$. For $q=4$, the crossing estimates are not uniform in boundary conditions and a more specific analysis, also based on the parafermionic observable, must be performed. It is the subject of Section~\ref{sec:q=4} in this paper.

\subsection{Scaling relations}\label{sec:intro_scaling_relations}

In continuous phase transitions,  natural observables of the model decay algebraically. 
The behaviour at and near criticality is thus expected to be encoded by various {\em critical exponents} $\alpha$, $\beta$, $\gamma$, $\delta$, $\eta$, $\nu$, $\zeta$, $\iota$, $\xi_1$ and $\xi_4$  defined as follows (below $o(1)$ denotes a quantity tending to 0):
\begin{align*}
f''(p)&=|p-p_c|^{-\alpha+o(1)}&\text{ as $p\rightarrow p_c$},\\
\theta(p)&=(p-p_c)^{\beta+o(1)}&\text{as $p\searrow  p_c$}, \\
\chi(p)&=|p-p_c|^{-\gamma+o(1)}&\text{as $p\rightarrow  p_c$,}\\
\phi^1_{p_c,h}[0\longleftrightarrow\ghost]&=h^{1/\delta+o(1)}&\text{as $h\rightarrow 0$,}\\
\phi_{p_c}[0\longleftrightarrow x]&=|x|^{-\eta+o(1)}&\text{as $|x|\rightarrow\infty$},\\
\pi_1(R)&=R^{-\xi_1+o(1)} &\text{ as $R\rightarrow\infty$},\\
\xi(p)&=|p-p_c|^{-\nu+o(1)}&\text{as $p\rightarrow  p_c$},\\
\phi_{p_c}[|\sfC|\ge n]&=n^{-\zeta+o(1)} &\text{as $n\rightarrow \infty$},\\
\Delta_{p_c}(R)&=R^{-\iota+o(1)} &\text{as } R\rightarrow\infty,\\
\pi_4(R)&=R^{-\xi_4+o(1)} &\text{ as $R\rightarrow\infty$},
\end{align*}
 where all the quantities above were already defined in previous sections, except that $0\longleftrightarrow  \mathfrak g$ is the event that 0 is connected to the ghost by an open path, $|\sfC|$ is the number of vertices in the cluster $\sfC$ of the origin, and $f(p)$ and $\chi(p)$ are the thermodynamical quantities respectively called the {\em free-energy}  and the {\em susceptibility} defined\footnote{The definition of $f(p)$  for $q=1$ is slightly different and is given by $f(p):=\phi_{p}[1/|\mathsf C|]$.}
  by
 \begin{align*}
 f(p)&:=\lim_{n\rightarrow\infty}-\tfrac1{|\Lambda_n|}\log Z^0(\Lambda_n,p),\\
\chi(p)&:=\phi_{p}[|\sfC| \ind_{|\sfC| < \infty}].
 \end{align*}
Let us mention that the first equation only applies when $f''(p)$ diverges as $p$ approaches $p_c$, which is to say that the phase transition is of second order. 
  
These exponents are quantities of central interest in physics and have been the object of many studies. A beautiful prediction is that these exponents should depend on each other via {\em scaling relations}:
\begin{align}
\eta&=2\xi_1,\tag{{\bf R1}}\label{eq:etaxi}\\
\label{eq:zetaxi}\zeta&=\xi_1/(2 - \xi_1)\tag{{\bf R2}},\\
  \delta &= (2 - \xi_1)/\xi_1,\label{eq:deltaxi}\tag{{\bf R3}}\\
 \beta&=\nu\xi_1  \tag{{\bf R4}},\label{R4}\\
   \gamma&=(2-2\xi_1)\nu, \tag{{\bf R5}}\label{R5}\\
   \alpha&=2-2\nu\tag{{\bf R6}}\label{eq:alphanu}.
\end{align}
An important feature of the relations above is that they are independent of $q$: the exponents vary from model to model, but not the formulae. Relations  {\bf R1--6} were proved for Bernoulli percolation  (i.e.~cluster-weight $q=1$) in a milestone paper by Harry Kesten~\cite{Kes87} without any of them being computed, or indeed even be shown to exist (see also~\cite{Nol08,GPS,DumManTas20b}). For the random-cluster with $q=2$, critical exponents were calculated independently~\cite{McCWu83,DumGarPet14,Dum13} and were observed to satisfy {\bf R1--6}. Let us mention that similar relations should hold in all dimensions that are below the so-called {\em upper-critical dimension} (with certain values of $2$ replaced by the dimension $d$). We refer to a paper by Borgs, Chayes, Kesten and Spencer~\cite{BCKS99} (see also~\cite{BCKS}) for a discussion of this phenomenon for Bernoulli percolation.

Kesten's analysis in the case of Bernoulli percolation was relying on another scaling relation, sometimes referred to as {\em Kesten's scaling relation}, stating that 
$\nu(2-\xi_4)=1$ for $q=1$.
 It was observed in~\cite{DumGarPet14} that this equality fails for $q=2$ (one can also check this using the table gathering the predicted exponents below). We will show that it should be replaced by the following {\em generalized Kesten's scaling relation}, 
\begin{equation}\label{eq:nu_iota}
\nu(2-\iota)=1\tag{{\bf R7}}.
\end{equation}
Note that the second property of Theorem~\ref{thm:delta} shows that $\xi_4>\iota$ so that $\nu(2-\xi_4)=1$ fails not only at $q=2$ but {\em for every} $q>1$.

Before discussing the main results, let us mention the predicted values for the different exponents. 
The first three  scaling relations enable us to express $\delta$, $\eta$ and $\zeta$ in terms of $\xi_1$ only. This is particularly interesting since $\xi_1$ is measurable in terms of the scaling limit of interfaces at criticality. 
The relations {\bf R6} and {\bf R7} link $\alpha$, $\nu$ and $\iota$. This is again very useful since it was noted in the previous section  how $\iota$ can be obtained from the understanding of the scaling limit of interfaces. An alternative approach to computing these three exponents would be to first obtain $\alpha$, which may perhaps be derived using exact integrability of the random-cluster model, see~\cite[Sec.~12.8]{Bax89} and Section~\ref{sec:1.5} for more details. 
Finally, {\bf R4} and {\bf R5} express $\beta$ and $\gamma$ in terms of $\xi_1$ and $\nu$, so that one can obtain all the exponents from $\xi_1$ and $\iota$ (or $\alpha$). 

Conformal invariance enables to predict that the scaling limit of the random-cluster model with cluster-weight $q\in[0,4]$ is related to  CLE$(\kappa)$ (see~\cite{SchSheWil09} and the discussion in~\cite{GarWu18}), from which $\xi_1$ and $\iota$ can be deduced. This leads to the following table, where all the exponents are expressed in terms of $\kappa$.

\begin{center}
\begin{tabular}{|c|c|c|c|c|c|c|}
\hline
{\bf Exponent} & {\bf Definition} &{\bf $q\in[0,4]$} & $q=1$ & $q=2$ & $q=3$ & $q=4$\\  \hline
%$\sigma$ & $\sigma(q):=\tfrac{2}{\pi}\arcsin(\tfrac{\sqrt q}2)$ & $\ $ & $\tfrac13$ & $\tfrac12$ & $\tfrac23$ & $1$\\ \hline
$\kappa$ & $\kappa(q):=4\pi/\arccos(-\tfrac{\sqrt q}2)$ & $\kappa$%$\frac{8}{1+\sigma} $
 & $6$ & $\tfrac{16}3$ & $\tfrac{24}5$ & $4$\\ \hline
 $\alpha$ & $f''(p)=|p-p_c|^{-\alpha+o(1)}$  & $\tfrac23\tfrac{16-3\kappa}{8-\kappa}$%$\frac{2(2\sigma-1)}{3\sigma} $ 
 &$-\tfrac23$ & 0 & $\tfrac13$ & $\tfrac23$\\ \hline
 $\beta$ & $\theta(p)=(p-p_c)^{\beta+o(1)}$ & $\tfrac{3\kappa-8}{12}$%$\frac{2-\sigma}{12}$ 
 & $\tfrac5{36}$ & $\tfrac18$ & $\tfrac19$& $\tfrac1{12}$\\ \hline
 $\gamma$ & $\chi(p)=|p_c-p|^{-\gamma+o(1)}$ & $\tfrac8\kappa-\tfrac23+\tfrac23\tfrac{\kappa}{8-\kappa}$%$\frac{\sigma^2+2\sigma+4}{6\sigma}$ 
 & $\tfrac{43}{18}$ & $\tfrac74$ & $\tfrac{13}9$& $\tfrac76$\\ \hline
 $\delta$ & $\theta(p_c,h)=h^{1/\delta+o(1)}$ & $\tfrac{(8+\kappa)(8+3\kappa)}{(8-\kappa)(3\kappa-8)}$%$\frac{(2+\sigma)(4+\sigma)}{\sigma(2-\sigma)}$ 
 & $\tfrac{91}5$ &$15$ &$14$ & $15$\\  \hline
 $\eta$ & $\phi^0_{p_c}[0\longleftrightarrow x]=|x|^{-\eta+o(1)}$ & $\tfrac{(8-\kappa)(3\kappa-8)}{16\kappa}$%$\frac{\sigma(2-\sigma)}{2(1+\sigma)}$
 & $\tfrac5{24}$ & $\tfrac14$ & $\tfrac4{15}$ & $\tfrac14$\\  \hline
 $\nu$ & $\xi(p)=|p_c-p|^{-\nu+o(1)}$ & $\tfrac8{3(8-\kappa)}$%$\frac{1+\sigma}{3\sigma}$ 
 & $\tfrac43$ & 1  & $\tfrac56$ & $\tfrac23$\\ \hline
  $\zeta$ & $\phi^0_{p_c}[|\sfC|\ge n]=n^{\zeta+o(1)}$ & $\tfrac{(8-\kappa)(3\kappa-8)}{(8+\kappa)(3\kappa+8)}$%$\frac{\sigma(2-\sigma)}{(2+\sigma)(4+\sigma)}$ 
  & $\frac5{91}$ & $\frac1{15}$  & $\frac1{14}$ & $\frac1{15}$\\ \hline
 $\xi_1$ & $\pi_1(R)=R^{-\xi_1+o(1)}$& $\tfrac{(8-\kappa)(3\kappa-8)}{32\kappa}$%$\frac{\sigma(2-\sigma)}{4(1+\sigma)}$
 & $\tfrac5{48}$ &$\tfrac18$ &  $\tfrac2{15}$  & $\tfrac18$\\ \hline
 $\xi_4$ &$ \pi_4(R)=R^{-\xi_4+o(1)}$& $-\tfrac{\kappa}{8}+4+\tfrac{6}{\kappa}$%$\frac{3\sigma^2+10\sigma+3}{4(1+\sigma)} $ 
 & $\tfrac54$& $\tfrac{39}{24}$& $\tfrac{33}{20}$& $2$ \\ \hline
$\iota$ & $ \Delta(R)=R^{-\iota+o(1)}$&  $\tfrac{3\kappa}{8}-1$%$\frac{2-\sigma}{1+\sigma} $ 
& $\ $& $1$& $\tfrac{4}{5}$& $\tfrac12$ \\ \hline
\end{tabular}
\end{center}

In this paper, we prove {\bf R1--7} for the random-cluster model with general cluster-weights~$q\in[1,4]$, except for~${\bf R6}$ when~$\alpha$ is negative. We insist on the fact that the random-cluster models belong to different universality classes when~$q$ varies from~$q=1$ to~$q=4$, so that this paper provides the first generic derivation of these relations for different universality classes. As in~\cite{Kes87}, we do not claim to show that any of these exponents exist, nor do we compute their values; the actual statements of the scaling relations with no reference to the exponents are given in the three theorems below. Note nonetheless that if one makes the assumption of algebraic decay with the proper exponent, the statements below imply the scaling relations mentioned above.

Below, we assume that~$1<q\le 4$ as the case~$q=1$ is already known. We start by the two simplest scaling relations~\eqref{eq:etaxi} and~\eqref{eq:zetaxi} involving only quantities at~$p=p_c$ and~$h=0$. The theorem is an easy consequence of uniform crossing estimates obtained for the random-cluster at criticality, see e.g.~\cite{DumSidTas16}. While the result is not especially complicated, we chose to include it here for completeness. Introduce the following quantity for every~$n>0$,
\begin{align}\label{eq:phi}
	\varphi(n):=\min\{r \in \bbN :\, r^2\pi_1(r)\ge n\}.
\end{align}

\begin{theorem}[Scaling relations at criticality]\label{thm:scaling_rel_crit}
	Fix~$1 < q\le 4$. For~$x\in\bbZ^2$ and~$n\ge1$,
	\begin{align}\label{eq:relation2}
		\phi_{p_c}[0\longleftrightarrow x]&\asymp \pi_1(|x|)^2,\\
		\label{eq:relation3}
		\phi_{p_c}[|\sfC|\ge n]&\asymp \pi_1(\varphi(n)).
	\end{align}
\end{theorem}

We now turn to the scaling relation~\eqref{eq:deltaxi} involving the magnetic field. For~$q=2$,~\eqref{eq:deltaxi} was proved in~\cite{CGN14} using the GKS inequality but this inequality is not available for general random-cluster models. A fact which came as a surprise to us is that~\eqref{eq:deltaxi}  can be derived for every~$1\le q\le 4$ without referring to any other result of the paper (see Section~\ref{sec:14}). 

\begin{theorem}[Scaling relation with magnetic field]\label{thm:scaling_rel_h}
	Fix~$1 < q\le 4$. For~$h>0$,	
	\begin{align}
		\phi_{p_c,h}[0\longleftrightarrow \mathfrak g]&\asymp \pi_1(\varphi(\tfrac1h)).
	\end{align}
\end{theorem}

The scaling relations~\eqref{R4}--\eqref{eq:nu_iota} are the most difficult ones as they involve
 the random-cluster model at~$p$ near~$p_c$ and rely heavily on the stability in the near-critical regime. 
 
 \begin{theorem}[Scaling relations near-critical regime]\label{thm:scaling_rel_nc}
    Fix~$1< q \le 4$. For~$p>p_c$ (and~$p\ne p_c$ for the second),
    \begin{align}
	    \theta(p)&\asymp\pi_1(\xi(p)), \label{eq:main_theta}\\
        \chi(p)&\asymp\xi(p)^2\pi_1(\xi(p))^2,\label{eq:main_chi}\\
		\Delta_{p_c}(\xi(p))&\asymp \xi(p)^{-2}|p-p_c|^{-1},\label{eq:L_Delta} \\
		f''(p)&\asymp\sum_{\ell\le \xi(p)}\ell\Delta_{p_c}(\ell)^2\label{eq:lp2}.
	\end{align}
\end{theorem}
 
Note that assuming that~$\iota$ exists, we get a very different behaviour depending on whether it is smaller or larger than~$1$, or correspondingly whether~$\alpha$ is positive or negative, i.e.~whether the random-cluster model undergoes a second-order or higher-order phase transition. The former occurs when~$\nu<1$, i.e.~conjecturally when~$q\in(2,4]$. When~$\nu=1$, which is conjectured to correspond only to~$q=2$, all the exponents are known and~$f''(p)$ blows up logarithmically (in particular it satisfies the scaling relation as well). 
When~$\nu>1$,~$f''(p)$ remains bounded in the vicinity of~$p_c$, and the phase transition becomes of third-order (or higher). 
The exponent~$\alpha$ may still be defined using the third derivative of~$f$, which is supposed to diverge at~$p_c$. 
We are currently only able to derive an upper bound on~$f'''$; the lower bound is unavailable even for Bernoulli percolation. 
We refer to Remark~\ref{rem:alpha_negative} for details.

\subsection{Two complementary results on the order parameter of Potts models}

This section gathers two satellite results that are of interest on their own and that do not necessarily fit in the storyline of the previous sections.  For~$q=2$, it is already known that  the value of~$\xi_1$ is equal to~$1/8$, so that our paper provides a new proof of the following immediate corollary using the Edwards-Sokal coupling~\cite[Section~1.4]{Gri06}. 
We include it since the Ising model with magnetic field, contrary to the case~$h= 0$, is not integrable and hence notoriously difficult to study. As mentioned previously, it was obtained in \cite{CGN14} using alternative arguments.
\begin{theorem}\label{thm:h}
Let ~$m(\beta,h)$ be the spontaneous magnetization of the Ising model on~$\bbZ^2$ at inverse-temperature~$\beta$ and magnetic field~$h$. For every~$h\in(0,1)$,
$$m(\beta_c,h)\asymp h^{1/15}.$$
\end{theorem}

When~$1\le q\le 3$, it was proved in~\cite{DumManTas20} that~$\pi_1(R)\pi_4(R)\ge cR^{c-2}$ for every~$R\ge1$ and some constant~$c > 0$. 
In Remark~\ref{rmk:lower bound Delta q=4}, we show that~$\pi_1(R)\Delta(R)\ge cR^{c-2}$  also for~$q=4$. 
From these inequalities, using~\eqref{eq:main_theta},~\eqref{eq:L_Delta} and~\eqref{eq:delta_thm4}, one may deduce the result below, which should be understood as~$\beta < 1$ for~$1\le q\le 3$ and~$q=4$. The result for~$q = 1$ (that is for Bernoulli percolation) was already obtained by Kesten and Zhang~\cite{KesZha}; we expect~$\beta < 1$ to be valid for all~$1 \le q \le 4$.

\begin{theorem}[non-differentiability of the order parameter]\label{prop:beta}
For every~$1\le q\le 3$ or~$q=4$,  there exists~$c>0$ such that for every~$p\ge p_c$,
    \begin{align}
\theta(p)&\ge c(p-p_c)^{1-c}.\label{eq:bbbc}\end{align}
In particular, the spontaneous magnetization~$m(\beta)$ of the 2, 3 and 4-state Potts model satisfies ~$m(\beta)\ge c(\beta-\beta_c)^{1-c}$ for~$\beta\ge \beta_c$.
\end{theorem}

\subsection{Open questions}\label{sec:1.5}

The present paper opens many doors in the study of the critical regime of random-cluster models (and more generally planar dependent percolation models). We now mention a few open questions that in our opinion deserve attention. We refrain ourselves from asking the obvious question of proving conformal invariance of the model, and focus on questions that are directly related to the current work.

 Let us start by a question concerning scaling relations, namely whether one can prove~\eqref{eq:alphanu} when~$q\le 2$. As mentioned above, in this case~$f''(p)$ is expected to remain bounded when~$p$ tends to~$p_c$, but one may consider the behaviour of~$f'''(p)$ to make sense of~$\alpha$. Remark~\ref{rem:alpha_negative} of the present paper shows that the critical exponent~$\alpha$ defined like this satisfies~$2\nu\le 2-\alpha$, leaving the following question open (note that this is also open for~$q=1$).
\begin{question}
	Prove that for every~$1\le q<2$, one has~$2\nu\ge 2-\alpha$.
\end{question}

Another natural question is to derive critical exponents for random-cluster models. The scaling relations enable one to deduce certain exponents from others, and we may therefore choose which exponents to try to derive. From this point of view, the exponent~$\alpha$ is particularly tempting since it implies directly~$\nu$, and also since exact integrability often provides physicists and mathematicians with closed formulae that may lead to~$\alpha$. We refer to~\cite{Bax89} for more details on this and summarize the discussion in the following question.
\begin{question}
Obtain~$\alpha$ using exact integrability to understand the near-critical behaviour of the free energy.
\end{question}
Another approach consists in deriving the exponents having as a basis the assumption of conformal invariance. In this case, we know that the scaling limit of the family of boundaries of clusters should be a
Conformal Loop Ensemble as mentioned in the introduction. As a consequence,~$\xi_1$ seems very easy to deduce from conformal invariance. Note that~$\xi_1$ and~$\iota$ are sufficient to derive the other exponents, and that~$\iota$ has the advantage of being a quantity which is computable using the scaling limit at criticality (the exponents~$\beta$,~$\gamma$,~$\delta$, and~$\nu$ involve values of~$p\ne p_c$ and should therefore be difficult to compute directly using only conformal invariance).  At the light of the quasi-multiplicativity property of~$\Delta_{p_c}(r,R)$, the following question seems  tractable.

\begin{question}
	Compute~$\iota$ assuming conformal invariance of interfaces at criticality.
\end{question}

 Let us finish this section by mentioning that~\cite{DumGarPet14} emphasizes a self-organized mechanism in the way new edges occur as~$p$ increases in Grimmett's monotone coupling (see~\cite{Gri06} for details). The authors argued that edges appear in clouds  and that the understanding of these clouds would be crucial towards the construction of the near-critical scaling limit, would anybody manage to construct the conformally invariant scaling limit at~$p_c$. The current work answers a number of questions and conjectures asked in this paper (including Conjecture 4.1 and 4.2 since~$\xi(p)$ is explicitly known, see e.g.~\cite{BefDum12b} and references therein), but does not provide direct insight on the structure of these clouds. We therefore conclude with the following question.
\begin{question}
What does the present work tell us about clouds (in the sense of~\cite{DumGarPet14}) in Grimmett's monotone coupling?
\end{question}

Almost everything is known about the random-cluster model with~$q=2$ on the square lattice, since the conformal invariance of the model and its interfaces was proved~\cite{Smi10,CheDumHon12a}. It is therefore only natural to discuss the question of the construction of the near-critical scaling limit in this context, especially since one expects subtle differences with the corresponding result for Bernoulli percolation (see~\cite{GPS}).

\begin{question}
	Construct the near-critical scaling limit of the model,
	i.e.~the limits of random-cluster models on~$\tfrac1R\mathbb Z^2$ at~$p$ such that~$R$ is of order of $\lambda L(p)$, 
	where~$\lambda$ is a fixed strictly positive parameter. One may start by studying the case of~$q=2$.
\end{question}

Recently, rotational invariance of the critical random-cluster model was obtained in~\cite{DKMO20}. 
This rotational invariance is expected to carry over to the near-critical regime. 
The arguments developed here, combined with those of~\cite{DumGarPet14}, should be relevant for the next question.

\begin{question}
Prove that the near-critical scaling limit of the model is invariant under rotations.
\end{question}

\subsection*{Organisation of the paper}

Section~\ref{sec:2} provides the necessary background to our paper. 
Section~\ref{sec:9} studies the dependency of crossing probabilities on boundary conditions (see Theorem~\ref{thm:boosting_pair}) 
and introduces the notion of boosting pair of boundary conditions. 
Section~\ref{sec:10} contains the proof of points~(ii),~(iv) and~(v) of Theorem~\ref{thm:delta}. This is the core of our paper, and indeed its biggest innovation. 
Section~\ref{sec:a} initiates the connection between the quantity~$\Delta_p$ and covariances, in particular proving Theorem~\ref{thm:delta}(i). 
Section~\ref{sec:b} provides the lower bound on~$\Delta_p$ given by Proposition~\ref{prop:lower bound Delta}. 
Section~\ref{sec:7} contains the proof of the stability below the correlation length: Theorem~\ref{thm:stability} and Theorem~\ref{thm:delta}(iii). 
Finally, Section~\ref{sec:all_scaling_relations} contains the derivation of the scaling relations.

\subsection*{A word about constants}
We will often work with~$p\in(0,1)$ and a spatial scale~$R \leq L(p)$. 
Unless stated otherwise, constants~$c, (c_i)_{i\ge 0}, C$ and~$(C_i)_{i\ge 0}$ are assumed uniform in~$(p,R)$ as above, with the assumption that~$p$ is not close to~$0$ or~$1$. They are, however, allowed to depend on the threshold~$\delta$ used in the definition of~$L(p)$; recall that this threshold is assumed small, but fixed. 
We do not discuss the dependence in~$q$ of constants, but the careful reader will notice that they may be rendered uniform in~$q$, potentially outside of the vicinity of~$1$ and~$4$. 

We reiterate that the constants in the notation~$\asymp$,~$\les$ and~$\ges$ also follow the same principle.

\section{Preliminaries}\label{sec:background}\label{sec:2}

This section briefly recalls some tools for the study of the planar random-cluster model.
Some sections are new, for instance Section~\ref{sec:11}. We recommend that the readers quickly browse through this section, even if they are already comfortable with the basics of the random-cluster model.

\subsection{Elementary properties of the random-cluster model}\label{sec:1.7}

We will use standard properties of the random-cluster model. They can be found in~\cite{Gri06}, and we only recall them briefly below. 
Fix a subgraph~$G=(V,E)$ of~$\bbZ^2$. 

\bigbreak\noindent
{\em Monotonic properties.} 
An event~$A$ is called {\em increasing} if for any~$\omega\le\omega'$ (for the partial ordering on~$\{0,1\}^E$),~$\omega\in A$  implies that~$\omega'\in A$.
Fix~$q\ge1$,~$1\ge p'\ge p\ge0$,~$h'\ge h\ge0$, and some boundary conditions~$\xi'\ge\xi$, where~$\xi'\ge\xi$ means that any wired vertices in~$\xi$ are also wired in~$\xi'$. Then, for every  increasing events~$A$ and~$B$,
\begin{align*}\tag{FKG}\label{eq:FKG} 
	\phi_{G,p,h}^\xi[A\cap B]&\ge \phi_{G,p,h}^\xi[A]\phi_{G,p,h}^{\xi}[B],\\
	\tag{$h$-MON}\label{eq:hmon} 
	\phi_{G,p,h'}^\xi[A]&\ge \phi_{G,p,h}^\xi[A],\\
	\tag{$p$-MON}\label{eq:pmon} 
	\phi_{G,p'}^\xi[A]&\ge \phi_{G,p}^\xi[A],\\
	\tag{CBC}\label{eq:CBC} 
	\phi_{G,p,h}^{\xi'}[A]&\ge \phi_{G,p,h}^\xi[A].
\end{align*}

The inequalities above will respectively be referred to as the {\em FKG inequality}, the {\em monotonicity in}~$h$ and~$p$, and the {\em comparison between boundary conditions}. 

\bigbreak\noindent{\em Spatial Markov property.} For any configuration~$\omega'\in\{0,1\}^E$ and any~$F\subset E$,
\begin{equation}\tag{SMP}\label{eq:SMP} 
	\phi_{G,p,h}^\xi[\cdot_{|F}\,|\,\omega_e=\omega'_e,\forall e\notin F]\ge \phi_{H,p,h}^{\xi'}[\cdot],
\end{equation}
where~$H$ denotes the graph induced by the edge-set~$F$, and~$\xi'$ the boundary conditions on~$H$ defined as follows: 
$x$ and~$y$ on~$\partial H$ are wired if they are connected in~$(\omega')_{|E\setminus F}^\xi$. 

\bigbreak\noindent
{\em Dual model.} Define the
dual graph~$G^*=(V^*,E^*)$ of~$G$ in the usual way: 
place dual sites at the centers of the faces of~$G$ (when considering a graph on the plane, 
the external face must be counted as a face of the graph), and for every bond~$e\in E$, place a dual bond
between the two dual sites corresponding to faces bordering~$e$. Given a
subgraph configuration~$\omega$, construct a configuration~$\omega^*$ on~$G^*$ by
declaring any bond of the dual graph to be open (resp.\ closed) if the
corresponding bond of the primal lattice is closed (resp.\ open) for the
initial configuration. The new configuration is called the \emph{dual
  configuration} of~$\omega$.
The dual model on the dual graph given by the dual configurations then
corresponds to a random-cluster measure 
with the same parameter~$q$, a dual parameter~$p^*$ satisfying~$$\frac{p^*p}{(1-p^*)(1-p)}=q,$$ 
and dual boundary conditions. We do not want to discuss too much the details of how dual boundary conditions are defined (we refer to~\cite{Gri06}) and simply mention that the dual of free boundary conditions are the wired ones, and vice versa. Note that the critical point is self-dual in the sense that~$p_c^*=p_c$.

\bigbreak
\noindent
{\em Loop model.} The loop representation of a configuration on~$G$ is supported on the {\em medial graph} of~$G$ defined as follows. 
Let~$(\bbZ^2)^\diamond$ be the {\em medial} lattice, with vertex-set given by the midpoints of edges of~$\bbZ^2$ and edges between pairs of nearest vertices (i.e.~vertices at a distance~$\sqrt 2/2$ of each other). It is a rotated and rescaled version of~$\bbZ^2$.  
For future reference, note that the faces of ~$(\bbZ^2)^\diamond$ contain either a vertex of~$\bbZ^2$ or one of~$(\bbZ^2)^*$.
The edges of the medial lattice~$(\bbZ^2)^\diamond$ are considered oriented in counterclockwise direction around each face containing a vertex of~$\bbZ^2$. 
Let~$G^\diamond$ be the subgraph of~$(\bbZ^2)^\diamond$ spanned by the edges of~$(\bbZ^2)^\diamond$ adjacent to a face corresponding to a vertex of~$G$. 

Let~$\omega$ be a configuration on~$G$. 
Draw self-avoiding paths on~$G^\diamond$ as follows: a path arriving at a vertex of the medial lattice 
always takes a~$\pm \pi/2$ turn at vertices so as not to cross the edges of~$\omega$ or~$\omega^*$ (see Figure~\ref{fig:domain}). 
The loop configuration thus defined is formed of possibly several paths going from boundary to boundary, as well as disjoint loops; 
together these form a partition of the edges of~$G^\diamond$. 

\subsection{Crossing and arm events probabilities below the characteristic length}\label{sec:crossing estimates}

As it is often the case when investigating the critical behaviour of lattice models, we will need to use crossing estimates in rectangles and more generally in quads, as well as estimates on certain universal and non-universal critical exponents. Such crossing estimates initially emerged in the study of Bernoulli percolation in the late seventies under the coined name of Russo-Seymour-Welsh theory~\cite{Rus78,SeyWel78}. 

The main technical tool that we will use is the following result on crossing estimates and arm events probabilities.

\begin{theorem}[Crossing estimates below the characteristic length]\label{thm:RSWnear}
	For~$\rho,\ep>0$, there exist~$c'>0$ and~$c=c(\rho,\ep)>0$ 
	such that for every~$p$,  every~$1 \le n\le L(p)$, every graph~$G$ containing~$[-\ep n,(\rho+\ep)n]\times[-\ep n,(1+\ep)n]$ and every boundary conditions~$\xi$,
    \begin{equation}\label{eq:RSWnear}\tag{RSW}	
    	c\le \phi_{G,p}^\xi[\calC([0,\rho n]\times[0,n])]\le 1-c.
    \end{equation}
    Moreover, if~$A_n$ denotes the event that there exists an open circuit surrounding 
   $\La_n$ in ${\rm Ann}(n,2n)$, 
    \begin{equation}\label{eq:RSWA}\tag{RSW'}
		\phi_{{\rm Ann}(2n,n),p}^0[A_n]\ge c' > 0.
	\end{equation}
\end{theorem}

Since the result is not formally proved anywhere, we include it here. It basically consists in gathering different known results.
\begin{proof}
We start with~\eqref{eq:RSWnear}. By duality and comparison between boundary conditions \eqref{eq:CBC}, it suffices to show that for~$p<p_c$ and~$\xi=0$, we have that 
\[
\phi_{\overline R,p}^0[\calC(R)]\ge c,
\]
where~$R:=[0,\rho n]\times[0,n]$ and~$\overline R=[-\ep n,(\rho+\ep)n]\times[-\ep n,(1+\ep)n]$. 

The RSW theorem extracted from~\cite{DumTas18} gives the existence of~$C=C(\rho)>0$ such that for every~$n$ and~$p$,
\begin{equation}
	\phi_{p}[\calC(R)]\ge \tfrac1C  \phi_{p} [ \calC(\Lambda_n)]^C\ge \tfrac1C\delta^C ,
\end{equation} 
where in the second inequality we used the definition of~$L(p)$ and the fact that~$n\le L(p)$.

Consider the event~$\calE^*$ that there exists a dual-open circuit in the annulus~$\overline R\setminus R$ surrounding~$R$.
Then~\eqref{eq:CBC},~\eqref{eq:pmon} and the fact that~$\calE^*$ is decreasing imply that ~$\phi_p[\calE^*|\calC(R)] \geq \phi_{\overline R\setminus R, p_c}^1[\calE^*]$.
The result of~\cite{DumSidTas16} states that the latter probability is bounded from below by~$c_0=c_0(\rho,\eps)>0$ independently of~$n$.
The spatial Markov property and the comparison between boundary conditions allow us to conclude that 
\begin{equation}
	\phi_{\overline R,p}^0[\calC(R)]\ge\phi_{p}[\calC(R)|\calE^*]\ge \phi_{p}[\calE^*|\calC(R)]\phi_{p}[\calC(R)]\ge \frac{c_0}C  \delta^C.
	\end{equation} 
	This concludes the proof of~\eqref{eq:RSWnear}. For~\eqref{eq:RSWA}, use the FKG inequality and the fact that there exists a circuit in~${\rm Ann}(n,2n)$ surrounding the origin if the rectangle~$[-\tfrac53n,\tfrac53n]\times[\tfrac43n,\tfrac53n]$ as well as its rotations by angles~$\tfrac\pi2$,~$\pi$ and~$\tfrac{3\pi}2$ are all crossed in the long direction.
\end{proof}

The previous theorem has classical applications for the probability of so-called arm events. 
A self-avoiding path of type~$0$ or~$1$ connecting the inner to the outer boundary of an annulus is called an {\em arm}. We say that an arm is {\em of type~$1$} if it is composed of primal edges that are all open, and {\em of type~$0$} if it is composed of dual edges that are all dual-open.
For~$k\ge1$ and~$\sigma\in\{0,1\}^k$\,, define~$A_\sigma(r,R)$ to be the event that there exist~$k$ \emph{disjoint} arms from the inner to the outer boundary of~$\Ann(r,R)$ which are of type~$\sigma_1,\dots, \sigma_k$, when indexed in counterclockwise order. 

To simplify the notation, we introduce~$\pi_\sigma(p,r,R)$ for the~$\phi_{p}$-probability of~$A_\sigma(r,R)$. We drop~$p$ or~$r$ from the notation when~$p=p_c$ or~$r$ is the smallest integer such that~$\pi_\sigma(p,r,R)>0$ for all~$R>r$. Finally, when~$\sigma=1010\dots$ has length~$k$, we write the subscript~$k$ instead of~$\sigma$. 
For every~$\sigma$,~$ \pi_\sigma(R)$ decays algebraically with~$R$~\cite{DumSidTas16} and the scale invariance prediction suggests the existence of a critical exponent~$\xi_\sigma$ such that~$
\pi_\sigma(R)=R^{-\xi_\sigma+o(1)}$ as~$R$ tends to infinity. 

We also introduce~$A_{\sigma}^+(r,R)$ to be the same event as~$A_\sigma(r,R)$, except that the paths must lie in the upper half-plane~$\mathbb H:=\mathbb Z\times\mathbb Z_+$ and are indexed starting from the right-most. Introduce  its probability~$\pi_\sigma^+(p,r,R)$ and the associated exponent~$\xi_\sigma^+$.

We will need the following near-critical estimates on certain arm event probabilities.

\begin{proposition}[estimates on certain arm events]
	Fix~$1\le q\le 4$. There exists~$c > 0$ such that, for~$p\in(0,1)$ and~$1\le r\le R\le L(p)$,
    \begin{align}
        \pi_2(p,r,R)&\ges (r/R)^{1-c},\label{eq:TWO ARM}\\
        (r/R)^{1/2-c}\les \pi_1(p,r,R)&\les (r/R)^c.\label{eq:LOWER_BOUND_ONE_ARM}
    \end{align} 
\end{proposition}

\begin{proof}
The bound \eqref{eq:TWO ARM} follows from the fractal structure of interfaces, which in turn follows from Theorem~\ref{thm:RSWnear} and~\cite[Theorem~\ref{thm:RSWnear}]{AizBur99}. 
The argument is classical and we omit it.

The inequality on the right of \eqref{eq:LOWER_BOUND_ONE_ARM} is standard. 
The one on the left follows readily from \eqref{eq:TWO ARM} and the FKG inequality. 
\end{proof}

\begin{proposition}[quasi-multiplicativity of the one-arm]
  Fix ~$1\le q\le 4$. For every~$1\le r\le \rho\le R\le L(p)$,
  \begin{equation}\label{eq:MULTIPLICATIVITY}
\pi_1(p,r,\rho)\pi_1(p,\rho, R)\asymp \pi_1(p,r,R).
  \end{equation}
\end{proposition}

\begin{proof}
This is a standard consequence of Theorem~\ref{thm:RSWnear}.\end{proof}

\subsection{Couplings via exploration}\label{sec:coupling exploration}\label{sec:7.1}
 
In this section we present a technique for coupling different random-cluster measures in an increasing fashion by exploring the graph edge by edge, which we formalise using decision trees as follows. 
 Consider a graph~$G=(V,E)$ with~$n$ edges and~$U=(U_{e})_{e \in E}$ a family of independent uniform random variables in~$[0,1]$.
For a~$n$-tuple~$e=(e_1,\dots,e_n)$ of edges and for~$t\le n$, 
write~$e_{[t]}=(e_1,\dots,e_t)$ (with the convention~$e_{[0]}=\emptyset$) and~$U_{{[t]}}=(U_{e_1},\dots,U_{e_t})$. 

\begin{definition}[decision tree, stopping time]~\\
	A {\em decision tree} is a pair~${\bf T}=(e_1, (\psi_t)_{2 \le  t \le n})$, where~$e_1 \in E$, 
	and for each~$2 \le t\le n$ the function~$\psi_t$ takes a pair~$(e_{[t-1]}, U_{{[t-1]}})$ as an input 
	and returns an element~$e_t\in E\setminus\{e_1,\dots,e_{t-1}\}$. 
A {\em stopping time} for ~${\bf T}$ is a random variable~$\tau$ taking values in~$\{1,\dots,n, \infty\}$ 
	which is such that~$\{\tau \le t\}$ is measurable in terms of~$(e_{[t]},U_{{[t]}})$.
\end{definition}
We will say that the decision tree {\em reveals} one by one the edges of~$E$; 
the edges~$e_{[t]}$ are the edges {\em explored} at time~$t$. 
Less formally, a decision tree takes~$U~$ as an input and reveals edges one  after the other.  
It always starts from the same fixed~$e_1\in E$ (which corresponds to the root of the decision tree), 
then queries the value of~$U_{e_1}$. 
After that, it continues inductively as follows: 
at step~$t>1$,
the function~$\psi_t$, which should be interpreted as the decision rule at time~$t$, takes the locations and the values of the explored edges at time~$t-1$, and decides of the next edge to reveal. 

\begin{remark}
The theory of (random) decision trees played a key role in computer science  (we refer the reader to the survey~\cite{buhrman2002complexity}), but also found many applications in other fields of mathematics. In particular, random decision trees (sometimes called randomized algorithms) were used in~\cite{schrammsteif} to study the noise sensitivity of Boolean functions, for instance in the context of percolation theory. It was also used in~\cite{DumRaoTas17} in combination with the OSSS inequality (which was originally introduced in~\cite{OSSS}) to prove sharpness of random-cluster models.
\end{remark}

Decision trees may be used to construct random-cluster measures in a step-by-step fashion. This technique is generic and may be applied to so-called monotonic measures (see e.g.~\cite{Gri06}).
A key feature of this construction is that it enables one to do it with two (or more) random-cluster measures simultaneously. In this case, the decision tree produces couplings of these measures. Since we are mostly interested in couplings, we directly explain the construction for a pair of configurations. For~$1\le t\le n$, we extend the notation~$e_{[t]}$ and~$U_{{[t]}}$ with the notation~$\omega_{{[t]}}=(\omega_{e_1},\dots,\omega_{e_t})$. Below, we use the notation~$G_t$ for the graph~$G$ minus the edges~$e_1,\dots,e_t$.

\begin{proposition}\label{prop:coupling}
	Fix a finite subgraph~$G = (V,E)$ of~$\bbZ^2$. Consider~$0 \le p\le p'\le 1$,~$q\geq1$ 
	and ~$\xi\le\xi'$ boundary conditions. 
	Let~$\mathbf T=(e_1, (\psi_t)_{2 \le  t \le n})$ be a decision tree 
	and~$(U_e)_{e \in E}$ be a set of i.i.d.~uniform random variables under some measure~$\bbP_{\bf T}$.
	Define~$\omega,\omega' \in \{0,1\}^E$ by the following inductive procedure: for every~$0\le t<n$,
\begin{align*}
    \omega_{e_{t+1}} 
    &:= \ind(U_{e_t}\ge \phi_{G_t,p}^{\xi_t}[e_t \text{ closed}]);\\
   \omega'_{e_{t+1}} &:= \ind(U_{e_t}\ge \phi_{G_t,p}^{\xi_t'}[e_t \text{ closed}]),
   \end{align*}
	where~$\xi_t$ and~$\xi'_t$ are the boundary conditions induced by~$\omega_{{[t]}}^\xi$ and~$(\omega'_{{[t]}})^{\xi'}$, respectively (when~$t=0$, these are~$\xi$ and~$\xi'$). Then,~$\mathbb P_{\bf T}$-almost surely, for every stopping time~$\tau$ for~$\mathbf T$, we have that
	\begin{itemize}[noitemsep]
	\item~$\omega_{{[\tau]}}\le \omega'_{{[\tau]}}$,
	\item 	conditionally on~$(\tau,\omega_{{[\tau]}},\omega'_{{[\tau]}})$, 
	$\omega$ and~$\omega'$ on~$G_\tau$ have law~$\phi_{G_\tau, p}^{\xi_{\tau}}$ 
	and~$\phi_{G_\tau, p'}^{\xi'_{\tau}}$. 
	\end{itemize}
\end{proposition}

Note that for~$\tau=0$, we obtain that~$\omega$ and~$\omega'$ have laws~$\phi_{G,p}^{\xi}$ and~$\phi_{G,p'}^{\xi'}$, respectively, and that~$\omega \le \omega'$ a.s.. The procedure above may be applied to infinite-volume measures as long as~$\bf T$ is such that a.s. all edges are eventually queried.

\begin{proof}
	That~$\omega_{{[\tau]}}\le\omega'_{{[\tau]}}$ is proved by induction and uses the monotonic property of random-cluster measures  mentioned in Section~\ref{sec:1.7}. 
	That~$\omega_{{[\tau]}}$ and~$\omega'_{{[\tau]}}$ have the right laws 
	follows immediately from the spatial Markov property \eqref{eq:SMP}.
\end{proof}

\begin{remark}While the definition indicates that~$\mathbf T$ looks at~$U_{{[t]}}$ in order to decide the next queried edge~$e_{t+1}$ (and hence~$U_{e_{t+1}}$), we will often describe~$\mathbf T$ as choosing~$e_{t+1}$ in function of~$\omega_{{[t]}}$ and~$\omega'_{{[t]}}$, which are in turn functions of~$( e_{[t]}, U_{{[t]}})$. 
\end{remark}

\begin{remark}\label{rem:T_stopping_time}
Due to Proposition~\ref{prop:coupling}, we may construct an increasing coupling between ~$\phi_{G, p}^{\xi}$ and~$\phi_{G, p'}^{\xi'}$ by switching between decision trees at stopping times. Indeed, if we start the coupling by following a decision tree~$\mathbf T$, but stop the procedure at some stopping time~$\tau$, 
then we may complete it with any increasing coupling of~$\phi_{G_\tau, p}^{\xi_\tau}$ and~$\phi_{G_\tau, p'}^{\xi'_\tau}$.
We will often use this property, sometimes continuing with a specific coupling, other times with an arbitrary one.  
\end{remark}

We now discuss a few examples of decision trees and the couplings they produce.

\paragraph{Example 1} The deterministic decision tree~$\mathbf T$ for which the order~$e_1,\dots,e_n$ is fixed. 

\paragraph{Example 2} 
The decision tree~$\mathbf T$ that explores the clusters of~$\partial G$ in~$\omega'$. 
Formally, this decision tree is defined using a growing sequence~$\partial G=V_0\subset V_1\subset \cdots\subset V$ that represents the sets of vertices that the decision tree found to be connected to~$\partial G$ at time~$t$.

Fix an arbitrary ordering of the edges in~$E$ and set~$V_0=\partial G$. 
Now, for~$t \geq 0$, assume that~$e_{[t]}$ and~$V_t\subset V$ have been constructed and distinguish between two cases:
\begin{itemize}[noitemsep,nolistsep]
\item If there exists an unexplored edge connecting a vertex~$x \in V_t$ to a vertex~$y \notin V_t$, then 
reveal ~$e_{t+1}= xy$ (if several choices for~$xy$ exists, choose one according to some arbitrary order) and set
~$$V_{t+1}:=\begin{cases}V_t\cup\{y\}&\text{ if }\omega'_{e_{t+1}}=1\\ V_t&\text{ otherwise}.\end{cases}$$
\item If no edge as above exists, then set~$ e_{t+1}$ to be the smallest~$e\in E\setminus e_{[t]}$ for some arbitrary order and set~$V_{t+1}=V_t$.
\end{itemize}

The coupling~$\mathbb P_{\mathbf T}$ has the following useful property when~$p=p'$. 
If~$\tau$ denotes the first time that the decision tree finds no unexplored edge between~$V_t$ and~$V_t^c$ (note that~$\tau$ is a stopping time),
then all edges bounding the unexplored region~$E\setminus e_{[\tau]}$ are closed in~$\omega'_{[\tau]}$, hence also in~$\omega_{[\tau]}$. 
As a consequence, at every subsequent step in the coupling process, edges will be sampled with the same rule in the two configurations, hence the configurations will be equal on~$E\setminus e_{[\tau]}$. Equivalently, they will only (possibly) differ for edges that are connected to~$\partial G$ in~$\omega'$, thus leading to the following conclusion when combined with  Theorem~\ref{thm:RSWnear}.

\begin{proposition}[mixing]\label{prop:mixing}
    There exists~$c_{\rm mix}>0$ such that for every~$p\in(0,1)$ and every~$r\le R$ with~$R/r$ large enough, every~$G\supset\Lambda_{R}$ 
    and every event~$A$ depending on edges in~$\Lambda_r$, we have that for every two boundary conditions~$\xi$ and~$\psi$,
    \begin{align}\label{eq:mixing b.c.}
    	|\phi_{G,p}^\psi[A]-\phi_{G,p}^\xi [A]|
    	\le (r/R)^{c_{\rm mix}}\,\phi_{G,p}^\psi [A].
    \end{align}
    In particular, for any two events~$A$ and~$B$ depending on the edges inside~$\Lambda_r$ and outside~$\Lambda_R$, respectively,
    \begin{equation}\tag{Mix}\label{eq:mix}
    	| \phi_{G,p}^{\xi}[A\cap B]-\phi_{G,p}^{\xi}[A]\phi_{G,p}^{\xi}[B]|\le (r/R)^{c_{\rm mix}}\,  \phi_{G,p}^{\xi}[A]\phi_{G,p}^{\xi}[B].
    \end{equation}
\end{proposition}

One may be surprised at first sight not to see any reference to the characteristic length in this statement, yet one should remember that the rate of decay is in fact {\em faster} when~$p$ is away from~$p_c$. The previous proposition simply state a universal bound on the rate of mixing valid for every~$p\in(0,1)$.

Notice also that in \eqref{eq:mixing b.c.} the event $A$ is not assumed increasing nor is there any assumption of ordering between the boundary conditions~$\xi$ and~$\psi$. 
These assumptions would greatly simplify the proof. 

\begin{proof}
Notice that applying \eqref{eq:mixing b.c.} to~$\La_R$ and using \eqref{eq:SMP} implies directly~\eqref{eq:mix}. We therefore focus on proving \eqref{eq:mixing b.c.}.

By duality, it suffices to prove the statement for~$p\le p_c$. 
Set~$\rho = \sqrt{rR}$ and let~$\Omega$ be a subgraph of~$G$ containing~$\La_\rho$.
We first compare the probability of~$A$ under free boundary conditions to that under arbitrary boundary conditions~$\psi$. 
For reasons which will be apparent later, we do this on~$\Omega$. 

For~$\psi$ boundary conditions on~$\partial \Omega$, using the increasing coupling~$\mathbb P_{\bf T}$ between~$\phi_{\Omega,p}^0$ and~$\phi_{\Omega,p}^\psi$ described above, we find
	\begin{align*}
		\phi_{\Omega,p}^\psi [A] - \phi_{\Omega,p}^0 [A] 
		&= \bbP_{\bfT} [\omega \notin A \text{ but } \omega' \in A]\\
		&\le \phi_{\Omega,p}^\psi  [\omega' \in A \text{ and } \partial \La_r \xlra{\omega'} \partial G]\\
		&\le \phi_{\Omega\setminus\La_r,p_c}^1[\partial\Lambda_r\longleftrightarrow\partial G]\phi_{\Omega,p}^\psi [A] \\
		&\le (r/R)^c\phi_{\Omega,p}^\psi[A],
	\end{align*}
	for some constant~$c >0$. The first inequality is due to the property of the coupling, 
	the second to \eqref{eq:SMP}, \eqref{eq:CBC} and \eqref{eq:pmon} and the third to \eqref{eq:LOWER_BOUND_ONE_ARM}.
	In conclusion,
	\begin{align}\label{eq:oj}
		\phi_{\Omega,p}^\psi [A]\le \phi_{\Omega,p}^0[A]/(1-(r/R)^c).
	\end{align}	
	The above applies in particular to~$\Omega = G$; let us now obtain a converse bound in this case.
	 
	Start by observing that, for any fixed~$\Omega$ as above 
	\begin{align*}
		 \phi_{G,p}^0[A] 
		 = \sum_{\xi \text{ b.c. on }\partial \Omega} \phi_{\Omega,p}^\xi [A] \phi_{G,p}^0[\omega_{|G\setminus\Omega} \text{ induces } \xi \text{ on }\partial\Omega]
		 \le \phi_{\Omega,p}^0[A]/(1-(r/R)^c).
	\end{align*}
	Fix now some boundary conditions~$\psi$ on~$G$.
	For a configuration~$\omega$ on~$G$, let~$\Omega(\omega)$ be the set of vertices that are {\em not} connected to~$\bbZ^2\setminus\La_{R-1}$. 
	Then, for~$\psi$ a boundary condition on~$G$, 
	\begin{align*}
		\phi_{G,p}^\psi [A]&\ge \phi_{G,p}^\psi [A,\partial \La_\rho \not\longleftrightarrow \partial \La_R]\\
		&=\sum_{\Omega \supset \La_{\rho} }\phi_{\Omega,p}^0[A]\,\phi_{G,p}^\psi [\Omega(\omega)=\Omega]\\
		&\ge (1-(r/R)^c)\phi_{G,p}^0[A] \phi_{G,p}^\psi [\partial \La_\rho \not\longleftrightarrow \partial \La_R]\\
		&\ge (1-(r/R)^c)^2 \phi_{G,p}^0[A],
		\end{align*}
		where in the second inequality we used \eqref{eq:oj} and the fact that 
		\[
		\sum_{\Omega \supset \La_{\rho} }\phi_{G,p}^\psi [\Omega(\omega)=\Omega] = \phi_{G,p}^\psi [\partial \La_\rho \not\longleftrightarrow \partial \La_R].
		\] 
		In the last inequality we used that 
		\[
		\phi_{G,p}^\psi [\partial \La_\rho \not\longleftrightarrow \partial \La_R] \leq \phi_{G,p_c}^\psi [\partial \La_\rho \not\longleftrightarrow \partial \La_R]
		\leq (\rho/R)^{2c}
		\]
		 for some constant~$c > 0$, by \eqref{eq:LOWER_BOUND_ONE_ARM}.
		
	Using the inequality above and \eqref{eq:oj} applied to~$G$, we conclude that 
	\begin{align*}
		\big|\phi_{G,p}^\psi [A] - \phi_{G,p}^0 [A]\big|
		\leq 2(r/R)^{c} \phi_{G,p}^0[A].
	\end{align*}
	Applying the above to two arbitrary boundary conditions~$\psi$ and~$\xi$ on~$\partial G$ and using the triangular inequality, 
	we conclude that for all~$r/R$ large enough, 
	\begin{align*}
		\big|\phi_{G,p}^\psi [A] - \phi_{G,p}^\xi [A]\big|
		\leq 4(r/R)^{c}\phi_{G,p}^0[A]
		\leq 4(r/R)^{c}(1+2(r/R)^c)\phi_{G,p}^\psi[A].
	\end{align*}
	By assuming again that~$r/R$ is large enough and modifying the constant in the exponent, we may eliminate the prefactor~$4$, and obtain~\eqref{eq:mixing b.c.}.
\end{proof}

\paragraph{Example 3} Alternatively, one may consider the decision tree~$\mathbf T$ that explores the dual clusters of~$\partial G^*$ in~$\omega^*$. 
We do not define this decision tree formally as it is almost identical to that of the previous example. 
We simply mention that, when coupling two measures with~$p=p'$ using~$\mathbf T$, 
differences only occur for edges that are connected in~$\omega^*$ to~$\partial G$. 

\begin{remark}
In spite of the constructions above, 
we are unaware of the existence of a coupling of random-cluster models with boundary conditions~$\xi\le\xi'$ and same edge-parameter~$p=p'$  that combines the properties of  Examples 2 and 3. Namely a coupling for which only edges connected in both~$\omega'$ and~$\omega^*$ to~$\partial G$ may have different states in the two configurations. 
\end{remark}

\begin{remark}
Even though the uniform variables~$(U_e)_{e\in E}$ are attached to the edges, 
the order in which these are revealed by~$\mathbf T$ has an influence on the final couple of configurations~$(\omega, \omega')$. 
Indeed, consider~$G = \Lambda_R$, parameters~$p = p'$ and boundary conditions~$\xi = 0$ and~$\xi' = 1$; 
let~$e$ be one of the edges containing the origin.  
In the coupling produced with the decision tree of Example~$1$, 
$\omega_e$ may differ from~$\omega'_e$ when~$e$ is not connected to~$\partial G$ in~$\omega'$, while this is impossible with the one produced by Example 2.\end{remark}

\subsection{Equivalence~$L(p)-\xi(p)$: proof of Theorem~\ref{thm:L_equiv_xi}}\label{sec:11}

We will show the following (stronger) proposition.
\begin{proposition}\label{prop:estimate subcritical}
There exist~$c,C>0$ such that for every~$p\le p_c$ and~$x\in \bbZ^2$,
\begin{equation}\label{eq:fundamental}
	\exp[-C|x|/L(p)]\les \phi_{p}[\Lambda_{L(p)}\longleftrightarrow \Lambda_{L(p)}(x)]\les\exp[-c|x|/L(p)].
\end{equation}
\end{proposition}

Before proving this proposition, we explain how it implies the theorem.
\begin{proof}[Theorem~\ref{thm:L_equiv_xi}] 
    For~$p < p_c$, the proof is immediate thanks to the definition of~$\xi(p)$ and the fact that 
    \begin{equation}\label{eq:fundweak}
    	p^{2|\Lambda_{L(p)}|}\phi_{p}[\Lambda_{L(p)}\longleftrightarrow \Lambda_{L(p)}(x)]
    	\le\phi_p[0\longleftrightarrow x]
    	\le \phi_{p}[\Lambda_{L(p)}\longleftrightarrow \Lambda_{L(p)}(x)].
    \end{equation}
    
    For~$p > p_c$ we proceed by duality. 
    Notice that 
    \begin{align*}
    \pi_1(p,n)-\theta(p) 
    &= \phi_{p}[0 \longleftrightarrow \partial \La_n \text{ but } 0 \not\longleftrightarrow \infty] \\
    &\le \sum_{k\geq 0} \phi_{p}[(k+\tfrac12,0) \xlra{\omega^*} \partial \La_{k \vee n}(k,0)]
	\les\exp[-c|x|/L(p^*)].
    \end{align*}
	Indeed, any configuration contributing to the second probability contains a dual circuit of length at least~$n$, surrounding~$0$ 
	and passing through some point~$(0,k)$ of the horizontal axis. The last inequality is due to the subcritical case already established. 
	
	Conversely, due to \eqref{eq:CBC},
	\begin{align*}
	 \phi_{p}[0 \longleftrightarrow \partial \La_n \text{ and } 0 \not\longleftrightarrow \infty] 
	 \geq \phi_{\La_{2n},p_c}^0[0 \longleftrightarrow \partial \La_n] \phi_{p}[\La_{2n}\not\longleftrightarrow \infty ].
	\end{align*}
	As~$n$ tends to infinity, the first term in the right-hand side above decays at most polynomially due to \eqref{eq:LOWER_BOUND_ONE_ARM} and \eqref{eq:mixing b.c.}, 
	while the second is lower bounded by~$\exp[-C|x|/L(p^*)]$ due to the subcritical case and the FKG property. 
	
	The two inequalities above show that~$\xi(p)\asymp L(p^*)$. That~$L(p^*)\asymp L(p)$ follows directly by duality from the definition of the characteristic length.
\end{proof}	
We now turn to the proof of Proposition~\ref{prop:estimate subcritical}.

\begin{proof}[Proposition~\ref{prop:estimate subcritical}]
Set~$L=L(p)$. We assume that~$x\in L\mathbb Z^2$; the general case can be solved similarly. 
We start with the lower bound. Consider the shortest family of vertices of~$y_i \in L\mathbb Z^2$ with~$0=y_0,\dots,y_k=x$. 
Let~$A_L(y)$ be the event that there exists a circuit in~$\Lambda_{2L}(y)$ surrounding~$\Lambda_L(y)$. If~$A_L(y_j)$ occurs for every~$0\le j\le k$, then~$\Lambda_L$ is connected to~$\Lambda_L(x)$. We deduce from
the FKG inequality and Theorem~\ref{thm:RSWnear} that
\begin{equation}
\phi_p[\Lambda_L\longleftrightarrow \Lambda_L(x)]\ge \phi_p[A_L]^{k+1}\ge \exp[-C|x|/L],
\end{equation}
where the last inequality follows from Theorem~\ref{thm:RSWnear} and~$C > 0$ is some universal constant.

For the upper bound, we start by observing that by~\eqref{eq:mix} and  the RSW theorem from~\cite{DumTas18}, we have that, for some constant~$C$, 
\begin{align}\label{eq:ihuu}
\phi^1_{\Lambda_{4L},p}[\Lambda_L\longleftrightarrow\partial\Lambda_{2L}]
&\les \phi_{p}[\Lambda_L\longleftrightarrow\partial\Lambda_{2L}]\le C\phi_{p}[\calC(\Lambda_L)]^{1/C}\le C\delta^{1/C}<8^{-49},
\end{align}
provided that~$\delta$ in the definition of~$L(p)$ is chosen sufficiently small. 
Now, if~$\Lambda_L$ and~$\Lambda_L(x)$ are connected, then there must exist a sequence of~$N \ge |x|/L$ distinct vertices~$0=y_1,\dots,y_N=x$ contained in~$L\bbZ^2$ 
such that 
\begin{itemize}[noitemsep,nolistsep]
\item~$\|y_i-y_{i+1}\|_\infty= L$ for every~$i$,
\item~$\Lambda_L(y_i)$ is connected to~$\partial\Lambda_{2L}(y_i)$ for every~$i$.
\end{itemize} 
Choose out of these the first subsequence of vertices~$(y_i)_{i \in I}$ for the lexicographical order which contains $k = N/49$ vertices which are all at distance at least~$8L$ of each other (the existence of such a subsequence is due to the pigeonhole principle). 
The union bound over the possible choices of~$y_1,\dots,y_N$ (of which there are at most~$8^N$), 
the spatial Markov property~\eqref{eq:SMP} and the comparison between boundary conditions~\eqref{eq:CBC} imply that
\begin{align*}
	\phi_p[\Lambda_L\longleftrightarrow \Lambda_L(x)]\le\sum_{N\ge |x|/L} 8^N\phi^1_{\Lambda_{4L},p}[\Lambda_L\longleftrightarrow\partial\Lambda_{2L}]^{N/49}.
\end{align*}
The desired upper bound follows from the above using \eqref{eq:ihuu}.
\end{proof}

\begin{remark}
The previous proof is probably the place where the strongest condition on $\delta$ is imposed (remember that we already fixed $\delta<1/2$ to guarantee infinite characteristic length at $p_c$). 
\end{remark}

The next corollary is a useful estimate that we will invoke later in the article.
\begin{corollary}\label{rmk:1}
	There exists~$c=c(\delta)>0$ such that for every~$p>p_c$ and~$k\ge1$,
	\begin{equation}
		\phi_{p}[\Lambda_{kL(p)}\leftrightarrow\infty]\ge 1-\exp[-ck].
	\end{equation}
\end{corollary}

\begin{proof}
Note that the previous proof implies that for some constant~$c>0$, we have that for every~$p>p_c$,
\[
\phi_{p^*}[\Lambda_{L(p)}\longleftrightarrow\partial\Lambda_{kL(p)}]\le \exp[-ck].
\]
By the same counting argument as in the proof of the supercritical case of Theorem~\ref{thm:L_equiv_xi}, 
we deduce from the above that the probability that there exists a circuit in~$\omega^*$ surrounding~$\Lambda_{k L(p)}$ is bounded from above by
$\sum_{j \geq k} \phi_{p^*}[\Lambda_{L(p)}\longleftrightarrow\partial\Lambda_{j L(p)}]\les \exp[-ck]$.
\end{proof}

\section{Boosting pairs of boundary conditions for flower domains}\label{sec:boost}\label{sec:9}

Fix~$q \in (1,4]$ for the whole section; we will omit it from the notation. 
The results of this section do not apply to~$q = 1$. 

\subsection{Flower domains}\label{sec:flower_domains}

we start by introducing the crucial notion of flower domain. 

\begin{definition}[Flower domain]
	An  \emph{inner flower domain} on~$\La_R$ is a simply connected finite domain~$\calF$ containing~$\La_R$,
	whose boundary is formed of a sequence of arcs~$(a_{j}a_{j+1})_{j = 1,\dots,2k}$ (with the convention~$a_{2k+1} = a_1$)
	where each point~$a_{j}$ is on~$\partial \La_R$.
		
	An  \emph{outer flower domain} on~$\La_R$ is the complement~$\calF$ of a simply connected finite domain, 
	with~$\La_R^c \subset \calF$ and 
	whose boundary is formed of a sequence of primal arcs~$(a_{j}a_{j+1})_{j = 1,\dots,2k}$ (with the convention~$a_{2k+1} = a_1$)
	where  each point~$a_{j}$ is on~$\partial \La_R$.

		The arcs of the boundary are called primal and dual petals depending on whether~$j$ is even or odd respectively. In both cases, we identify~$\calF$ to the graph formed of the edges strictly inside~$\calF$, plus the edges on the dual petals.

	For~$\eta > 0$, the flower domain~$\calF$ is said to be {\em~$\eta$-well-separated} 
	if the distance between any two distinct points~$a_i$ and~$a_j$ is greater than~$\eta R$. 
	
	A boundary condition~$\xi$ is said to be coherent with~$\calF$ if all vertices of any primal petal are wired together
	and all vertices of dual petals (except the endpoints) are wired to no other vertex of~$\partial \calF$. 
\end{definition}

Formally flower domains should be defined as the couple formed of~$\calF$ and of the points~$a_1,\dots, a_{2k}$;
we will, however, allow ourselves this small abuse of notation.
See Figure~\ref{fig:flowerdefinition} for an illustration. 
When considering a flower domain with a coherent boundary condition, 
it will be useful to view the flower domain as containing the edges of the primal petals which are conditioned to be opened. We will also often identify dual arcs~$(a_ja_{j+1})$ with the dual path made of the dual edges~$e^*$ with~$e$ incident to~$x\in (a_ja_{j+1})$ and~$y\notin \calF$, and assume it is made of open dual edges.

Notice that, when~$\calF$ has at least four petals, there are several boundary conditions that are coherent with~$\calF$ as different primal petals may be wired together or not. 

	\begin{figure}
	\begin{center}
	\includegraphics[width=0.8\textwidth]{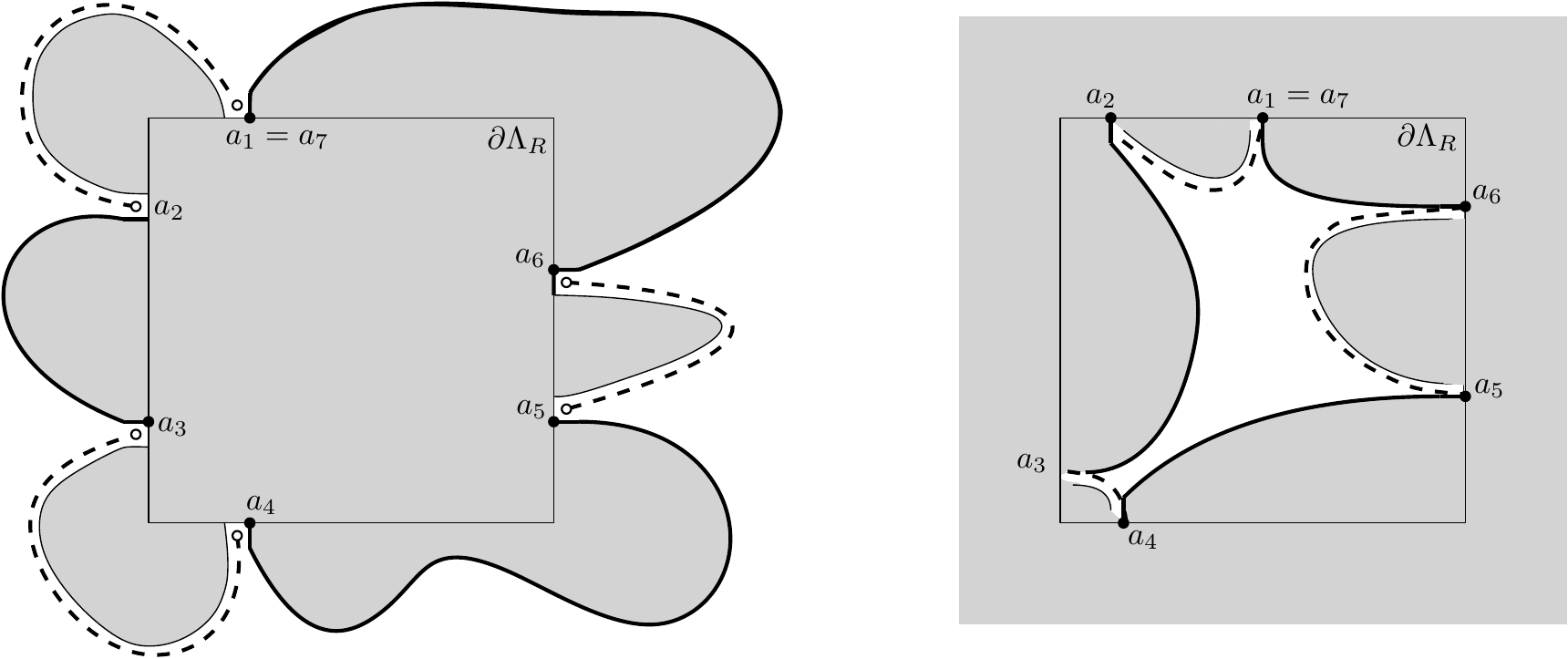}
	\caption{An inner flower domain on the left and an outer one on the right.}
	\label{fig:flowerdefinition}
	\end{center}
	\end{figure}
\paragraph{Example} 
The example that we will most commonly use is that of a flower domain explored from the inside or outside. 
Consider~$r < R$ and let~$\omega$ be a configuration on the annulus~$\Ann(r,R)$. 

The {\em inner flower domain~$\calF$ from~$\La_{R}$ to~$\La_r$} is obtained as follows. 
Consider all interfaces of~$\omega$ contained in~$\Ann(r,R)$ starting on~$\partial \La_R$;
these are paths in the loop representation of the random cluster model with endpoints on~$\partial \La_R$ or~$\partial \La_r$. 
Write~$\mathsf{Exp}$ for the set of edges adjacent or intersecting any such interface. 
Loosely speaking, these are the edges revealed during the exploration of the interfaces starting on~$\partial \La_R$. 

If at least one such interface has an endpoint on~$\partial \La_r$, define~$\calF$ as being the connected component of~$\La_r$ in~$\La_R \setminus \mathsf{Exp}$.
Otherwise~$\calF$ is not defined (formally set~$\calF = \emptyset$ in this case). 
Observe that the number of interfaces between~$\partial \La_R$ and~$\partial \La_r$ is necessarily even. 
Their endpoints~$\partial \La_r$ naturally partition~$\partial \calF$ into primal and dual petals. See Figure~\ref{fig:flowerdefinition} for an illustration.

To define the {\em outer flower domain from~$\La_{r}$ to~$\La_R$}, similarly explore the interfaces starting on~$\partial \La_r^c$. 

\begin{lemma}\label{lem:well-separated}
	For every~$\eps >0$, there exists~$\eta > 0$ such that 
	for any~$p \in (0,1)$, any $R <L(p)$, and any boundary conditions~$\xi$,
	\begin{align*}
		\phi_{\Lambda_{2R}}^\xi [\calF \text{ exists but is not~$\eta$-well-separated}] <\eps,
	\end{align*}
	when~$\calF$ denotes the inner flower domain explored from~$\La_{2R}$ to~$\La_R$, and 
	\begin{align*}
		\phi_{\Lambda_{R}^c}^\xi [\calF \text{ exists but is not~$\eta$-well-separated}] <\eps,
	\end{align*}
	when~$\calF$ denotes the outer flower domain explored from~$\La_{R}$ to~$\La_{2R}$.
\end{lemma}

\begin{proof}
We treat the case of inner flower domains, that of outer domains can be solved similarly. 
For~$\calF$ to exist but not be~$\eta$-well-separated, it needs to contain a primal or dual petal of diameter smaller than~$\eta R$. 
We will exclude below the possibility of a small dual petal, the case of a primal one is identical. 

Divide~$\partial\Lambda_R$ into arcs~$\ell_1,\dots,\ell_N$ of length~$2\eta R$ successively overlapping on a segment of length $\eta R$. 
Let~$A_{10}^\square(\ell_i)$ and~$A_{101}^\square(\ell_i)$ be the events that there exist two or three, respectively,  arms of alternating type
contained in~$\Ann(R,2R)$ from~$\ell_i$ to~$\partial \La_{2R}$; for the three arms, we require two primal ones with a dual one in between. 
Notice that, if~$\calF$ contains a dual petal of diameter smaller than~$\eta R$, then there exists at least one arc~$\ell_i$ for which~$A_{101}^\square(\ell_i)$ occurs. 
Our goal is therefore to bound the probability of the events~$A_{101}^\square(\ell_i)$. 

From Theorem~\ref{thm:RSWnear} it is easily deduced by an exploration argument that for each~$i$,
\begin{equation}\label{eq:ah1}
	\phi_{\Lambda_{2R}}^\xi[A_{101}^\square(\ell_i)]\le C_0\eta^{c_0}\,\phi_{\Lambda_{2R}}^\xi[A_{10}^\square(\ell_i)],
\end{equation}
for universal constants~$c_0,C_0 > 0$. 
Indeed, explore first the interface from~$\ell_i$ to~$\partial \La_{2R}$ closest to a chosen endpoint of~$\ell_i$.
The existence of such an interface is synonymous to~$A_{10}^\square(\ell_i)$. 
Condition next on this interface, and bound the probability of existence of the second primal arm. 
The conditioning induces both positive and negative information on the remaining edges, but 
the information favorable to the existence of a primal path is at every scale around $\ell_i$ at a macroscopic distance from~$\ell_i$.
This suffices to obtain the extra polynomial term using \eqref{eq:LOWER_BOUND_ONE_ARM}. 
	
To bound the probabilities of the events~$A_{10}^\square(\ell_i)$, define~$\mathbf N$ to be  the number of disjoint clusters crossing~${\rm Ann}(R,2R)$ from inside to outside. Then
\begin{equation}\label{eq:ah2}
	\sum_{i=1}^k\phi_{\Lambda_{2R}}^\xi[A_{10}^\square(\ell_i)]\le \phi_{\Lambda_{2R}}^\xi[\mathbf N]\le C_1,
\end{equation}
 where the first inequality is a deterministic bound, and the second uniform bound on the expectation of~$\mathbf N$ which is a standard consequence of Theorem~\ref{thm:RSWnear} sketched below:
 observe that there exists~$c_2>0$ such that for every~$k\ge0$, 
 \begin{align}\label{eq:ah23}
 \phi_{\Lambda_{2R}}^\xi[\mathbf N\ge k+1|\mathbf N\ge k]\le 1-c_2. 
 \end{align} 
 Indeed, conditionally on the~$k$ first clusters (in clockwise order around~$\partial \La_R$ starting from some arbitrary point), observe that the complement $\Omega$  in~${\rm Ann}(R,2R)$ of these clusters is a subset of~${\rm Ann}(R,2R)$ with free boundary conditions on the part of the boundary that lies strictly inside~${\rm Ann}(R,2R)$; on the rest of the boundary, the boundary conditions are dominated by wired ones. Then, a dual path disconnecting~$\partial\Lambda_R$ from~$\partial\Lambda_{2R}$ in~$\Omega$ occurs with probability at least~$c_2>0$ by Theorem~\ref{thm:RSWnear}. 
 This proves \eqref{eq:ah23}, which in turn implies \eqref{eq:ah2}.
 
 Combining \eqref{eq:ah1} and \eqref{eq:ah2}, and adding a factor~$2$ to account for the existence of small primal petals, we find
 \[
 \phi_{\Lambda_{2R}}^\xi [\calF \text{ exists but is not~$\eta$-well-separated}] \le 2C_0\eta^{c_0} \, \sum_{i=1}^k\phi_{\Lambda_{2R}}^\xi[A_{10}^\square(\ell_i)]\le 2C_0C_1\eta^{c_0}.
 \]
Fixing~$\eta$ small enough concludes the proof.
\end{proof}

\begin{definition}[Double four-petal flower domain]
Fix~$1 \le r < R$. 
We say that there exists a double four-petal flower domain between~$\La_r$ and~$\La_R$ if
\begin{itemize}
	\item the outer flower domain~$\calF_{out}$ explored from~$\La_{(rR)^{1/2}}$ to~$\La_{R}$ 
	exists,	is~$1/2$-well-separated and has exactly four petals~$P_1^{out},\dots,P_4^{out}$;
	\item the inner flower domain~$\calF_{in}$ explored from~$\La_{(rR)^{1/2}}$ to~$\La_{r}$ 
	exists, is~$1/2$-well-separated and has exactly four petals~$P_1^{in},\dots,P_4^{in}$;
	\item~$P_1^{in}$ is connected to~$P_1^{out}$ and~$P_3^{in}$ to~$P_3^{out}$ in~$\omega\cap \calF_{in}^c \cap \calF_{out}^c$;
	\item~$P_2^{in}$ is connected to~$P_2^{out}$ and~$P_4^{in}$ to~$P_4^{out}$ in~$\omega^*\cap \calF_{in}^c \cap \calF_{out}^c$.	
\end{itemize}
\end{definition}

The advantage of the double four-petal flower domain is that it can be explored from~$\partial\La_{(rR)^{1/2}}$ towards the inside and outside, 
and limits the interaction between the configurations in~$\calF_{in}$ and~$\calF_{out}$.

\begin{lemma}\label{lem:D4PFD}
	For any~$\eta > 1~$, there exists~$c = c(\eta) > 0$ such that 
	for any~$p \in (0,1)$, any ~$R <L(p)$ large enough and any boundary conditions~$\xi$ on~$\Ann := \Ann((1-\eta)R, (1 + 2\eta)R)$
	\begin{align*}
		\phi_{\Ann}^\xi [\text{there exists a double four-petal flower domain between~$\La_R$ and~$\La_{(1 + \eta) R}$}] > c.
	\end{align*}
\end{lemma}

\begin{proof}
We recommend taking a look at Figure~\ref{fig:double-flower}. Set~$R'':=(1+\eta)R$ and~$R':=(RR'')^{1/2}$.
Consider the rectangles 
\begin{align*}
{\rm Rect}_1&:=[R,R'']\times[0,\eta R], \\
{\rm Rect}_1'&:=[R,R'']\times[-\eta R,0],
\end{align*} as well as their rotations~${\rm Rect}_2,{\rm Rect}_3,{\rm Rect}_4$ and~${\rm Rect}'_2,{\rm Rect}'_3,{\rm Rect}_4'$ by angles of~$\pi/2,\pi,3\pi/2$.

Define the event~$E$ that~${\rm Rect}_1,{\rm Rect}'_2,{\rm Rect}_3,{\rm Rect}_4'$ (resp.~${\rm Rect}'_1,{\rm Rect}_2,{\rm Rect}'_3,{\rm Rect}_4$) are crossed by a primal (resp.~dual) path from~$\partial\Lambda_{R}$ to~$\partial\Lambda_{R''}$. Theorem~\ref{thm:RSWnear} implies the existence of~$c_0>0$ such that 
\begin{equation}\label{eq:hk00}
\phi^\xi_{\rm Ann}[E]\ge c_0.
\end{equation}
Let~$F$ be the event that there exist exactly two clusters in~${\rm Ann}(R,R')$ crossing from~$\partial\Lambda_R$ to~$\partial \Lambda_{R'}$, and exactly two clusters in~${\rm Ann}(R',R'')$ crossing from~$\partial\Lambda_{R'}$ to~$\partial \Lambda_{R''}$. We claim that 
\begin{equation}\label{eq:hk}
\phi^\xi_{\rm Ann}[F|E]\ge c_1.
\end{equation}
When~$F\cap E$ occurs, both flower domains~$\calF_{in}$ and~$\calF_{out}$ exist, have four petals and are~$1/2$-well separated. 
Thus, \eqref{eq:hk00} and \eqref{eq:hk} imply directly the result, and we focus next on its proof. 

Let~$\Gamma_1$ be the bottom boundary of the cluster of~$\{R\}\times[0,\eta R]$ in~${\rm Rect}_1\cup {\rm Rect}_1'$,  
$\Gamma_2$ be the left boundary of the cluster of~$[0,\eta R]\times\{R\}$ in~${\rm Rect}_2\cup {\rm Rect}'_2$,   
$\Gamma_3$ be the top boundary of the cluster of~$\{-R\}\times[-\eta R,0]$ in~${\rm Rect}_3\cup {\rm Rect}_3'$ and 
$\Gamma_4$ be the right boundary of the cluster of~$[-\eta R,0]\times\{-R\}$ in~${\rm Rect}_4\cup {\rm Rect}'_4$. 
Notice that each~$\Gamma_i$ may be explored by a standard procedure, starting from the common inner corner of~${\rm Rect}_i$ and~${\rm Rect}_i'$.
Moreover,~$E$ implies that each~$\Gamma_i$ is contained in~${\rm Rect}_i\cup {\rm Rect}'_i$ and ends on~$\partial \La_{R''}$.

Condition now on a realisation of~$\Gamma_1,\dots,\Gamma_4$ with the properties above. 
By Theorem~\ref{thm:RSWnear}, one may construct with uniformly positive probability four primal paths and four dual ones as in Figure~\ref{fig:double-flower}, namely:
two primal paths connecting~$\Gamma_1$ and~$\Gamma_2$ in the top-right corners of~${\rm Ann}(R,R')$ and~${\rm Ann}(R',R'')$, respectively; 
two primal paths connecting~$\Gamma_3$ and~$\Gamma_4$ in the bottom-left corners of the same annuli; 
two dual paths connecting the dual vertices adjacent to~$\Gamma_2$ and~$\Gamma_3$ in the top-left corners and 
and two more between the dual vertices adjacent to~$\Gamma_4$ and~$\Gamma_1$ in the bottom-right corner of these two annuli. 
When all these paths exist, both~$E$ and~$F$ occur. This concludes the proof of \eqref{eq:hk} and therefore the whole argument.
\end{proof}

	\begin{figure}
	\begin{center}
	\includegraphics[width=0.6\textwidth]{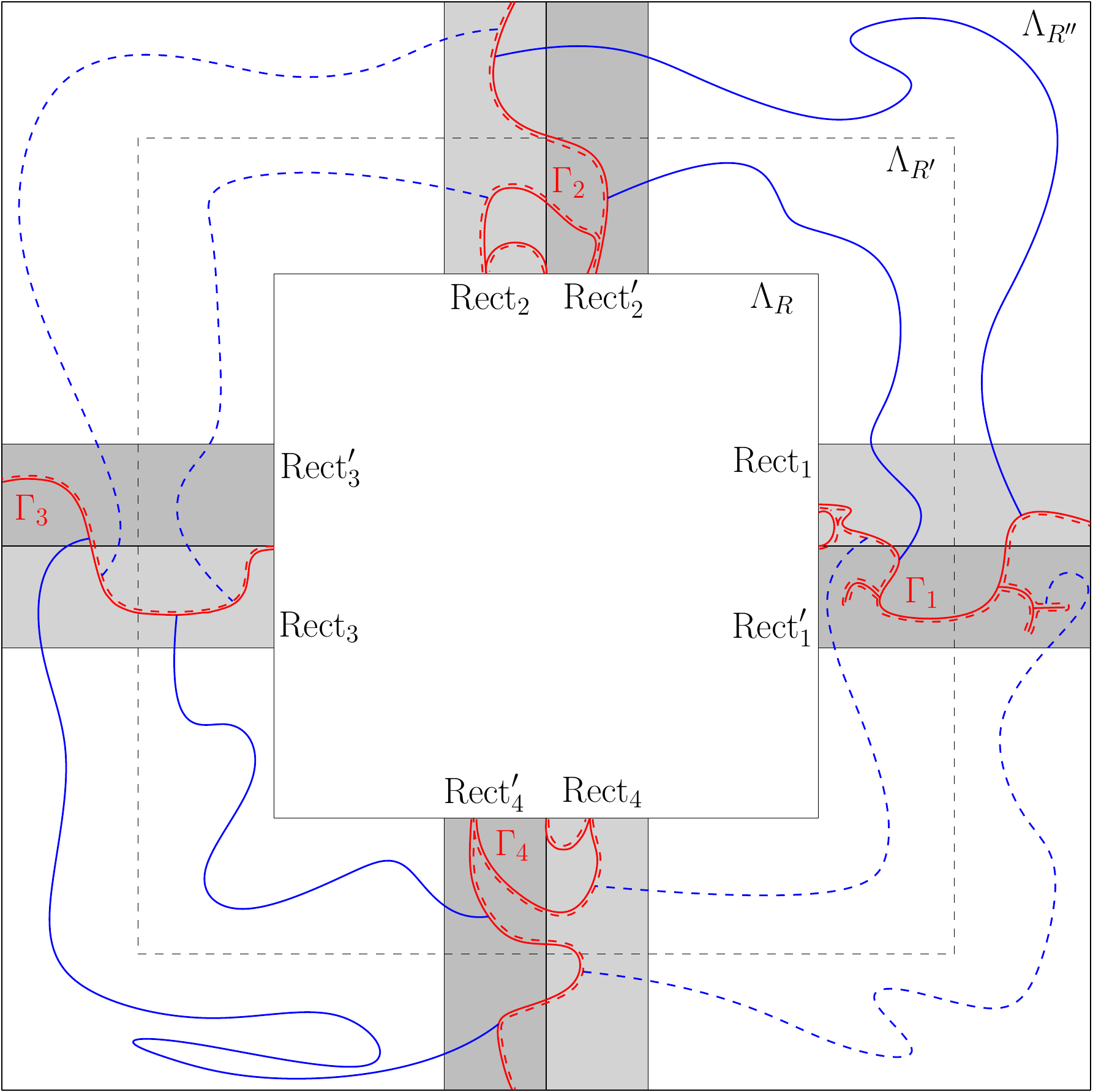}
	\caption{In grey and darker grey (respectively): the rectangles~${\rm Rect}_i$ and~${\rm Rect}_i'$. We depicted in red some possible realization for the paths~$\Gamma_1,\dots,\Gamma_4$. In blue, possible realizations of paths and dual paths enforcing the event in blue.}
	\label{fig:double-flower}
	\end{center}
	\end{figure}

\subsection{Boosting pair of boundary conditions}

The goal of this section is to study how changes of boundary conditions impact crossing probabilities.

\begin{definition}
	A {\em boosting pair} of boundary conditions for a flower domain~$\calF$ 
	is a pair of boundary conditions~$(\xi,\xi')$ such that 
	\begin{itemize}[noitemsep,nolistsep]
	\item~$\xi$ and~$\xi'$ are compatible with~$\calF$, 
	\item~$\xi \le \xi'$, 
	\item there exist two primal petals of~$\calF$ that are wired together in~$\xi'$ but not in~$\xi$.
	\end{itemize}
\end{definition}

In a slight abuse of notation, we will henceforth also call a pair~$(\zeta,\zeta')$ of boundary conditions on~$\calF$ boosting 
if there exists a boosting pair~$(\xi,\xi')$ of boundary conditions on~$\calF$ (in the sense of the definition above) such that~$\zeta \le \xi < \xi'\le \zeta'$. 

Recall from Section~\ref{sec:crossing estimates} that $A_R$ is the event that there exists a circuit in~${\rm Ann}(R,2R)$ surrounding~$0$.
The object of this section is to prove the following theorem.

\begin{theorem}\label{thm:boosting_pair}
    Fix~$q\in(1,4]$. For every~$\eta>0$, there exists~$\delta=\delta(\eta,q)>0$ such that the following holds.
    \begin{itemize}
 		\item[(i)] 
		For every~$p \in (0,1)$, every~$R \le L(p)$, 
		every~$\eta$-well-separated inner flower domain~$\calF$ on~$\La_{2R}$, 
		every  boosting pair~$(\xi,\xi')$ of boundary conditions on~$\calF$, 
	    and every~$\eta$-regular quad~$(\calD,a,b,c,d)$ of size~$R$,
	    \begin{align}
		    \phi^{\xi'}_{\calF}[\calC(\calD)]  &\ge \phi^{\xi}_{\calF}[\calC(\calD)]+\delta  \label{eq:boosting_pair_i},\\	
	    	\phi^{\xi'}_{\calF}[A_{R}]&\ge \phi^{\xi}_{\calF}[A_{R}]+\delta.\label{eq:boosting_pair_circ_i}
	    \end{align}
		\item[(ii)]
		For every~$p \in (0,1)$, every~$R \le L(p)$, 
		every~$\eta$-well-separated outer flower domain~$\calF$ on~$\La_R$, 
		every  boosting pair~$(\xi,\xi')$ of boundary conditions on~$\calF$, 
		and every~$\eta$-regular quad~$(\calD,a,b,c,d)$ of size~$R$ translated in such a way that it is contained in~$\Ann(2R, 4R)$,
	    \begin{align}
		    \phi^{\xi'}_{\calF}[\calC(\calD)] &\ge \phi^{\xi}_{\calF}[\calC(\calD)]+\delta,  \label{eq:boosting_pair_o}\\
	    	\phi^{\xi'}_{\calF}[A_R]&\ge \phi^{\xi}_{\calF}[A_R]+\delta.  \label{eq:boosting_pair_circ}
    	\end{align}
	
    \end{itemize}
\end{theorem}

In light of the RSW theory, the crossing probabilities~$\phi^{\xi}_{\calF}[\calC(\calD)]$ and~$\phi^{\xi'}_{\calF}[\calC(\calD)]$ are bounded away from~$0$ and~$1$ by constants depending only on~$\eta$. 
Above we are concerned with the amount by which such crossing probabilities increase when the boundary conditions change from~$\xi$ to~$\xi'$. 
Indeed, the theorem states that the increase is positive, uniformly in the scale, the quad to be crossed and the boosting pair of boundary conditions. 
The rest of the section is dedicated to proving Theorem~\ref{thm:boosting_pair}. \smallskip

The following lemma is the cornerstone for the proof.
For a quad~$(\calD,a,b,c,d)$, let~${\rm mix}$ be the boundary conditions on~$\calD$ corresponding to the partitions containing~$(ab)$,~$(cd)$ and singletons, and~${\rm mix}'$ be the one containing~$(ab)\cup(cd)$ and singletons. 

\begin{lemma}\label{lem:boost}
For every~$q>1$,~$p \in (0,1)$, and every quad~$(\calD,a,b,c,d)$, we have
\begin{equation}
\phi_\calD^{\rm mix'}[\calC(\calD)]= \frac{q}{1+(q-1)\phi_\calD^{\rm mix}[\calC(\calD)]}\phi_\calD^{\rm mix}[\calC(\calD)].
\end{equation} 
\end{lemma}

Notice that the ratio in the right-hand side above is always larger than~$1$, and considerably so when~$\phi_\calD^{\rm mix}[\calC(\calD)]$
is far from~$1$. 

\begin{proof}
Let~$w(\omega):=(\tfrac p{1-p})^{|\omega|}q^{k(\omega^{\rm mix})}$ and~$w'(\omega):=(\tfrac p{1-p})^{|\omega|}q^{k(\omega^{\rm mix'})}$ and observe that 
\begin{equation*}
w(\omega)=\begin{cases} w'(\omega)&\text{ if }\omega\in\calC(\calD),\\ qw'(\omega)&\text{ if }\omega\notin\calC(\calD).\end{cases}
\end{equation*}
Now, set 
\begin{align*}
Z[\calC(\calD)]&:=\sum_{\omega\in\calC(\calD)}w(\omega)\quad\ ,\quad Z[\calC(\calD)^c]:=\sum_{\omega\notin\calC(\calD)}w(\omega),\\
Z'[\calC(\calD)]&:=\sum_{\omega\in\calC(\calD)}w'(\omega)\quad,\quad
Z'[\calC(\calD)^c]:=\sum_{\omega\notin\calC(\calD)}w'(\omega).
\end{align*}
Then
\begin{align*}
{\phi^{\rm mix'}_\calD[\calC(\calD)]}
=\frac{Z'[\calC(\calD)]}{Z'[\calC(\calD)]+Z'[\calC(\calD)^c]}
&=\frac{Z[\calC(\calD)]}{Z[\calC(\calD)]+\tfrac1qZ[\calC(\calD)^c]}\\
&=\frac{\phi^{\rm mix}_\calD[\calC(\calD)]}{\phi^{\rm mix}_\calD[\calC(\calD)]+\tfrac1q(1-\phi^{\rm mix}_\calD[\calC(\calD)])},
\end{align*}
which is the desired equality. 
\end{proof}

	\begin{figure}
	\begin{center}
	\includegraphics[width=01.0\textwidth]{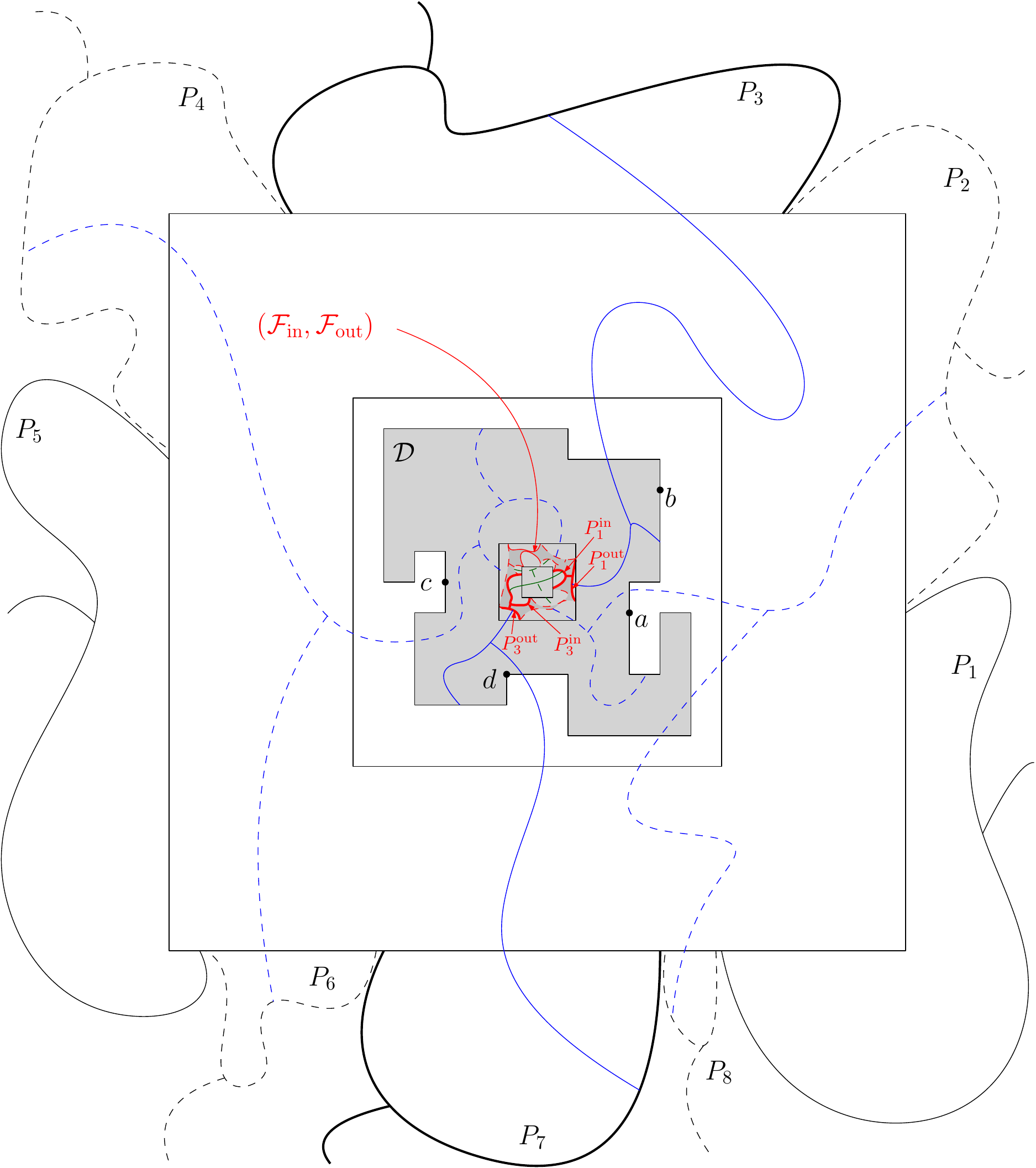}
	\caption{The domain~$\mathcal D$ with the four marked points~$a,b,c,d$ in grey. The outer flower~$\calD$ is in black. In this case,~$P_3$ is assumed to be connected to~$P_7$ in~$\xi'$ but not in~$\xi$. 
	The red depicts the double flower-domain~$(\calF_{\rm in},\calF_{\rm out})$ which is explored first. 
	Then, conditionally on the realization of~$(\calF_{\rm in},\calF_{\rm out})$, the conditions for the event~$E$ to occur are depicted in blue (note that the blue connections from~$\calF_{\rm out}$ to~$\calF$ do not necessarily need to cross the arcs~$(ab)$ and~$(cd)$ of~$\calD$). 
	At the time~$\tau$ of the procedure, the red and blue parts have been revealed and only the inside of~$\calF_{\rm in}$ is unexplored. 
	Then, the event~$\calC(\calD)$ depends on the connection inside~$\calF_{\rm in}$ between its primal petals, 
	which, with positive probability, are connected in~$\omega'$, but not in~$\omega$  (see the green paths).
	}
	\label{fig:h1}
	\end{center}
	\end{figure}
	
\begin{proof}[Theorem~\ref{thm:boosting_pair}]
	We will focus on inner flower domains; the case of outer flower domains is very similar. 
	Let~$\calF$ be an~$\eta$-well-separated inner flower domain on~$\Lambda_{2R}$ 	
	and let~$(\xi,\xi')$  be a  boosting pair of boundary conditions. 
	Write~$P_1,\dots, P_{2k}$ for the petals of~$\calF$ in counter-clockwise order, indexed in such a way that~$P_1$ is primal. 
	Fix~$i$ and~$j$ odd such that~$P_i$ is wired to~$P_j$ in~$\xi'$ but not in~$\xi$. 
	Below,~$c_0,\dots,c_4$ will denote strictly positive constants depending only on~$\eta$. 
	
	We start by proving~\eqref{eq:boosting_pair_i}. We recommend to take a look at Figure~\ref{fig:h1}.
	Fix an~$\eta$-regular quad~$(\calD,a,b,c,d)$ of size~$R$ and translate everything in such a way that the box~$\Lambda_{\eta R}$ 
	is included in~$\calD$. 
	Consider the event~$E$ that there exists a double four-petal flower domain~$(\calF_{\rm in}, \calF_{\rm out})$ between~$\La_{\eta R/4}$ and~$\La_{\eta R/2}$
	and that
	\begin{itemize}[noitemsep]
	\item~$P_1^{\rm out}$ and~$P_3^{\rm out}$ are connected to~$P_i$ and~$P_j$ in~$\calF \cap \calF_{\rm out}$, respectively;  
	\item~$P_1^{\rm out}$ and~$P_3^{\rm out}$ are not connected to each other, nor to any other primal petal of~$\calF$ in~$\calF \cap \calF_{\rm out}$;
	\item~$P_1^{\rm out}$ and~$P_3^{\rm out}$ are connected to the arcs~$(ab)$ and~$(cd)$ in~$\calD \cap \calF_{\rm out}$ respectively.
	\item~$P_2^{\rm out}$ and~$P_4^{\rm out}$ are dually-connected to the arcs~$(bc)$ and~$(da)$ in~$\calD \cap \calF_{\rm out}$.
	\end{itemize}
 	Theorem~\ref{thm:RSWnear} and Lemma~\ref{lem:D4PFD} give	\begin{equation}
		\phi^{\xi}_\calF[E]\ge c_0 >0.
	\end{equation}
	
	Consider the coupling~$\mathbb P_{\bf T}$ between~$\phi^{\xi}_\calF$ and~$\phi^{\xi'}_\calF$, 
	obtained by first exploring~$(\calF_{\rm in}, \calF_{\rm out})$ in~$\omega$, 
	then revealing all the edges in~$\calF_{\rm out}\cap\calF$, then all those inside~$\calF_{\rm in}$. 
	Let $\tau$ be the stopping time corresponding to the end of the second step of this exploration.
	
	Suppose that~$\omega_{[\tau]}\in E$. 
	Then~$\calD$ is crossed in~$\omega$ if and only if the petals~$P_1^{\rm in}$ and~$P_3^{\rm in}$ are connected inside~$\calF_{\rm in}$. 
	Write~$\zeta < \zeta'$ for the two boundary conditions on~$\calF_{\rm in}$ which are coherent with the flower domain structure 
	(with~$P_1^{\rm in}$ wired to~$P_3^{\rm in}$ in~$\zeta'$, but not in~$\zeta$). 
	Then the boundary conditions induced by~$\omega_{[\tau]}$ and~$\xi$ on~$\calF_{\rm in}$ are~$\zeta$. 
	Moreover, since~$\omega' \geq \omega$ and due to the wiring of~$P_i$ and~$P_j$ in~$\xi'$, 
	$\omega_{[\tau]}'$ induces the boundary conditions  on~$\calF_{\rm in}$ that dominate~$\zeta'$.
	Thus, 
	\begin{align*}
		\phi^{\xi'}_\calF[\calC(\calD)]-\phi^{\xi}_\calF[\calC(\calD)]
		&=\bbP_{\bf T}[\omega'\in\calC(\calD),\,\omega\notin\calC(\calD)]\\
		&\geq \bbE_{\bf T}\Big[\big( \phi_{\calF_{\rm in}}^{\zeta'}[P_1^{\rm in} \lra P_3^{\rm in}]-\phi_{\calF_{\rm in}}^{\zeta}[P_1^{\rm in} \lra P_3^{\rm in}]\big)
		\mathbf 1_{\omega_{[\tau]}\in E}\Big]\\
		&\ge c_1 \bbP_{\bf T}[\omega_{[\tau]}\in E]\ge c_1\,c_0 >0, 
	\end{align*}
	where~$c_1>0$ is given by Lemma~\ref{lem:boost} and Theorem~\ref{thm:RSWnear}.
	This concludes the proof of~\eqref{eq:boosting_pair_i}.
	\bigbreak
	
	We turn to~\eqref{eq:boosting_pair_circ_i} and refer the reader to Figure~\ref{fig:h2}.
	Write~$z$ for the point~$(3R/2, 0)$. 
	Consider the coupling~$\bbP_\bfT$ between~$\phi^{\xi}_\calF$ and~$\phi^{\xi'}_\calF$ obtained by first exploring the  
	double flower domain~$(\calF_{\rm in}, \calF_{\rm out})$ in~$\omega$ between~$\La_{R/16}(z)$ and~$\La_{R/8}(z)$, 
	then the configurations inside~$\calF_{\rm in}$ and finally those in~$\calF_{\rm out}$. 
	If no double flower domain exists, reveal all remaining edges in arbitrary order. 
	Write~$\tau_1$ for the stopping time marking the end of the exploration of~$(\calF_{\rm in}, \calF_{\rm out})$, 
	and~$\tau_2$ the stopping time after further exploring~$\calF_{\rm in}$.

	Condition on a pair of configurations~$(\omega_{[\tau_1]}, \omega'_{[\tau_1]})$ such that the double flower domain exists, 
	and write~$\zeta < \zeta'$ for the two boundary conditions on~$\calF_{\rm in}$ that are coherent with the flower domain structure. 
	The configuration~$\omega$ inside~$\calF_{\rm in}$ is sampled according to a convex combination of~$\phi_{\calF_{\rm in}}^\zeta$
	and~$\phi_{\calF_{\rm in}}^{\zeta'}$. 
	Indeed, the coefficient~$\la$ for the latter measure is given by 
	the probability that~$P_1^{\rm out}$ is connected to~$P_3^{\rm out}$ in~$\calF_{\rm out}$, including via the boundary conditions~$\xi$. 
	Similarly, the law of~$\omega'$ inside~$\calF_{\rm in}$ dominates a convex combination of~$\phi_{\calF_{\rm in}}^\zeta$
	and~$\phi_{\calF_{\rm in}}^{\zeta'}$, with the coefficient~$\la'$ for the latter given by 
	the probability that~$P_1^{\rm out}$ is connected to~$P_3^{\rm out}$ in~$\calF_{\rm out}$, including via the boundary conditions~$\xi'$. 
	
	As shown in the first point of the proof with the event~$E$, 
	there is a uniformly positive probability $c_2 > 0$ for $P_1^{\rm out}$ and~$P_3^{\rm out}$ to be connected
	in~$\calF_{\rm out}$ to~$P_i$ and~$P_j$ respectively, but not to each other. 
	Thus,~$\la' \geq \la + c_2$. 
	Applying Lemma~\ref{lem:boost} and Theorem~\ref{thm:RSWnear} we find
	\begin{align*}
	\bbP_\bfT [ P_1^{\rm in} \xlra{\omega' \cap \calF_{\rm in}} P_3^{\rm in}
	\text{ but }  P_1^{\rm in} \nxlra{\omega \cap \calF_{\rm in}} P_3^{\rm in}
	\,|\, \omega_{[\tau_1]},  \omega'_{[\tau_1]} \text{ s.t.~$(\calF_{\rm in}, \calF_{\rm out})$ exists}] &\\
	\geq 
	(\phi_{\calF_{\rm in}}^{\zeta'} [P_1^{\rm in} \xlra{\calF_{\rm in}} P_3^{\rm in}] - 
	\phi_{\calF_{\rm in}}^{\zeta} [P_1^{\rm in} \xlra{\calF_{\rm in}} P_3^{\rm in}]) 
	(\la - \la')
	&\geq c_3, 
	\end{align*}
	for some~$c_3 >0$. 
	
	Write~$F$ for the event that~$(\calF_{\rm in}, \calF_{\rm out})$ exists and that~$P_1^{\rm in}$ is connected to~$P_3^{\rm in}$
	inside~$\calF_{\rm in}$ in~$\omega'$, but not in~$\omega$. 
	This event is measurable in terms of the configurations at the stopping time~$\tau_2$. 
	Finally, write~$H$ for the event that in~$\omega$, 
	\begin{itemize}[noitemsep,nolistsep]
	\item in~$\calF_{\rm out} \cap {\rm Ann}(R,2R)$, 
	$P_1^{\rm out}$ is connected to~$P_3^{\rm out}$ by a primal path, 
	\item~$P_2^{\rm out}$ is connected to~$\La_{R}$ by a dual path, and
	\item~$P_4^{\rm out}$ is connected to~$\La_{2R}^c$ by a dual path.
	\end{itemize}
	Notice that, if~$H$ occurs, the primal path connecting~$P_1^{\rm out}$ to~$P_3^{\rm out}$ needs to ``wind around''~$\La_{R}$. 
	Thus,~$H$ may be understood as the connection between~$ P_1^{\rm in}$ and~$P_3^{\rm in}$ in~$\calF_{\rm in}$ being ``pivotal'' for~$A_{R}$ (for~$\omega$). 
	By Theorem ~\ref{thm:RSWnear}, 
	\begin{align*}
	\bbP_{\bfT}[\omega \in H\,| \, (\omega_{[\tau_2]},\omega'_{[\tau_2]})\text{ such that~$F$ occurs}] > c_4, 
	\end{align*}
	for some~$c_4 > 0$. 
	Moreover, when~$F$ and~$H$ occur, then~$A_{R}$ occurs for~$\omega'$, but not for~$\omega$. 
	Thus,
	\begin{align*}
		\phi^{\xi^1}_\calF[A_{R}]-\phi^{\xi^0}_\calF[A_{R}]
		=\bbP_{\bf T}[\omega\in H \text{ and } F] \ge c_3\,c_4 >0, 
	\end{align*}
	which is the desired conclusion. 
	\end{proof}
		\begin{figure}
	\begin{center}
	\includegraphics[width=0.55\textwidth]{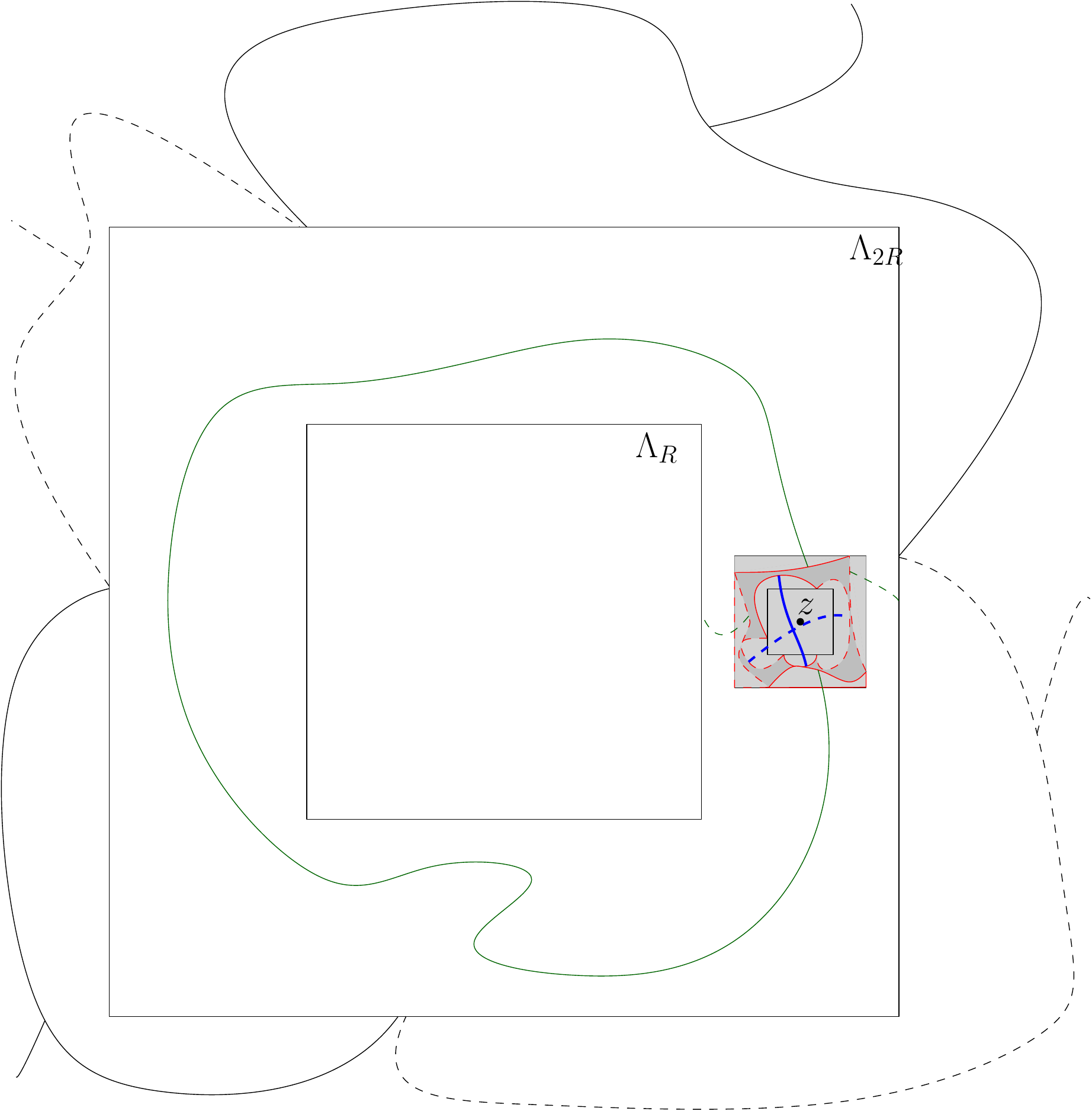}
	\caption{A depiction of the annulus configuration in the proof of \eqref{eq:boosting_pair_circ_i}. In black, the flower domain~$\calF$. The red part corresponds to the double flower-domain~$(\calF_{\rm in},\calF_{\rm out})$ explored at time~$\tau_1$. The blue part corresponds to what is explored at time~$\tau_2$, i.e.~a connection in~$\omega'$ between two wired petals of~$\calF_{\rm in}$ together with a dual connection in~$\omega^*$ between two free petals. Then after time~$\tau_2$, we construct the event~$H$ (in green). }
	\label{fig:h2}
	\end{center}
	\end{figure}
	
	\begin{remark}
		The proof of~\eqref{eq:boosting_pair_circ_i} may appear contradictory, 
		as we are first arguing that~$P_1^{\rm out}$ and~$P_3^{\rm out}$ may appear wired in~$\omega'$ but not in~$\omega$, 
		then we focus on the event~$\omega \in H$ which ensures that~$P_1^{\rm out}$ and~$P_3^{\rm out}$ are connected in both~$\omega$ and~$\omega'$. 
		The reader should keep in mind that the configurations~$\omega$ and~$\omega'$ in~$\calF_{\rm in}$ are sampled before sampling the configurations in~$\calF_{\rm out}$, 
		and their laws are obtained by averaging over the possible configurations in~$\calF_{\rm out}$. 
		
		Alternatively, one may imagine a coupling where~$(\calF_{\rm in}, \calF_{\rm out})$ is explored first, 
		then the configurations in~$\calF_{\rm out}$ are revealed, then those in~$\calF_{\rm in}$
		and finally the configurations in~$\calF_{\rm out}$ are resampled. 
		In this context, we are investigating the situation where, 
		in the first sampling of the configurations in~$\calF_{\rm out}$, 
		$P_1^{\rm out}$ and~$P_3^{\rm out}$ are connected to~$P_i$ and~$P_j$, respectively, but not to each other, 
		then, in the sampling inside~$\calF_{\rm in}$,~$P_1^{\rm in}$ and~$P_3^{\rm in}$ are connected in~$\omega'$, but not in~$\omega$, 
		and finally, in the second sampling in~$\calF_{\rm out}$,~$H$ occurs for~$\omega$. 
		
	\end{remark}

\subsection{Crossing quads produce boosting pairs}

This section is concerned with the following result which roughly states that 
conditioning on the existence of the crossing of a quad has the same effect as a boosting pair of boundary conditions. 

\begin{proposition}\label{prop:cross_is_boosting}
	For any~$\eta > 0$, there exists a constant~$c = c(\eta,q) > 0$ such that the following holds. 
	Fix~$p$ and let $(\calD,a,b,c,d)$ be an~$\eta$-regular discrete quad at scale~$R\le L(p)$. 
	Then there exists a coupling via decision trees~$\bbP$ of~$\phi_{p}$ and~$\phi_{p}[\cdot\,|\calC(\calD)]$,	
	and a stopping time~$\tau$ such that, when~$\tau < \infty$,
	$\calF_\tau = \bbZ^2\setminus e_{[\tau]}$ is a~$1/2$-well-separated outer flower domain on~$\La_{2R}$, 
	and the boundary conditions induced by~$\omega'_{[\tau]}$ on~$\calF_\tau$ are a boost of those induced by~$\omega_{[\tau]}$. 
	Finally, 
	\begin{align*}
	\bbP[\tau < \infty] \geq c.
	\end{align*}
\end{proposition}

\begin{remark}\label{rem:circ_is_boosting}
	We may replace~$\calC(\calD)$ by~$A_{R/2}$ in the statement above. 
	Indeed, if we set~$x = (0,3R/4)$, observe that~$A_{R/2}$ induces a vertical crossing of the~$\eta$-regular discrete quad~$\La_{R/4}(x)$ 
	(with~$a,b,c,d$ the corners of~$\La_{R/4}(x)$, starting with the top-left one). 
	As a consequence, the measure~$\phi_{p}[\cdot\,|A_{R/2}]$ dominates~$\phi_{p}[\cdot\,|\calC(\La_{R/4}(x))]$.
	Now, apply the decision tree provided by Proposition~\ref{prop:cross_is_boosting} to couple in an increasing fashion the measures 
	$\phi_p$,~$\phi_{p}[\cdot\,|\calC(\La_{R/4}(x))]$ and~$\phi_{p}[\cdot\,|A_{R/2}]$; call~$\omega \le \omega' \le \omega''$ the resulting configurations. 
	Then, if~$\tau < \infty$, 
	the boundary conditions induced by ~$\omega''_{[\tau]}$ on~$\calF_\tau$ dominate those induced by ~$\omega'_{[\tau]}$,
	and hence are a boost of those induced by~$\omega_{[\tau]}$.
\end{remark}

The proof of Proposition~\ref{prop:cross_is_boosting} is based on the following lemma, 
which allows us to ``lengthen'' crossings at a small cost.
For an~$\eta$-regular discrete quad~$(\calD,a,b,c,d)$ at some scale~$R$ and for some $m \le \eta R/4$, 
define two modified quads~$(\calD^+,a^+,b^+,c^+,d^+)$ and~$(\calD^-,a^-,b^-,c^-,d^-)$ as follows (see Figure~\ref{fig:D+-} for an illustration). 
The domains~$\calD^+$ and~$\calD^-$ are formed by the union of~$\calD$ 
with the set of edges of~$\bbZ^2 \setminus \calD$ that are at a~$\ell^\infty$ distance at most~$m$ from
the arcs~$(ab)$ and~$(cd)$ (respectively~$(bc)$ and~$(da)$),  
but at least distance~$m$ from~$a$,~$b$,~$c$ and~$d$. 
The point~$a^+$ is the first point of~$\partial \calD^+$ in counter-clockwise order after~$a$ 
that is at a distance~$m$ from~$\partial \calD$; the point~$b^+$ is the last such point before~$b$. 
The points~$c^+$ and~$d^+$ are defined similarly in terms of~$c$ and~$d$ respectively. 
A similar definition applies to~$a^-$,~$b^-$,~$c^-$ and~$d^-$. 
Notice that $(\calD^+,a^+,b^+,c^+,d^+)$ and~$(\calD^-,a^-,b^-,c^-,d^-)$ are both discrete quads and that 
\begin{align*}
\calC(\calD^+) \subset \calC(\calD) \subset \calC(\calD^-).
\end{align*}
 
\begin{figure}
    \begin{center}
        \includegraphics[width = 0.34\textwidth, page = 1]{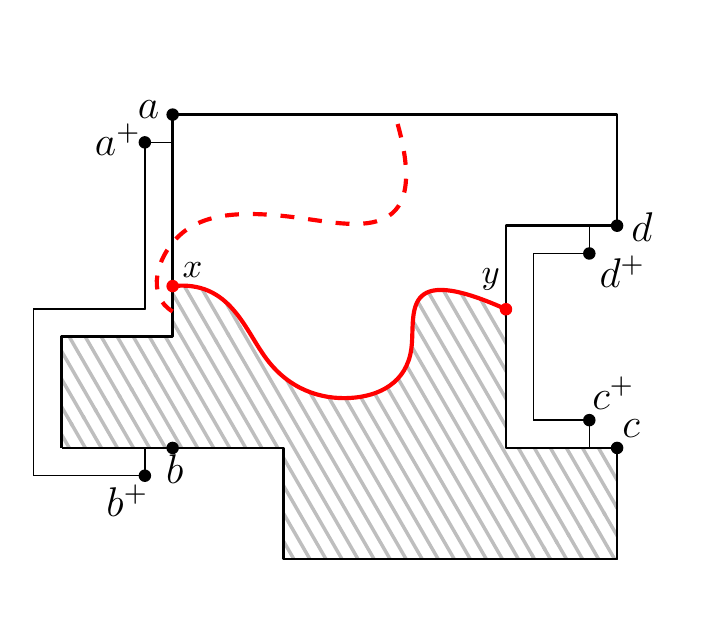}
        \includegraphics[width = 0.34\textwidth, page = 2]{D+-.pdf}
                \includegraphics[width = 0.3\textwidth, page = 3]{D+-.pdf}
        \caption{{\em Left:} For the lowest crossing~$\Gamma$ of~$\calD$ not to be connected to~$(a^+b^+)$, 
            there needs to exist a dual path from~$\La_m(x)$ to the arc~$(d^+a^+)$ in the unexplored region above~$\Gamma$.   \newline
            {\em Middle:} At the time~$\tau_2 < \infty$, when~$\Gamma^+$ and~$\Gamma^-$ have been revealed, 
            the region of~$\calD^+\setminus \calD$ above~$\Gamma^+$ and that of~$\calD^-\setminus \calD$ left of~$\Gamma^-$ are unexplored. \newline
            {\em Right:} The unexplored regions may be used to connect~$\Gamma^+$ to the primal petals of~$\calF$ in~$\omega$ 
            and~$\Gamma^-$ to the dual petals of~$\calF$ in~$\omega^*$. This ensures that the boundary conditions induced by~$\omega'_{[\tau]}$ on~$\calF$
            are a boost of those induced by~$\omega_{[\tau]}$.}
		\label{fig:D+-}
    \end{center}
\end{figure}
 
 \begin{remark}
 The choice of~$m\le \eta R/4$ and the fact that~$\calD$ is~$\eta$-regular guarantee that~$\calD^+\setminus\calD$ and~$\calD^-\setminus\calD$ 
are at a distance at least~$\eta R/4$ of each other and are each made of two separate connected components. In particular, the complement of~$\calD^+\cup\calD^-$ is connected.
 \end{remark}

\begin{lemma}\label{lem:lengthen_crossings}
	For any~$\eta > 0$, there exist~$C=C(\eta)>0$, $\eps=\eps(\eta) > 0$ such that the following holds. 
	For any~$p \in (0,1)$,~$R \le L(p)$, $m < \eta R/4$
	and any~$\eta$-regular discrete quad~$(\calD,a,b,c,d)$ at scale~$R \le L(p)$.
	\begin{align}\label{eq:lengthen_crossings}
		\phi_{p}[\calC(\calD)] - \phi_{p} [\calC(\calD^+)]  &\le C (m/R)^{\eps},\\
				\phi_{p}[\calC(\calD^-)] - \phi_{p} [\calC(\calD)]  &\le C (m/R)^{\eps}.
	\end{align}
\end{lemma}

\begin{proof}
	We will focus on the first inequality; the second one is the same inequality applied to the dual model. 
	Since~$\calC(\calD^+) \subset \calC(\calD)$, we are searching for an upper bound on 
	$\phi_{p}[\calC(\calD) \setminus \calC(\calD^+)]$. 
	When~$\calC(\calD)$ occurs, let~$\Gamma$ be the ``lowest'' crossing of~$\calD$, that is the open path closest to the arc~$(bc)$, 
	with endpoints~$x \in (ab)$ and~$y \in (cd)$. Write~$\mathsf{Under}(\Gamma)$ for the set of edges of~$\calD$ between~$\Gamma$ and~$(bc)$. 
	
 	If~$\calC(\calD) \setminus \calC(\calD^+)$ occurs, then at least one of the following four events needs to occur
	\begin{itemize}
	\item[(i)]~$x$ is at a distance at most~$\sqrt {mR}$ from~$a$;
	\item[(ii)]~$x$ is at a distance at least~$\sqrt {mR}$ from~$a$, but~$\Gamma$ is not connected to~$(a^+b^+)$ in~$\calD^+$;
	\item[(iii)]~$y$ is at a distance at most~$\sqrt {mR}$ from~$d$;
	\item[(iv)]~$y$ is at a distance at least~$\sqrt {mR}$ from~$d$, but~$\Gamma$ is not connected to~$(c^+d^+)$ in~$\calD^+$.
	\end{itemize} 
	Next we bound the probability of each of the four events described above. 
	
	If the event in (i) occurs, there exists a primal arm, namely~$\Gamma$, contained in~$\calD$ from~$\La_{\sqrt {mR}}(a)$ to distance~$\eta R$; and therefore using \eqref{eq:LOWER_BOUND_ONE_ARM} we get
	\[
	\phi_p[(i)\text{ occurs}]\le \pi_1(p; \sqrt {mR}, \eta R) \le C(\eta) (m/R)^\eps.
	\]
		
	If the event in (ii) occurs, then there exists a dual-open path from~$\La_{m}(x)$ to~$(d^+a^+)$ contained in~$\calD^+ \setminus \mathsf{Under}(\Gamma)$; see Figure~\ref{fig:D+-} (left diagram). 
	It is a standard consequence of Theorem~\ref{thm:RSWnear} (the near-critical RSW) that there exist~$\eps, C > 0$ such that, 
	for any realisation of~$\Gamma$ such that (i) fails,
	\begin{align*}
		\phi_{p}[ \La_{m}(x)\xlra{\omega^* \cap (\calD^+ \setminus \mathsf{Under}(\Gamma))}(da)\,|\, \Gamma \text{ and } \omega \text{ on } \mathsf{Under}(\Gamma)] 
		\le C (m/R)^\eps.
	\end{align*}
	Indeed, notice that the conditioning above induces both positive and negative information on~$ \mathsf{Under}(\Gamma)^c$.
	However, due to the~$\eta$-regularity of~$\calD$, the information that is favorable to the existence of dual connections is at every scale around $x$ at a macroscopic distance from 
	$\La_{m}^c(x) \cap (\calD^+ \setminus \mathsf{Under}(\Gamma))$; see Figure~\ref{fig:D+-} (left diagram).
	
	The bounds above also apply to the events in (iii) and (iv). When combined using a union bound, we find
	$$\phi_{p}[\calC(\calD) \setminus \calC(\calD^+)]  \le 4 C (m/R)^{\eps}.$$
\end{proof}

\begin{remark}
	Alternatively, one may use bounds on the probability of the three-arm event in the half-plane to prove the above. 
	We prefer the strategy above as it adapts easily to fractal quads of bounded extremal distance; 
	the cases (i) and (ii) (and (iii) and (iv), respectively) should then be distinguished using extremal distance rather than the geometric average of~$m$ and~$R$.
	Above, the quad~$\calD$ is assumed to be~$\eta$-regular simply for convenience. 
\end{remark}

\begin{proof}[Proposition~\ref{prop:cross_is_boosting}]
	Fix~$\eta > 0$ and let~$C=C(\eta)$ and~$\eps=\eps(\eta)$ be the constants given by Lemma~\ref{lem:lengthen_crossings} for this value of~$\eta$. 
	Below,~$c_0,\dots,c_3$ denote strictly positive constants depending only on~$\eta$. 
	
	For~$p$ and~$R$ with~$R \le L(p)$, let~$m = c_0 R$ be such that 
	$$ C(m/R)^\eps < \tfrac14 \inf_{\calD}	\phi_{p}[\calC(\calD)]\phi_{p}[\calC(\calD)^c],$$
	where the infimum is taken over all~$\eta$-regular quads~$\calD$ at scale~$R$. 
	By Theorem~\ref{thm:RSWnear}, the infimum is uniformly positive, and therefore~$c_0 >0$ above does indeed depend only on~$\eta$. 
	
	Fix~$\calD$ as in the statement of the proposition. 
	Then, in any increasing coupling~$\bbP$ of~$\phi_{p}[\cdot]$ and~$\phi_{p}[\cdot\,|\calC(\calD)]$ (note that~$\calC(\calD)$ occurs automatically for~$\omega'$),
	\begin{align*}
	\bbP[\omega' \in \calC(\calD^+), &\,\omega \notin \calC(\calD^-)]
\\
	& \geq 1-\phi_p[\calC(\calD)]- \phi_p[\calC(\calD^-) \setminus \calC(\calD)] - \phi_p[ \calC(\calD^+)^c\setminus\calC(\calD)^c \,|\, \calC(\calD)] \\
	& \geq \phi_p[\calC(\calD)^c] -   C(m/R)^\eps - \frac{C(m/R)^\eps}{\phi_{p}[\calC(\calD)]}  \\
	& \geq \tfrac12\phi_p[\calC(\calD)^c] \geq c_1 >0.
	\end{align*}

	Next, we create an increasing coupling~$\bbP$ between~$\phi_{p}$ and~$\phi_{p}[\cdot\,|\calC(\calD)]$ using a specific decision tree described below. 
	Start by exploring the lowest crossing~$\Gamma^+$ in~$\calD^+$ from~$(a^+b^+)$ to~$(c^+d^+)$ (that is the crossing closest to~$(b^+c^+)$) 
	in the larger configuration~$\omega'$. 
	Write~$\tau_1$ for the stopping time when this crossing is found; set~$\tau_1 = \infty$ if no such crossing exists. 
	If~$\tau_1 < \infty$, explore the ``right-most'' dual crossing~$\Gamma^-$ in~$\calD^-$ from~$(b^-c^-)$ to~$(d^-a^-)$ 
	(that is the one closest to~$(c^-d^-)$) in the configuration~$\omega^*$;
	define~$\tau_2$ for the stopping time when this crossing is found (if~$\tau_1 = \infty$ or no such crossing is found, set~$\tau_2 = \infty$). 
	Notice that 
	\begin{align*}
		\bbP[\tau_2 < \infty]
		= \bbP[\omega' \in \calC(\calD^+),\omega \notin \calC(\calD^-)]
		\geq c_1.	
	\end{align*}
	Assuming that~$\tau_2 < \infty$, the explored edges~$e_{[\tau_2]}$ are those of~$\calD^+$ below~$\Gamma^+$ and those of~$\calD^-$ right or~$\Gamma^-$. 
	In particular, the edges of~$\calD^+\setminus\calD$ that are above~$\Gamma^+$, 
	as well as those of~$\calD^- \setminus \calD$ that are left of~$\Gamma^-$ are unexplored. See Figure~\ref{fig:D+-} (middle diagram). 
	
	Next, explore the double four-petal flower domain~$(\calF_{\rm in}, \calF_{\rm out})$ between~$\La_{3R/2}$ and~$\La_{2R}$; 
	let~$\tau_3$ be the stopping time marking the end of this exploration 
	(with~$\tau_3 = \infty$ if no double four-petal flower domain exists or if~$\tau_2 = \infty$). 
	Due to Lemma~\ref{lem:D4PFD}, 
	\begin{align*}
		\bbP[\tau_3<\infty \,|\,(\omega_{[\tau_2]}, \omega'_{[\tau_2]}) \text{ s.t. } \tau_2 < \infty] \geq c_2.
	\end{align*}
	
	Finally, reveal the configurations in the unexplored regions of~$\calF_{\rm in}$ and write~$\tau_4$ for the stopping time marking the end of this stage. 
	Let~$H$ be the event that~$P_1^{\rm in}$ and~$P_3^{\rm in}$ are connected by paths of~$\omega' \cap (\calF_{\rm in}^c \setminus \calD)$ to~$\Gamma^+$
	and ~$P_2^{\rm in}$ and~$P_4^{\rm in}$ are connected by paths of~$\omega^* \cap (\calF_{\rm in}^c \setminus \calD)$ to~$\Gamma^-$.
	Due to  Theorem~\ref{thm:RSWnear} (see Figure~\ref{fig:D+-}, right diagram), 
	\begin{align*}
		\bbP[H \,|\,(\omega_{[\tau_3]}, \omega'_{[\tau_3]}) \text{ s.t. } \tau_3 < \infty] \geq c_3.
	\end{align*}
	Special care should be taken as the primal connections occur in~$\omega'$ while the dual ones occur in~$\omega$.
	This may be easily overcome by considering connections in pre-determined disjoint regions\footnote{
	One may be tempted to ask for both the primal and dual connections to occur in the same configuration, for instance in~$\omega$. 
	This would be conceptually simpler, 
	but would require a stronger RSW result, as the sections of~$\Gamma^+$ outside of~$\calD$ are wired in~$\omega'$, but not in~$\omega$. This stronger RSW result is true for~$q<4$ (see~\cite{DumManTas20}), but it is expected to be wrong for~$q=4$.}. 

	Now, if~$\tau_3 < \infty$ and~$H$ occurs,  
	then~$P_1^{\rm out}$ and~$P_3^{\rm out}$ are disconnected in~$\omega \cap \calF_{\rm out}^c$, 
	but are connected in~$\omega'\cap \calF^c_{\rm out}$. 
	Set~$\tau = \tau_4$ if the above occurs and~$\tau = \infty$ otherwise. 
	Then~$\tau$ satisfies the requirements of the proposition and 
	$$ \bbP [\tau < \infty] \geq c_1c_2c_3>0.$$	
\end{proof}

We conclude this section by the following lemma, which is straightforward application of Theorem~\ref{thm:RSWnear}. 
This lemma will be useful in the future sections.

\begin{lemma}\label{lem:boosting_out_to_boosting_in}
	For any~$\eta > 0$, there exists a constant~$c = c(\eta) > 0$ such that the following holds. 
	Fix~$p$ and~$\calF$ an~$\eta$-well-separated outer flower domain on~$\La_{R}$ for some~$R\le L(p)$. 
	Let~$\xi,\xi'$ be a boosting pair of boundary conditions on~$\calF$ and~$x$ be a point of~$\Ann(2R,4R)$. 
	Then, there exists a coupling via decision trees~$\bbP$ of~$\phi_{\calF,p}^\xi$ and~$\phi_{\calF,p}^{\xi'}$
	and a stopping time~$\tau$ such that, when~$\tau < \infty$,
	$\calF_\tau = \calF\setminus e_{[\tau]}$ is a~$1/2$-well-separated inner flower domain on~$\La_{R/4}(x)$, 
	and the boundary conditions induced by~$\omega'_{[\tau]}$ on~$\calF_\tau$ are a boost of those induced by~$\omega_{[\tau]}$. 
	Finally, 
	\begin{align*}
	\bbP[\tau < \infty] \geq c.
	\end{align*}
\end{lemma}
The proof is similar (yet much easier) to the one of Theorem~\ref{thm:boosting_pair}. 
\begin{proof}
	 Write~$P_i$,~$P_j$ for two petals of~$\calF$ that are wired in~$\xi'$, but not in~$\xi$. 
	
	Start by exploring the double four-petal flower domain~$(\calF_{\rm in},\calF_{\rm out})$ between~$\La_{R/4}(x)$ and~$\La_{R/2}(x)$.
	If no such double flower domain exists, set~$\tau = \infty$ and proceed in an arbitrary way. 
	If~$(\calF_{\rm in},\calF_{\rm out})$ exists, continue by revealing the configurations in~$\calF \cap \calF_{\rm out}$. 
	Write~$H$ for the event that in~$\omega \cap \calF_{\rm out}$, 
	$P^{\rm out}_1$ is connected to~$P_i$,~$P^{\rm out}_3$ is connected to~$P_j$, but that~$P^{\rm out}_1$ and~$P^{\rm out}_3$
	are not connected to each other, nor to any other petal of~$\calF$. 
	Theorem~\ref{thm:RSWnear} implies that 
	\begin{align}\label{eq:hhHH}
	\bbP[H \,|\,\calF_{\rm in},\calF_{\rm out} ] \geq c, 
	\end{align}
	where~$c > 0$ depends only on~$\eta$. 
	If~$H$ occurs, set~$\tau$ to be the stopping time at which the configurations in~$\calF_{\rm out}$ have been revealed; 
	otherwise set~$\tau = \infty$.
	
	Then due to Lemma~\ref{lem:D4PFD} and \eqref{eq:hhHH},~$\bbP[H ] \ges c$. 
	Finally, when~$\tau < \infty$, 
	the boundary conditions induced by~$\omega'_{[\tau]}$ on~$\calF_\tau = \calF_{\rm in}$ are indeed a boost of those induced by~$\omega_{[\tau]}$, since~$P^1_{\rm in}$ is connected to~$P^3_{\rm in}$  in~$\omega'_{[\tau]}$, but not in~$\omega_{[\tau]}$.
\end{proof}

\section{Properties of the mixing rate}\label{sec:coupling_Delta}\label{sec:10}

Fix~$q \in (1,4]$ and~$\eta > 0$ for the whole of this section; 
all constants, including those in~$\les$ and~$\asymp$, may depend on~$\eta$.  
In this section, we always work with a single edge-parameter~$p\in(0,1)$, 
and therefore omit it often from the measure~$\phi_{G,p}$ for notational convenience.

\subsection{No coupling induces boosting boundary conditions}

The main result of this subsection is be the following; 
it is the cornerstone to the other results proved later in this and other sections. 

\begin{theorem}\label{thm:coupling}
	For any~$\eta >0$,~$p$, and~$4r < R \le L(p)$.
	\begin{itemize}
	\item[(i)] Let~$\calG$ be an~$\eta$-well-separated inner flower domain on~$\La_R$ and~$\xi,\xi'$ be a boosting pair of boundary conditions on~$\calG$. 
	There exists an increasing coupling~$\bbP$ of~$\phi_{\calG,p}^\xi$ and~$\phi_{\calG,p}^{\xi'}$ on~$\calG \setminus \La_r$ via decision trees, 
	and a stopping time~$\tau$ with the following property.  
	When~$\tau <\infty$,~$\calF_\tau = \calG \setminus e_{[\tau]}$ is a~$1/2$-well-separated inner flower domain on~$\La_r$, 
	and the boundary conditions induced by~$(\omega'_{[\tau]})^{\xi'}$ on~$\calF_\tau$ are a boost of those induced by~$\omega_{[\tau]}^\xi$. 
	Moreover, if we write~$\zeta$ and~$\zeta'$ for the boundary conditions induced by~$\omega'$ and~$\omega$ on~$\partial \La_r$, then 
	\begin{equation}\label{eq:eta boost in} 
\bbP[\tau < \infty] 
		\asymp \bbP[\zeta \neq\zeta' ]. 
	\end{equation}
	\item[(ii)] 	Let~$\calG$ be an~$\eta$-well-separated outer flower domain on~$\La_r$ and~$\xi,\xi'$ be a boosting pair of boundary conditions on~$\calG$. 
	There exists a increasing coupling~$\bbP$ of~$\phi_{\calG,p}^\xi$ and~$\phi_{\calG,p}^{\xi'}$ on~$\calG \cap \La_R$ via decision trees, 
	and a stopping time~$\tau$ with the following property.  
	When~$\tau <\infty$,~$\calF_\tau = \calG\setminus e_{[\tau]}$ is a~$1/2$-well-separated outer flower domain on~$\La_R$, 
	and the boundary conditions induced by~$(\omega'_{[\tau]})^{\xi'}$ on~$\calF_\tau$ are a boost of those induced by~$\omega_{[\tau]}^\xi$. 
	Moreover, if we write~$\zeta$ and~$\zeta'$ for the boundary conditions induced by~$\omega'$ and~$\omega$ on~$\partial \La_R$, then 
	\begin{equation}\label{eq:eta boost in} 
	\bbP[\tau < \infty] 
		\asymp \bbP[\zeta \neq\zeta' ]. 
\end{equation}
	\end{itemize}
\end{theorem}

\begin{remark}
In the previous theorem we sample edges only outside $\Lambda_r$ in Case (i), and only inside $\Lambda_R$ in Case (ii), but, of course, once this is done, one may use an arbitrary coupling inside $\Lambda_r$ or outside $\Lambda_R$ to obtain couplings in the whole flower domain $\calG$. We stated the previous theorem in this setting to be able to reuse the coupling in applications we have in mind.
\end{remark}

The first point of the theorem should be understood as follows. 
We reveal the configurations~$\omega$ and~$\omega'$ in the coupling~$\bbP$ starting from the outside and moving inwards; 
while doing this, we follow the difference between the boundary conditions that the configurations induce on the unexplored region. 
If this difference survives until the whole of~$\calG \setminus \La_r$ is revealed, 
then there is a positive probability that it survives as a significant difference, namely in the form of a boosting pair of boundary conditions on a well-separated flower domain. 

The second point is analogous, with the revealment starting inside and moving outwards.

\begin{proof}[Theorem~\ref{thm:coupling}]
	We will only prove point (i); the proof of point (ii) is identical.
	Fix~$p \in (0,1)$,~$4r \le  R \le L(p)$ and~$\calG$,~$\xi$ and~$\xi'$ as in the statement. 
	All constants below are independent of~$r$~$R$,~$\calG$,~$\xi$ and~$\xi'$, unless explicitly stated. 
	When referring to connections in configurations~$\omega$ and~$\omega'$ 
	with laws ~$\phi_{\calG,p}^\xi$ and~$\phi_{\calG,p}^{\xi'}$, respectively, 
	we will implicitly include connections that use the boundary conditions. 
	In other words, we omit the superscript in the notation~$\omega^\xi$ and~$(\omega')^{\xi'}$.

	First, for any increasing coupling~$\bbP$ between~$\phi_\calG^\xi$ and~$\phi_\calG^{\xi'}$ obtained by a decision tree, 
	and any stopping time~$\tau$ with the properties of the theorem, 
	$\bbP[\tau < \infty] \le \bbP[\zeta \neq \zeta']$.
	Indeed, the requirement that the boundary conditions induced by~$\omega'_{[\tau]}$ on~$\calF_\tau$ are a boost of those induced by~$\omega_{[\tau]}$, 
	imposes that the boundary conditions induced by~$\omega$ and~$\omega'$ on~$\La_{r}$ are distinct. 
	The rest of the proof is dedicated to the converse bound.
	\medbreak

	Assume for simplicity that~$R = 4^k r$ for some integer~$k \geq 2$.
	By monotonicity, it suffices to treat the case where~$\xi$ and~$\xi'$ are identical, 
	with the exception of two petals that are wired together in~$\xi'$, but not in~$\xi$. 
	Assume this is the case, and write $P_1$ and~$P_3$ for the two primal petals of~$\calG$ that are wired together in~$\xi'$, but not in~$\xi$
	(contrary to what the notation suggests,~$P_1$ and~$P_3$ need not be separated by a single dual petal).
		
	Below, we describe an increasing coupling~$\bfP$ of  
	$\phi_\calG^\xi$ and~$\phi_\calG^{\xi'}$ on~$\calG \cap \La_r^c$ obtained through a decision tree. 
	The actual coupling of the theorem is a slight variation of~$\bfP$ that we will describe at the end of the proof.  
	
	At any time~$t$, write~$\calF_t$ for the connected component of~$\La_r$ in~$\calG \setminus e_{[t]}$.
	We will abusively consider that all edges of~$\calG$ that are not revealed at time~$t$ 
	and which are not part of~$\calF_t$, are revealed instantaneously. 
	Thus, from now on, we have~$\calF_t = \calG \setminus e_{[t]}$.
	We will identify~$\omega_{[t]}$ and~$\omega_{[t]}'$ to the boundary conditions they induce on~$\calF_t$. 
	
	Write~$T$ for the  first time~$t$ when~$\omega_{[t]}$ and~$\omega'_{[t]}$ induce the same boundary conditions on~$\calF_t$,
	and~$T = \infty$ if the boundary conditions are never the same. 
	If~$T <\infty$, the configurations~$\omega,\omega'$ in~$\calF_{T}$ are identical, regardless of the decision tree used after~$T$. 
	Thus, when this occurs, reveal the rest of the edges using lexicographical order.

	The coupling proceeds in several stages numbered~$j = 0,\dots, k$. 
	If at any point~$T$ occurs, the procedure described below stops, and the revealment by lexicographical order is used. 
	Stage~$j$ corresponds to revealing information in~$\La_{4^{k-j} r}^c := \calG \setminus \La_{4^{k-j} r}$.
	We will write~$\tau_j$ for the stopping time that marks the end of stage~$j$.
	At time~$\tau_j$, the revealed edges are those of the cluster of~$P_1$ and~$P_3$ in~$\omega' \cap \La_{4^{k-j} r}^c$, 
	any edges of~$\La_{4^{k-j} r}^c$ adjacent to this cluster, 
	as well as any edges separated from~$\La_r$ by the two categories of edges mentioned above. 
	
	Let us now describe precisely the revealment algorithm. \bigbreak 
	
	\begin{center} Revealment algorithm\end{center}
	\bigbreak
	\noindent{\bf Stage 0:}
	Reveal the cluster of~$P_1$ and~$P_3$ inside~$\omega' \cap \La_R^c$  in arbitrary order. 
	Let~$\tau_{0}$ be the stopping time marking the end of this stage.  \smallskip 

	\noindent{\bf Stage~$j+1$:}	
	Fix~$j\geq0$ and assume the coupling defined up to time~$\tau_j$ (see Figure~\ref{fig:exploration} for an illustration). Stage~$j+1$ is itself formed of two steps. 
	As already mentioned, this is only valid when~$T > \tau_{j}$.
	Write~$\rho = 4^{k-j-1} r$. 
	In~$\omega'_{[\tau_j]}$, the boundary of~$\calF_{\tau_j}$ is formed of dual arcs outside of~$\La_{4\rho}$ 
	with endpoints on~$\partial \La_{4\rho}$, 
	along with points on~$\La_{4\rho}$ that are connected in~$\omega'_{\tau_j}$ to~$P_1$ or~$P_3$. 
	Call the latter points the ``wired'' points of~$\partial \calF_{\tau_j}$.
	Since~$T$ has not yet occurred, there exists at least one wired point on~$\partial \calF_{\tau_j}$. 
	
	\begin{figure}
	\begin{center}
	\includegraphics[width=0.3\textwidth, page =2]{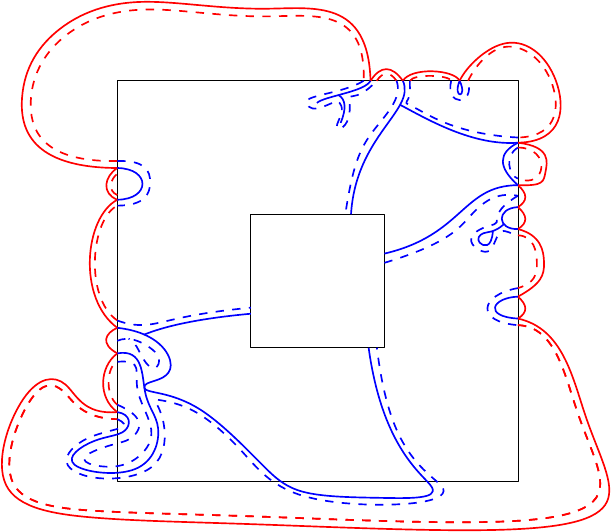}\hspace{0.02\textwidth}
	\includegraphics[width=0.3\textwidth, page =3]{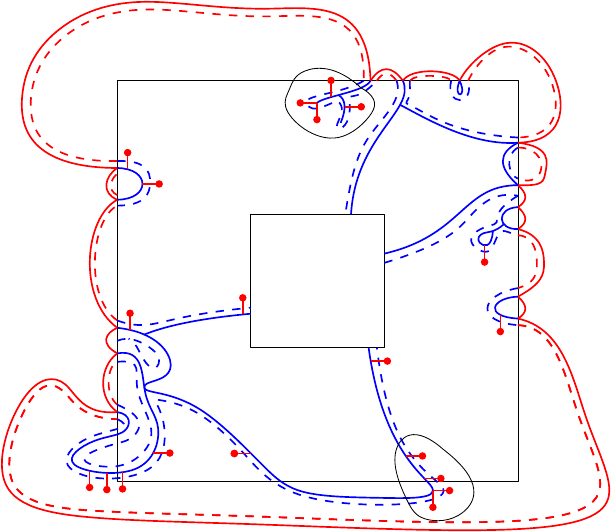}\hspace{0.02\textwidth}
	\includegraphics[width=0.3\textwidth, page =4]{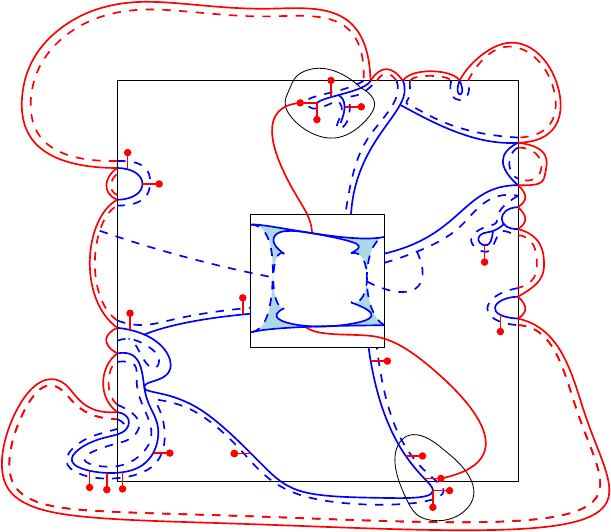}\hspace{0.02\textwidth}
	\caption{Red solid lines represent open edges in~$\omega'$, while blue ones represent open edges in~$\omega$. 
		Dotted lines represent closed (or equivalently open dual edges) in the configurations corresponding to their color. 
		{\em Left:} At the end of stage~$j$, the revealed edges are those of the cluster of~$P_1$ and~$P_3$ in~$\omega' \cap \La_{4\rho}^c$ 
		and its boundary. 
		Then, by time~$\tau_{j+1/2}$, the interfaces in~$\omega$ starting at the wired vertices previously exposed are explored until they touch~$\La_{2\rho}$.
		At this time, the unexplored region is a flower domain that is likely well-separated. 
		{\em Middle:} Assuming that the primal petals of~$\calF_{\tau_{j+1/2}}$ are all wired together in~$\omega$, 
		we turn our attention to the set~$\calA$ of points that lie on dual petals of~$\calF_{\tau_{j+1/2}}$, but which are connected to~$P_1$ or~$P_3$
		in~$\omega'_{[\tau_{j+1/2}]}$. If this set has a positive probability to be connected to~$\La_\rho$, then we call~$\tau_{j+1/2}$ promising. 
		{\em Right:} For~$\tau_{j+1/2}$ promising, we may connect two separate regions of~$\calA$ to the two primal external petals of a double flower-domain at a smaller scale. Then these petals will be connected in~$\omega'$ but not in~$\omega$.}
	\label{fig:exploration}
	\end{center}
	\end{figure}
	
	First, reveal all interfaces in~$\omega \cap \calF_{\tau_j} \cap \La_{2\rho}^c$ that start at wired points of~$\partial \calF_{\tau_j}$.
	This may be done by tracking the left and right boundaries of the clusters of each wired point of~$\partial \calF_{\tau_j}$ 
	until they finish on~$\partial \calF_{\tau_j}$ or they reach~$\partial \La_{2\rho}$.
	Write~$\tau_{j+1/2}$ for the stopping time that marks the end of this step. 
	Formally, if~$T$ occurs before time~$\tau_{j+1/2}$, set~$\tau_{j+1/2} = \infty$. 
	 
	Next, explore the cluster of~$P_1$ and~$P_3$ inside~$\omega' \cap \La_\rho^c$ and write~$\tau_{j+1}$ 
	for the stopping time marking the end of this stage (set~$\tau_{j+1} = \infty$ if~$T$ occurs before the end of this stage). In fact, we will sometimes require that the clusters are explored in a certain order (see the coupling ${\rm P}$ defined by  Lemma~\ref{lem:reconstruction} for details).
	\smallskip
	
	\noindent{\bf After stage~$k$:} Assuming that~$T > \tau_k$, reveal all remaining edges in lexicographical order.
\bigbreak
	Each stopping time~$\tau_{j+1/2}$ 
	will be declared promising or not (see the precise definition below), depending on the configuration at that stage 
	and on some constant threshold~$\delta > 0$ to be fixed below.
	Fix~$0\le j <k$ and assume~$\tau_{j+1/2} < \infty$ (otherwise~$\tau_{j+1/2}$ is not promising). 
	Then~$\calF_{\tau_{j+1/2}}$ is either a simply connected domain containing~$\La_{2\rho}$ 
	and $\omega_{[\tau_{j+1/2}]}$ induces free boundary conditions on it, 
	or it is an inner flower domain on~$\La_{2\rho}$ and $\omega_{[\tau_{j+1/2}]}$ induces  coherent boundary conditions on it.
	The former occurs when  none of the revealed interfaces in the first step of stage~$j$ reaches~$\La_{2\rho}$.
	
	First, we analyse the case where~$\omega_{[\tau_{j+1/2}]}$ induces free boundary conditions~$\calF_{\tau_{j+1/2}}$. 
	Write~$\mathcal A$ for the set of vertices on~$\partial \calF_{\tau_{j+1/2}}$ that are connected to~$P_1\cup P_3$ 
	in~$\omega'_{[\tau_{j+1/2}]}$. 
	Notice that, due to the exploration procedure, 
	for any edge~$uv$ with~$u \in \partial \calF_{\tau_{j+1/2}}$ and~$v \in \calF_{\tau_{j+1/2}}^c$, 
	$v$ is connected in~$\omega'_{[\tau_{j+1/2}]}$ to~$P_1$ or~$P_3$. 
	Thus,~$\mathcal A$ is exactly the set of vertices~$u \in \partial \calF_{\tau_{j+1/2}}$  
	for which there exists an explored edge~$uv$ with~$v \in \calF_{\tau_{j+1/2}}^c$ that is open in~$\omega'_{[\tau_{j+1/2}]}$.
	In particular, in the boundary conditions induced by~$\omega'_{[\tau_{j+1/2}]}$ on~$\calF_{\tau_{j+1/2}}$, 
	all vertices of~$\mathcal A$ are wired together and all other boundary vertices are free 
	(that is they are not wired to any other boundary vertices). 
	
	In this case, we call~$\tau_{j+1/2}$ {\em promising} if 
	\begin{align*}
		\phi_{\calF_{\tau_{j+1/2}}}^{\omega'_{[\tau_{j+1/2}]}}[\calA \lra \La_{\rho}] \geq \delta.
	\end{align*}
	
	Next, we turn our attention to the case where~$\calF_{\tau_{j+1/2}}$ is a flower domain.
	Write~$\calA$ for the set of points on the dual petals of~$\calF_{\tau_{j+1/2}}$ 
	that are connected to~$P_1$ or~$P_3$ in~$\omega'_{[\tau_{j+1/2}]}$. 
	By the same argument as in the previous case, in the boundary conditions imposed by~$\omega'_{[\tau_{j+1/2}]}$ on~$\calF_{\tau_{j+1/2}}$
 	there exists a single wired component formed of~$\calA$ along with all primal petals of ~$\calF_{\tau_{j+1/2}}$.
		
	If $\calF_{\tau_{j+1/2}}$ is not~$\delta$-well-separated, then we say that~$\tau_{j+1/2}$ is not promising. 
	If $\calF_{\tau_{j+1/2}}$ is~$\delta$-well-separated 
	and contains at least two primal petals that are not wired in~$\omega_{[\tau_{j+1/2}]}$, then we call~$\tau_{j+1/2}$ promising. 
	Finally, if~$\calF_{\tau_{j+1/2}}$ is~$\delta$-well-separated and all its primal petals are wired together in~$\omega_{[\tau_{j+1/2}]}$, 
	we call~$\tau_{j+1/2}$ promising if and only if  
	\begin{align*}
		\phi_{\calF_{\tau_{j+1/2}}}^{\omega'_{[\tau_{j+1/2}]}}[\calA \lra \La_{\rho}] \geq \delta.
	\end{align*}
	
	For formal reasons, set~$\tau_{-1/2} = 0$ and call it promising. 
	We now state two results that are instrumental in the proof of the theorem. 
	Roughly speaking, the first one states that if~$\tau_{j+1/2}$ is not promising, then it is very likely for~$T$ to arise before~$\tau_{j+1}$. 
	The second lemma states that as soon as~$\tau_{j+1/2}$ is promising, then there exists a coupling guaranteeing a fairly good probability that the flower domain on~$\partial\Lambda_r$ is~$1/2$-well-separated and the boundary conditions in~$\omega$ and~$\omega'$ induced on this domain correspond to a boosting pair.
		
	\begin{lemma}\label{lem:deconstruction}
		For any~$\eps > 0$, we may choose~$\delta=\delta(\ep) >0$ (independent of~$r$,~$R$,~$p$,~$\calG$,~$\xi$ and~$\xi'$) so that 
		for every~$0 \le j < k$,
		\begin{align*}
			\bfP[\text{$\tau_{j+1/2}$ not promising and~$T > \tau_{j+1}$}\, | \, \omega_{[\tau_{j}]},  \omega'_{[\tau_{j}]}]<\eps.
		\end{align*}
	\end{lemma}
	
	\begin{lemma}\label{lem:reconstruction}
		For any~$\delta > 0$, there exists~$c(\delta) > 0$ with the following property. 
		Fix some~$-1 \le j \le k-1$ 
		and a realisation of~$\tau_{j+1/2}$,~$\omega_{[\tau_{j+1/2}]}$,~$\omega'_{[\tau_{j+1/2}]}$ and~$\calF_{\tau_{j+1/2}}$ 
		for which~$\tau_{j+1/2}$ is promising. 
		Then there exists an increasing coupling 
		${\rm P}$ of~$\phi_{\calF_{\tau_{j+1/2}}}^{\omega_{[\tau_{j+1/2}]}}$ and~$\phi_{\calF_{\tau_{j+1/2}}}^{\omega'_{[\tau_{j+1/2}]}}$ via decision trees, 
		and a stopping time~$\sigma$ with the following property.  
		When~$\sigma <\infty$,~$\calF_\sigma$ is a~$1/2$-well-separated inner flower domain on~$\La_r$, 
		and the boundary conditions induced by~$\omega'_{[\sigma]}$ on~$\calF_\sigma$ are a boost of those induced by~$\omega_{[\sigma]}$. 
		Moreover, 
		\begin{align}\label{eq:reconstruction}
			{\rm P}[\sigma < \infty] \geq c(\delta) \pi_4(p;r,4^{k-j}r) .
		\end{align}
	\end{lemma}
	
	\begin{remark}
	
		It is essential in the second lemmata that~$c(\delta)$ is allowed to depend on~$\delta$, but that 
		it only appears as a multiplicative constant in~\eqref{eq:reconstruction}. 
		Indeed, the upper bound in \eqref{eq:reconstruction} depends on the ratio between the 
		scales of~$\calF_{\tau_{j+1/2}}$  and~$\calF_\sigma$ in a way that is uniform in~$\delta$. 
	\end{remark}

	Before proving the two lemmas, let us conclude the proof of the theorem. 
	Fix~$\eps > 0$ so that~$\pi_4(p;r,4^j r) \geq (2\eps)^{j}$ for all~$j \geq 0$ with~$4^j r \le L(p)$. 
	Due to~\eqref{eq:RSWnear}, ~$\eps$ may be chosen independently of~$r$ or~$p$. 
	Let~$\delta = \delta(\eps)$ be the quantity given by Lemma~\ref{lem:deconstruction} for this value of~$\eps$. 
	Below,~$c_0$ and~$c_1$ stand for strictly positive constants that may depend on~$\eps$, 
	but not on~$r$,~$R$ or~$\calG$. 
		
	First observe that, due to Lemma~\ref{lem:reconstruction} and Theorem~\ref{thm:boosting_pair},
	there exists~$c_0 > 0$ such that for any~$1 \le j \le k$,
	\begin{align}\label{eq:reconstruction2}
		 \bfP[T = \infty] \geq \phi_\calG^{\xi'}[\calC(\La_r)] - \phi_\calG^\xi[\calC(\La_r)] 
		 \geq  c_0 \,(2\eps)^{j}\, \bfP[\text{$\tau_{k-j-1/2}$ promising}].
	\end{align}
	Indeed, the first inequality follows from the more general observation that~$\bfP[T = \infty]$ 
	bounds the distance in total variation between the restrictions of~$\phi_\calG^{\xi'}$ and~$\phi_\calG^\xi$ to~$\La_r$. 
	For the second inequality, 
	consider the increasing coupling of~$\phi_\calG^{\xi'}$ and~$\phi_\calG^\xi$ obtained by following~$\bfP$ 
	up to the stopping time~$\tau_{k-j-1/2}$, then, if~$\tau_{k-j-1/2}$ is promising, 
	using the coupling~$\rm P$ of Lemma~\ref{lem:reconstruction} and applying Theorem~\ref{thm:boosting_pair}. 
	
	Next, we claim that there exists some~$\ell \geq 1$ such that 
	if~$T = \infty$, then with positive probability there exists one promising stopping time among the last~$\ell$ ones.  
	It is essential here that~$\ell$ is independent of~$R/r$. 	
	Indeed, fix~$\ell \geq 1$, and observe that
	\begin{align*}
		&\bfP[T = \infty \text{ and~$\tau_{k-1/2},\dots,\tau_{k-\ell+1/2}$ not promising}] \\
		&\quad  = \sum_{j = \ell}^{k} \bfP[T = \infty ;\text{ ~$\tau_{k-1/2},\dots,\tau_{k-j+1/2}$ not promising;~$\tau_{k-j-1/2}$ promising}]  \\
		&\quad  \le \sum_{j = \ell}^{k}\eps^{j-1} \bfP[\text{$\tau_{k-j-1/2}$ promising}]  \\
		&\quad  \le c_0 \, \bfP[T = \infty]\, \sum_{j = \ell}^{k}  2^{-j}
		\le  c_0 \, 2^{-\ell+1} \, \bfP[T = \infty].
	\end{align*}
	(Note that this is where we use the convention that~$\tau_{-1/2}$ is promising.)
	The first inequality is obtained by repeated applications of Lemma~\ref{lem:deconstruction};
	the second is due to~\eqref{eq:reconstruction2}.
	Thus, if we fix~$\ell \geq - \log_2  c_0 + 2$, then indeed, 
	\begin{align}\label{eq:promising}
		\bfP[\text{at least one of~$\tau_{k-1/2},\dots,\tau_{k-\ell+1/2}$ is promising} \,|\, T = \infty ] \geq \tfrac12.
	\end{align}
	
	We are now ready to define the coupling~$\bbP$ of the theorem.
	Follow~$\bfP$ up to the first promising stopping time ~$\tau_{j+1/2}$ with~$j \geq k-\ell$. 
	Then follow the coupling~$\rm P$ of Lemma~\ref{lem:reconstruction}; write~$\tau_{\text{final}}$ for the stopping time described in said lemma. 
	If~$\bfP$ does not encounter a promising stopping time~$\tau_{j+1/2}$ with~$j \geq k-\ell$, set~$\tau_{\text{final}} = \infty$. 
	Then, due to Lemma~\ref{lem:reconstruction} and to~\eqref{eq:promising},
	\begin{align*}
		\bbP[\tau_{\rm final} < \infty \,|\, T = \infty ] \geq \tfrac12 c_1 \,c_2^\ell .
	\end{align*}
	Since~$\eps$ and~$\ell$ are independent of~$r$ and~$R$, the right hand side is bounded away from~$0$
	uniformly in~$r \le R/4$. Moreover,~$\tau_{\rm final}$ satisfies the conditions stated in the theorem. 
	Multiply the above by~$\bbP[T = \infty ] = \bbP[\zeta \neq \zeta']$ to obtain the desired inequality. 
\end{proof}

\begin{proof}[Lemma~\ref{lem:deconstruction}]
	Fix~$\eps > 0$ and fix a realisation of~$\tau_j$,~$\omega_{[\tau_j]}$ and~$\omega'_{[\tau_j]}$.
	Lemma~\ref{lem:well-separated} ensures that, by choosing~$\delta > 0$ small enough, we have 
	\begin{align*}
		\bfP[\text{$\calF_{\tau_{j+1/2}}$ is a flower domain which is not~$\delta$-well-separated}\, | \, \omega_{[\tau_{j}]},  \omega'_{[\tau_{j}]}]<\eps/2.
	\end{align*}
	Suppose now that either~$\calF_{\tau_{j+1/2}}$ is a~$\delta$-well-separated flower domain, 
	or that~$\omega_{\tau_{j+1/2}}$ induced free boundary conditions on it. 
	In either case, if~$\tau_{j+1/2}$ is not promising, 
	it is because~$\calA$ has a conditionally small probability to be connected to~$\La_{\rho}$ in~$\omega'$. 
	
	Continue the coupling~$\bfP$ by revealing first the connected component~$\sfC$ of~$\calA$ in~$\omega' \cap \La_{\rho}^c$,
	then the rest of the connected component of~$P_1$ and~$P_3$.
	If~$\sfC$ does not reach~$\La_{\rho}$, then it is entirely surrounded by closed edges of~$\omega'$. 
	Thus, the boundary conditions in~$\omega$ and~$\omega'$ on the complement of~$\sfC$ are identical, 
	which is to say that~$T$ occurs before~$\tau_{j+1}$. 
	Thus
	\begin{align*}
		\bfP[T > \tau_{j+1}\,|\, \omega_{[\tau_{j+1/2}]},  \omega'_{[\tau_{j+1/2}]}] \leq
		\phi_{\calF_{\tau_{j+1/2}}}^{\omega'_{[\tau_{j+1/2}]}}[\calA \lra \La_{\rho}]< \delta.
	\end{align*}
	By choosing~$\delta > 0$ small enough, both displays may be rendered smaller than~$\eps /2$. 
	Apply the union bound to obtain the desired inequality. 
\end{proof}

\begin{proof}[Lemma~\ref{lem:reconstruction}]
	Fix~$\delta > 0$ and a realisation~$\calF$ of~$\calF_{\tau_{j+1/2}}$, with boundary conditions~$\zeta$ and~$\zeta'$ 
	induced by~$\omega_{[\tau_{j+1/2}]}$ and~$\omega'_{[\tau_{j+1/2}]}$, respectively, on~$\calF$.
	Assume that~$\calF$,~$\zeta$ and~$\zeta'$ satisfy the assumptions of the lemma. 
	There are three situations that need to be analysed; we will proceed in increasing order of difficulty. Set~$\rho:=4^{k-j}r$.
	\medskip 
	
	\noindent{\bf Case 1:~$\calF$ is a flower domain with two petals wired in~$\zeta'$ but not in~$\zeta$.}
	First assume that~$\calF$ is a~$\delta$-well-separated flower domain 
	containing two primal petals~$P_i$ and~$P_j$ that are not wired together in~$\zeta$. 
	Recall that~$P_i$ and~$P_j$ are necessarily wired together in~$\zeta'$.
	
	Then~$\rm P$ is constructed as follows. 
	Attempt to explore the double four-petal flower domain~$(\calF_{\rm in}, \calF_{\rm out})$ between~$\La_r$ and~$\La_{2r}$ 
	for the configuration~$\omega$. 
	If no such double four-petal flower domain exists set~$\sigma = \infty$ and reveal the remaining edges in lexicographical order.
	If it exists, reveal the configurations in~$\calF_{\rm out} \cap \calF$ and let~$\sigma$ be the stopping time marking the end of this stage. 
	
	Let~$H$ be the event that the double four-petal flower domain~$(\calF_{\rm in}, \calF_{\rm out})$ exists, 
	that the two primal petals~$P_1^{\rm out}$ and~$P_3^{\rm out}$ 
	of ~$\calF_{\rm out}$ are connected in~$\omega'_{[\sigma]}$ to~$P_i$ and~$P_j$, respectively, 
	while  $P_i$ and~$P_j$ are not connected to each other 
	nor to any primal petal of~$\calF$ in~$\omega_{[\sigma]}$. 
	It is a standard consequence of the separation of arms that there exists a constant~$c(\delta) > 0$ such that 
	\begin{align*}
		{\rm P}[H\,|\,(\calF_{\rm in}, \calF_{\rm out}) \text{ double four-petal flower domain}] 
		\geq c(\delta) \, \pi_4(p;r,4\rho). 
	\end{align*}
	Lemma~\ref{lem:D4PFD} ensures that the event in the conditioning has uniformly positive probability, 
	and hence so does~$H$.  
	Finally, notice that when~$H$ occurs, the boundary conditions induced by~$\omega'_{[\sigma]}$ on~$\partial \calF_{\rm in}$ are a boost of those induced by~$\omega_{[\sigma]}$. This concludes the proof in this case. 
	\medskip 

	\noindent{\bf Case 2:~$\calF$ is not a flower domain.}
	Next, assume that~$\calF$ is not a flower domain, and therefore that~$\zeta = 0$. 
	Start the coupling by attempting to reveal 
	the double four-petal flower domains~$(\calF_{\rm in}, \calF_{\rm out})$ between~$\La_{r}$ and~$\La_{2r}$
	and~$(\overline{\calF}_{\rm in}, \overline{\calF}_{\rm out})$ between~$\La_{\rho}$ and $\La_{3\rho/2}$, respectively, for the configuration~$\omega$ (if~$\rho = r$, only perform the latter exploration). 
	If these two double four-petal flower domains exist, proceed by revealing the configurations in~$\overline{\calF}_{\rm in}\cap {\calF}_{\rm out}$; call~$\sigma_0$ the stopping time marking the end of this stage. 
	
	Let~$H$ be the event that the two double four-petal flower domains above exist,
	that~$\overline{P}_1^{\rm in}$ is connected to~${P}_1^{\rm out}$ 
	and~$\overline{P}_3^{\rm in}$ to~${P}_3^{\rm out}$ in~$\omega \cap \overline{\calF}_{\rm in}\cap {\calF}_{\rm out}$
	and that~$\overline{P}_2^{\rm in}$ is connected to~${P}_2^{\rm out}$ 
	and~$\overline{P}_4^{\rm in}$ to~${P}_4^{\rm out}$ in~$\omega^* \cap \overline{\calF}_{\rm in}\cap {\calF}_{\rm out}$.
	As in the first case, it is a consequence of the separation of arms and Lemma~\ref{lem:D4PFD}
	that there exists~$c_0 > 0$ such that 
	\begin{align}\label{eq:rmH}
		{\rm P}[H] \geq c_0 \, \pi_4(p;r,4\rho). 
	\end{align}
	
	Next we reveal the configurations in~$\calF \cap \overline\calF_{\rm out}$ in a fashion described below. 
	Write~$H'$ for the event that~$\overline{P}_1^{\rm out}$ and~$\overline{P}_3^{\rm out}$ are connected
	to each other in~$\omega' \cap \overline\calF_{\rm out}$, but not in~$\omega \cap \overline\calF_{\rm out}$.
	Observe that if~$H$ and~$H'$ occur, then the boundary conditions induced on~$\calF_{\rm in}$ 
	by~$\omega\cap\calF_{\rm in}$ and~$\omega'\cap\calF_{\rm in}$, respectively, form a boosting pair. 
	At this stage, it will be useful to introduce the following claim.

	\begin{claim}\label{claim:tubes}
		There exists~$c_1 = c_1(\delta)> 0$ and 
		four disjoint domains~$D_1,\dots, D_4$ contained in~$\calF \cap  \overline\calF_{\rm out}$ (they depend on $\calA$) such that
		\begin{align}
			\phi_{\calD_1}^{\omega'_{[\sigma_0]}, 0}[\calA  \longleftrightarrow \overline P_1^{\rm out}]> c_1, \qquad &\phi_{\calD_3}^{\omega'_{[\sigma_0]}, 0}[\calA  \longleftrightarrow \overline P_3^{\rm out}]> c_1,\label{eq:bound_conf1}\\
			 \phi_{\calD_2}^{\omega_{[\sigma_0]}, 1}[\partial \calF  \xlra{\omega^*}\overline P_2^{\rm out}]> c_1,
			\qquad 
			& \phi_{\calD_4}^{\omega_{[\sigma_0]}, 1}[\partial \calF  \xlra{\omega^*} \overline P_4^{\rm out}]> c_1,\label{eq:bound_conf2}
		\end{align}
		where~$\omega'_{[\sigma_0]}, 0$ (and~$\omega_{[\sigma_0]}, 1$) are the boundary conditions on~$\partial D_i$ induced 
		by the configuration equal to~$\omega'$ (and~$\omega$, respectively), completed in the unexplored region of~$D_i^c$ by closed (respectively,  open) edges. 
	\end{claim}

	Essentially~$D_1,\dots, D_4$ should be viewed as disjoint tubes connecting the petals~$\overline P_i^{\rm out}$
	to four disjoint regions of~$\partial \calF$. Each tube should be thought of as thick enough to contain either primal or dual paths with positive probability (for instance we may think of each tube as being constructed using a finite number of rectangles of aspect ratio~$2$). 
	Moreover, we should further require that the regions of~$\partial \calF$ contained in~$D_1$ and~$D_3$ contain sufficiently many points of~$\calA$ that a connection reaches these points with positive probability. 
	Let us delay the proof of this technical result and finish the proof of the lemma in this case. 
	
	Fix the four domains~$D_1,\dots, D_4$ given by the claim. 
	Reveal the configuration in~$\calF \cap \overline\calF_{\rm out}$ by first revealing the configurations outside $D_1,\dots, D_4$, 
	then in each of the quads $D_1,\dots D_4$. 
	Due to \eqref{eq:CBC},
	\begin{align*}
		{\rm P}\big[
		\calA  \xlra{\omega' \cap D_1} \overline P_1^{\rm out}, \, 
		\calA  \xlra{\omega' \cap D_3} \overline P_3^{\rm out},\,
		\partial \calF  \xlra{\omega^* \cap D_2} \overline P_2^{\rm out},\,
		\partial \calF  \xlra{\omega^* \cap D_4} \overline P_4^{\rm out}
		\,\big|\, (\omega_{[\sigma_0]},\omega_{[\sigma_0]}')\big]
		\geq c_1^4,
	\end{align*}	
	whenever~$(\omega_{[\sigma_0]},\omega_{[\sigma_0]}')$ are such that~$H$ occur. 
	Notice now that, since all points of~$\calA$ are wired together in~$\omega'$ outside of~$\calF$, 
	but that all the boundary of~$\calF$ is free in~$\omega$, 
	if the event above occurs, then so does~$H'$. 
	Thus, 
	\begin{align*}
		{\rm P}[H' \cap H] \geq c_1^4{\rm P}[H] \geq c_0 \, \pi_4(p;r,4\rho).
	\end{align*}
	If we now set~$\sigma$ to be the stopping time at which all of~$\calF_{\rm in}^c$ has been explored and~$H'$ has been found to occur, 
	and~$\sigma = \infty$ otherwise, then~$\sigma$ satisfies the properties claimed in the lemma.

		\medskip 
	
	\noindent{\bf Case 3:~$\calF$ is a flower domain but the petals wired in~$\zeta'$ are wired in~$\zeta$.}
	The construction in this is very similar to that of Case 2.  
	Start off by revealing 
	the double four-petal flower domains~$(\calF_{\rm in}, \calF_{\rm out})$  between~$\La_{r}$ and~$\La_{2r}$
	and~$(\overline{\calF}_{\rm in}, \overline{\calF}_{\rm out})$ between~$\La_{\rho}$ and ~$\La_{3\rho/2}$, respectively, for the configuration~$\omega$. 
	Define~$H$ and~$H'$ in the same way as in Case 2.
	The only difference is in the way in which the configurations in~$\calF \cap \overline\calF_{\rm out}$ are revealed so that
	${\rm P}[H'\,|\, H]$ is bounded below by a constant depending only on~$\delta$. 

	Write~$P_1,\dots, P_{2k}$ for the petals of~$\calF$. 
	By the union bound and the assumption on~$\calA$, there exists a free petal~$P_j$ such that 
	\begin{align*}
		\phi^{\zeta'}_\calF [P_j \cap \calA \xlra{\omega'} \La_\rho] \geq  \delta/k.
	\end{align*}
	Recall that~$\calF$ is~$\delta$-well-separated, and therefore~$k$ is bounded in terms of~$\delta$ only. 
	A construction similar to that of the claim may be performed, 
	with~$\calA$ replaced by~$\calA \cap P_j$ and~$\partial \calF$ replaced by~$P_j$. 
	We conclude in the same way as in Case 2.
\end{proof}

We now turn to the proof of Claim~\ref{claim:tubes}. 
This type of construction is usually tedious because of the general form of~$\calF$, but may be performed fairly explicitly. 
We warn the reader of the difficulties arising from the potential bottlenecks of~$\calF$ and the fact that the RSW result \eqref{eq:RSWnear} is not valid with arbitrary boundary conditions in arbitrary quads when $q=4$. 
Nevertheless, with sufficient care, the domains~$D_1,\dots,D_4$ may be constructed as unions of small squares~$\La_{c \rho n}(x)$ with~$x \in c\rho \bbZ^2$ and~$c > 0$ a small constant independent of~$\rho$. 
We propose below an alternative, more innovative construction of~$D_1,\dots, D_4$.
We will use the following result from~\cite{DumManTas20} (see Remark~4.2 to be precise).

The extremal distance between the arcs~$(ab)$ and~$(cd)$ of a quad~$(\calD, a,b,c,d)$ is the unique value 
$\ell = \ell_\calD[(ab),(cd)]$ such that there exists a conformal transformation~$\psi$ from~$\calD$ (seen as a domain in the continuum)
to~$(0, 1)\times(0, \ell)$, with~$a, b, c, d$ being mapped (by the continuous extension of~$\psi$) to the corners of~$[0, 1] \times [0, \ell]$, in counterclockwise order, starting with the lower-left corner.

\begin{proposition}[RSW in terms of extremal distance]\label{prop:conformal_RSW}
	For all $L > 0$, there exists $\eta(L) >0$ such that, for all $1\le q\le 4$, $p \in (0,1)$, $R < L(p)$,
	and $(\calD, a,b,c,d)$ a discrete quad contained in~$\La_R$, 
	if $\ell_\calD[(ab),(cd)] \le L$ then
	\[\phi_{\calD,p}^{1/0}[\calC(\calD)] \geq\eta(L),\]
	where~$1/0$ denotes the boundary condition on~$\calD$ where the arcs~$(ab)$ and~$(cd)$ are wired and the rest of the boundary is free. 
\end{proposition}

While the result is stated for $p=p_c$ only in \cite{DumManTas20}, the reader will easily check that its proof extends readily to the near-critical regime using~\eqref{eq:RSWA}.

\begin{proof}[Claim~\ref{claim:tubes}]
	Write~$B_s$ for the euclidian ball of~$\bbR^2$ of radius~$s$ centred at~$0$, and~$B_s(z)$ for its translate by~$z \in \bbR^2$.
	Let~$\psi$ be a conformal map from~$\calF \cap \overline\calF_{\rm out}$ to some~$B_1 \setminus B_s$. 
	The existence of such a map is given by the uniformization theorem for the topological annulus $\calF \cap \overline\calF_{\rm out}$. 
	Note that~$s$ is determined by~$\calF \cap \overline\calF_{\rm out}$. 
	Moreover, since~$\partial \calF$ intersects~$\La_{2\rho}$ and 
	$\partial \overline \calF$ is contained in~$\Ann(\rho,\frac32\rho)$,~$s$ is bounded uniformly away from~$0$ and~$1$, 
	and the endpoints of $\psi(P_1),\dots,\psi(P_4)$ are uniformly far from each other. 
	
	Fix some small positive constant~$c_0 = c_0(\delta) < (1-s)/16$ which will be chosen below and will depend only on~$\delta$. 
	Let~$a_0,\dots, a_K = a_0$ (with~$K =  2\pi/c_0$) be points on the circle~$\partial B_1$ indexed in counterclockwise order and at a distance~$c_0$ of each other. 
	Write~$(a_{i}a_{i+1})$ for the arc of~$\partial B_1$ contained between~$a_i$ and~$a_{i+1}$;
	identify~$\psi^{-1}(a_{i}a_{i+1})$ to the corresponding vertices of~$\partial \calF$. Let $\calA_i:=\calA \cap \psi^{-1}(a_{i}a_{i+1})$.

	Next, we argue that there exist two indices~$i,j$ with~$|i - j| \geq 8$ (modulo~$K$) such that 
	\begin{align}\label{eq:two_arcs}
		\phi_{\calF}^{\zeta'}	[\calA_i \longleftrightarrow \psi^{-1}(B_{2c_0}(a_i))^c] 	&\geq c_0\delta/2,\\
		\phi_{\calF}^{\zeta'}	[\calA_j\longleftrightarrow \psi^{-1}(B_{2c_0}(a_j))^c] 	&\geq c_0\delta/2.\nonumber
	\end{align}
	Indeed, if the above fails, then there necessarily exists~$i$ such that for all~$j \notin \{i,\dots, i+7\}$, 
	\begin{align}\label{eq:two_arcs1}
		\phi_{\calF}^{\zeta'}	[\calA _j \longleftrightarrow \psi^{-1}(B_{2c_0}(a_j))^c] 	< c_0\delta/2.
	\end{align} 
	Assuming that this is the case, explore the open cluster~$\mathbf C$ of~$\calA \setminus \psi^{-1}(a_i a_{i+8})$. 
	Due to the small value of~$c_0$ and to the assumption above, 
	we conclude by the union bound that 
	\begin{align}\label{eq:two_arcs2}
		\phi_{\calF}^{\zeta'}	[\mathbf C \text{ intersects } \La_{\rho}] 	<  \delta/2.
	\end{align} 
	If~$\mathbf C$ does not intersects~$\La_{\rho}$, the measure on~$\calF\setminus \mathbf C$ obtained by conditioning on~$\mathbf C$ has free boundary conditions for all vertices adjacent to~$\mathbf C$. 
	Then, due to Proposition~\ref{prop:conformal_RSW} applied repeatedly to the dual model, we conclude that for~$c_0$ small enough, 
	\begin{align}\label{eq:two_arcs3}
		\phi_{\calF}^{\zeta'}	[\calA \longleftrightarrow \La_{\rho} \,|\, \mathbf C \text{ does no intersect } \La_{\rho}] 	<  \delta/2.
	\end{align} 
	Combining \eqref{eq:two_arcs2} and \eqref{eq:two_arcs3} we conclude that~$\phi_{\calF}^{\zeta'}	[\calA \lra \La_{\rho}] <  \delta$,
	which contradicts the assumption that the stopping time at which~$\calF$ was discovered was promising. 
	Thus \eqref{eq:two_arcs} is proved. 
	
Fix $i,j$ with $|i - j| \geq 8$ (modulo~$K$) and satisfying \eqref{eq:two_arcs}. For what comes next, we refer to Figure~\ref{fig:claim}. 
For $k = 1,2,3,4$, set  
\[
(D_k,a_k,b_k,c_k,d_k):=\psi^{-1}[(D'_k,a'_k,b'_k,c'_k,d'_k)],
\] 
where the quads $(D_k',a_k',b_k',c_k',d_k')$ satisfy:
\begin{itemize}[noitemsep]
\item $D'_k\subset B_1\setminus B_s$ for every $k$; 
\item the extremal distances $\ell_{\calD_k}[(a_kb_k),(c_kd_k)]$ belong to $(\kappa,1/\kappa)$ with $\kappa=\kappa(c_0)\in(0,\infty)$ for every $k$;
\item $(a'_kb'_k)$ is equal to $(a_{i-3}a_{i+3})$, $(a_{i+3}a_{j-3})$, $(a_{j-3}a_{j+3})$ and $(a_{j+3}a_{i-3})$ for $k=1,2,3$ and~$4$, respectively;
\item $(c'_kd'_k)=\psi(\overline P_k^{\rm out}$) for every $k$;
\item $D'_1$ and $D'_3$ contain $B_{3c_0}(a_i)$ and $B_{3c_0}(a_j)$ respectively.
\end{itemize} 
The construction of the domains $(D'_k,a'_k,b'_k,c'_k,d'_k)$ is straightforward. 

We now check \eqref{eq:bound_conf1} and \eqref{eq:bound_conf2}. We start by the latter and focus on the first inequality. The boundary conditions induced by $\{\omega_{[\sigma_0]},1\}$ on $D_2$ are free on $\overline P_2^{\rm out}$ and on $\psi^{-1}((a_{i+3}a_{j-3}))$. Therefore, the result follows directly from Proposition~\ref{prop:conformal_RSW} and duality. 
We now turn to \eqref{eq:bound_conf1} and focus on the first inequality. Let 
\begin{align*}
\ell&:=\partial\psi^{-1}(B_{2c_0}(a_i))^c\setminus\partial (\calF\cap\overline\calF^{\rm out}).
\end{align*}
	First, a mixing-type argument using Proposition~\ref{prop:conformal_RSW} and \eqref{eq:two_arcs} combine to give
	 \begin{align*}
		\phi_{\calD_1}^{\omega'_{[\sigma_0]},0}	[\calA_i\longleftrightarrow \ell] \ge c_2
		\phi_{\calF}^{\zeta'}	[\calA_i \longleftrightarrow \ell] \stackrel{\eqref{eq:two_arcs}}\ge c_3(\delta).
		\end{align*}
		Second, if $\ell^\pm:=\ell\cap\psi^{-1}(H^\pm)$, where $H^+$ is the  half-plane on the left of the line going from 0 to $\tfrac12(a_i+a_{i+1})$ (where left is understood when looking in the direction $\tfrac12(a_i+a_{i+1})$) and $H^-=\mathbb C\setminus H^+$. 
		The union bound implies that $\#$ may be chosen equal to $+$ or $-$ so that
		\begin{align*}
		\phi_{\calD_1}^{\omega'_{[\sigma_0]},0}	[\calA_i\longleftrightarrow \ell^\#] 
		\ge \tfrac12\phi_{\calD_1}^{\omega'_{[\sigma_0]},0}	[\calA_i\longleftrightarrow \ell] .
		\end{align*}
		We assume below that $\#=+$ (the same can be done for $\#=-$). Condition on the left-most path $\Gamma$ going from $\calA_i$ to $\ell^+$. Then, conditioned on $\Gamma$, the domain carved from $D_1$ by removing $\Gamma$ and all the edges explored to determine $\Gamma$ has wired boundary conditions on $\overline P_1^{\rm out}$, wired on $\Gamma$, and boundary conditions dominating the free boundary conditions  elsewhere. Since the extremal distance between $\overline P_1^{\rm out}$ and $\Gamma$ in this new domain is larger than $\kappa'=\kappa'(\kappa,\delta)>0$, we deduce from Proposition~\ref{prop:conformal_RSW} that 
		\[
		\phi_{\calD_1}^{\omega'_{[\sigma_0]},0}[\calA_i\longleftrightarrow\overline P_1^{\rm out}|\calA_i\longleftrightarrow \ell^+]\ge c_4.
		\]
		Combining the last three displayed equations gives the first inequality of \eqref{eq:bound_conf1}, and therefore concludes the proof.
	\end{proof}

	\begin{figure}
	\begin{center}
	\includegraphics[width=1.05\textwidth]{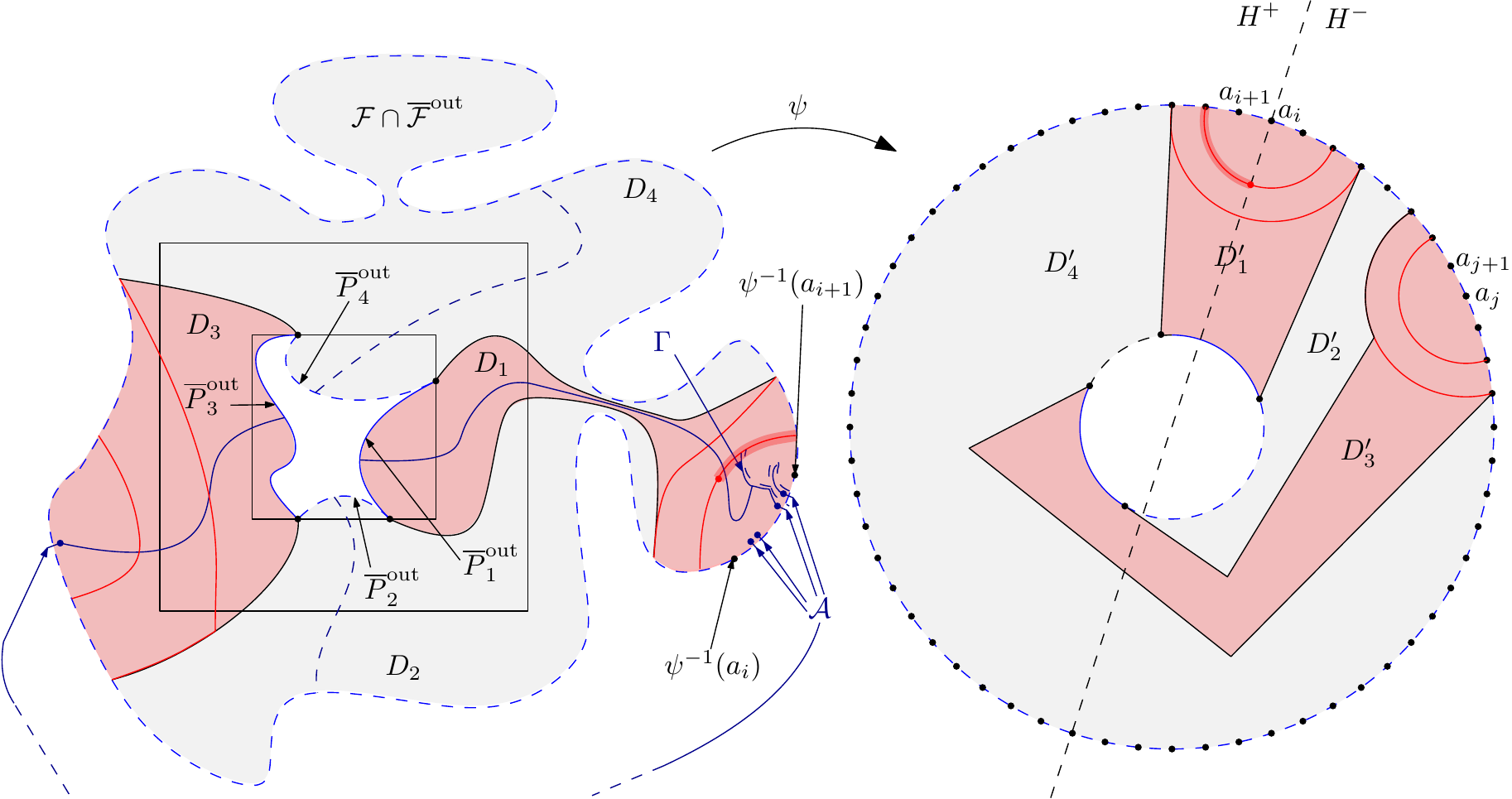}
	\caption{The uniformization map $\Psi$ from $\calF\cap\overline\calF^{\rm out}$ into $B_1\setminus B_s$. On the right, the black dots denote the $a_i$. We depicted the balls $B_{2c_0}(a_i)$ and $B_{3c_0}(a_i)$. The four domains $D_1',\dots,D_4'$ can be chosen in many ways. In bigger red, the path $\ell^+$ and its preimage $\psi^{-1}(\ell^+)$. The path $\Gamma$ from $\mathcal A$ into $\psi^{-1}(\ell^+)$ is drawn in dark blue.}
	\label{fig:claim}
	\end{center}
	\end{figure}

\subsection{Mixing rate versus coupling}

In this section we estimate the probability under the coupling of Theorem~\ref{thm:coupling} that~$\tau$ occurs. 
The relevant result of this section is the following theorem.

\begin{theorem}\label{thm:Delta_coupling}
	For any~$p\in(0,1)$ and any~$4r < R \le L(p)$, the following holds. 
	Fix~$\calG$ an~$\eta$-well-separated inner flower domain on~$\La_R$ or outer flower domain on~$\La_r$
	and~$\xi$ and~$\xi'$ a boosting pair of boundary conditions on~$\calG$.
	Let~$\bbP$ be the coupling of Theorem~\ref{thm:coupling} 
	(points (i) or (ii), depending on whether~$\calG$ is an inner or outer flower domain), 
	and recall the stopping time~$\tau$ associated to~$\bbP$. 
	Then 
	\begin{equation}\label{eq:Delta_coupling1}
		 \bbP[\tau < \infty] \asymp \Delta_p(r,R). 
	\end{equation}
	As a consequence, 
	if~$\calG$ is an inner flower domain and~$H$ is either the crossing event of an~$\eta$-regular quad at scale~$r$ or~$H = A_{r/2}$,
	or 	if~$\calG$ is an outer flower domain and~$H$ is either the crossing event of an~$\eta$-regular quad at scale~$R$
	translated so that it is contained in~$\Ann(R,2R)$ or~$H = A_{R}$, then 
	\begin{equation}\label{eq:Delta_coupling2}
		\phi_\calG^{\xi'}[H] - \phi_\calG^{\xi}[H]
		\asymp \Delta_p(r,R). 
	\end{equation}
\end{theorem}

When~$\calG$ is an inner flower domain, \eqref{eq:Delta_coupling2} should be understood as follows. 
Any boosting pair of boundary conditions at scale~$R$ boosts the probability of any crossing event at scale~$r$ 
by a quantity comparable to~$\Delta_p(r,R)$. 
The same holds when the boundary conditions are at scale~$r$ and the crossing event is at scale~$R$.
Recall that~$\Delta_p(r,R)$ was defined in terms of the boost that a specific pair of boundary conditions at scale~$R$ 
has on a specific crossing event at scale~$r$. 
Thus, in addition to stating that the boost of any pair of boundary conditions on any crossing event is comparable, 
the proposition also links the boost from outside in to that from inside out. 

The rest of this section is dedicated to showing Theorem~\ref{thm:Delta_coupling}.
The proof is split into several steps, each corresponding to a lemma. 
First, as a consequence of Theorem~\ref{thm:coupling}(i), 
we prove that the influence of boundary conditions on the crossing of any two regular quads is comparable.

\begin{lemma}\label{lem:Delta_out_in}
	Fix~$p\in(0,1)$ and~$2r < R \le L(p)$. 
	Let~$\calG$ be an~$\eta$-well-separated inner flower domain on~$\La_R$
	and~$\xi$ and~$\xi'$ be a boosting pair of boundary conditions on~$\calG$. 
	Moreover, let~$H$ be either the crossing event of an~$\eta$-regular quad at scale~$r$ or~$H = A_{r/2}$. 
	Then
	\begin{align}\label{eq:Delta_out_in0}
		\phi_{\calG}^{\xi'}[H] - \phi_{\calG}^{\xi}[H] 
		\asymp 	\phi_{\calG}^{\xi'}[\calC(\Lambda_r)] - \phi_{\calG}^{\xi}[\calC(\Lambda_r)].
	\end{align}
	In particular, 
	\begin{align}\label{eq:Delta_out_in2}		
		\phi_{\La_R}^1[H] - \phi_{\La_R}^0[H] \asymp \Delta_p(r,R).
	\end{align}
\end{lemma}

A particular case of~\eqref{eq:Delta_out_in2} shows that~$\Delta_p(1,R)\asymp \Delta_p(R)$.
In addition, by taking~$H = \calC(\La_{r/2})$,~\eqref{eq:Delta_out_in2} also proves that 
\begin{align}\label{eq:Delta_lengthen_in}
	\Delta_p(r,R)\asymp \Delta_p(r/2,R), 
\end{align}
and more generally that replacing~$r$ by a multiple of~$r$ only affects~$\Delta_p$ by a multiplicative constant. 

As will be apparent from the proof, 
an equivalent of~\eqref{eq:Delta_out_in0} may also be proved for outer flower domains on~$\La_r$ 
and crossing events at scale~$R$. However, an equivalent of~\eqref{eq:Delta_out_in2} cannot be shown with the same proof, as least for now (but will be later). 

\begin{proof}
	Fix~$p$,~$r$,~$R$,~$\calG$,~$\xi$,~$\xi'$ and~$H$ as in the statement. 
	Let~$\bbP$ be the coupling of Theorem~\ref{thm:coupling}~(i) between~$\phi_{\La_R}^{\xi'}$ and~$\phi_{\La_R}^{\xi}$.
	Then, using the notation of the theorem,
	\begin{align*}
		\phi_{\calG}^{\xi'}[H] - \phi_{\calG}^{\xi}[H]
		=\bbP[\omega' \in H ,\omega \notin H] 
		\le \bbP[\zeta \neq \zeta'] 
		\les \bbP[\tau < \infty].
	\end{align*}
	Moreover, by Theorem~\ref{thm:boosting_pair}, 
	\begin{align*}
		\phi_{\calG}^{\xi'}[H] - \phi_{\calG}^{\xi}[H]
		\geq \bbE\big[\ind_{\{\tau < \infty\}}\big(\phi_{\calF_\tau}^{\omega'_{[\tau]}}[H] - \phi_{\calF_\tau}^{\omega_{[\tau]}}[H]\big)\big]
		\ges \bbP[\tau < \infty], 
	\end{align*}
	where~$\bbE$ stands for the expectation associated to~$\bbP$.
	The two displays above imply that 
	\begin{align}\label{eq:Delta_out_in_tau}
		\phi_{\calG}^{\xi'}[H] - \phi_{\calG}^{\xi}[H]\asymp \bbP[\tau < \infty].
	\end{align}
	Apply this to a generic~$H$ and to~$H = \calC(\La_r)$ to obtain~\eqref{eq:Delta_out_in0}.
	Finally,~\eqref{eq:Delta_out_in0} applied with~$\calG = \La_R$,~$\xi = 0$,~$\xi' = 1$ gives~\eqref{eq:Delta_out_in2}.
\end{proof}

Next, we deduce an interpretation of~$\Delta_p(r,R)$ as a covariance between events at scales~$r$ and~$R$. 
Considering the symmetry between~$r$ and~$R$ below, this result is used to link the influence from outside in to that from inside out. 

\begin{lemma}\label{lem:Delta_in_out}	
	For every~$p$ and every~$4 r\le R \le  L(p)$,
    \begin{align}\label{eq:Delta_in_out}	
    	\mathrm{Cov}[A_{r/2}, A_R] 
		\asymp \phi_{\La_r^c}^1[A_R] - 	\phi_{\La_r^c}^0[A_R]
		\asymp \Delta_p(r,R).
    \end{align}
\end{lemma}

\begin{proof}
	We start by proving the equivalence between~$\Cov(A_{r/2},A_R)$ and~$\Delta(r,R)$. 
	By the monotonicity of boundary conditions~\eqref{eq:CBC}, 
	\begin{align}\label{eq:acov}
    	\Cov(A_{r/2}, A_R) 
    	= \phi[A_R](\phi[A_{r/2}\,|\, A_R] -\phi[A_{r/2}])
    	\le \phi_{\La_R}^1[A_{r/2}] -\phi_{\La_R}^0[A_{r/2}]
    	\les \Delta(r,R),	
	\end{align}
	where the last inequality is due to Lemma~\ref{lem:Delta_out_in}.
	
	For the converse bound, some additional work is needed. 
	Let~$\bbP$ denote the coupling between~$\phi$ and~$\phi[\cdot\,|A_{r/2}]$ inside~$\La_{R/2}$ produced by applying Remark~\ref{rem:circ_is_boosting} then Theorem~\ref{thm:coupling}(ii). 
	Recall that~$\zeta$ and~$\zeta'$ denote the boosting pair of boundary conditions on~$\La_{R/2}^c$ 
	induced by~$\omega$ and~$\omega'$, respectively. 
	Complete the coupling outside of~$\La_{R/2}$ 
	by an arbitrary increasing coupling. 
	Then, Theorems~\ref{thm:boosting_pair} and~\ref{thm:coupling}(ii) show that
	\begin{align}
		{\rm Cov}(A_{r/2},A_R)&=\bbP[\omega' \in A_R,\, \omega \notin A_R]\nonumber\\
		&		\geq \bbE \big[\ind_{\{\tau < \infty\}}\big( \phi_{\calF_\tau}^{\omega_{[\tau]}'}[A_R]- \phi_{\calF_\tau}^{\omega_{[\tau]}}[A_R]\big)\big]\nonumber\\
		&\ges \bbP[\tau < \infty] \ges \bbP[\zeta' \neq \zeta].\label{eq:acov2}
	\end{align}
	Using that~$\phi[A_{r/2}]\ges 1$, we deduce from the display above that 
	$\Cov(A_{r/2}, A_R)  \ges \bbP[\zeta' \neq \zeta]$.	
	Moreover	
	\begin{align}\label{eq:AR/2}
		\bbP[\zeta' \neq \zeta]
		&\ges \bbP[\omega' \in A_{R/2},\, \omega \notin A_{R/2}] 
		\ge \Cov(A_{r/2}, A_{R/2}) \quad \text{ and} \\
		\bbP[\zeta' \neq \zeta]
		&\ges \bbP[\omega' \notin A_{R/2},\, \omega \in A_{R/2}^*] 
		\ge \Cov(A_{r/2}, A_{R/2}^*),\nonumber
	\end{align}
	where~$A_{R/2}^* = \{\La_{R/2} \not\longleftrightarrow \partial \La_{R}\}$ is the event~$A_{R/2}$ applied to the dual model. 
	Divide the equations above by~$\phi[A_{R/2}]$ and~$\phi[A_{R/2}^*]$, respectively, both of which are uniformly positive quantities. 
	Using the monotonicity of boundary conditions~\eqref{eq:CBC}, we conclude that 
	\begin{align*}
		\bbP[\zeta' \neq \zeta]
		\ges \phi[ A_{r/2}\,|\, A_{R/2}] - \phi[ A_{r/2}\,|\, A_{R/2}^*]
		\ge  \phi_{\La_R}^1[ A_{r/2}] - \phi_{\La_R}^0[ A_{r/2}]
		\ges \Delta(r,R). 
	\end{align*}
	Together with~\eqref{eq:acov2} and~\eqref{eq:acov} this shows that 
	\begin{align}\label{eq:acov3}
    	\Delta(r,R) \les \bbP[\zeta' \neq \zeta]\les	\Cov(A_{r/2}, A_R) \les  \Delta(r,R).
	\end{align}
	
	Finally, we turn to the equivalence between $\Cov(A_{r/2}, A_R)$ and $\phi_{\La_r^c}^1[A_R] -	\phi_{\La_r^c}^0[A_R]$.
	The monotonicity of boundary conditions \eqref{eq:CBC} shows that
	\begin{align*}
		\mathrm{Cov}(A_{r/2}, A_R)
		\le \phi[A_R\,|\,A_{r/2}] - \phi[A_R]
		\le \phi_{\La_r^c}^1[A_R] -	\phi_{\La_r^c}^0[A_R].
	\end{align*}
	Conversely, 
	\begin{align*}
		\phi_{\La_r^c}^1[A_R] -	\phi_{\La_r^c}^0[A_R]
		&\leq \phi[A_R\,|\,A_{r}] - \phi[A_R\,|\, A_r^*] &\text{by~\eqref{eq:CBC}}&\\
		&\les \Cov(A_R,A_{r}) - \Cov(A_R,A_{r}^*)&\hspace{-1cm}\text{as~$\phi[A_{r}]\ges1$ \&~$\phi[A_r^*]\ges1$}&\\
		&\les  ( \phi[A_r \,|\, A_R] - \phi[A_r])+ ( \phi[A_r^*] - \phi[A_r^* \,|\, A_R])&\text{since~$\phi[A_R]\le1$}&\\
		&\les \Delta_p(2r,R) &\text{by~\eqref{eq:CBC} and \eqref{eq:Delta_coupling2}}&\\
		&\les \Delta_p(r,R) &\text{by~\eqref{eq:Delta_lengthen_in}}&\\
		&\les 		\mathrm{Cov}[A_{r/2}, A_R]  &\text{by~\eqref{eq:acov3}.}&
	\end{align*}
	The last two displays prove that~$\mathrm{Cov}[A_{r/2}, A_R] \asymp \phi_{\La_r^c}^1[A_R] -	\phi_{\La_r^c}^0[A_R]$. 
\end{proof}

Notice that~\eqref{eq:acov}--\eqref{eq:AR/2} also show that 
\begin{align}\label{eq:Delta_lengthen_out}
	\Delta_p(r,R) \ges \mathrm{Cov}[A_{r/2}, A_R] \ges \mathrm{Cov}[A_{r/2}, A_{R/2}] \ges \Delta_p(r,R/2) \ge \Delta_p(r,R) ,
\end{align}
and more generally that replacing~$R$ by a constant multiple of~$R$ only affects~$\Delta_p$ by a multiplicative constant. 

Finally, we claim that all boosting pairs of boundary conditions at scale~$R$
influence~$A_{r/2}$ by a similar amount; the same holds from the inside out.

\begin{lemma}\label{lem:Delta_boosing_pair}
	Fix~$p$ and~$4 r\le R \le  L(p)$. Then \\
	(i) 
	for any~$\eta$-well-separated inner flower domain~$\calG$ on~$\La_R$ and
	any boosting pair of boundary conditions~$\xi,\xi'$ on~$\calG$,
	\begin{align}\label{eq:dbp1}
		\phi_{\calG}^{\xi'} [A_{r/2}] - \phi_{\calG}^{\xi} [A_{r/2}] 
		\asymp \phi_{\La_R}^1 [A_{r/2}] - \phi_{\La_R}^{0} [A_{r/2}] ;
	\end{align}
	(ii) 
	for any~$\eta$-well-separated outer flower domain~$\calG$ on~$\La_r$ and
	any boosting pair of boundary conditions~$\xi,\xi'$ on~$\calG$,
	\begin{align}\label{eq:bdp2}
		\phi_{\calG}^{\xi'} [A_R] - \phi_{\calG}^{\xi} [A_R] 
		\asymp \phi_{\La_r^c}^{1} [A_R] - \phi_{\La_r^c}^{0} [A_R] .
	\end{align}
\end{lemma}

\begin{figure}
	\begin{center}
	\includegraphics[width=0.6\textwidth]{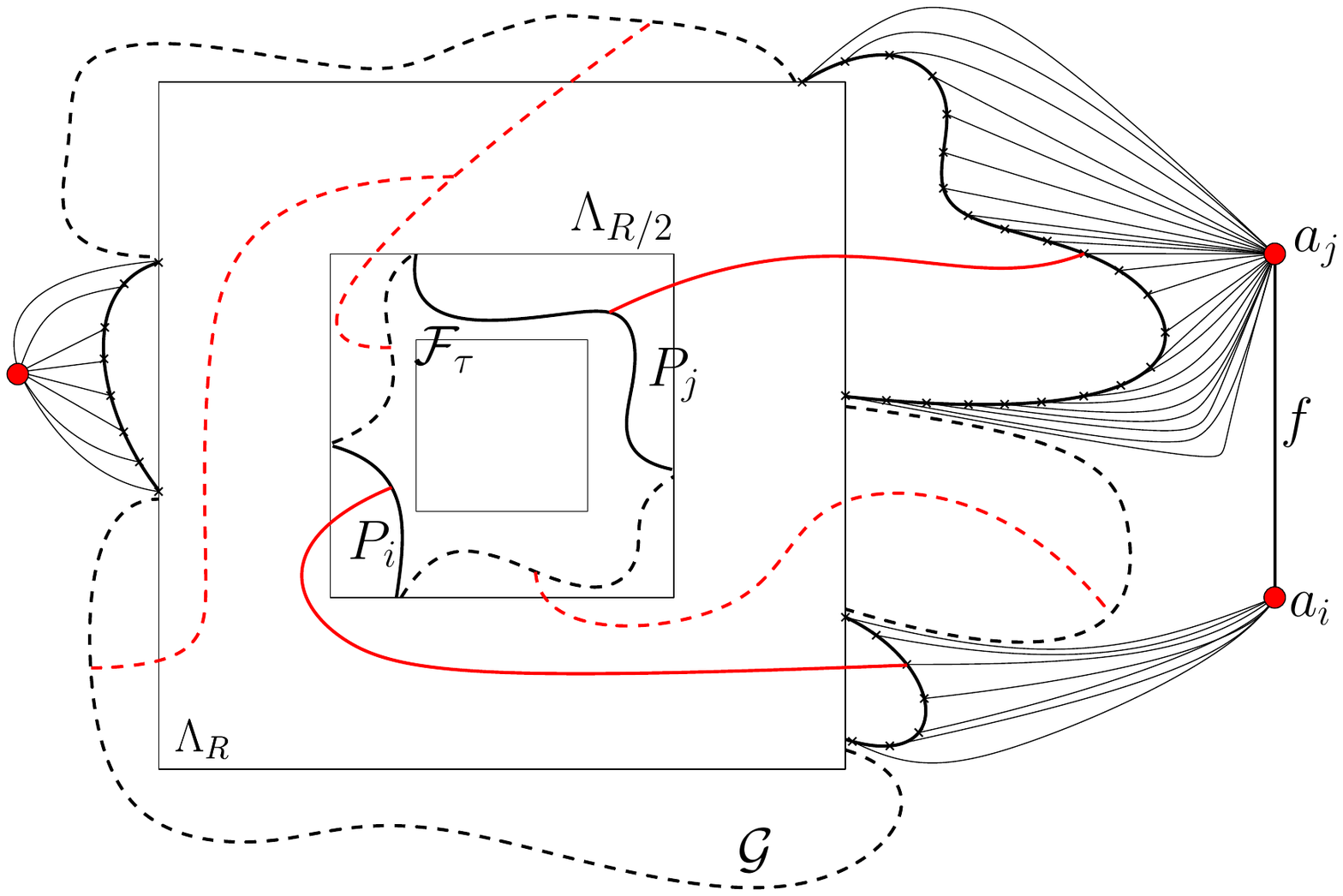}
	\caption{The graph~$G$ (the black edges between the true vertices of the petals and the vertices in red mean that they are all merged into the red vertex, or equivalently that we added {\em open} edges between them and the red vertices. We also depicted the flower domains~$\calG$ and~$\calF_\tau$, as well as the event~$H$, that corresponds to the occurrence of the red paths.
	} 
	\label{fig:graphG}
	\end{center}
	\end{figure}

\begin{proof}
	We will only treat point~$(i)$ as point~$(ii)$ is identical. 
	
	By the monotonicity of boundary conditions, 
 	\begin{align*}
		\phi_{\calG}^{\xi'} [A_{r/2}] - \phi_{\calG}^{\xi} [A_{r/2}] 
		\le \phi_{\La_R}^1 [A_{r/2}] - \phi_{\La_R}^{0} [A_{r/2}].
	\end{align*}
	We turn to the converse bound. We recommend to take a look at Figure~\ref{fig:graphG}.
	By monotonicity, we may assume that~$\xi,\xi'$ are both coherent with~$\calG$, 
	and that there exists exactly one pair of sets in the partition~$\xi$ that are wired together in~$\xi'$ (due to the coherence condition, each such set contains at least one primal petal of~$\calG$).
	Consider the graph~$G$ obtained from~$\calG$ as follows. 
	All vertices contained in a non-singleton set of the partition~$\xi$ 
	are collapsed to a single point (in particular, all points on each primal petal of~$\calG$ get collapsed together). 
	Write~$a_1,\dots, a_k$ for the points thus obtained. 
	Then there exist two distinct points~$a_i$ and~$a_j$ 
	such that the corresponding groups of petals are wired in~$\xi'$. 
	Finally,~$G$ is the graph obtained after the collapsing procedure described above, 
	with an additional edge~$f$ between~$a_i$ and~$a_j$. 
	There is an obvious correspondence between the edges of~$\calG$ and those of~$G\setminus \{f\}$, 
	and we will identify them from now on. Let~$\phi_G$ be the random-cluster measure on the finite graph~$G$ (note that~$G$ is not a subgraph of~$\bbZ^2$).
	
	With the construction above,~$\phi^{\xi}_{\calG}$ and~$\phi^{\xi'}_\calG$ are simply the restrictions of 
	$\phi_G[\cdot\,|\omega_f = 0]$ and~$\phi_G[\cdot\,|\omega_f = 1]$, respectively, to~$\calG$. 
	As such,~$\phi_G[A_{r/2}] \ges 1$ and Bayes's formula imply that
	\begin{align}
    	\phi_{\calG}^{\xi'} [A_{r/2}] - \phi_{\calG}^{\xi} [A_{r/2}]
    	&= \phi_G [A_{r/2}|\omega_f =1] - \phi_G [A_{r/2}|\omega_f =0] \nonumber\\
    	&\geq \phi_{G} [A_{r/2}|\omega_f =1] - \phi_G[A_{r/2}]\nonumber\\
		& \ges  \phi_{G} [\omega_f =1|A_{r/2}] - \phi_{G}[\omega_f =1].
    \label{eq:qm1}	
	\end{align}
	Let now~$\bbP$ be the coupling between~$\phi_{G}$ and~$\phi_{G} [\cdot\,|A_{r/2}]$ obtained as follows: 
	\begin{itemize}[noitemsep]
		\item reveal the edges of~$\La_{R/2}$ in the order dictated by Remark~\ref{rem:circ_is_boosting}. 
		If this stage produces an outer flower domain with a boosting pair of boundary conditions, 
		apply Theorem~\ref{thm:coupling}(ii) up to the associated stopping time~$\tau$;
		\item reveal all remaining edges of~$\calG$;
		\item reveal the state of~$f$. 
	\end{itemize}
	If~$\tau < \infty$, 
	there exist two primal petals~$P_i$ and~$P_j$ of~$\calF_\tau$ that are wired in~$\omega_{[\tau]}'$ but not in~$\omega_{[\tau]}$.
	Choose arbitrarily one such pair of petals. 
	Let~$H = H(\calF_\tau, \omega_{[\tau]},\omega_{[\tau]}')$ be the event that~$a_i$ is connected to~$P_i$ in~$\omega \cap \calF_\tau$, 
	that~$a_j$ is connected to~$P_j$ in~$\omega \cap \calF_\tau$, 
	but that~$a_i$ and~$a_j$ are not connected to each other or to any other primal petal of~$\calF_\tau$ in~$\omega \cap \calF_\tau$.

	Since~$\calF_\tau$ is~$1/2$-well-separated and~$\calG$ is~$\eta$-well-separated, and since~$\calF_\tau \cap \calG$ contains the annulus~$\Ann(R/2,R)$, 
	by standard applications of Theorem~\ref{thm:RSWnear}, 
	we find that
	\begin{align*}
		\bbP(H \,|\, \tau < \infty,\calF_\tau, \omega_{[\tau]},\omega_{[\tau]}') \ges 1.
	\end{align*}
	The occurrence of~$H$ may be determined before revealing the state of the edge~$f$.
	Moreover, if~$H$ occurs, then the endpoints~$a_i$ and~$a_j$ of~$f$
	are connected in~$\omega' \setminus \{f\}$, but not in~$\omega \setminus \{f\}$.
	Indeed,~$\omega'$ dominates~$\omega$, which implies that 
	$a_i$ is connected to~$P_i$ (in~$\omega' \cap \calF_\tau$), which is connected to~$P_j$ (in~$\omega'_{[\tau]}$), 
	which in turn is connected to~$a_j$ (in~$\omega' \cap \calF_\tau$).
		
	It follows that, at the last step of the coupling, if~$\tau < \infty$ and~$H$ occurs, 
	there is a probability~$\frac{(1-p)q}{p + (1-p)q}$ that~$f$ is closed in~$\omega$ but open in~$\omega'$. 
	To summarise, we find
	\begin{align}\label{eq:qm2}
		\phi_{G} [\omega_f =1|A_{r/2}]- \phi_{G}[\omega_f =1]
		= \bbP[\omega_f = 0, \omega'_f = 1] 
		\ges \bbP[H \text{ and } \tau < \infty] 
		\ges \bbP[\tau < \infty].
	\end{align}
	Finally, by Theorem~\ref{thm:coupling}(ii), the fact that~$\phi_{G} [A_{R/2}] \ges 1$ 
	and the comparison between boundary conditions~\eqref{eq:CBC}, we have 
	\begin{align}\label{eq:qm3}
		 \bbP[\tau < \infty]
		 &\ges \bbP[\text{$\omega$ and~$\omega'$ induce different b.c. on~$\La_{R/2}^c$}]\\
		 &\ge \bbP[\omega' \in A_{R/2} \text{ and } \omega \notin A_{R/2}]\nonumber\\
		 &\ges \Cov_G(A_{r/2}, A_{R/2})
		 \ges \phi_{\La_R}^{1} [A_{r/2}]- \phi_{\La_R}^{0} [A_{r/2}]
		, \nonumber
	\end{align}
	where~$\Cov_G$ is the covariance under~$\phi_G$ and the last inequality is given by~\eqref{eq:Delta_out_in2}. 
	Equations~\eqref{eq:qm1}--\eqref{eq:qm3} prove 	\[\phi_{\calG}^{\xi'} [A_{r/2}] - \phi_{\calG}^{\xi} [A_{r/2}] 
		\ges \phi_{\La_R}^1 [A_{r/2}] - \phi_{\La_R}^{0} [A_{r/2}],\] as desired. 
\end{proof}

We are finally ready to prove Theorem~\ref{thm:Delta_coupling}.

\begin{proof}[Theorem~\ref{thm:Delta_coupling}]
	We focus on the case where~$\calG$ is an inner flower domain on~$\La_R$; 
	the case where~$\calG$ is an outer flower domain is identical. 
	Recall that~$\xi$ and~$\xi'$ form a boosting pair of boundary conditions on~$\calG$ and that~$\tau$ is the stopping time 
	associated to the coupling of Theorem~\ref{thm:coupling} between~$\phi_{\calG}^{\xi'}$ and~$\phi_{\calG}^{\xi}$.
	Due to Lemmata~\ref{lem:Delta_out_in}, we find that
	\begin{equation*}
		\bbP[\tau < \infty] 
		\asymp \phi_{\calG}^{\xi'} [A_{r/2}]- \phi_{\calG}^{\xi} [A_{r/2}].
	\end{equation*}
	Now, Lemmata~\ref{lem:Delta_boosing_pair} and \ref{lem:Delta_boosing_pair} indicate that the right-hand side above is of order~$\Delta_p(r,R)$, and the proof is complete. 
\end{proof}

\subsection{$\Delta_p$ controls the mixing rate: proof of Theorem~\ref{thm:delta}(v)}

For the lower bound, \eqref{eq:Delta_in_out} gives that
\begin{equation*}
	\max\Big\{\Big|\frac{\phi_p[A\cap B]}{\phi_p[A]\phi_p[B]}-1\Big|\, :\, A\in\mathcal F(\Lambda_r),
	B\in\mathcal F(\mathbb Z^2\setminus\Lambda_R)\Big\}
	\ge \frac{\phi_p[A_{r/2}\cap A_R]}{\phi_p[A_{r/2}]\phi_p[A_R]}-1
	\ges \Delta_p(r,R).
\end{equation*}
For the upper bound, use the spatial Markov property \eqref{eq:SMP} to get that for every~$A\in\mathcal F(\Lambda_r)$ and~$B\in\mathcal F(\mathbb Z^2\setminus\Lambda_R)$, 
\begin{equation}\label{eq:delta_thm2b}
	\Big|\frac{\phi_p[A\cap B]}{\phi_p[A]\phi_p[B]}-1\Big|\le \max\Big\{\Big|\frac{\phi_{\Lambda_R,p}^{\xi'}[A]}{\phi_{\Lambda_R,p}^{\xi}[A]}-1\Big|\ :\xi,\xi'\text{ b.c. on }\partial\Lambda_R\Big\}.  
\end{equation}
Now, consider the coupling~$\mathbb P$ between~$\omega^0$ and~$ \omega^1$ constructed in Theorem~\ref{thm:Delta_coupling} 
on~$\La_R \setminus \La_{2r}$ and boundary conditions~$0$ and~$1$ respectively on~$\La_{R}$ (which form a boosting pair). 
By applying the same decision tree for the boundary conditions~$\xi$ and~$\xi'$, we obtain two additional configurations
$\omega^\xi,\omega^{\xi'}$ with laws~$\phi_{\Lambda_R}^{\xi}$ and~$\phi_{\Lambda_R}^{\xi'}$, respectively,
and such that~$\omega^0 \leq \omega^\xi \leq \omega^1$ and~$\omega^0 \leq \omega^{\xi'} \leq \omega^1$.
If~$\zeta$ and~$\zeta'$ are the boundary conditions induced on~$\partial\Lambda_{2r}$ by~$\omega^0$ and~$\omega^1$, we see that
	   \[
	   \phi_{\Lambda_R,p}^{\xi'}[A]-\phi_{\Lambda_R,p}^\xi[A]\le \mathbb P[\zeta\ne \zeta']\max_{\psi}\phi_{\Lambda_{2r},p}^\psi[A]\les \Delta_p(r,R)\phi_{\Lambda_R,p}^\xi[A],
	   \]
	   where in the second inequality we used Theorem~\ref{thm:coupling} as well as the mixing property to replace~$\phi_{\Lambda_{2r},p}^\psi[A]$ by~$\phi_{\Lambda_R,p}^\xi[A]$.  This concludes the proof.

\subsection{Quasi-multiplicativity of~$\Delta_p$: proof of Theorem~\ref{thm:delta}(ii)}

The following corollary is a slight generalisation (in the introduction it is the case $r=0$ only) of the quasi-multiplicativity property of Theorem~\ref{thm:delta}.

\begin{corollary}[Quasi-multiplicativity of~$\Delta_p$]\label{cor:Delta_quasi}
	For any~$p$ and~$r < n <R \le L(p)$, 
	\begin{align}\label{eq:delta_quasi}
		\Delta_p(r,R) \les \Delta_p(r,n)\Delta_p(n,R) \les  \Delta_p(r,R).
	\end{align}
\end{corollary}

\begin{proof}
	Let~$\bbP$ be the coupling between~$\phi_{\La_R}^{1}$ and~$\phi_{\La_R}^{0}$ in~${\rm Ann}(n,R)$ given by Theorem~\ref{thm:coupling}(i)
	and let~$\zeta$ and $\zeta'$ be the boundary conditions induced by~$\omega$ and~$\omega'$ on~$\partial\La_{n}$.  
	Complete the coupling inside~$\La_n$ by an arbitrary increasing coupling of~$\phi_{\La_n}^{\zeta}$ and~$\phi_{\La_n}^{\zeta'}$.
	Write~$\bbE$ for the expectation associated to~$\bbP$. 
	
	On the one hand, Theorem~\ref{thm:Delta_coupling} and the comparison between boundary conditions~\eqref{eq:CBC} yields
	\begin{align*}
		\Delta_p(r,R) 
		&= \bbE[\ind_{\{\zeta \neq \zeta'\}}( \phi_{\La_n}^{\zeta'}[\calC(\La_r)]- \phi_{\La_n}^{\zeta}[\calC(\La_r)])]\\
		&\le \bbP[ \zeta \neq \zeta' ]( \phi_{\La_n}^{1}[\calC(\La_r)]- \phi_{\La_n}^{0}[\calC(\La_r)])\\
		& \les \Delta_p(n,R)\Delta_p(r,n).
	\end{align*}
	
	On the other hand, recall that when~$\tau < \infty$, 
	$\calF_\tau$ is a~$1/2$-well-separated inner flower domain on~$\La_n$
	and that~$\omega_{[\tau]}$ and~$\omega'_{[\tau]}$ induce a boosting pair of boundary conditions on~$\calF_\tau$. 
	Thus, Theorem~\ref{thm:Delta_coupling} implies
	\begin{align*}
		\Delta_p(r,R) 
		&\geq \bbE [\ind_{\{\tau < \infty\}}( \phi_{\calF_\tau}^{\omega'_{[\tau]}}[\calC(\La_r)]- \phi_{\calF_\tau}^{\omega_{[\tau]}}[\calC(\La_r)])]\\
		&\ges  \bbP[\tau < \infty] \Delta_p(n,R)  \\
		& \ges\Delta_p(r,n)\Delta_p(n,R).
	\end{align*}
\end{proof}

\subsection{Mixing rate versus pivotality: proof of Theorem~\ref{thm:delta}(v)}\label{sec:influence vs. pivotality}

This section concerns the proof of \eqref{eq:delta_pi4_improvement}. This property is not used in the rest of the paper, but it proves that Kesten's scaling relation does not extend with~$\pi_4(p,r,R)$ instead of~$\Delta_p(r,R)$. 

The upper bound follows trivially from Proposition~\ref{prop:mixing}, hence we are left with proving 
\begin{align}\label{eq:Delta_improvement}
	\Delta_p(r,R) \ges (R/r)^c\, \pi_4(p;r,R),
\end{align}
for some constant~$c > 0$ and all~$p$ and~$r \leq R \leq L(p)$. 

\begin{remark}Note that the bound~$\Delta_p(r,R)\ges \pi_4(p,r,R)$ follows readily from Theorems~\ref{thm:coupling} and~\ref{thm:Delta_coupling}. 
Thus, it is the polynomial improvement of~$(R/r)^c$ that is the core of the above inequality. 
The rest of the section is devoted to the proof of \eqref{eq:Delta_improvement}.
\end{remark}

We will use the following notation.  
Let~$r < R$,~$\calF$ be an outer flower domain on~$\La_r$ and~$\calG$ be an inner flower domain on~$\La_R$, 
each containing exactly four petals denoted by~$P_1,\dots, P_4$ and~$P_1',\dots, P_4'$, respectively.
Then~$\xi^1$ (resp.~$\xi^0$) are the boundary conditions on~$\calF$ which are coherent with its flower domain structure
and in which the primal petals~$P_1$ and~$P_3$ are wired (resp.~not wired) together. 
The similarly defined boundary conditions on~$\calG$ are written~$\zeta^1$ and~$\zeta^0$, respectively. 

For~$i,j \in \{0,1\}$, denote by~$\phi_{\calF\cap\calG}^{\xi^i\cup\zeta^j}$ the measure on the subgraph~$\calF\cap \calG$ with
the boundary condition~$\xi^i\cup\zeta^j$ which is the partition of~$\partial(\calF\cap\calG)=\partial\calF\cup\partial\calG$ given by the   union of the partitions~$\xi^i$ of the inner boundary and~$\zeta^j$ of the outer one.

For configurations on~$\calF \cap \calG$, define the events
\begin{align*}
&A_4(\calF,\calG) := \{ P_1 \xlra{\omega} P_1',\, P_3 \xlra{\omega} P_3',\, P_2 \xlra{\omega^*}P_2',P_4 \xlra{\omega^*}P_4'\},\\
&\tilde A_4(\calF,\calG) := \{ P_1 \xlra{\omega'} P_1',\, P_3 \xlra{\omega'} P_3',\, P_2 \xlra{\omega^*}P_2',P_4 \xlra{\omega^*}P_4'\}.
\end{align*}

Also write~$\phi_{\calF \cap \La_R}^{\xi^i\cup 1}$ and~$\phi_{\calF \cap \La_R}^{\xi^i\cup 0}$ for the measure on~$\calF \cap \La_R$
with boundary conditions~$\xi^i$ on~$\partial \calF$ and respectively  	wired and free boundary conditions on~$\partial \La_R$. 
Define~$A_4(\calF,R)$ and~$\tilde A_4(\calF,R)$ as in the last display, with~$P_1',\dots, P_4'$ all replaced by~$\partial \La_R$.

The following lemma states that the probability of the four-arm event~$A_{4}(\calF,\calG)$
increases substantially if we allow the primal arms to be in~$\omega'$ rather than in~$\omega$. It is particularly important in the lemma below that~$\calF$ and~$\calG$ are not assumed to be well-separated.

\begin{lemma}\label{lem:tildeAA}
	There exists $\delta > 0$ such that the following holds. 
	For any~$p \in (0,1)$ and~$1 \leq r \leq L(p)$, 
	any outer flower domain~$\calF$ on~$\La_r$ and any inner flower domain~$\calG$ on~$ \La_{2r}$, 
	both containing exactly four petals, and any~$i \in \{0,1\}$, 
	there exists a coupling~$\bbP$ of~$\phi_{\calF\cap\calG}^{\xi^i\cup\zeta^0}$ and~$\phi_{\calF\cap\calG}^{\xi^i\cup\zeta^1}$ such that 
	\begin{align}\label{eq:tildeAA}
	\bbP[\tilde A_{4}(\calF,\calG)]
	\geq (1 + \delta) \bbP[A_{4}(\calF,\calG)].
	\end{align}
\end{lemma}

\begin{proof}
	Fix~$p$,~$r$,~$i$,~$\calF$ and~$\calG$ as above. 
	We will first treat the case where both flower domains are well-separated, then use it to solve the general case. 
	\bigbreak
	
	\noindent{\em Proof when~$\calF$ and~$\calG$  are~$1/2$-well-separated.}
	Let~$x = (3n/2,0)$ and start the coupling by exploring the double four-petal flower domain~$(\calF_{\rm in},\calF_{\rm out})$
	between~$\La_{n/8}(x)$ and~$\La_{n/4}(x)$ in~$\omega$. 
	If this stage fails, reveal the rest of the configuration in an arbitrary increasing fashion. 
	If~$(\calF_{\rm in},\calF_{\rm out})$ exists, continue by revealing the configuration inside~$\calF_{\rm in}$. 
	Write~$\xi_{{\rm in}}^{0}<\xi_{{\rm in}}^{1}$ for the two boundary conditions on~$\calF_{\rm in}$ which are coherent with the flower domain structure. 
	
	We now use the same argument as in the proof of \eqref{eq:boosting_pair_circ_i} to study the connection probability between 
	$P_1^{\rm in}$ and~$P_3^{\rm in}$ inside~$\calF_{\rm in}$ in~$\omega$ and~$\omega'$. 
	The conditional law of  $\omega$ in~$\calF_{\rm in}$ is 
	\begin{align*}
	(1-\lambda) \phi_{\calF_{\rm in}}^{\xi_{\rm in}^{0}} + \lambda \phi_{\calF_{\rm in}}^{\xi_{\rm in}^{1}}\qquad
	 \text{ with }\qquad\la  := \phi_{\calF \cap \calG}^{\xi^i \cup\zeta^0}[P_1^{\rm out} \xlra{\calF_{\rm out}} P_3^{\rm out}\,|\, \calF_{\rm in}, \calF_{\rm out}],
	\end{align*}
	while that of~$\omega'$ dominates
	\begin{align*}
	(1-\lambda') \phi_{\calF_{\rm in}}^{\xi_{\rm in}^{0}} + \lambda' \phi_{\calF_{\rm in}}^{\xi_{\rm in}^{i}}\qquad
	 \text{ with }\qquad \la'  := \phi_{\calF \cap \calG}^{\xi^i\cup\zeta^1}[P_1^{\rm out} \xlra{\calF_{\rm out}} P_3^{\rm out}\,|\, \calF_{\rm in}, \calF_{\rm out}].
	\end{align*}
	Thus,~$\la' - \la$ may be lower bounded by the probability that~$P_1^{\rm out}$ and~$P_3^{\rm out}$
	are connected in~$\calF_{\rm out}$ to the two primal petals of~$\calG$, but not to each other. 
	By Theorem~\ref{thm:RSWnear}, we conclude that~$\la' - \la \ges 1$, and by Lemma~\ref{lem:boost} that 
	\begin{align}\label{eq:tildeAAPP}
	 \bbP[\omega' \in \{P_1^{\rm in} \xlra{ \calF_{\rm in}} P_3^{\rm in}\} \text{ but }
	 \omega \notin \{P_1^{\rm in} \xlra{ \calF_{\rm in}} P_3^{\rm in}\} \,|\, \calF_{\rm in}, \calF_{\rm out} ] \ges 1. 
	\end{align}
	Finally, reveal the configurations on~$\calF_{\rm out}$. 
	
	Let~$H$ be the event that in~$\calF_{\rm out}$:
	\begin{itemize}[noitemsep]
	\item~$P_1^{\rm out}$ is connected to~$P_1$ in~$\omega$,
	\item~$P_3^{\rm out}$ is connected to ~$P_1'$	 in~$\omega$,
	\item~$P_3$ is connected to~$P_3'$ in~$\omega$,
	\item~$P_2$,~$P_2'$ and~$P_4^{\rm out}$ are connected in~$\omega^*$, and
	\item~$P_4$,~$P_4'$ and~$P_2^{\rm out}$ are connected in~$\omega^*$.
	\end{itemize}
	In other words,~$H$ is the event that the connection between~$P_1^{\rm in}$ and~$P_3^{\rm in}$ inside~$\calF_{\rm in}$ 
	is pivotal for~$A_{4}(\calF,\calG)$. See Figure~\ref{fig:improvement} for an illustration. 
	Theorem~\ref{thm:RSWnear} and the well-separation of~$\calF$,~$\calF_{\rm out}$ and~$\calG$ imply that 
	\begin{align}\label{eq:tildeAAH}
	 \bbP[H \,|\, \calF_{\rm in}, \calF_{\rm out} \text{ and~$(\omega,\omega')$ on~$\calF_{\rm in}$}] \ges 1. 
	\end{align}
	
	\begin{figure}
	\begin{center}
	\includegraphics[width = 0.65\textwidth]{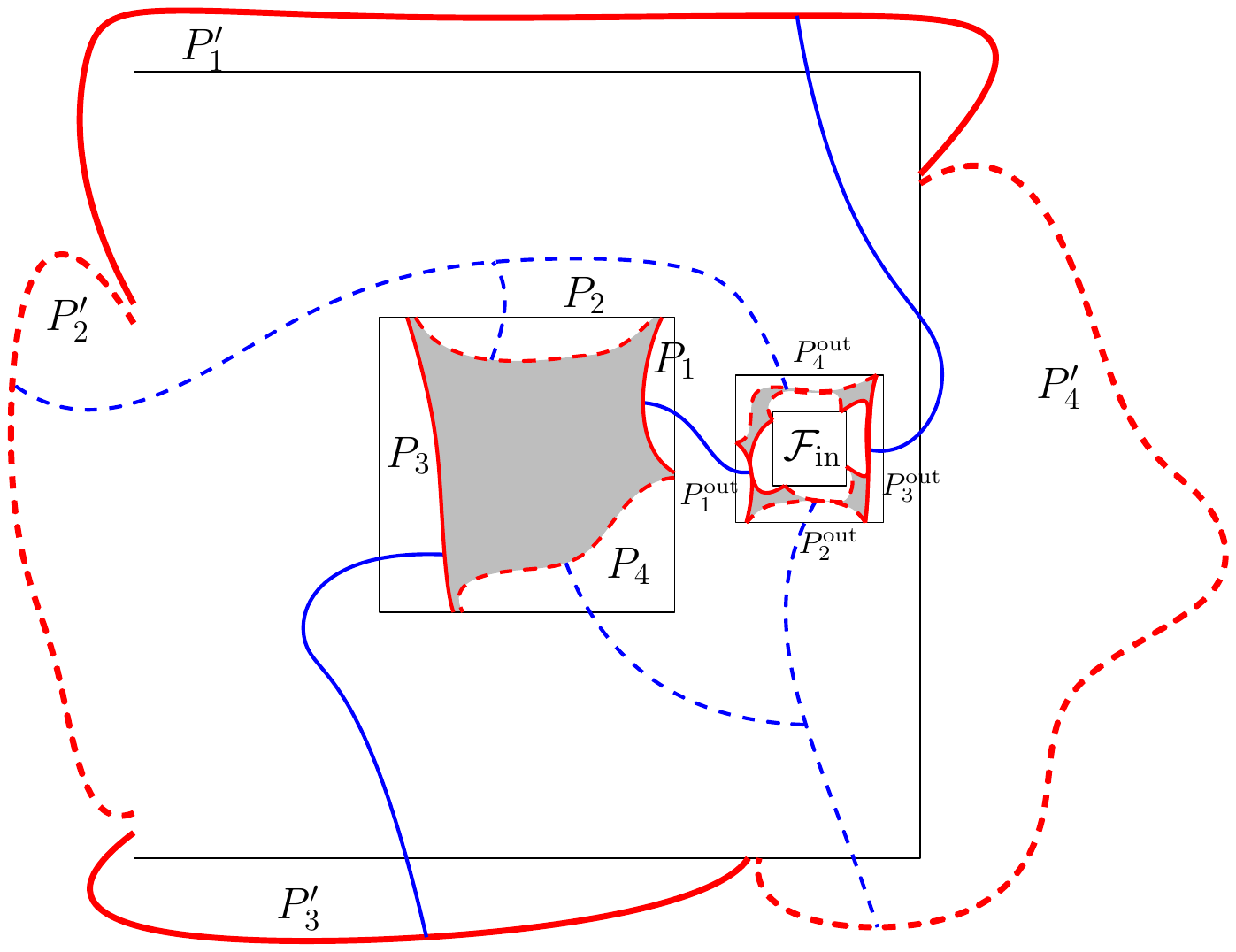}
	\caption{In the graph $\calF\cap \calG$, we first explore the double four-petal flower domain $(\calF_{\rm in}, \calF_{\rm out})$, 
	then reveal the configurations in $\calF_{\rm in}$ and $\calF_{\rm out}$. If $H$ occurs (see the blue paths), 
	then $A_{4}(\calF,\calG)$ depends on the connection inside $\calF_{\rm in}$ between its primal petals.
	If this connection occurs in $\omega'$ but not in $\omega$, then the configurations are in $\tilde A_{4}(\calF,\calG) \setminus A_{4}(\calF,\calG)$.}
	\label{fig:improvement}
	\end{center}
	\end{figure}
	Now, since~$A_{4}(\calF,\calG)$ implies the occurrence of~$\tilde A_{4}(\calF,\calG)$, we conclude that
	\begin{align*}
		&\bbP[\tilde A_{4}(\calF,\calG)] - \bbP[A_{4}(\calF,\calG)]\geq \bbP[ (\calF_{\rm in}, \calF_{\rm out}) \text{ exist},\,
		P_1^{\rm out} \xlra{ \calF_{\rm in}} P_3^{\rm out} \text{ in~$\omega'$ but not in~$\omega$}, \, H].
	\end{align*}
	Finally, \eqref{eq:tildeAAPP}, \eqref{eq:tildeAAH} and Lemma~\ref{lem:D4PFD} together imply that the right-hand side is larger than $c_0>0$.
	This concludes the proof of \eqref{eq:tildeAA} when~$\calF$ and~$\calG$ are~$1/2$-well-separated.  
	\bigbreak
	
	\noindent{\em Proof when~$\calF$ and~$\calG$ are not~$1/2$-well-separated.}
	We will construct the coupling~$\bbP$ between~$\phi_{\calF\cap\calG}^{\xi^i,\zeta^0}$ and~$\phi_{\calF\cap\calG}^{\xi^i,\zeta^1}$ in two steps. 
	We start by exploring the outer flower domain ~$\overline \calF$ between~$\La_r$ and~$\La_{5r/8}$ in~$\omega$
	and the inner flower domain~$\overline \calG$ from~$\La_{2r}$ to~$\La_{7r/8}$ also in~$\omega$. 
	Say that~$(\overline \calF,\overline \calG)$ is {\em good} if 
	\begin{itemize}[noitemsep]
	\item~$\overline \calF$ and~$\overline \calG$ each contain exactly four petals 
	and are~$1/2$-well-separated;
	\item the four petals of~$\overline \calF$ are connected in~$\calF\setminus \overline \calF$ to the corresponding petals of~$\calF$ 
	by paths of alternating types in~$\omega$;
	\item the four petals of~$\overline \calG$ are connected in~$ \calG\setminus \overline \calG$ to the corresponding petals of~$\calG$ 
	by paths of alternating types in~$\omega$.
	\end{itemize} 
	If $(\overline \calF,\overline \calG)$ is not good, complete~$(\omega,\omega')$ inside~$\overline \calF\cap \overline \calG$
	using an arbitrary increasing coupling. 
	When~$(\overline \calF,\overline \calG)$ is good, 
	the first case may be applied to~$\overline \calF$ and~$\overline \calG$,
	and provides a way to complete~$(\omega,\omega')$ inside~$\overline \calF\cap \overline \calG$ 
	so that 
	\begin{align*}
		\bbP[  \tilde A_{4}(\overline\calF,\overline\calG)  \setminus A_{4}(\overline\calF,\overline\calG)
		\,|\,\text{$(\overline \calF,\overline \calG)$ good}] \ges 1,
	\end{align*}
	Moreover, notice that in this case~$A_{4}(\calF,\calG)$ and~$\tilde A_{4}(\calF,\calG)$ occur if and only if 
	$A_{4}(\overline\calF,\overline\calG)$ and~$\tilde A_{4}(\overline\calF,\overline\calG)$, respectively, do. 
	By summing the display above, we conclude that 
	\begin{align*}
		\bbP[\tilde A_{4}(\calF,\calG)] - 	\bbP[A_{4}(\calF,\calG)]
		\ges \bbP[\text{$(\overline \calF,\overline \calG)$ good}].
	\end{align*}
	Finally, it is a standard consequence of the RSW theory (specifically of the separation of arms) that 
	\begin{align*}
	\bbP[\text{$(\overline \calF,\overline \calG)$ good}]
	\ges \bbP[A_{4}(\calF,\calG)],
	\end{align*} 
	where the constant in~$\ges$ does not depend on~$\calF$ and~$\calG$. 
	The last two displays provide the desired conclusion.
\end{proof}

\begin{corollary}\label{cor:induction}
	For~$\rho$ large enough,~$p \in (0,1)$ and~$r < R \leq L(p)$ with~$R = (\rho^2 + 2)^k r$, 
	any outer flower domain~$\calF$ on~$\La_r$ with exactly four petals
	and any~$i \in \{0,1\}$, 
	there exists an increasing coupling~$\bbP$ of~$\phi_{\calF \cap \La_R}^{\xi^i\cup0}$ and~$\phi_{\calF \cap \La_R}^{\xi^i\cup1}$ such that 
	\begin{align}\label{eq:FR}
		\bbP[ \tilde A_4(\calF,R)]
		\geq (1 + \tfrac\delta2)^k \, \bbP[  A_4(\calF,R)],
	\end{align}
	where~$\delta > 0$ is the constant given by Lemma~\ref{lem:tildeAA}.
\end{corollary}

\begin{proof}
	The proof proceeds by induction on~$k$. 
	Consider the value of~$\rho>1$ fixed; it will be chosen at the end of the proof 
	and it will be apparent that it is independent of~$p$,~$r$,~$R$ or~$\calF$.
	The case~$k = 0$ is trivially true. 
	
	Fix~$k \geq 0$ and assume that~\eqref{eq:FR} holds for this value of~$k$; we will now prove~\eqref{eq:FR} for~$k+1$. 
	Let~$\calF$ be an inner flower domain on~$\La_r$ and fix boundary conditions~$\xi \in \{\xi^0,\xi^1\}$ on~$\partial \calF$. 
	The coupling~$\bbP$ is built in three steps: 
	\bigbreak
	
	\noindent {\em Step 1:}
	Explore the inner flower domain from~$\partial \La_{2r\rho}$ to~$\partial\La_{2r}$ 
	and the outer one from~$\partial\La_{2r\rho}$ to~$\partial\La_{2r\rho^2}$ in~$\omega$;
	call these~$\calG_{{\rm in}}$ and~$\calG_{{\rm out}}$, respectively. 
	We will abuse notation by identifying~$(\calG_{{\rm in}},\calG_{{\rm out}})$ with the entire configuration~$(\omega,\omega')$ on~$\calG_{{\rm in}}^c \cap \calG_{{\rm out}}^c$. 
	Say that ~$(\calG_{{\rm in}},\calG_{{\rm out}})$ is {\em good} if both~$\calG_{{\rm in}}$ and~$\calG_{{\rm out}}$ have exactly four petals 
	$P^{\rm in}_1,\dots, P^{\rm in}_{4}$ and~$P^{\rm out}_{1},\dots, P^{\rm out}_{4}$, respectively, and if 
	in~$\omega \cap (\calG_{{\rm in}}^c \cap \calG_{{\rm out}}^c)$ there exist open paths connecting 
	$P^{\rm in}_{1}$ to~$P^{\rm out}_{1}$ and~$P^{\rm in}_{3}$ to~$P^{\rm out}_{3}$
	and dual-open paths connecting 	$P^{\rm in}_{2}$ to~$P^{\rm out}_{2}$ and~$P^{\rm in}_{4}$ to~$P^{\rm out}_{4}$.
	As before, we use the notation~$\zeta_{{\rm in}}^{0}$,~$\zeta_{{\rm in}}^{1}$ for the two boundary conditions on~$\calG_{{\rm in}}$ coherent with it being a flower domain, and~$\zeta_{{\rm out}}^{0}$ and~$\zeta_{{\rm out}}^{1}$ for those on~$\calG_{{\rm out}}$. 
	\bigbreak
	
	\noindent {\em Stage 2:}
	If~$(\calG_{{\rm in}},\calG_{{\rm out}})$ is not good, we sample the rest of the configurations according to an arbitrary increasing coupling. 
	If~$(\calG_{{\rm in}},\calG_{{\rm out}})$ is good, 
	observe that the law of~$\omega$ in~$\calG_{\rm out}$, conditionally on the revealed set, is a linear combination 
	$$(1-\lambda) \phi_{\calG_{{\rm out}} \cap \La_R}^{\zeta_{{\rm out}}^{0}\cup 0} + \lambda \phi_{\calG_{{\rm out}} \cap \La_R}^{\zeta_{{\rm out}}^{1}\cup 0}\qquad
	\text{where}\qquad\la  := \phi_{\calF \cap \La_R}^{\xi\cup0}[P^{\rm in}_{1}\xlra{\calG_{{\rm in}}} P^{\rm in}_{3} \,|\,(\calG_{{\rm in}}, \calG_{{\rm out}}) ].$$
	Moreover, the conditional law of~$\omega'$ dominates
	\begin{align}\label{eq:omega^+}
		(1-\lambda) \phi_{\calG_{{\rm out}} \cap \La_R}^{\zeta_{{\rm out}}^{0}\cup 1} 
		+ \lambda \phi_{\calG_{{\rm out}} \cap \La_R}^{\zeta_{{\rm out}}^{1}\cup 1}, 
	\end{align}
	for the same value~$\la$. 
	
	Since~$\tilde A_4(\calG_{{\rm out}}, R)$ is increasing in~$\omega'$ and decreasing in~$\omega$, 
	we may use~\eqref{eq:FR} for~$k$ (which is our induction hypothesis) applied between~$\calG_{\rm out}$ and~$\La_R$,
	once with the boundary conditions $\zeta_{{\rm out}}^{0}$ on $\calG_{\rm out}$ and once with the boundary conditions $\zeta_{{\rm out}}^{1}$, 
	to produce an increasing coupling of~$\omega$ and~$\omega'$ on~$\calG_{\rm out}$ so that 
	\begin{align}\label{eq:omega^+2}
		\bbP[\tilde A_4(\calG_{{\rm out}},R) | \, (\calG_{{\rm in}},\calG_{{\rm out}}) \text{ good}] 
		&\geq (1 + \tfrac\delta2)^k\, \bbP[  A_4(\calG_{{\rm out}},R)| \, (\calG_{{\rm in}},\calG_{{\rm out}})\text{ good}].
	\end{align}

	\bigbreak 
	\noindent {\em Step 3:}
	Finally, we sample~$(\omega,\omega')$ in~$\calG_{{\rm in}} \cap \calF$ according to a specific coupling between the measures in this region,
	with boundary conditions induced by the previously revealed parts of~$\omega$ and~$\omega'$, respectively.
	Recall that this stage is reached only if~$(\calG_{{\rm in}},\calG_{{\rm out}})$ is good
	and that the revealed configurations suffice to decide whether~$\tilde A_4(\calG_{{\rm out}}, R)$ occurred or not.
	
	If $\tilde  A_4(\calG_{{\rm out}}, R)$ did not occur, complete the coupling in an arbitrary increasing way. 
	If~$\tilde  A_4(\calG_{{\rm out}}, R)$ did occur, 
	then the boundary conditions induced by the already revealed parts of~$\omega$ on~$\calG_{\rm in}$ are equal to~$\zeta_{{\rm in}}^{0}$.
	Indeed, in~$\omega \cap \calG_{\rm out}$,~$P^{\rm out}_{1}$ and~$P^{\rm out}_{3}$ are disconnected from each other due to the dual arms. 
	On the contrary, the boundary conditions induced by the revealed region of~$\omega'$ on~$\calG_{\rm in}$ dominate~$\zeta_{{\rm in}}^{1}$, 
	since in~$\omega' \cap \calG_{\rm out}$,~$P^{\rm out}_{1}$ and~$P^{\rm out}_{3}$ are connected to~$\partial \La_R$, which is wired. 
	Thus, we may apply Lemma~\ref{lem:tildeAA} to continue the coupling so that 
	\begin{align}
	& \bbP[\tilde A_4(\calF,\calG_{{\rm in}}) | \, 
	(\calG_{{\rm in}},\calG_{{\rm out}}) \text{ good  and }\tilde A_4(\calG_{{\rm out}}, R)]\nonumber\\
	&\quad  \geq (1+\delta) \bbP[ A_4(\calF,\calG_{{\rm in}}) | \, 
	(\calG_{{\rm in}},\calG_{{\rm out}}) \text{ good and } \tilde A_4(\calG_{{\rm out}}, R)].
	\label{eq:omega^+3}
	\end{align}
	This concludes the construction of the coupling~$\bbP$; next we show that~$\bbP$ satisfies~\eqref{eq:FR}. 
	\bigbreak

	If~$(\calG_{{\rm in}},\calG_{{\rm out}})$ is good and~$\tilde A_4(\calG_{{\rm out}}, R)$ and~$\tilde A_4(\calF,\calG_{{\rm in}})$ both occur,
	then so does~$\tilde A_4(\calF,R)$.
	By~\eqref{eq:omega^+2},~\eqref{eq:omega^+3}, 
	and the fact that $(\calG_{{\rm in}},\calG_{{\rm out}})$ are determined by~$\omega$ alone, we find 
	\begin{align}\label{eq:good_G}
		&\bbP[\tilde A_4(\calF,R)] \nonumber \\
		&\geq (1+\delta)(1+\tfrac{\delta}2)^{k}
		\sum_{(\calG_{{\rm in}},\calG_{{\rm out}}) \text{ good}}\phi_{\calF}^{\xi\cup0} [(\calG_{{\rm in}},\calG_{{\rm out}})]
		\phi_{\calF}^{\xi\cup0}[A_4(\calG_{{\rm out}},R) \cap A_4(\calF,\calG_{{\rm in}})|\,(\calG_{{\rm in}},\calG_{{\rm out}})] \nonumber\\
		&= (1+\delta)(1+\tfrac{\delta}2)^{k}
		\sum_{(\calG_{{\rm in}},\calG_{{\rm out}}) \text{ good}}\phi_{\calF}^{\xi\cup0} [(\calG_{{\rm in}},\calG_{{\rm out}})]
		\phi_{\calF}^{\xi\cup0}[A_4(\calF,R)| \, (\calG_{{\rm in}},\calG_{{\rm out}})] \nonumber\\
		&= (1+\delta)(1+\tfrac{\delta}2)^{k}
		\phi_{\calF}^{\xi\cup0}[A_4(\calF,R) \text{ and } (\calG_{{\rm in}},\calG_{{\rm out}}) \text{ good}].
	\end{align}
	The first equality is due to the fact that, conditionally on a good~$(\calG_{{\rm in}},\calG_{{\rm out}})$, 
	$A_4(\calF,R)$ occurs if and only if both~$A_4(\calG_{{\rm out}},R)$ and~$A_4(\calF,\calG_{{\rm in}})$ do. 
	The last equality is obtained directly by summation. 
	
	Observe now that for~$A_4(\calF,R)$ to occur,~$\calG_{\rm in}$ and~$\calG_{\rm out}$ need to have at least four petals each.
	Moreover, if they each have exactly four petals, then these need to be connected in such a way that 
	$(\calG_{{\rm in}},\calG_{{\rm out}})$ is good. 
	Thus
	\begin{align}\label{eq:bad_G}
	\phi_{\calF}^{\xi\cup0}[A_4(\calF,R),(\calG_{{\rm in}},\calG_{{\rm out}}) \text{ not good}]
	\le 
	&\phi_{\calF}^{\xi\cup0}[A_4(\calF,R),\calG_{{\rm in}} \text{ has at least 6 petals}] \nonumber\\ 	
	+&\phi_{\calF}^{\xi\cup0}[A_4(\calF,R),\calG_{{\rm out}} \text{ has at least 6 petals}].
	\end{align}
	We will argue that both of the terms in the right-hand side are small compared to the quantity $\phi_{\calF}^{\xi,0}[A_4(\calF,R)]$, 
	provided~$\rho$ is large enough. 

	Indeed, due to the quasi-multiplicativity of the four-arm event -- here applied to the slightly unusual event~$A_4(\calF,R)$ -- 
	and the mixing property~\eqref{eq:mix}, 
	there exists a universal constant~$C$ (that does not depend on~$\calF$,~$p$,~$r$ or~$R$) such that 
	\begin{align*}
		\phi_{\calF}^{\xi\cup0}[A_4(\calF,R),\calG_{{\rm in}} \text{ has at least 6 petals}] 
		&\le C \phi_{\calF}^{\xi\cup0}[A_4(\calF,2r)] \phi_{\calF}^{\xi\cup0}[A_4(2\rho r, R)] \pi_6(2r,2\rho r)\\
		&\le C^2 \phi_{\calF}^{\xi\cup0}[A_4(\calF,R)] \frac{\pi_6(2r,2\rho r)}{\pi_4(2r,2\rho r)}.
	\end{align*}
	Due to Theorem~\ref{thm:RSWnear},~$\rho > 0$ may be chosen independently of~$\calF$,~$p$,~$r$ or~$R$ so that
	\begin{align*}
		C^2 \frac{\pi_6(2r,2\rho r)}{\pi_4(2r,2\rho r)} \le \delta/4.
	\end{align*}
	Assuming this is the case, and by the same reasoning for the second term of~\eqref{eq:bad_G}, we find
	\begin{align*}
		(1 + \delta) \phi_{\calF}^{\xi,0}\big[A_4(\calF,R),(\calG_{{\rm in}},\calG_{{\rm out}}) \text{ good}\big] 
		\geq (1+ \delta/2)\phi_{\calF}^{\xi,0}\big[A_4(\calF,R)\big], 
	\end{align*}
	which, when inserted in~\eqref{eq:good_G}, proves~\eqref{eq:FR} for~$R = (\rho^2 + 2)^{k+1} r$.
\end{proof}

\begin{proof}[\eqref{eq:Delta_improvement}]
	Fix~$\delta > 0$ and~$\rho$ given by Corollary~\ref{cor:induction}.
	Due to Theorem~\ref{thm:delta}(ii), 
	it suffices to prove the statement for~$R = 2(\rho^2 + 2)^k r$.  From now on, fix such values~$r$ and~$R$. 
	
	We use the same notation as in the proof of Corollary~\ref{cor:induction} 
	and construct a coupling~$\bbP$ between~$\phi_{\La_R}^0$ and~$\phi_{\La_R}^1$ similar to that of the previous proof. 

	First, explore the double four-petal flower domain between ~$\La_{r}$ and~$\La_{2r}$; 
	call it~$(\calG_{{\rm in}},\calG_{{\rm out}})$. 
	If no such double four-petal flower domain exists, proceed with an arbitrary coupling. 
	If~$(\calG_{{\rm in}},\calG_{{\rm out}})$ exists, use the coupling provided by Corollary~\ref{cor:induction} to complete~$(\omega,\omega')$ in~$\calG_{\rm out}$ so that, 
	\begin{align}
		\bbP[ A_4(\calG_{{\rm out}},R) \,| \, (\calG_{{\rm in}},\calG_{{\rm out}})] 
		&\geq (1 + \tfrac\delta2)^k\, \bbP[   A_4(\calG_{{\rm out}},R)\,| \, (\calG_{{\rm in}},\calG_{{\rm out}})].
	\end{align}
	Finally, use an arbitrary increasing coupling inside~$\calG_{\rm in}$. 
	
	When~$(\calG_{{\rm in}},\calG_{{\rm out}})$ exists and~$\tilde A_4(\calG_{{\rm out}},R)$ occurs, 
	the boundary conditions imposed by the revealed portions of~$\omega$ and~$\omega'$
	are equal to~$\zeta^0_{\rm in}$ and dominate~$\zeta^1_{\rm in}$, respectively. 
	Using Theorem~\ref{thm:boosting_pair} and Lemma~\ref{lem:D4PFD}, we conclude that 
	\begin{align*}
	\Delta_p(r,R) 
	&= \bbP\big[ \omega' \in \calC(\La_r), \omega \notin \calC(\La_r) \big]\\
	&\ges \bbP \big[\tilde A_4(\calG_{{\rm out}},R), \, (\calG_{{\rm in}},\calG_{{\rm out}}) \text{ exists}\big] \\
	&\ges (1+ \tfrac\delta2)^k \bbP \big[A_4(\calG_{{\rm out}},R), \, (\calG_{{\rm in}},\calG_{{\rm out}}) \text{ exists}\big] \\
	&\ges (R/2r)^\eps \pi_4 (r,R),
	\end{align*}
	where~$\eps = \log(1+\frac\delta2)/\log{(\rho^2 + 2)} > 0$.
\end{proof}

\section{Derivatives in terms of~$\Delta_p$}\label{sec:a}

\subsection{Derivatives for crossing and arm events}

In this section we obtain expressions for the derivatives of probabilities of crossing events and arm events in terms of~$\Delta_p$. 
In addition, we upper bound the derivative of the mixing rate~$\Delta_p$ by similar expressions.
The relevant results are Proposition~\ref{prop:deriv_cross} and~\ref{prop:deriv_Delta_ub}, respectively. 
These will hold within the critical window, and are instrumental in proving the main stability results Theorem~\ref{thm:stability} and~\eqref{eq:delta_thm5}.

We start by a proposition which states a slightly weaker form of Corollary~\ref{cor:sta}, 
but extends the expression to logarithmic derivatives of probabilities of arm events. 

\begin{proposition}	\label{prop:deriv_cross}
    Fix~$\eta>0$. For~$p\in(0,1)$ and every~$\eta$-regular quad~$(\calD,a,b,c,d)$ at scale~$R\le L(p)$,
    \begin{equation}\label{eq:deriv_cross}
    	R^2\Delta_p(R)+\sum_{\ell=R}^{L(p)}\ell \Delta_p(\ell)\Delta_p(R,\ell)	
		\les \tfrac{{\rm d}}{{\rm d}p}\phi_{p}[\calC(\calD)]
		\les \sum_{\ell=1}^R \ell\Delta_p(\ell)+\sum_{\ell=R}^{L(p)}\ell \Delta_p(\ell)\Delta_p(R,\ell),
    \end{equation}
    where the constants in~$\les$ depend on~$\eta$.
    Moreover, for any~$\sigma \in \{0,1\}^k$ and any~$r \le R \le L(p)$, 
    \begin{equation*}
    	 R^2\Delta_p(R)+\sum_{\ell=R}^{L(p)}\ell \Delta_p(\ell)\Delta_p(R,\ell)	\les\tfrac{{\rm d}}{{\rm d}p}\log\phi_{p}[A_\sigma(r,R)]\les
		  \sum_{\ell=1}^R \ell\Delta_p(\ell)+\sum_{\ell=R}^{L(p)}\ell \Delta_p(\ell)\Delta_p(R,\ell),
    \end{equation*}
    where the constants in~$\les$ depend on~$\sigma$.
\end{proposition}

After proving Proposition~\ref{prop:lower bound Delta} in the next section, we may use the quasi-multiplicativity of~$\Delta_p$ to replace the first term~$\sum_{\ell\le R}\ell\Delta_p(\ell)$ in the upper bounds above by~$R^2\Delta_p(R)$, hence deducing an up to constant estimate on the derivative. This is stated in Corollary~\ref{cor:deriv}.

\begin{remark}\label{rem:iota<1}
	The formula~\eqref{eq:deriv_cross} for the derivative of the crossing probability of~$\calD$ 
	is significantly different than the one for percolation, 
	since edges far from~$\calD$ -- which correspond to the second term in the formula -- may contribute substantially. 
	Indeed, if we accept the asymptotic~$\Delta(r,R) \asymp (r/R)^{\iota}$, and if~$\iota < 1$, 
	then the edges at distance~$L(p)$ contribute most to~\eqref{eq:deriv_cross}; when~$\iota > 1$ however, 
	the derivative is governed by the contribution of edges close to~$\calD$.  
	See Section~\ref{sec:scaling_iota} for consequences of these two types of behaviour. 
\end{remark}

The proof of Proposition~\ref{prop:deriv_cross} is based on the following lemma, 
which controls the influence of each edge on a crossing event. 
Below, we call~$e$ a (open) {\em $r$-pivotal} in~$\omega$ for a crossing event in~$(\calD,a,b,c,d)$ 
if~$\omega$  contains a crossing of~$\calD$ from~$(ab)$ to~$(cd)$, but~$\omega \cap (\calD\setminus\Lambda_r(e))$ does not. 
Call this event~${\rm Piv}_{r,e}(\calD)$.

\begin{lemma}\label{lem:Cov_Delta}
	There exists~$c > 0$ such that, for any~$\eta>0$, the following holds. 
	For~$p \in (0,1)$,~$R \le L(p)$, every~$\eta$-regular quad~$(\calD,a,b,c,d)$ of size~$R$
	and every edge~$e$ at a distance~$n$ from~$\partial\calD$, 
	\begin{align}
		\Delta_p(R)\les \mathrm{Cov}_{p}[\omega_e;\calC(\calD)] &\les \sum_{r = n/2}^{3R}\frac {\Delta_p(r)}r\phi_p[{\rm Piv}_{r,e}(\calD)]
		&& \text{ if~$n \le 2R$}, \label{eq:Cov_Delta1}\\
		\mathrm{Cov}_{p}[\omega_e;\calC(\calD)] &\asymp  \Delta_p(n)\Delta_p(R,n) 
		&&\text{ if~$2R \le n \le 2L(p)$}, \label{eq:Cov_Delta2}\\
		\mathrm{Cov}_{p}[\omega_e;\calC(\calD)] &\les \Delta_p(L(p))\Delta_p(R,L(p))e^{-cn/L(p)} 
		&& \text{ if~$n \geq 2L(p)$},\label{eq:Cov_Delta3} 
	\end{align}
	where the constants in~$\asymp$ and~$\les$ depend on~$\eta$. 
	
	The same bounds hold with 
	$\mathrm{Cov}_{p}[\omega_e; A_{\sigma}(r,R)]/\phi_p[A_{\sigma}(r,R)]$ instead of~$\mathrm{Cov}_{p}[\omega_e;\calC(\calD)]$, 
	with constants which depend on~$\sigma$;~$n$ should then be replaced with the distance from~$e$ to~$\partial \Ann(r,R)$.
\end{lemma}

\begin{remark}\label{rem:Cov_Delta}
	Due to the choice of~$\calD$, all edges~$e$ inside~$\calD$ are in case~\eqref{eq:Cov_Delta1}.	
	By~\eqref{eq:RSWnear}, $\phi_p[{\rm Piv}_{r,e}(\calD)]$ is uniformly bounded away from~$0$ for all~$R \le r\le 2R$. 
	These terms of the sum account for a contribution of order at least~$\Delta_p(R)$ to the right-hand side of~\eqref{eq:Cov_Delta1}.
	When~$e$ is in the ``bulk'' of~$\calD$ (that is at a distance of order~$R$ from~$\calD^c$), 
	the lower and upper bounds in~\eqref{eq:Cov_Delta1} are comparable.
	
	In addition we point out that the same results hold for covariances under a measure~$\phi_{G}^\xi$ for any sub-graph~$G$ of~$\bbZ^2$ that contains~$\La_{2R}$, any boundary conditions~$\xi$, and any edge~$e$ closer to~$\calD$ than to~$G^c$. The proof below adapts readily. 
\end{remark}

Theorem~\ref{thm:delta}(i) is a particular case of Lemma~\ref{lem:Cov_Delta}. 
\begin{proof}[Theorem~\ref{thm:delta}(i)]
	Apply \eqref{eq:Cov_Delta1} and observe that the right-hand side is smaller than $\sum_{r = \eta R}^R\frac1r\Delta_p(r) \les \Delta_p(R)$.
\end{proof}

\begin{proof}[Lemma~\ref{lem:Cov_Delta}]
We will prove here formulas for the crossing of a quad~$(\calD,a,b,c,d)$; 
the proof for arm events is similar and entails standard adaptations. 
Write~$\calC$ instead of~$\calC(\calD)$. 
Recall that 
\begin{align*}
    \mathrm{Cov}[\omega_e;\calC] =\phi[\omega_e](\phi[\calC \,|\, \omega_e = 1] -\phi[\calC])
    \asymp \phi[\calC \,|\, \omega_e = 1] -\phi[\calC \,|\, \omega_e = 0].
\end{align*}
We will prove each of the equations separately, starting with the second. 
\medskip 

\noindent{\em Proof of~\eqref{eq:Cov_Delta2}:}
	Construct an increasing coupling~$\bbP$ between~$\phi[\cdot\,|\,\omega_e = 0]$ and ~$\phi[\cdot\,|\,\omega_e = 1]$ 
	producing configurations~$\omega$ and~$\omega'$ as follows: the coupling contains three stages based on 
	Theorem~\ref{thm:coupling}~(ii), Lemma~\ref{lem:boosting_out_to_boosting_in} and  Theorem~\ref{thm:Delta_coupling}, respectively. We refer to Figure~\ref{fig:coupling far} for a picture.

	\begin{enumerate}[noitemsep,nolistsep]
\item Apply the coupling inside~$\La_{n/4}(e)$ between 
	~$\phi[\cdot\,|\,\omega_e = 0]$ and ~$\phi[\cdot\,|\,\omega_e = 1]$ 
	using the procedure of Theorem~\ref{thm:coupling}(ii), 
	up to the associated stopping time, which we denote by~$\tau_1$. 
	If~$\tau_1  = \infty$, continue the coupling using an arbitrary increasing coupling. 
	Recall that when~$\tau_1 < \infty$, 
	$\calF_{\tau_1}$ is a~$1/2$-well-separated outer flower domain on~$\La_{n/4}$
	and that~$\omega_{[\tau_1]}$ and~$\omega'_{[\tau_1]}$ induce a boosting pair of boundary conditions on~$\calF_{\tau_1}$. 
	
\item Continue the coupling on~$\La_{n/2}^c$ between 
	$\phi^{\omega'_{[\tau_1]}}_{\calF_{\tau_1}}$ and~$\phi^{\omega_{[\tau_1]}}_{\calF_{\tau_1}}$
	using the procedure of Lemma~\ref{lem:boosting_out_to_boosting_in}, 
	up to the associated stopping time, which we denote by~$\tau_2$. 
	If~$\tau_2  = \infty$, continue the coupling using an arbitrary increasing coupling. 
	Recall that when~$\tau_2 < \infty$, 
	$\calF_{\tau_2}$ is a~$1/2$-well-separated inner flower domain on~$\La_{n/2}$
	and that~$\omega_{[\tau_2]}$ and~$\omega'_{[\tau_2]}$ induce a boosting pair of boundary conditions on~$\calF_{\tau_2}$. 

\item Complete the coupling inside~$\calF_{\tau_2}$ using an arbitrary increasing coupling of 
	$\phi^{\omega'_{[\tau_2]}}_{\calF_{\tau_2}}$ and~$\phi^{\omega_{[\tau_2]}}_{\calF_{\tau_2}}$. 
\end{enumerate}
	\begin{figure}
	\begin{center}
	\includegraphics[width = 0.9\textwidth]{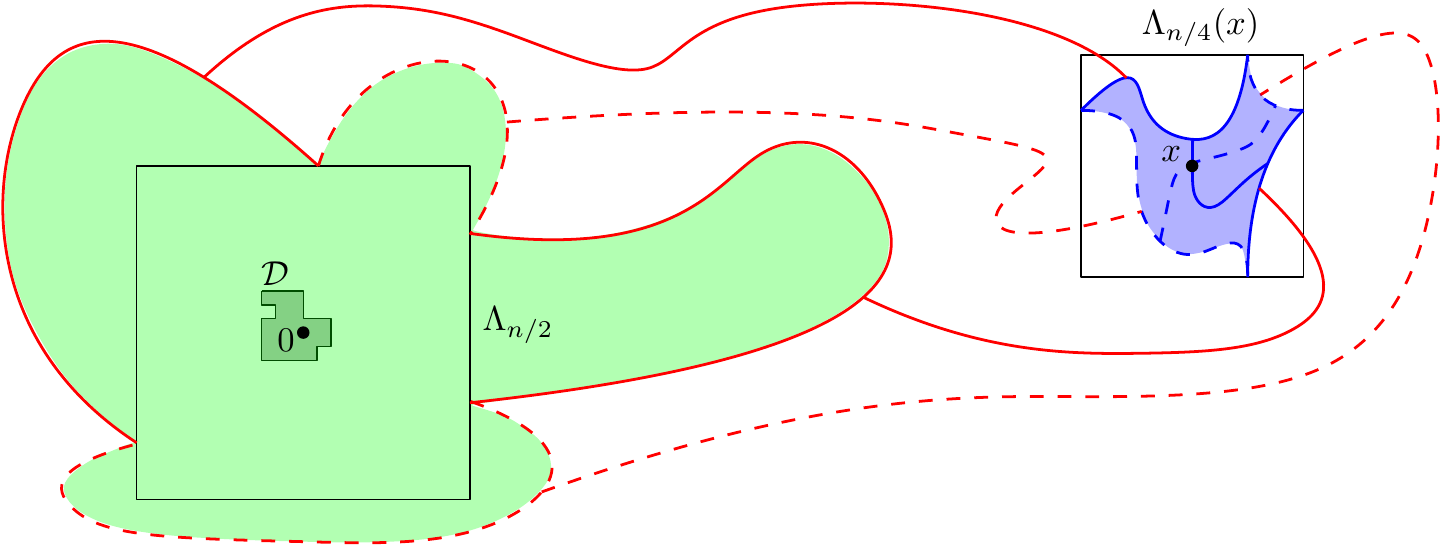}
	\caption{In blue, the edges discovered until $\tau_1$. In red, the edges discovered between $\tau_1$ and $\tau_2$. Finally in green, the edges that are discovered afterwards. Note that the domain $\calD$ is a priori much smaller than the box $\Lambda_{n/2}$. }
	\label{fig:coupling far}
	\end{center}
	\end{figure}	
For the upper bound observe that, since~$\calD \subset \La_{n/2}$ and~$\La_{n/2}(e)$ are disjoint, if~$\zeta$ and~$\zeta'$ are the boundary conditions induced on~$\calF_{\tau_1}$, Theorem~\ref{thm:Delta_coupling} gives
	\begin{align}\label{eq:cov_delta2ub}
		 \phi[\calC \,|\, \omega_e = 1] -\phi[\calC\,|\,\omega_e = 0]
		& \le (\phi_{\La_{n/2}}^1[\calC] - \phi_{\La_{n/2}}^0[\calC])\,
	\mathbb P[\zeta\ne \zeta']
		\les \Delta_p(R,n/2)\Delta_p(n).
	\end{align}
	Finally,~$\Delta_p(R,n/2)$ may be replaced by~$\Delta_p(R,n)$ due to~\eqref{eq:Delta_lengthen_out}.
For the lower bound, Theorem~\ref{thm:Delta_coupling}  gives
	\begin{align*}
		\phi[\calC\,|\, \omega_e = 1] - \phi[\calC\,|\, \omega_e = 0]
		&= 
		\bbP[\omega' \in \calC,\omega\notin \calC ]\\
		&\geq \bbE \big[\ind_{\{\tau_2 < \infty\}}\big( \phi_{\calF_{\tau_2}}^{\omega'_{[{\tau_2}]}}[\calC]- \phi_{\calF_{\tau_2}}^{\omega_{[{\tau_2}]}}[\calC]\big)\big]\\
		&\ges \bbP[\tau_1 < \infty]\, \bbP[\tau_2 < \infty\,|\, \tau_1 < \infty]\, \Delta_p(R,n/2).
	\end{align*}	
	The second term above is of constant order by Lemma~\ref{lem:boosting_out_to_boosting_in}.
	The first term is of order~$\Delta_p(1,n/4)$ by Theorem~\ref{thm:Delta_coupling}. Since~$\Delta_p(1,n/4)\ge \Delta_p(n)$ and~$\Delta_p(R,n/2)\ge \Delta_p(R,n)$, we obtain the result.
	\hfill~$\diamond$
\medskip 

\noindent{\em Proof of~\eqref{eq:Cov_Delta3}:}	
	We focus here on the case~$p < p_c$; when~$p > p_c$ the proof is obtained by duality. 
	Construct the following coupling~$\bbP$ between~$\phi_p[\cdot\,|\,\omega_e=0]$ and~$\phi_p[\cdot\,|\,\omega_e=1]$:
	\begin{enumerate}[noitemsep,nolistsep]
		\item Use the coupling given by Theorem~\ref{thm:coupling}(ii) between~$\phi_p[\cdot\,|\,\omega_e=0]$ and~$\phi_p[\cdot\,|\,\omega_e=1]$ inside~$\La_{L(p)}(e)$; 
		let~$\xi$ and~$\xi'$ be the boundary conditions induced by the revealed configurations~$\omega$ and~$\omega'$ 
		on the boundary of~$\bbZ^2\setminus\La_{L(p)}(e)$;
		\item Continue the coupling in~$\bbZ^2\setminus\La_{L(p)}$ using the decision tree that explores the connected components of~$\La_{L(p)}(e)$ 
		in~$\omega'$ (see Example 2 of Section~\ref{sec:coupling exploration}); 
		let~$\zeta,\zeta'$ be the boundary conditions induced by the revealed regions of~$\omega$ and~$\omega'$ on~$\partial \La_{L(p)}$;
		\item Use an arbitrary increasing coupling of~$\phi_{\La_{L(p)},p}^\zeta$ and~$\phi_{\La_{L(p)},p}^{\zeta'}$.
	\end{enumerate}
	If~$\xi = \xi'$, then~$\omega$ and~$\omega'$ are identical in the complement of~$\La_{L(p)}(e)$. 
	Moreover, as explained in Section~\ref{sec:coupling exploration}, for~$\zeta$ to differ from~$\zeta'$, 
	$\omega'$ must contain a connection between~$\La_{L(p)}(e)$ and~$\La_{L(p)}$. Thus, we find
	\begin{align}
		\phi_p[\calC \,|\, \omega_e = 1] -\phi_p[\calC\,|\, \omega_e = 0] 
		&\le \bbE\big[\ind_{\{\xi \neq\xi'\}} \ind_{\omega'\in\{\La_{L}(e) \xlra{} \La_{L(p)}\}} ( \phi^{\zeta'}_{\La_L,p}[ \calC] - \phi^{\zeta}_{\La_L,p}[\calC])\big ]\label{eq:cov_delta3ub}\\
		&\les \Delta_p({L(p)}) \phi_{\bbZ^2\setminus\La_{L(p)}, p}^1[\La_{L(p)}\longleftrightarrow\La_{L(p)}(e) ] \Delta_p(R,{L(p)}) .
		\nonumber
	\end{align}
	In the second inequality we used the monotonicity of boundary conditions~\eqref{eq:CBC} and Theorem~\ref{thm:Delta_coupling}.
	Finally, due to the mixing property of Proposition~\ref{prop:mixing} and Proposition~\ref{prop:estimate subcritical}, 
	the second term of the last product is bounded above by~$\exp[-cn/L(p)]$ for some positive constant~$c$. 
\medskip \hfill~$\diamond$

\noindent{\em Proof of~\eqref{eq:Cov_Delta1}:}
For the lower bound, translate~$\calD$ and~$e$ such that~$e\notin \La_{\eta R/2}$ and~$\La_{\eta R/2} \subset \calD$. 
Since~$\calD$ is assumed~$\eta$-regular, this is possible. 
We produce a coupling $\bbP$ between~$\phi_p[\cdot\,|\,\omega_e=0]$ and~$\phi_p[\cdot\,|\,\omega_e=1]$ as follows. 
\begin{enumerate}[noitemsep,nolistsep]
		\item Explore the double four-petal flower domain~$(\calF_{\rm in}, \calF_{\rm out})$ in~$\omega$ between ~$\La_{\eta R/4}$ and ~$\La_{\eta R/8}$; 
if no such  double four-petal flower domain exists, reveal the rest of the configurations in arbitrary order.
\item Continue~$\bbP$ by sampling~$\omega$ and~$\omega'$ inside~$\calF_{\rm in}$. 
\item Reveal the configurations in~$\calF_{\rm out}$. 
\end{enumerate}

 Now,  if~$(\calF_{\rm in}, \calF_{\rm out})$ exists, the argument of~\eqref{eq:Cov_Delta2} shows that 
\begin{align*}
\phi[P_1^{\rm out} \xlra{\calF_{\rm out}} P_3^{\rm out} | \omega_e =1, \calF_{\rm in},  \calF_{\rm out}] -
\phi[P_1^{\rm out} \xlra{\calF_{\rm out}} P_3^{\rm out} | \omega_e =0, \calF_{\rm in},  \calF_{\rm out}] \ges \Delta_p(R).
\end{align*}
Indeed,~$e$ is at a distance comparable to~$R$ from~$(\calF_{\rm in}, \calF_{\rm out})$; 
for details see the proof of point~\eqref{eq:boosting_pair_circ_i} of Theorem~\ref{thm:boosting_pair}.
Note that this is simply a statement on the conditional probability of the connections between the petals~$P_1$ and~$P_3$, 
we do not reveal the configurations~$\omega$ and~$\omega'$ on~$\calF_{\rm out}$. 

Then, as in the proof of~\eqref{eq:boosting_pair_circ_i}, the previous estimate gives that
\begin{align*}
    \phi[P_1^{\rm in} \xlra{\calF_{\rm in}} P_3^{\rm in} | \omega_e =1, \calF_{\rm in},  \calF_{\rm out}] -
    \phi[P_1^{\rm in} \xlra{\calF_{\rm in}} P_3^{\rm in} | \omega_e =0, \calF_{\rm in},  \calF_{\rm out}] \ges \Delta_p(R).
\end{align*}
 
Let~$H$ be the event that~$P_1^{\rm out}$ and~$P_3^{\rm out}$ are connected in~$\omega \cap \calF_{\rm out} \cap \calD$ to the arcs~$(ab)$ and~$(cd)$, respectively, while~$P_2^{\rm out}$ and~$P_4^{\rm out}$ are connected in~$\omega^* \cap \calF_{\rm out} \cap \calD$ to the arcs~$(bc)$ and~$(da)$ respectively.
By Theorem~\ref{thm:RSWnear}, 
\begin{align*}
    \phi[H\, | \,(\calF_{\rm in},  \calF_{\rm out}) \text{ and } (\omega,\omega)' \text{ on~$\calF^c_{\rm out}$}] \ges 1.
\end{align*}
Finally, when~$H$ occurs,~$\calD$ is crossed if and only if~$P_1^{\rm in}$ and~$P_3^{\rm in}$ are connected inside~$\calF_{\rm in}$. 
As a consequence, the two displays above and Lemma~\ref{lem:D4PFD} conclude that 
\begin{align*}
	&\phi_p[\calC \,|\, \omega_e = 1] -\phi_p[\calC\,|\, \omega_e = 0]\\
	&\qquad\ges \bbP[(\calF_{\rm in},  \calF_{\rm out}) \text{ exists}]
	\bbP[P_1^{\rm in} \xlra{\omega' \cap \calF_{\rm in}} P_3^{\rm in} \text{ but }P_1^{\rm in} \nxlra{\omega \cap \calF_{\rm in}} P_3^{\rm in}\,|\, (\calF_{\rm in},  \calF_{\rm out}) \text{ exists}]\\
	&\qquad\ges \Delta_p(R).
\end{align*}

We turn to the upper bound. 
For~$s > \lceil \log n \rceil$, let~$\calP_s = {\rm Piv}_{2^s,e}(\calD) \setminus {\rm Piv}_{2^{s-1},e}(\calD)$,
and set~$\calP_s = {\rm Piv}_{2^s,e}(\calD)$ when~$s = \lceil \log n \rceil$.
These events partition~$\calC$, and we find that
\begin{align*}
	\mathrm{Cov}_{p}(\omega_e,\calC)
	&= \sum_{s> \log n}(\phi [\omega_e | \calP_s]-\phi[\omega_e])\phi_p[\calP_s]\\
    &\le \sum_{s > \log n}(\phi_{\Lambda_{2^{s-1}}(e)}^1[\omega_e]-\phi_{\Lambda_{2^{s-1}}(e)}^0[\omega_e])\phi_p[{\rm Piv}_{2^s,e}(\calD)]\\
    &\les \sum_{r\ge n/2}\tfrac{1}r\Delta_p(r)\phi_p[{\rm Piv}_{r,e}(\calD)].
\end{align*}
where the first inequality is based on the spatial Markov property, the fact that~$\calP_s$ is measurable in the edges outside~$\La_{2^{s-1}}$
and the inclusion~$\calP_s \subset {\rm Piv}_{2^s,e}(\calD)$. 
The second is a simple reindexing of the sum that uses~\eqref{eq:Delta_lengthen_out} 
and the monotonicity in~$r$ of the events~${\rm Piv}_{r,e}(\calD)$.
To obtain the upper bound in~\eqref{eq:Cov_Delta1}, it suffices to notice that~${\rm Piv}_{r,e}(\calD)$ may not occur for~$r \geq 3R$. 
\end{proof}

\begin{remark}\label{rem:dtv}
	The upper bounds in \eqref{eq:Cov_Delta2} and \eqref{eq:Cov_Delta3} 
	may be shown to apply to any event~$H$ which depends only on the edges in~$\La_R$. 
	Indeed, for~$H$ increasing the proof above applies mutatis mutandis. 
	When~$H$ is not increasing, some additional care needs to be taken when bounding 
	$|\phi[H \,|\, \omega_e = 1] - \phi[H \,|\, \omega_e = 0]|$ in \eqref{eq:cov_delta2ub} and \eqref{eq:cov_delta3ub}. 
	Notice, however, that in both equations one may produce a coupling of~$\phi[\cdot \,|\, \omega_e = 0]$ and~$\phi[\cdot \,|\, \omega_e = 1]$ 
	that produces equal configurations inside~$\La_R$ with probability~$1 - O(\Delta_p(R))$ and~$1 - O(\Delta_p(R,n))$, respectively. 
	As a consequence, the terms on the right-hand side in \eqref{eq:Cov_Delta2} and \eqref{eq:Cov_Delta3} 
	bound from above the distance in total variation between the restrictions of 
	$\phi[\cdot\,|\, \omega_e = 1]$ and~$\phi[\cdot\,|\, \omega_e = 0]$ to~$\La_R$. 
\end{remark}

\begin{proof}[Proposition~\ref{prop:deriv_cross}]
	We focus here on the expression for the crossing probability of a quad;
	the proof for the derivatives of probabilities of arm events is identical. 

	Fix~$p$ and some~$\eta$-regular quad~$(\calD,a,b,c,d)$ at scale~$R\le L(p)$.
	By grouping the contributions of different edges depending on their distance to the origin and the boundary of~$\Lambda_R$, 
	and using Lemma~\ref{lem:Cov_Delta} we have
		\begin{align*}
		\tfrac{{\rm d}}{{\rm d}p}\phi_{p}[\calC(\calD)]
		&= \frac1{p(1-p)} \sum_{e} \mathrm{Cov}_p(\omega_e, \calC(\calD)) 
		\ges R^2 \Delta_p(R) + \sum_{n =R}^{L(p)}  n \Delta_p(n)\Delta_p(R,n),
	\end{align*}
	since there are order~$R^2$ edges in the case~\eqref{eq:Cov_Delta1} and order~$n$ edges at distance~$n$ from~$\partial \calD$. 
	
	For the upper bound we use the three upper-bounds of Lemma~\ref{lem:Cov_Delta}.
	We split the edges into three categories:
	those at a distance less than~$2R$ from~$\partial \calD$, 
	those at a distance between~$2R$ and~$L(p)$ from~$\partial \calD$,
	and those at a distance larger than~$L(p)$ from~$\partial \calD$.

	For the two last categories,  keeping in mind that there are~$O(n)$ edges at a distance exactly~$n \geq 2R$ from~$\partial \calD$,
	and applying \eqref{eq:Cov_Delta2} and \eqref{eq:Cov_Delta3}, we have
	\begin{align}\label{eq:cat3}
    	\sum_{e:\,{\rm dist}(e, \partial \calD) \geq 2R} \!\!\!\!\!\!\mathrm{Cov}_p(\omega_e, \calC(\calD)) 
    	&\les \sum_{n= 2R}^{2L(p)} n \Delta_p(n)\Delta_p(R,n) 
    	 + \sum_{n\geq 2L(p)} n \Delta_p(L(p))\Delta_p(R,L(p))e^{-cn/L(p)}\nonumber\\
    	& \les  R^2 \Delta_p(R) + \sum_{n= R}^{L(p)} n \Delta_p(n)\Delta_p(R,n).
	\end{align}
	Indeed, the second sum of the middle term may be ignored due to the exponential factor. 
	The first term in the second line is only important in the particular situation when~$R$ is very close to~$L(p)$.
	 
	We turn to the edges in the first category: those close to~$\partial \calD$. We refer to Figure~\ref{fig:derivative} for an illustration. 
	Fix~$n\le 2R$. We start by studying the contribution to the derivative of the edges at a distance~$n$ of the arc~$(ab)$ (the cases of other arcs $(bc)$, $(cd)$ and $(da)$),
	and we precisely focus on the term in the right-hand side of \eqref{eq:Cov_Delta1} for these edges and corresponding 
	to some fixed~$n/2 \leq r\leq 3R$.
	Consider a family of vertices~$a=x_0,\dots,x_k$ found counterclockwise along~$(ab)$ at a~$\ell^\infty$ distance~$r$ from each other, with the last vertex~$x_k$ at distance at most~$r$ of~$b$ (when~$r$ is large~$k$ is equal to 0). Consider the event~$E_i$ that~$(x_{i-4}x_{i-3})$ is dual connected to~$(da)$ and~$(x_{i+3}x_{i+4})$ is connected to~$(cd)$, with the convention that if~$i\le 4/\eta$ or~$i\ge k-4/\eta$ then~$E_i$ is the full event. We claim that, for every~$e\subset\Lambda_r(x_i)$, 
	\[
	\phi_p[{\rm Piv}_{r,e}(\calD)]\les \phi_p[E_i].
	\]
	Indeed, the inequality is trivial when~$i\le 4/\eta$ or~$i\ge k-4/\eta$ since $E$ is the full event. 
	For the remaining values of~$i$ (which exist only when~$r\ll \eta R$), 
	observe that for~${\rm Piv}_{r,e}(\calD)$ to occur, $\Lambda_{2r}(x_i)$ needs to be connected by a primal path to $(cd)$
	and by dual paths to $(bc)$ and $(da)$. 
	In particular, the interface starting from~$d$ and delimiting the primal cluster of~$(cd)$ in~$\calD$
	necessarily hits~$\Lambda_{2r}(x_i)$ before hitting the arc $(ab)$. 
	Let $\Gamma$ be the section of this interface from $d$ up to the first time it hits ~$\Lambda_{2r}(x_i)$. 
	Conditionally on $\Gamma$, one may use Theorem~\ref{thm:RSWnear} to construct a dual path connecting~$\Gamma$ to~$(x_{i-4}x_{i-3})$ 
	and a primal path connecting~$\Gamma$ to~$(x_{i+3}x_{i+4})$  with probability uniformly bounded away from zero.
	These paths, together with $\Gamma$, induce the connections required by $E_i$. 
	
	Notice now that any edge at distance $n$ from $(ab)$ is contained in at least one $\La_r(x_i)$. 
	Conversely there are~$O(r)$ vertices in each~$\Lambda_{r}(x_i)$ at a distance exactly~$n$ from~$(ab)$. 
	Thus, summing over the edges~$e$ at a distance~$n$ from~$(ab)$ gives 
	\[
	\sum_{e:{\rm dist}(e,(ab))=n}\phi_p[{\rm Piv}_{r,e}(\calD)]\les r \sum_{i=0}^k\phi_p[E_i]\les r,
	\]
	where the second inequality is due to the fact that at most~$O(1)$ events~$E_i$ can occur simultaneously. 	
	
	One may do the same with edges at a distance~$n$ of~$(bc)$,~$(cd)$ and~$(da)$.
	Finally, summing the previous displayed equation over $r$ and using~\eqref{eq:Cov_Delta1}, we find
	\begin{align*}
    	\sum_{e:{\rm dist}(e,\partial \calD)\le 2R}\Cov_p(\omega_e,\calC(\calD))
    	&\les \sum_{n=1}^{2R}\sum_{r=n/2}^{3R}\tfrac1r\Delta_p(r)\sum_{e:{\rm dist}(e,\partial \calD)=n}\phi_p[{\rm Piv}_{r,e}(\calD)]\\
    	&\les \sum_{n=1}^{2R}\sum_{r=n/2}^{3R}\Delta_p(r)
    	\les \sum_{r=1}^{R}r \Delta_p(r).
	\end{align*}
	The above combined with \eqref{eq:cat3} implies the upper bound in \eqref{eq:deriv_cross}.
\end{proof}

\begin{figure}
	\begin{center}
	\includegraphics[width=0.6\textwidth]{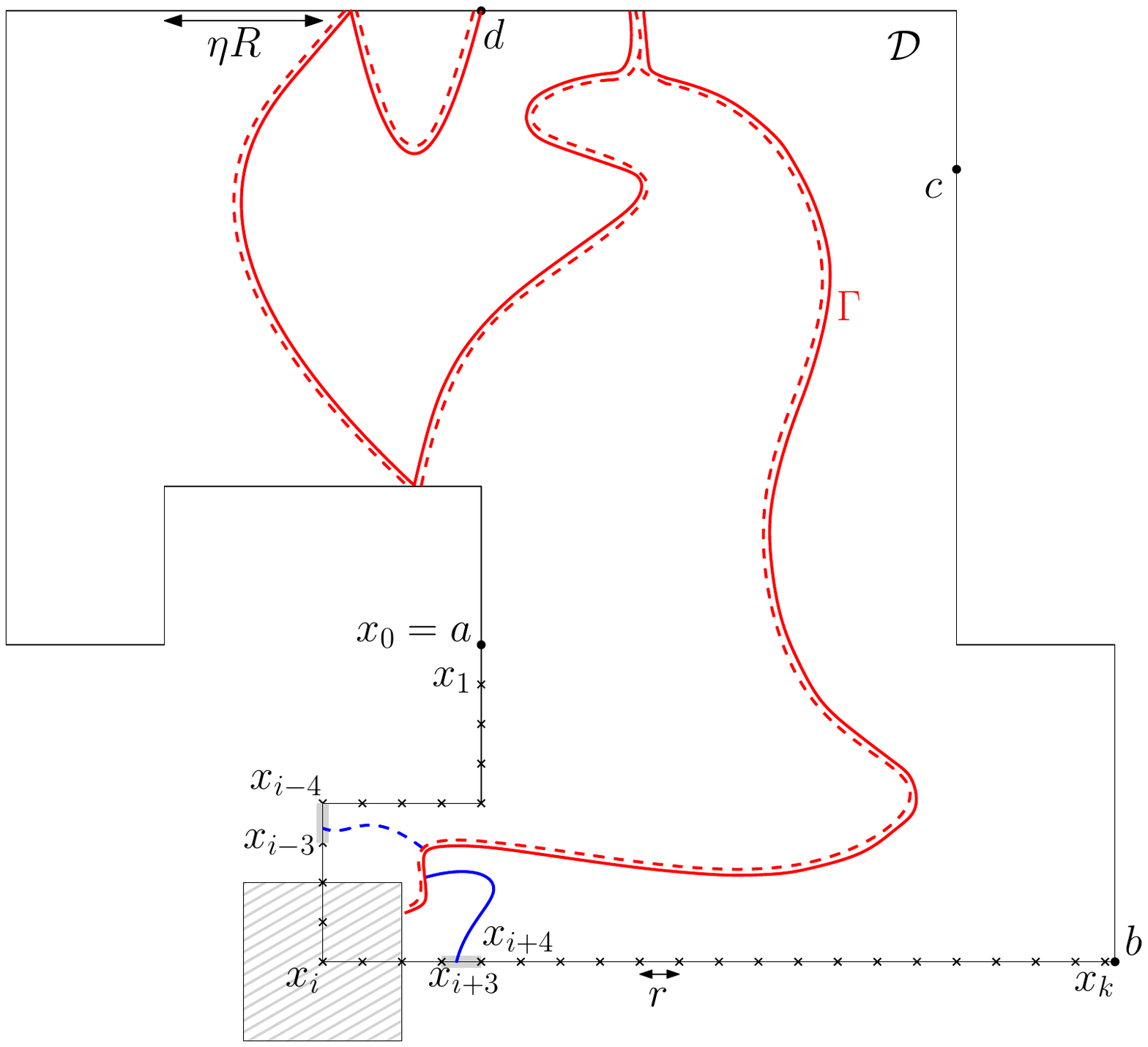}
	\caption{A depiction of the points~$x_0,\dots,x_k$ (note that~$x_k$ is not necessarily equal to~$b$). 
	The occurrence of ${\rm Piv}_{r,e}(\calD)$ for some $e \in \La_r(x_i)$ induces the existence of the exploration path~$\Gamma$ (in red).
	When orienting~$\Gamma$ in the direction of its exploration, that is from $d$ to $\La_{2r}(x_i)$, 
	$\Gamma$ has dual open edges on its right and primal ones on its left. 
	Thus, conditionally on $\Gamma$, the blue paths occur with uniformly positive probability, and induce the occurrence of $E_i$. } \label{fig:derivative}
	\end{center}
	\end{figure}

\subsection{Derivative for the mixing rate}
	
\begin{proposition}\label{prop:deriv_Delta_ub}
    For every~$p$ and two edges~$e$ and~$f$ at a distance~$R\le L(p)$ of each other, 
    \begin{equation}\label{eq:deriv_Delta_ub}
        \Big|\tfrac{{\rm d}}{{\rm d}p}\log \Cov_p(\omega_e,\omega_f)\big|
        \les \sum_{\ell=1}^R\ell\Delta_p(\ell) + \sum_{\ell=R}^{L(p)} \ell \Delta_p(\ell)\Delta_p(R,\ell).
    \end{equation}
\end{proposition}

Exactly as for crossing events, one may use the next section to replace~$\sum_{\ell\le R}\ell \Delta_p(\ell)$ by~$R^2\Delta_p(R)$; see Corollary~\ref{cor:deriv} of Section~\ref{sec:7}.

\begin{remark}\label{rmk:h}
    Since~$\Cov_p(\omega_e,\omega_f)$ was shown in~\eqref{eq:Cov_Delta2}
    to be comparable to~$\Delta_p(R)^2$, the above should be understood as a bound on the logarithmic derivative of~$\Delta_p(R)$. 
    Indeed,~\eqref{eq:deriv_Delta_ub} combined with~\eqref{eq:Cov_Delta2} yields
    \begin{align}\label{eq:ub_deriv_Delta_int}
    	\Big|\log \frac{\Delta_p(R)}{\Delta_{p_c}(R)}\Big|
    	\les \int_{p_c}^p \Big(\sum_{\ell=1}^R\ell\Delta_u(\ell)+ \sum_{\ell=R}^{L(u)} \ell \Delta_u(\ell)\Delta_u(R,\ell)\Big) \,du,
    \end{align}
    for all~$p$ and~$R \le L(p)$.
    This inequality will be useful when~$R$ is of the same order as~$L(p)$, 
    in which case the terms with~$R\le \ell \le L(p)$
    may be absorbed by the first sum on the right-hand side. 

\end{remark}	

\begin{proof}
	Fix~$p$,~$e$ and~$f$ as in the statement. 	
	Then 
	\begin{align}\label{eq:CC'0}
    	\tfrac{{\rm d}}{{\rm d}p} \Cov_p(\omega_e,\omega_f)
		&= \sum_{g \in \bbE}  \phi_p[\omega_e   \omega_f   \omega_g ] - \phi_p[\omega_e   \omega_f ]\phi_p[\omega_g] - \phi_p[\omega_f   \omega_g ]\phi_p[\omega_e]\\
		&\qquad - \phi_p[\omega_e   \omega_g]\phi_p[\omega_f ] +2 \phi_p[\omega_e]\phi_p[\omega_f]\phi_p[\omega_g ].\nonumber
	\end{align}
	We will write these terms differently, depending on the position of~$g$. 
	First, consider an edge $g \in \La_{R/2}(e)$; the corresponding term may be written as
	\begin{align}\label{eq:CC'}
	\phi_p[\omega_f]\Big( \phi_p[\omega_e   \omega_g |\omega_f ] - \phi_p[\omega_g|\omega_f]\phi_p[\omega_e]- \phi_p[\omega_e|\omega_f]\phi_p[\omega_g ] 
		- \phi_p[\omega_e   \omega_g  ] +2 \phi_p[\omega_e]\phi_p[\omega_g ]\Big).
	\end{align}
	Write~$\ell$ for the distance between~$e$ and~$g$ and set~$G := \La_{3\ell/2}(e)$.	
	For any event~$A$ depending only on the edges in~$G$, we have 
	\begin{align*}
		\phi_p[A] = \sum_\xi \phi_p[B_\xi] \phi_{G,p}^\xi [A] \quad \text{ and }\quad  \phi_p[A \,|\,\omega_f] = \sum_\xi \phi_p[B_\xi\,|\,\omega_f] \phi_{G,p}^\xi [A],
	\end{align*}
	where the sum is over all boundary conditions~$\xi$ imposed on~$G$ by the configuration outside of~$G$, and~$B_\xi$ is the event that the boundary conditions induced on~$G$ are~$\xi$. Thus,~\eqref{eq:CC'} may be written as
	\begin{align*}
		\phi_p[\omega_f]\sum_\xi  (\phi_p[B_\xi\,|\,\omega_f] - \phi_p[B_\xi])
		(\underbrace{\phi_{G,p}^\xi[\omega_e   \omega_g] -\phi_{G,p}^\xi[\omega_g]\phi_p[\omega_e]- \phi_{G,p}^\xi[\omega_e]	\phi_p[\omega_g] + \phi_p[\omega_e]\phi_p[\omega_g]}_{F(\xi)}).
	\end{align*}
	By adding and subtracting~$\phi_{G,p}^\xi[\omega_e] \phi_{G,p}^\xi[\omega_g ]$, we find 
	\begin{align*}
		F(\xi) = \Cov_{G,p}^\xi(\omega_e,  \omega_g) +(\phi_{G,p}^\xi[\omega_g ] - \phi_p[\omega_g ] )( \phi_{G,p}^\xi[\omega_g]  -\phi_p[\omega_g]).
	\end{align*}
	where~$\Cov_{G}^\xi$ stands for the covariance under the measure~$\phi_{G,p}^\xi$. 
	Next, we apply Lemma~\ref{lem:Cov_Delta} to~$\phi_{G,p}^\xi$ instead of~$\phi_p$ 
	-- notice that the choice of~$G$ ensures that both~$e$ and~$g$ are at a distance of order $\ell$ from~$\partial G$, 
	and the lemma does indeed apply as explained in Remark~\ref{rem:Cov_Delta} -- to obtain
	\begin{align*}
		\Cov_{G}^\xi(\omega_e,  \omega_g) &\les \Delta_p(\ell)^2.
	\end{align*}
	In addition, for any~$\xi$, Theorem~\ref{thm:Delta_coupling} yields
	\begin{align*}
	\phi_{G,p}^\xi[\omega_g ] - \phi_p[\omega_g ] \les \Delta_p(\ell) \quad \text{ and }\quad
	\phi_{G,p}^\xi[\omega_e]  -\phi_p[\omega_e] \les \Delta_p(\ell).
	\end{align*}
	Then,~\eqref{eq:CC'} may be bounded above by 
	\begin{align}\label{eq:CC10}
		 \sum_\xi   \|F\|_\infty\cdot
|\phi_p[B_\xi|\omega_f] - \phi_p[B_\xi]|	\les \Delta_p(R) \Delta_p(\ell, R)\Delta_p(\ell)^2 
		\asymp \Delta_p(\ell) \Delta_p(R)^2,
	\end{align}
	since the total variation distance between the restrictions of~$\phi_p[\cdot\,|\omega_f]$ and~$\phi_p$ to~$G$ 
	is bounded by a universal multiple of~$\Delta_p(R) \Delta_p(\ell, R)$, as shown in Lemma~\ref{lem:Cov_Delta} (see also Remark~\ref{rem:dtv}).
	The second equivalence is given by the quasi-multiplicativity of~$\Delta_p$~\eqref{eq:delta_quasi}. 
	
	Edges~$g \in \La_{R/2}(f)$ have the same contribution to~\eqref{eq:CC'} due to symmetry. 
	Finally, we consider edges~$g$ that are outside~$\La_{R/2}(e)\cup\La_{R/2}(f)$; write~$r$ for the distance between~$g$ and~$e$. 
	Let~$G$ be the domain formed of the edges at a distance at least~$R/4$ from the segment uniting~$e$ and~$f$;
	notice that~$g$ is at a  distance at least~$R/4$ from~$G$. 
	Applying the same argument as in the first case, with the roles of~$\omega_f$ and~$\omega_g$ inverted, we find
	that the term corresponding to~$g$ in~\eqref{eq:CC'0} is
	\begin{align}
		&\phi_p[\omega_g]\sum_\xi  \big(\phi_p[\xi\,|\,\omega_g] - \phi_p[\xi]\big)\big(\Cov_{G,p}^\xi(\omega_e,  \omega_f) 
		+( \phi_{G,p}^\xi[\omega_e]  -\phi_p[\omega_e])(\phi_{G,p}^\xi[\omega_f ] - \phi_p[\omega_f ] ) \big) \nonumber \\
    	&\les \Delta_p(\ell \wedge L(p) )\Delta_p(R,\ell \wedge L(p))e^{-c r/L(p)} \Delta_p(R)^2\nonumber\\
		&\asymp \Delta_p(R)\Delta_p(\ell \wedge L(p) )^2 e^{-c r/L(p)}\label{eq:CC11},
	\end{align}
	where the sum in the left-hand side is over all the boundary conditions~$\xi$ on~$G$ imposed by the configuration outside of~$G$. 
	Indeed, by Lemma~\ref{lem:Cov_Delta}, the total variation distance between the restrictions of $\phi[\cdot\,|\,\omega_g]$ and~$\phi$ to~$G$ 
	is bounded by a universal multiple of~$\Delta_p(\ell \wedge L(p) )\Delta_p(R,\ell \wedge L(p))e^{-c r/L(p)}$,
	while each term in the second parentheses of the left-hand side is bounded by multiples of~$\Delta_p(R)^2$.
	
	Summing now over~$g$ and using~\eqref{eq:CC10} and~\eqref{eq:CC11}, we find 
	\begin{align*}
    	\tfrac{{\rm d}}{{\rm d}p} \Cov_p(\omega_e,\omega_f) 
		&\les \sum_{\ell =1}^{R/2} \ell \Delta_p(\ell) \Delta_p(R)^2 + \sum_{\ell >R/2} \ell \Delta_p(R)\Delta_p(\ell \wedge L(p) )^2 e^{-c r/L(p)}\\
		&\les \Delta_p(R)^2 \Big(\sum_{\ell =1}^{R} \ell \Delta_p(\ell) + \sum_{\ell =R}^{L(p)} \ell \Delta_p(R,\ell)\Delta_p(\ell)\Big).
	\end{align*}
	In the second inequality, we use the exponential term to eliminate all terms with~$\ell > L(p)$ and~\eqref{eq:Delta_lengthen_in} and~\eqref{eq:Delta_lengthen_out}
	to adjust the range of~$\ell$ in the sum. 
	Finally, divide by~$\Cov_p(\omega_e,\omega_f) \asymp \Delta_p(R)^2$ to obtain the desired result. 
\end{proof}

\section{Lower bound on~$\Delta_p(r,R)$: proof of Proposition~\ref{prop:lower bound Delta}}\label{sec:b}

In this section, we prove Proposition~\ref{prop:lower bound Delta}. The section will be divided in two as the case~$1\le q<4$ is quite different from the case~$q=4$.

\subsection{Lower bound on~$\Delta_p(r,R)$ for~$1\le q<4$}

In this section, we assume that~$1\le q<4$. We will prove the following stronger statement.
\begin{proposition}\label{prop:lower bound pi4 q<4}
Fix~$1\le q<4$. There exists~$\delta=\delta(q)>0$ such that for~$p \in (0,1)$ and~$r\le R\le L(p)$,
\begin{equation}\label{eq:stronger}
\pi_4(p,r,R)\ges (r/R)^{2-\delta}.
\end{equation}
\end{proposition}
This implies \eqref{eq:delta_thm4} since~$\Delta_p(r,R)\ge \pi_4(p,r,R)$ by the observation at the beginning of Section~\ref{sec:influence vs. pivotality} 
(one does not even require the polynomial improvement \eqref{eq:delta_pi4_improvement}).
Let us mention that \eqref{eq:stronger} is expected to fail for~$q=4$, which explains why the case~$q=4$ needs to be treated separately. 

\begin{proof}
By symmetry, we only need to treat the case~$p\ge p_c$. The case~$p=p_c$ was obtained in a recent paper~\cite[Prop.~6.8]{DumManTas20}. The argument extends readily to our setting. 

Indeed, if $E$ denotes the event that $\La_{3R}$ contains both an open circuit and
a dual open circuit surrounding  $\La_{2R}$, with the open one being connected to $\partial \La_{4R}$, 
then it is shown in~\cite{DumManTas20} that
$$\pi_4(p,r,R)\ge \phi_{p}[E] \inf_{\calD}\phi_{\calD,p}[ \mathbf M_r(\calD,R)],$$ 
where the infimum is taken over all simply connected domains containing $\La_{2R}$ and contained in $\La_{3n}$ 
and $\mathbf M_r(\calD,R)$ is an increasing random variable defined in~\cite[Sec.~1.5]{DumManTas20}. 
When $R \leq L(p)$, Theorem~\ref{thm:RSWnear} implies that $\phi_{p}[E] \ges 1$. 
Moreover, since $\mathbf M_r(\calD,R)$ is increasing, 
\begin{align*}
	\inf_{\calD}\phi_{\calD,p}[ \mathbf M_r(\calD,R)]
\geq  \inf_{\calD}\phi_{\calD,p_c}[ \mathbf M_r(\calD,R)] \ges (R/r)^{\delta_0},
\end{align*} 
for some universal $\delta_0 > 0$, 
where the last inequality is given by~\cite[Prop. 1.4]{DumManTas20}. Combining the above displays produces the desired bound.
\end{proof}

\begin{remark}
The proof looks simple but we wish to insist that the whole difficulty of the argument is in \cite{DumManTas20}.
\end{remark}

\subsection{Lower bound on~$\Delta_p(r,R)$ for~$q=4$}\label{sec:q=4}

The reasoning of the previous section does not apply for~$q=4$, as it is expected that
\[
\pi_4(p,r,R)\asymp (\tfrac rR)^2
\] (see Remark~\ref{rmk:ahah}). Nevertheless, this is not contradictory with the fact that~$\Delta_p(r,R)\ge \delta(\tfrac rR)^{2-\delta}$ as we know from Section~\ref{sec:influence vs. pivotality} that~$\Delta_p(r,R)$ is polynomially larger than~$\pi_4(p,r,R)$. 
Since we do not currently know how to prove that~$\pi_4(p,r,R)\asymp (\tfrac rR)^2$, 
we adopt a direct approach to prove Proposition~\ref{prop:lower bound Delta}, that does not involve the comparison to~$\pi_4(p,r,R)$.

\begin{proof}[ \eqref{eq:delta_thm4}] The lower bound on~$\Delta_p(r,R)$ follows directly from the combination of the next two propositions, which respectively correspond to the required bound at~$p_c$ and the stability of~$\Delta_p$ below the characteristic length. 
\end{proof}
\begin{proposition}[Lower bound at~$p_c$]\label{prop:lower_bound_Delta4c}
    For~$q=4$, there exists~$\delta_0>0$ such that for every~$R\ge 2r>0$,
    \begin{equation}\label{eq:lower_bound_Delta4c}
	    \Delta_{p_c}(r,R)\ges\pi_2(r,R)(r/R)\ges (r/R)^{2-\delta_0}.
    \end{equation}
\end{proposition}

\begin{proposition}[Stability of~$\Delta$]\label{cor:stability_Delta}
    For~$q=4$, every~$p$ and~$2\le r \le R\le L(p)$,
    \begin{equation}\label{eq:stability_Delta4}
	    \Delta_{p}(r, R) \asymp \Delta_{p_c}(r,R).
    \end{equation}
\end{proposition}

The section is divided as follows.
The proof makes use of the parafermionic observable introduced by Smirnov~\cite{Smi10}, we therefore start by recalling this notion. 
Then, we prove each proposition in a separate section, in the order in which they are stated.

\subsubsection{Parafermionic observable: a crash-course}

Since the parafermionic observable is used only in this section, we give minimal details and refer to the literature, for instance the review article~\cite{Dum17a} and the research paper~\cite{DumManTas20} (which is implementing a reasoning very similar to the present one, with similar notations).  We also recommend that the reader take a look at Figures~\ref{fig:domain}.

Parafermionic observables are usually defined in Dobrushin domains, 
but in the present study we limit ourselves to situations where the wired arc is reduced to a single point,
and therefore the domain has free boundary conditions. We do, however, authorise the domains to be non-simply-connected. 

Fix a finite connected subgraph~$\Omega$ of~$\mathbb Z^2$ and a vertex~$x\in\Omega$ with a neighbour~$x'$ in the infinite connected component of~$\mathbb Z^2\setminus \Omega$. 
Recall that~$\Omega^\diamond$ is spanned by the edges of~$(\bbZ^2)^\diamond$ bordering the faces of~$(\bbZ^2)^\diamond$ that contain vertices of~$\Omega$. As such the vertices of~$\Omega^\diamond$ have either degree~$2$ or~$4$ in~$\Omega^\diamond$ (here one should be careful when looking at a vertex of $\Omega^\diamond$ that is corresponding to an edge outside of $\Omega$ with both endpoints in $\Omega$: in this case we think of this vertex as being split into two ``prime ends'' of degree 2 in $\Omega^\diamond$, but this case does not occur for domains considered in this paper). 
Let~$e_x$ be the first edge of~$\Omega^\diamond$ bordering the face of~$(\bbZ^2)^\diamond$ containing $x$, 
when going around said face in clockwise order, starting from~$xx'$.

For a configuration~$\omega$ on~$\Omega$, let~$\overline\omega$ be the loop configuration on~$\Omega^\diamond$ associated to~$\omega$. 
In the loop configuration, the loop passing through~$e_x$ is called  \emph{exploration path};
it is denoted by~$\gamma=\gamma(\omega)$ and is oriented counterclockwise, so as to have primal open edges on its left and dual-open ones on its right.
For an edge~$e \in \gamma$, let~$\text{W}_{\gamma}(e,e_x)$ be the winding of~$\gamma$ between~$e$ and~$e_x$, 
that is the number of left turns minus the number of right turns taken by~$\gamma$ when going from~$e$ to~$e_x$ multiplied by~$\pi/2$. 

\begin{definition}\label{def:parafermionic observable}
    For~$(\Omega,x)$ as above, the {\em parafermionic observable}~$F=F_{\Omega,x}$  is defined for any (medial) edge~$e$ of~$\Omega^\diamond$ by
    \begin{equation*}
	      F(e)~:=~\phi^{0}_{\Omega,p_c,4}[\mathrm{W}_{\gamma}(e,e_x){\rm e}^{{\rm i}\mathrm{W}_{\gamma}(e,e_x)} \ind_{e\in \gamma}].
    \end{equation*}
 \end{definition}

The parafermionic observable satisfies a very special property first observed in~\cite{Smi10} (see also~\cite[Thm. 5.16]{Dum17a} for a statement with a similar notation). 
For every vertex of~$\Omega^\diamond$ with four incident edges in~$\Omega^\diamond$,
\begin{equation}\label{rel_vertex}
	\sum_{i=1}^4 \eta(e_i) F(e_i) = F(e_1) - {\rm i}F(e_2) - F(e_3) + {\rm i}F(e_4) = 0,
\end{equation}
where~$e_1$,~$e_2$,~$e_3$ and~$e_4$ are the four edges incident to~$v$, indexed in clockwise order,
and~$\eta(e_i)$ is the complex number of norm one with same direction as~$e_i$ and orientation from~$v$ towards the other endpoint of~$e_i$.
Summing this relation over all vertices of~$\Omega^\diamond$ of degree four, we obtain that 
\begin{equation}\label{eq:rel_vertex}
	\sum_{e\in \calC}\eta(e)F(e)=0,
\end{equation}
where~$\calC$ is the set of medial-edges of~$\Omega^\diamond$ having exactly one endpoint of degree two in~$\Omega^\diamond$, 
and~$\eta(e)$ is the complex number of modulus one, collinear with the edge~$e$ and oriented towards the outside of~$\Omega$. 

\begin{figure}
	\begin{center}
	\includegraphics[width=0.6\textwidth]{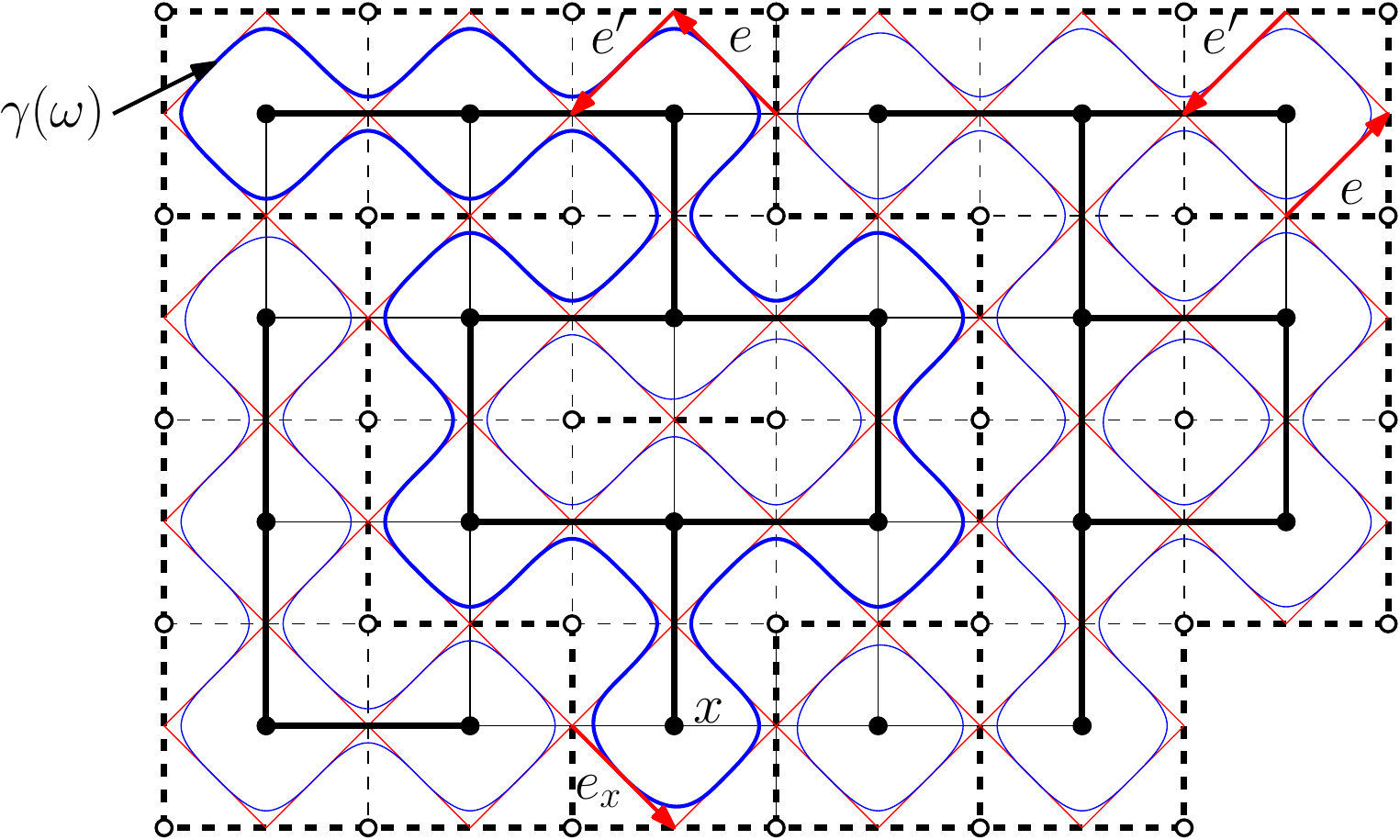}
	\caption{The primal and dual graphs in plain and dashed black lines. In bold (plain and dashed respectively), the configurations $\omega$ and $\omega^*$. The graph $\Omega^\diamond$ is in red and the configuration $\overline\omega$ in blue. The path $\gamma(\omega)$ is in bold blue. We also presented two examples of edges $e$ and $e'$, one for which the corresponding vertex of the primal graph has degree three, and one for which it has degree two.}
	\label{fig:domain}
	\end{center}
	\end{figure}
	
\begin{remark}
This relation should be understood as stating that
the contour integral of the parafermionic observable along the boundary of~$\Omega^\diamond$ is~$0$. 
While it is common to use the above in simply connected domains, we insist that~\eqref{eq:rel_vertex} is valid for any~$\Omega$ as above (it is important that $x$ is on the exterior face). 
\end{remark}

For random-cluster models with general values of~$q$ between 1 and 4, and the loop~$O(n)$ models, a similar property was used to obtain estimates for the two-point functions, see e.g.~\cite{DumSmi12,DumSidTas16,DumGla18,DumManTas20}. Here, we propose a new use of this property. 

For~$q = 4$, the complex phase of the observable (that is~${\rm e}^{{\rm i}\mathrm{W}_{\gamma}(e,e_x)}$) is invariant under addition of factors~$2\pi$ to the winding. Therefore, it is (almost) determined by the orientation of~$e$. 
Indeed, for any oriented edge~$e$ and any~$\omega$ such that~$\gamma$ passes through~$e$
\begin{align}\label{eq:etaW}
	\eta(e){\rm e}^{{\rm i}\mathrm{W}_{\gamma}(e,e_x)}
	=\begin{cases}
	+1 &\text{ if~$\gamma$ is oriented like~$e$},\\
	-1 &\text{ if~$\gamma$ is oriented opposite to~$e$}.	
	\end{cases}
\end{align}
Moreover, for~$y \in \partial \Omega$ 
(exclude from this explanation the situation where~$y$ has exactly two opposite neighbours in~$\Omega$), 
there exist two edges~$e,e' \in \calC$ bordering the face of~$(\bbZ^2)^\diamond$ containing~$y$
such that~$\gamma$ passes through~$e$ if and only if it passes through~$e'$, see Figure~\ref{fig:domain}. 
When this happens and~$y$ is not equal to~$x$,~$e$ and~$e'$ may be chosen such that~$\gamma$ always passes first through~$e$ towards the exterior of~$\Omega$, then through~$e'$, towards the interior.
Finally,~$\gamma$ performs~$4 - d_y$ positive~$\pi/2$ turns 
between the passage through~$e$ and~$e'$, where~$d_y$ is the degree of~$y$ inside~$\Omega$. 
Thus, 
\begin{align}\label{eq:wind_dif}
	\mathrm{W}_{\gamma}(e,e_x)  = \mathrm{W}_{\gamma}(e',e_x)  + (4-d_y)\pi/2.
\end{align}
When~$y=x$, the two contributions to \eqref{eq:rel_vertex} of~$e$ and~$e'$ are~$0$ and~$3\pi/2$, respectively. 
Inserting~\eqref{eq:etaW} and~\eqref{eq:wind_dif} into~\eqref{eq:rel_vertex} yields
\begin{equation}\label{eq:fundamental_holomorphicity}
	\sum_{\substack{y\in\partial\Omega\\y\ne x}}(4-d_y)\phi_{\Omega}^0[A(x,y)]=\tfrac{3\pi}2,
\end{equation}
where~$A(x,y)$ is the event that~$\gamma$ passes between~$y$ and its neighbour outside~$\Omega$.
When~$y$ has a neighbour in the infinite component of~$\Omega^c$, then~$A(x,y) = \{x \longleftrightarrow y\}$,
but since we do not ask~$\Omega$ to be simply connected, this is not always the case for other $y$. 

A more sophisticated analysis yields~\eqref{eq:fundamental_holomorphicity} 
also when some~$y \in \calC$ have exactly two opposite neighbours in~$\Omega$.

\subsubsection{Lower bound on~$\Delta_{p_c}(r,R)$ at~$p_c$: proof of Proposition~\ref{prop:lower_bound_Delta4c}}
 
    The second inequality follows from~\eqref{eq:LOWER_BOUND_ONE_ARM}, we therefore focus on the first one.
    By subtracting~\eqref{eq:fundamental_holomorphicity} for~$\Omega=\Lambda_R$ and~$\Omega={\rm Ann}(r,R)$ with~$x=(0,R)$,
    we find
    \begin{equation}\label{eq:obs_diff}
    	\sum_{y\in\partial\Lambda_R}(4-d_y)\big[\phi_{\Lambda_R}^0[A(x,y)]-\phi_{\Ann(r,R)}^0[A(x,y)]\big]
    	=\sum_{y\in \partial\Lambda_r}(4-d_y)\phi_{{\rm Ann}(r,R)}^0[A(x,y)].
    \end{equation}

    We will now estimate the terms on the left and right-hand sides of the above. 
  \begin{remark}  It is reasonable to expect that the left-hand side is of order~$R  \pi_1^+(R)^2\Delta_{p_c}(r,R)$, 
    while the right one of order~$r \pi_1^+(R)\pi_1^+(r)\pi_2(r,R)$, which would produce the desired result.  Nevertheless, we are not able to prove this due to boundary terms near corners of the inner box, which we cannot control. 
    Instead, we will use a more sophisticated strategy. 
Let us mention that the conjectural scaling limit of the model suggests that~$\pi_1^+(r,R)\asymp (r/R)$ and~$\pi_2(r,R)\asymp \sqrt{r/R}$, which would imply that~$\Delta_{p_c}(r,R)\asymp \sqrt{r/R}$.        \end{remark} 
    First, on the right-hand side, the event~$A(x,y)$ occurs if and only if~$x$ and~$y$ are connected in~$\omega$ and~$\Lambda_r$ and~$\Lambda_R$ are connected by a path in~$\omega^*$. Quasi-multiplicativity implies that
    \begin{align}\label{eq:obs_diff1}
        \sum_{y\in \partial\Lambda_r}(4-d_y)\phi_{{\rm Ann}(r,R)}^0[A(x,y)]
        \asymp \pi_1^+(R)\pi_2(r,R)\sum_{k\le r/2}\phi_\mathbb U^0[(k,0)\longleftrightarrow \partial\Lambda_r],
    \end{align}
    where~$\mathbb U:=\{x\in\mathbb Z^2:x_1\le0\text{ or }x_2\le0\}$ is the lower-left three-quarter plane. 
    
    On the left-hand side, for~$y\in\partial\Lambda_R$,~$A(x,y)$ corresponds to the event that~$x$ and~$y$ are connected.       
    Also, the measure~$\phi_{{\rm Ann}(r,R)}^0$ may be viewed as the measure in~$\Lambda_R$ conditioned on the event~$\{\La_r \equiv 0\}$ that every edge with one endpoint in~$\Lambda_{r-1}$ is closed.
    The previous study of~$\Delta_{p_c}(r,R)$ thus implies that for any~$r \le R/2$, 
        \begin{align*}
    	\phi_{\La_R}^0[A(x,y)]-\phi_{\La_R}^0[A(x,y)\,|\, \La_r\equiv 0] 
    	\asymp \Delta_{p_c}(r,R) \big[\phi_{\La_R}^0[A(x,y)]-\phi_{\La_R}^0[A(x,y)\,|\, \La_{R/2}\equiv 0]\big].
    \end{align*}
    We deduce that 
    \begin{equation}\label{eq:obs_diff2}
        \sum_{y\in\partial\Lambda_R}(4-d_y)\big[ \phi_{\Lambda_R}^0[A(x,y)]-\phi_{\Ann(r,R)}^0[A(x,y)]\big]
        \asymp   \Delta_{p_c}(r,R) \Sigma_R,
    \end{equation}
    where \begin{equation}\label{eq:Sigma}
    \Sigma_R := \sum_{y\in\partial\Lambda_R}(4-d_y)\big[ \phi_{\Lambda_R}^0[A(x,y)]-\phi_{\La_R}^0[A(x,y)\,|\, \La_{R/2}\equiv 0]\big]\end{equation} is a constant that depends on~$R$ only. 
    Inserting~\eqref{eq:obs_diff1} and~\eqref{eq:obs_diff2} into~\eqref{eq:obs_diff} gives
    \begin{align*}
    	\Delta_{p_c}(r,R) \Sigma_R \asymp \pi_1^+(R)\pi_2(r,R)\sum_{k\le r/2}\phi_\mathbb U^0[(k,0)\longleftrightarrow \partial\Lambda_r],
    \end{align*}
    for any~$r \le R/2$. Divide the relation above for an arbitrary~$r \le R/2$ with the same relation for~$r = R/2$ (we use that~$\Delta_{p_c}(r/2,R)\asymp \Delta_{p_c}(r,R)$) to find
    \begin{equation}\label{eq:obs_diff4}
    	\Delta_{p_c}(r,R)
    	\asymp \pi_2(r,R)\, \frac{\sum_{k\le r/2}\phi_\bbU^0[(k,0)\longleftrightarrow \partial\La_r]}{\sum_{k\le R/4}\phi_\bbU^0[(k,0)\longleftrightarrow \partial\La_R]}\ges \pi_2(r,R)(r/R).
    \end{equation}
    The inequality above is obtained by summing over~$j = 0,\dots, R/r$ and~$k = 1,\dots, r/2$ the inequality below :
        \begin{equation*}
    \phi_\mathbb U^0[(k+rj/2,0)\longleftrightarrow \partial\Lambda_R]\le\phi_\mathbb U^0[(k,0)\longleftrightarrow \partial\Lambda_r],
    \end{equation*}
   which is a direct consequence of the comparison between boundary conditions \eqref{eq:CBC}.

\begin{remark}\label{rmk:lower bound Delta q=4}
Applying~\eqref{eq:obs_diff4} with~$r=1$, the comparison between boundary conditions~\eqref{eq:CBC} and~\eqref{eq:LOWER_BOUND_ONE_ARM} imply that 
\[\pi_1(R)\Delta_{p_c}(R)\ges \frac{\pi_1(R)\pi_2(R)}{\sum_{k\le R/4}\phi_\bbU^0[(k,0)\longleftrightarrow \partial\La_R]}\ge \frac{\pi_1(R)\pi_2(R)}{\tfrac14R\pi_1(3R/4)}\ges \pi_2(R)/R\ges R^{c-2}.\] 
\end{remark}

\subsubsection{Stability of~$\Delta_p(r,R)$: proof of Proposition~\ref{cor:stability_Delta}}

The proof is based on the following quantity. Call a subgraph~$\Omega$ of~$\bbZ^2$ an {\em topological~$R$-annulus} if it is of the form~$\La_R \setminus H$ with~$H$ 
a simply connected domain such that~$\La_{R/8}\subset H \subset \La_{R/4}$; see Figure~\ref{fig:observable}. 
The boundary of~$\Omega$ is split into~$\partial \La_R$ and~$\partial_{\rm in}\Omega:=\partial\Omega\setminus\partial\Lambda_R$. 
Set
\begin{equation*}
	\bfN_{\Omega}(p):=\sum_{y\in \partial_{\rm in}\Omega}\phi_{\Omega,p}^0[y\longleftrightarrow \partial \La_{R/2}].
\end{equation*}
The first observation is the following lemma stating that for~$q=4$ the quantity~$\mathbf N_\Omega(p)$ does not depend on the choice of the~$R$-annulus~$\Omega$, or on the choice of~$p$ such that~$R\le L(p)$.
\begin{lemma}\label{lem:four_arm_N}
	For every~$p\ge p_c$ and~$R \le L(p)$, 
	\begin{align}\label{eq:four_arm_N}
    	\sup_\Omega\mathbf N_{\Omega}(p^*)
    	\asymp \sup_\Omega\mathbf N_\Omega(p_c)\asymp\inf_\Omega\mathbf N_{\Omega}(p_c)\asymp \inf_\Omega\mathbf N_{\Omega}(p),
	\end{align}
	where the infimum and supremum are taken over topological~$R$-annuli.
	\end{lemma}

\begin{figure}
	\begin{center}
	\includegraphics[width=0.6\textwidth]{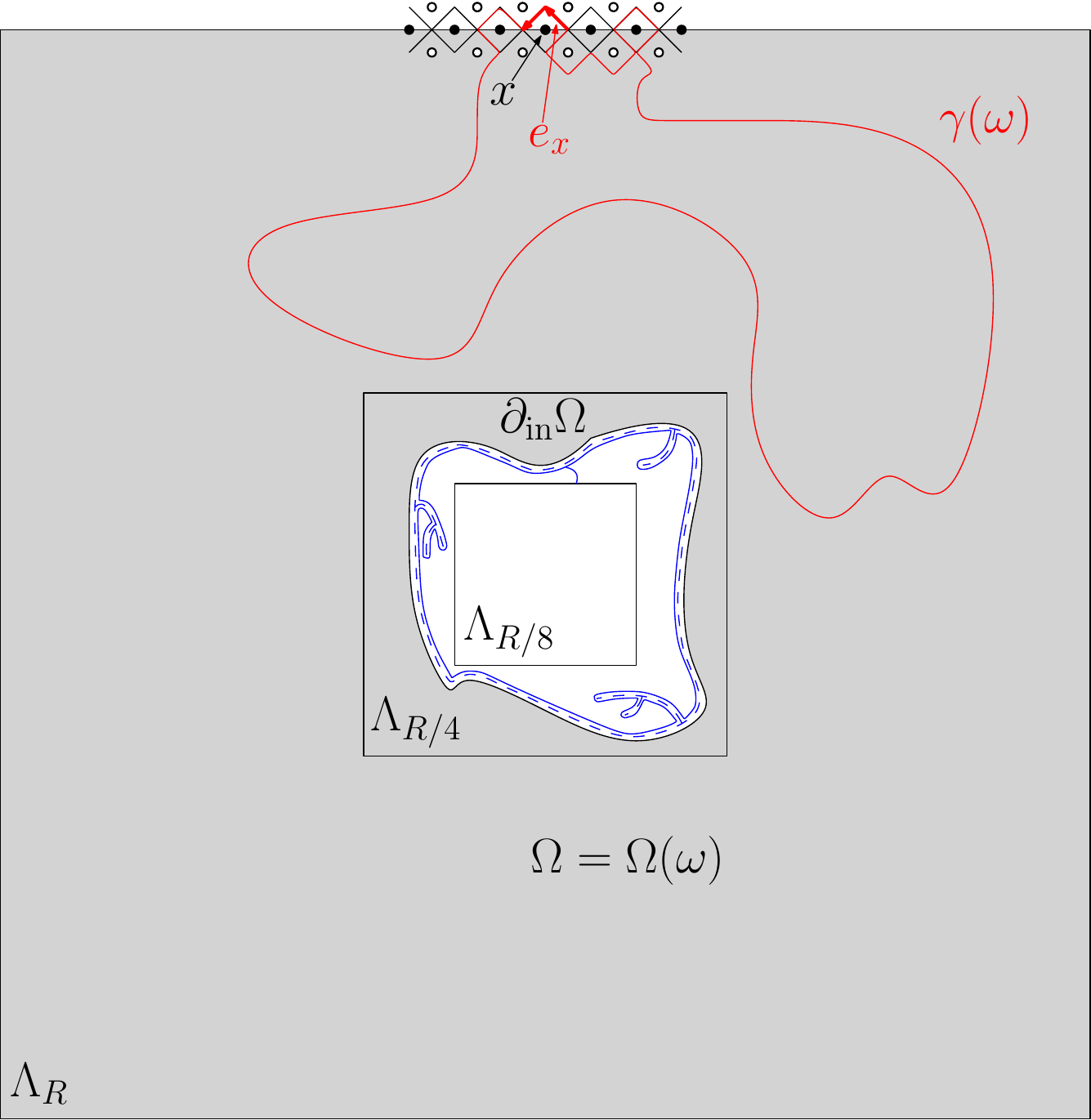}
	\caption{In grey the domain~$\mathcal H$ (which is an topological~$R$-annulus) carved by the event~$F$ (in blue). In red, the exploration path~$\gamma=\gamma(\omega)$ going around the boundary vertex~$x$.}
	\label{fig:observable}
	\end{center}
	\end{figure}
	
\begin{proof}
We start by proving that
	\begin{align}\label{eq:important1}
    	\sup_{\Omega} \mathbf N_{\Omega}(p_c) \asymp \Sigma_R/\pi_1^+(R)\asymp \inf_{\Omega} \mathbf N_{\Omega}(p_c),
   \end{align}
   where~$\Sigma_R$ was introduced in \eqref{eq:Sigma} in the previous section. 
   
    Fix~$\Omega$ an topological~$R$-annulus and set~$x := (0,R)$. 
	Then, due to~\eqref{thm:RSWnear}, for every~$y \in \partial_{\rm in}\Omega$,
	\begin{align*}
		\phi_{\Omega,p_c}^0[y\longleftrightarrow x \text{ and } \partial_{\rm in} \Omega \stackrel{*}{\longleftrightarrow}\partial \La_R] 
		\asymp \pi_1^+(R)\phi_{\Omega,p_c}^0[y\longleftrightarrow \partial \La_{R/2}].
	\end{align*}
	In particular, we have that 
	\begin{align}\label{eq:OmegaA}
		\mathbf N_{\Omega}(p_c)
		\asymp \frac{1}{ \pi_1^+(R)}\sum_{y \in \partial_{\rm in}\Omega}\phi_{\Omega,p_c}^0[A(x,y)].
	\end{align}
	Subtracting~\eqref{eq:fundamental_holomorphicity} for the domains~$\Lambda_R$ and~$\Omega$, 
	and following the proof of Proposition~\ref{prop:lower_bound_Delta4c}, we find that
    \begin{equation*}
    	\sum_{y \in \partial_{\rm in}\Omega}\phi_{\Omega,p_c}^0[A(x,y)]\asymp \Sigma_R,
    \end{equation*}
	which concludes the proof of~\eqref{eq:important1}.

 We now prove that for every~$p$ and~$R\le L(p)$,
 \begin{equation}\label{eq:important2}
 \inf_\Omega \mathbf N_\Omega(p)\les R^2\pi_4(p,R)\les \sup_\Omega \mathbf N_\Omega(p),
 \end{equation}
 where the infimum and supremum are taken over topological~$R$-annuli. Note that this inequality implies the result. Indeed, 
since~$\mathbf N_{\Omega}(p)$ is increasing in~$p$  and~$\pi_4(p^*,R)=\pi_4(p,R)$ (by duality), the previous inequality implies that for~$p\ge p_c$ and~$R\le L(p)=L(p^*)$,
	\begin{align*}
		\sup_\Omega\mathbf N_\Omega(p_c)\ge 
		\sup_\Omega\mathbf N_{\Omega}(p^*)\ges R^2\pi_4(p^*,R)=R^2\pi_4(p,R)
    	\ges  \inf_{\Omega} \mathbf N_{\Omega}(p)
	    \geq  \inf_{\Omega} \mathbf N_{\Omega}(p_c),
	\end{align*}
	which combines with~\eqref{eq:important1} to give the result.
	We now focus on the proof of~\eqref{eq:important2}.
	We proceed similarly to~\cite[Prop~6.8]{DumManTas20}, but from inside out.
   
    Let~$F$ be the event that in~$\Ann(R/8,R/4)$ there exists an open circuit surrounding~$\La_{R/8}$ which is connected to~$\La_{R/8}$, 
    as well as a dual-open circuit, which is necessarily outside of the open one; see Figure~\ref{fig:observable}. 
    If~$F$ occurs, let~$\calH = \calH(\omega)$ be the graph formed of the union of all open clusters that intersect~$\La_{R/8}$, 
    along with all finite components of~$\mathbb Z^2$ minus the said union. Then, due to the definition of~$F$,  
   ~$\La_{R/8}\subset \calH \subset \La_{R/4}$.
    Write~$\Omega(\omega) = \La_R\setminus \calH$ for the random topological~$R$-annulus formed by removing~$\calH$. 
    
When the measure~$\phi^0_{\La_R,p}$ is conditioned on~$F\cap\{\Omega(\omega)=\Omega\}$, 
    its restriction to~$\Omega$ is~$\phi^0_{\Omega,p}$. 
    Notice that if~$y \in \partial_{\rm in} \Omega$ is connected to~$\partial \La_{R/2}$ by an open path, then a four arm event to distance~$R/8$ occurs around~$y$. Thus, using the quasi-multiplicativity~\eqref{eq:mix} and Theorem~\ref{thm:RSWnear} (to bound the probability of~$F$ from below), we find that
    \begin{align*}
        R^2\pi_4(p,R)
        \ges  \phi_{\Lambda_R,p}[\mathbf N_{\Omega(\omega)}(p) \, \ind_{F}]\ges \inf_{\Omega}\bfN_{\Omega}(p).
    \end{align*}
     
    Conversely, the separation of arms for the four-arm event and Theorem~\ref{thm:RSWnear} show that for each~$y \in \Ann(5R/32, 7R/32)$,
    \begin{align*}
    	\phi_{\Lambda_R,p}[F \cap \{y \in \partial_{\rm in}\Omega(\omega)\} \cap \{y \longleftrightarrow \partial \La_{R/2}\}] \ges \pi_4(p,R).
    \end{align*}
    Summing over all~$y$ we find 
    \begin{align*}
        R^2\pi_4(p,R) \les \phi_{\Lambda_R,p}[\mathbf N_{\Omega(\omega)}(p)\ind_{F}]\le \sup_{\Omega }\bfN_{\Omega}(p).
    \end{align*}
\end{proof}

\begin{remark}
The previous proof shows as a byproduct that~$\pi_4(p,R)\asymp\pi_4(R)$ for every~$R\le L(p)$. 
It is somehow surprising that the stability of~$\pi_4$ below the correlation length can be directly extracted from the parafermionic observable.
Recall, however, that this is only valid for~$q = 4$. 
\end{remark}

\begin{remark}\label{rmk:ahah}
The previous proof also implies that~$R^2\pi_4(R)$ is of the order of~$\mathbf N_\Omega(p_c)$ for every topological~$R$-annulus~$\Omega$. In particular, taking~$\Omega=\Lambda_{R/8}$ gives  that
\begin{equation}
R^2 \pi_4(R)\asymp R\pi_1^+(R).
\end{equation}
The conjectural scaling limit of the model suggests that~$\pi_1^+(R)\asymp R^{-1}$, which would imply that~$\pi_4(R)\asymp R^{-2}$. We see that in this case the fact that~$\Delta_{p_c}(R)\gg \pi_4(R)$ is crucial for the bound~\eqref{eq:lower_bound_Delta4c}.
\end{remark}

\begin{figure}
	\begin{center}
	\includegraphics[width=1.0\textwidth]{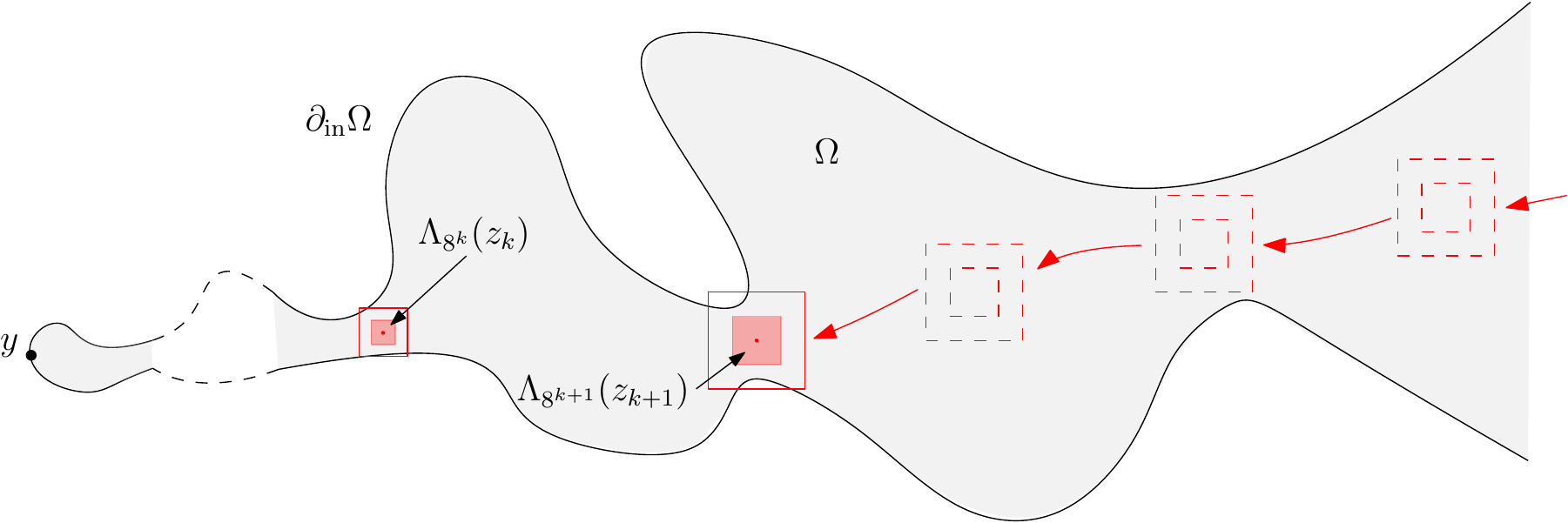}
	\caption{By sliding the boxes along a given path one by one from the exterior towards $y$, one finds the first time that the twice larger box $\Lambda_{2\cdot 8^k}(z_k)$ disconnects $y$ from the boundary. Note that by construction the first time it stops cannot be the quite the same for different $k$. In particular, the box $\Lambda_{8^k/2}(z_k)$ being at a distance $8^k/2$ of the boundary of $\Lambda_{8^k}(z_k)$, it is also at such a distance of $\partial_{\rm in}\Omega$. Since on contrary $\Lambda_{8^\ell/2}(z_\ell)$ is within a distance $8^\ell\times 3/2$ of it, we immediately deduce that the two boxes do not intersect as soon as $\ell<k$.}
	\label{fig:trick}
	\end{center}
	\end{figure}
\begin{proof}[Proposition~\ref{cor:stability_Delta}]
We treat the case~$p>p_c$; the case~$p<p_c$ can  be done similarly. Fix~$R\le L(p)$. 
The inequality~\eqref{eq:ub_deriv_Delta_int} implies that
	\begin{align*}
        \Big|\log \frac{\Delta_p(r,R)}{\Delta_{p_c}(r,R)}\Big|
        &\les \underbrace{\int_{p_c}^p \Big(\sum_{\ell=1}^R \ell\Delta_u(\ell)\Big)du}_{(A)}
        	+ \underbrace{\int_{p_c}^p \Big(\sum_{\ell=R}^{L(u)}\ell\Delta_u(\ell)\Delta_u(R,\ell)\Big) du}_{(B)}
\end{align*}
and we need to prove that~$(A)$ and~$(B)$ are bounded by constants that are independent of~$R$ and~$p>p_c$. 
The bound on~$(B)$ is easy to obtain since~\eqref{eq:deriv_cross} shows that~$(B) \les \phi_p[\calC(\Lambda_R)]  -  \phi_{p_c}[\calC(\Lambda_R)]\le 1$. We therefore focus on~$(A)$.

Choose the topological~$R$-annulus~$\Omega$ minimizing~$\mathbf N_\Omega(p)$ and observe that by Lemma~\ref{lem:four_arm_N}, $\mathbf N_\Omega(p)\asymp \mathbf N_\Omega(p_c)$. 
We claim that 
\begin{align}\label{eq:Nderiv}
	\tfrac{d}{d u} \log \bfN_\Omega (u)  \ges\sum_{\ell=1}^R \ell\Delta_u(\ell).
\end{align}
Observe that \eqref{eq:Nderiv} implies that
$(A)\les \log \mathbf N_\Omega(p)-\log \mathbf N_\Omega(p_c)\les1$ 
which concludes the proof of the proposition. 
Thus, we only need to prove \eqref{eq:Nderiv}, which we do next. 

Start by observing that
	\begin{align}\label{eq:Nderiv2}
		\tfrac{d}{d u} \log \bfN_\Omega(u) 
		= \sum_{y\in \partial_{\rm in}\Omega}\tfrac{d}{d u} \phi_{\Omega,u}^0[y\longleftrightarrow \partial \La_{R/2}]
		= \tfrac{1}{u(1-u)}\sum_{y\in \partial_{\rm in}\Omega}\,\sum_{e \in \Omega}\Cov[y\longleftrightarrow \partial \La_{R/2},\omega_e ].
	\end{align}
	Fix some~$y \in \partial_{\rm in}\Omega$ and set~$N:=\lfloor \log_8 R\rfloor$. 
	Then, there exist points~$z_1,\dots, z_{N-2}$ in~$\Omega$ 
	such that for each~$k$,~$\La_{8^k}(z_k) \subset \Omega$ but~$y$ is not connected to~$\partial \La_R$ in
	the subgraph~$\Omega \setminus \La_{2 \cdot 8^k}(z_k)$; the latter condition includes the situation where~$y \in  \La_{2 \cdot 8^k}(z_k)$.
	See Figure~\ref{fig:trick} 
	and its caption for an explanation of this elementary fact. 
	For~$k < \ell$, since~$\La_{2 \cdot 8^k}(z_k)$ intersects the boundary of~$\Omega$, but~$\La_{8^{\ell}}(z_{\ell}) \subset \Omega$, 
	we conclude that~$\La_{8^k/2}(z_k)$ and~$\La_{8^{\ell}/2}(z_{\ell})$ do not intersect.
	Thus, the boxes~$(\La_{8^k/2}(z_k))_{k \le N-2}$ are pairwise disjoint. 
	
	Now, for any~$e \in \La_{8^k/2}(z_k)$ for one such~$k$, a direct application of Theorem~\ref{thm:Delta_coupling} implies that
	\begin{align*}
		\Cov[y\longleftrightarrow \partial \La_{R/2},\omega_e ] \asymp \Delta_u(8^k)\phi_{\Omega,u}^0[y\longleftrightarrow \partial \La_{R/2}].
	\end{align*}
	Summing over~$e$ and keeping in mind that all covariances in~\eqref{eq:Nderiv2} are non-negative gives
	\begin{align*}
		\tfrac{d}{d u} \phi_{\Omega,u}^0[y\longleftrightarrow \partial \La_{R/2}]
		&\ges \sum_{k = 1}^{N -2}\sum_{e \in \La_{8^k/2}(z_k)} \Cov[y\longleftrightarrow \partial \La_{R/2},\omega_e ]\\
		&\ges \sum_{k = 1}^{N -2}8^{2k} \Delta_u(8^k)\phi_{\Omega,u}^0[y\longleftrightarrow \partial \La_{R/2}]\\
		&\ges \Big(\sum_{\ell = 1}^{R}\ell \Delta_u(\ell)\Big)\phi_{\Omega,u}^0[y\longleftrightarrow \partial \La_{R/2}].
	\end{align*}
	For the last inequality, we used the fact that~$\Delta_u(\ell) \asymp \Delta_u(8^k)$ for any~$8^k\le  \ell \le 8^{k+1}$.
	Finally, summing over~$y$, we find
	\begin{align*}
		\tfrac{d}{d u} \bfN_\Omega(u) 
		\ges 
		 \bfN_\Omega(u)\sum_{\ell = 1}^{R}\ell \Delta_u(\ell),
	\end{align*}
	which concludes the proof of~\eqref{eq:Nderiv} and of the whole proposition.
\end{proof}

\section{Proofs of the stability theorems}\label{sec:7}

In this section, we prove stability results for crossing probabilities, arm events and~$\Delta_p$. 
First, observe that due to the previous section, Propositions~\ref{prop:deriv_cross} and~\ref{prop:deriv_Delta_ub} immediately lead to the following corollary.
 
\begin{corollary}\label{cor:deriv}
    (i) Fix~$\eta>0$. For every~$p \in (0,1)$ and every~$\eta$-regular quad~$(\calD,a,b,c,d)$ at scale~$R\le L(p)$,
    \begin{equation}\label{eq:deriv_cross_}
    	\tfrac{{\rm d}}{{\rm d}p}\phi_{p}[\calC(\calD)]\asymp 
		 R^2\Delta_p(R)+\sum_{\ell=R}^{L(p)}\ell \Delta_p(\ell)\Delta_p(R,\ell),
    \end{equation}
    where the constants in~$\asymp$ depend on~$\eta$.
    
    (ii) For every~$p \in (0,1)$,~$\sigma \in \{0,1\}^k$ and~$r \le R \le L(p)$, 
    \begin{equation}\label{eq:deriv_arms_}
    	\tfrac{{\rm d}}{{\rm d}p}\log\phi_{p}[A_\sigma(r,R)]\asymp 
		  R^2\Delta_p(R)+\sum_{\ell=R}^{L(p)}\ell \Delta_p(\ell)\Delta_p(R,\ell),
    \end{equation}
    where the constants in~$\asymp$ depend on~$\sigma$.
    
    (iii) For every~$p \in (0,1)$ and two edges~$e$ and~$f$ at a distance~$R\le L(p)$ of each other, 
    \begin{equation}\label{eq:deriv_Delta_ub1}
        \big|\tfrac{{\rm d}}{{\rm d}p}\log \Cov_p(\omega_e,\omega_f)\big|
        \les R^2\Delta_p(R) + \sum_{\ell=R}^{L(p)} \ell \Delta_p(\ell)\Delta_p(R,\ell).
    \end{equation}
\end{corollary}

\begin{proof}
	By \eqref{eq:delta_thm4} and \eqref{eq:delta_thm3}, for any~$p \in (0,1)$ and~$R \leq L(p)$, 
	$\sum_{\ell=1}^R \ell\Delta_p(\ell) \les R^2\Delta_p(R)$. Insert this into 
	Propositions~\ref{prop:deriv_cross} and~\ref{prop:deriv_Delta_ub} to obtain the desired results. 
\end{proof}

We next prove the stability of crossing probabilities and arm events probabilities stated in Theorem~\ref{thm:stability}. 

\begin{proof}[Theorem~\ref{thm:stability}]
	Fix~$p \neq p_c$. 
	We prove~\eqref{eq:stability crossings}; the stability for arm events can be deduced similarly.
	Fix some~$\eta$-regular discrete quad~$(\calD,a,b,c,d)$ at some scale~$R\ge1$.
	The inequality~\eqref{eq:stability crossings} is trivial when~$R > L(p)$, and we focus on the case where~$R \le L(p)$.
	Applying~\eqref{eq:deriv_cross}  to~$u$ between~$p_c$ and~$p$, we find, 
    \begin{align}
	   	\tfrac{{\rm d}}{{\rm d}u}\phi_{u}[\calC(\calD)]
		 & \asymp R^2\Delta_u(R)+\sum_{\ell=R}^{L(p)}\ell \Delta_u(\ell) \Delta_u(R, \ell) + \sum_{\ell=L(p)}^{L(u)}\ell \Delta_u(\ell)\Delta_u(R,\ell)\nonumber\\
		 &\les (\tfrac R{L(p)})^\delta\, \Big[L(p)^2\Delta_u(L(p)) + \sum_{\ell=L(p)}^{L(u)}\ell \Delta_u(\ell)\Delta_u(L(p),\ell)\Big],\label{eq:ffbbb}
    \end{align}
 where the inequality is due to a simple computation based on the quasi-multiplicativity of $\Delta_p$ and the two inequalities from \eqref{eq:delta_thm4} (the constant $\delta$ is actually given by this displayed equation). 
    The terms for~$\ell \geq L(p)$ are dominated by the corresponding sum in the second line. 
    
    Corollary~\ref{cor:deriv} applied to~$\calC(\La_{L(p)})$ together with  \eqref{eq:ffbbb} imply that
     \begin{align*}
	   	\tfrac{{\rm d}}{{\rm d}u}\phi_{u}[\calC(\calD)]
		\les (\tfrac R{L(p)})^\delta\, \tfrac{{\rm d}}{{\rm d}u}\phi_{u}[\calC(\La_{L(p)})].
    \end{align*}
    Integrate the above between~$p_c$ and~$p$ to find
    \begin{align*}
    	|\phi_{p} [\calC(\calD)] - \phi_{p_c} [\calC(\calD)]|
		\les(\tfrac R{L(p)})^\delta	|\phi_{p} [\calC(\La_{L(p)})] - \phi_{p_c} [\calC(\La_{L(p)})]|\le (\tfrac R{L(p)})^\delta. 
    \end{align*}
\end{proof}

\begin{remark}\label{rem:improved_stability}
	One may also obtain the following improvement of~\eqref{eq:stability arm exponent} (similar to~\eqref{eq:stability crossings}):
	\begin{align}\label{eq:improved_stability}
		\exp[-C(R/L(p))^\delta]\le \frac{\pi_1(p,R)}{\pi_1(p_c,R)} \le \exp[C(R/L(p))^\delta]\qquad \text{ for all~$R \le L(p)$}, 
	\end{align}
	where~$\delta,C > 0$ are universal constants depending only on~$q$. 
	
		Moreover, since Corollary~\ref{cor:deriv} applies to logarithmic derivatives of any arm event probabilities, 
	\eqref{eq:improved_stability} also applies to~$\pi_\sigma(p,R)$ and~$\pi_\sigma^+(p,R)$, with~$C$ depending on the color sequence~$\sigma$. 
	Notice, however, that we do not claim to prove stability for probabilities of arm events in the half-plane with boundary conditions strictly on the half-plane. 
\end{remark}

We now turn to the proof of the stability of~$\Delta_p$. 
Let us note that for~$q = 4$, the stability of~$\Delta_p$ was proved also in Section~\ref{sec:q=4}, as a step towards~\eqref{eq:delta_thm4}. 

\begin{proof}[Theorem~\ref{thm:delta}(iii).]
     Fix~$p \neq p_c$,~$R \le L(p)$ and two edges~$e$ and~$f$ at a distance~$R$ of each other.
    For~$u$ between~$p$ and~$p_c$, Corollary~\ref{cor:deriv} gives
     \begin{align*}
	    \Big|\tfrac{{\rm d}}{{\rm d}u}\log \Cov_u(\omega_e,\omega_f)\big|
	    &\les R^2 \Delta_u(R)+\sum_{\ell=R}^{L(u)}\ell \Delta_u(\ell) \Delta_u(R,\ell)
    	\les \tfrac{{\rm d}}{{\rm d}u}\phi_u[\calC(\Lambda_R)].
    \end{align*}
    By integrating the above between~$p_c$ and~$p$, we find
    \begin{equation} 
    	\frac{\Cov_p(\omega_e,\omega_f)}{\Cov_{p_c}(\omega_e,\omega_f)}	\asymp 1.
	\end{equation}
	Apply now~\eqref{eq:Cov_Delta2} to deduce that 
   \begin{equation*}
    	\Delta_p(R) \asymp \sqrt{\Cov_p(\omega_e,\omega_f)}
		\asymp \sqrt{\Cov_{p_c}(\omega_e,\omega_f)} \asymp \Delta_{p_c}(R).
    \end{equation*}
 \end{proof}

\begin{remark}\label{rem:improved_stability}
	It will come to no surprise to the reader that the same type of improved stability as in \eqref{eq:improved_stability} may be shown for the covariance. Getting the same result for $\Delta_p$ itself seems more difficult as we crucially rely on an up-to-constant relation between $\Delta_p$ and the covariance, and that the derivative of $\Delta_p$ itself is less obvious.
\end{remark}

\section{Derivation of the scaling relations}\label{sec:all_scaling_relations}

This section is dedicated to proving the scaling relations (Theorems~\ref{thm:scaling_rel_crit},~\ref{thm:scaling_rel_h} and~\ref{thm:scaling_rel_nc}). 
The proof of the near-critical scaling relations (Theorem~\ref{thm:scaling_rel_nc}) is based on the stability below the characteristic length (Theorem~\ref{thm:stability} and Theorem~\ref{thm:delta}(iii)). With the latter results at our disposal, the proofs of the critical and near-critical scaling relations \eqref{eq:relation2}--\eqref{eq:L_Delta}  are very close to those for Bernoulli percolation and contain no significant innovation. For this reason, we are voluntarily quick on these proofs, trying to merely recall the crucial ingredients.

The main novelties in this section are the independent proof of the scaling relation involving the magnetic field (Section~\ref{sec:scaling_rel_h}), and the derivation of the scaling relation involving~$\alpha$.

\subsection{Scaling relations at criticality: proof of Theorem~\ref{thm:scaling_rel_crit}}\label{sec:scaling_rel_crit}\label{sec:5}

In this section we work with~$p = p_c$ and we drop it from the notation. 
We will prove a stronger result, which implies~\eqref{eq:relation2} when taking~$x\in\partial\Lambda_R$. 

\begin{lemma}\label{lem:0}
    Fix~$1\le q\le 4$. For every~$R\ge1$ and every~$x\in \Lambda_R$,
    \begin{equation}\label{eq:jj}
    	\pi_1(R)^{2}\les \phi_{\bbZ^2}[0\xlra{\Lambda_{2R}} x, 0\longleftrightarrow \partial\Lambda_R]
		\le \phi_{\bbZ^2}[0 \longleftrightarrow x]\les \pi_1(|x|)^{2}.
    \end{equation}
\end{lemma}

\begin{proof}
	The middle inequality is obvious. For the right one, observe that if~$0$ and~$x$ are connected, then 
 ~$0 \longleftrightarrow \partial \Lambda_r$ and~$x\longleftrightarrow \partial\Lambda_r(x)$, where~$r:=\lfloor |x|/3\rfloor$. 
The invariance under translations of~$\phi_{\bbZ^2}$, the mixing property~\eqref{eq:mix} and~\eqref{eq:MULTIPLICATIVITY} give that
  ~$$
    	\phi_{\bbZ^2}[0\longleftrightarrow x] 
		\les\pi_1(r)^2\les\pi_1(|x|)^2.
$$
    We now prove the left inequality.   The FKG inequality~\eqref{eq:FKG} implies that   \begin{align*}\phi_{\bbZ^2}[0\stackrel{\Lambda_{2R}}{\longleftrightarrow} x, 0\longleftrightarrow \partial\Lambda_R]&\ge \phi_{\bbZ^2}[A_R,0\longleftrightarrow \partial\Lambda_{2R}, x\longleftrightarrow \partial\Lambda_{3R}(x)]\\
   &		\geq  \phi_{\bbZ^2}[A_R]\pi_1(2R)\pi_1(3R)\\
	&	\ges \pi_1(R)^2,	\end{align*}
    where we used~\eqref{eq:RSWA} and~\eqref{eq:MULTIPLICATIVITY}.
\end{proof}

Recall that~$\sfC$ is the cluster of~$0$. Let~${\rm rad}(\sfC) := \max\{r:\sfC \text{ intersects }\partial \La_r \}$ be the radius of~$\sfC$. 

\begin{lemma}\label{prop:moments}
	For every~$R \geq1$, 
	\begin{align}\label{eq:moments}
    	R^2\pi_1(R)&\les\phi_{\bbZ^2}[|\sfC| \,| R\le{\rm rad}(\sfC)<4R] \le \sqrt{\phi_{\bbZ^2}[|\sfC|^2 \,| R\le{\rm rad}(\sfC)<4R]}\les R^2\pi_1(R).
	\end{align}
\end{lemma}

\begin{proof}
	Fix~$n \geq 1$.  The inequality in the middle is the Cauchy-Schwarz inequality. 
	For the first one,
	observe that~\eqref{eq:jj} and~\eqref{eq:RSWA} imply that
	\begin{align*}
		\phi_{\bbZ^2} [ |\sfC| \ind_{R\le{\rm rad}(\sfC)<4R} ] 
		\ge \sum_{x\in\Lambda_R}
			\phi_{\bbZ^2}[0\xlra{\Lambda_{2R}} x,\Lambda_{2R}\not\longleftrightarrow\partial\Lambda_{4R}]
		\ges |\Lambda_R|\pi_1(R)^2.
	\end{align*}
	Divide the above by~$\phi_{\bbZ^2}[R\le{\rm rad}(\sfC)<4R]\le \pi_1(R)$ to obtain the first inequality in \eqref{eq:moments}.  

    We turn to the last inequality of~\eqref{eq:moments}. We have 
	\begin{align*}
		\phi_{\bbZ^2}[|\sfC|^2 \ind_{R\le{\rm rad}(\sfC)<4R}]
    	&\le \sum_{x,y\in \Lambda_{4R}} \phi_{\bbZ^2}[0\longleftrightarrow x,0\longleftrightarrow y,0\longleftrightarrow \partial\Lambda_R].
	\end{align*}
	Fix~$x,y \in \Lambda_{4R}$ 
	and assume first that~$|x|\le |y| \le |x-y|$. 
	Set~$\ell := |x|$ and~$k := |y|$. 
	Observe that the event on the right induces the following events which are listed along with the order -- up to constants -- of their probabilities (which are obtained thanks to~\eqref{eq:MULTIPLICATIVITY}): 
	\begin{itemize}[noitemsep]		
		\setlength\itemsep{0cm}
		\setlength\topsep{0cm}
		\item~$0 \llra \partial \Lambda_{\ell/4 }$ -- of probability of order~$\pi_1(\ell)$;
		\item~$x \llra \partial \Lambda_{\ell/4}(x)$ -- of probability of order~$\pi_1(\ell)$;
		\item~$y \llra \partial \Lambda_{k/4}(y)$ -- of probability of order~$\pi_1(k)$;
		\item~$\partial \Lambda_{\min\{2\ell,k/4\}} \llra  \partial \Lambda_k$ -- of probability of order~$\pi_1(k)/\pi_1(\ell)$;
		\item~$\partial \Lambda_{\min\{2k,R\}} \llra  \partial \Lambda_{R}$ -- of probability of order~$\pi_1(R) /\pi_1(k)$.
	\end{itemize}
	Several iterations of the mixing property~\eqref{eq:mix} imply that
	\begin{align}\label{eq:pp}
		\phi_{\bbZ^2}[0\longleftrightarrow x,0\longleftrightarrow y,0\longleftrightarrow \partial\Lambda_R]
		&\les \pi_1(\ell)\pi_1(k) \pi_1(R).
	\end{align}
	Observe now that for~$1 \le \ell \le k\le 4R$, there are~$8\ell$ vertices~$x \in \bbZ^2$ with~$|x| = \ell$ 
	and~$8k$  vertices~$y$ with~$|y| = k$.
	Thus, 
	\begin{align*}
    	\sum_{\substack{x,y\in \Lambda_{4R}\\ |x| \le |y| \le |x-y|}}\phi_{\bbZ^2}[0\longleftrightarrow x,0\longleftrightarrow y,0\longleftrightarrow \partial\Lambda_R]
    	&\les  \sum_{1 \le \ell \le k \le 4R} \ell\pi_1(\ell)k\pi_1(k)\pi_1(R)	\les R^4\pi_1(R)^3,
	\end{align*}
	where in the last line we used that 
	\begin{equation}
	\sum_{k=1}^Rk\pi_1(k)\les R^2\pi_1(R)
	\end{equation}
	which is a consequence of~\eqref{eq:LOWER_BOUND_ONE_ARM} and \eqref{eq:MULTIPLICATIVITY}. The same upper bound may be obtained for any of the other five possible orderings of~$|x|, |y|, |x-y|$. 
	Overall, we conclude that 
	\begin{align*}
    	\phi_{\bbZ^2} [|\sfC|^2 \ind_{R\le{\rm rad}(\sfC)<4R}]\les R^4\pi_1(R)^3,
	\end{align*}
	which gives the result by dividing by	
	\begin{align}
	\phi_{\bbZ^2}[R\le{\rm rad}(\sfC)<4R]
	&\geq \phi_{\bbZ^2}[0 \llra \partial \Lambda_R,\Lambda_R\not\longleftrightarrow\partial\Lambda_{2R}]
	\ges\pi_1(R),\label{eq:l}
	\end{align}
where in the last inequality we used the mixing property~\eqref{eq:mix} and~\eqref{eq:RSWA}.
\end{proof}

\begin{proof}[Theorem~\ref{thm:scaling_rel_crit}] 
    Lemma~\ref{lem:0} applied with~$R=2|x|$ and~\eqref{eq:MULTIPLICATIVITY} directly imply~\eqref{eq:relation2}. 
	We turn to the proof of~\eqref{eq:relation3}.
	Fix~$R \geq 1$ and~$r:=\varphi(R)$.
    Let us start with the lower bound on~$\phi_{\bbZ^2}[|\sfC| \geq R]$. 
   	Let~$c \in (0,1)$ be the constant appearing in the first bound~$\les$ of \eqref{eq:moments}.
	Then, using the definition of~$\varphi$,~\eqref{eq:moments}, 	the Paley–Zygmund inequality and~\eqref{eq:l}, we find that 
	\begin{align*}
		\phi_{\bbZ^2}[|\sfC|\geq \tfrac{c}2R]&\ge \phi_{\bbZ^2}[|\sfC|\geq  \tfrac{c}2 r^2\pi_1(r)]\\
		&\ge \phi_{\bbZ^2}[|\sfC|\geq \tfrac{1}2\phi_{\bbZ^2}(|\sfC| \,|\,r\le {\rm rad}(\sfC)< 4r)]\\
		&\ges \frac{\phi_{\bbZ^2}[|\sfC| \,|\,r\le {\rm rad}(\sfC)<4r]^2}{\phi_{\bbZ^2}[|\sfC|^2 \,|\, r\le {\rm rad}(\sfC)<4r]}\,\phi_{\bbZ^2}[r\le {\rm rad}(\sfC)<4r]\\
		&\ges\phi_{\bbZ^2}[r\le {\rm rad}(\sfC)<4r]\\
		& \ges\pi_1(r).
	\end{align*}
	This concludes the proof of the lower bound since \eqref{eq:MULTIPLICATIVITY} implies that~$\phi( \tfrac{c}2R) \asymp  \phi(R) = r$. 
	
	We turn to the complementary upper bound. Using the Markov inequality in the third line and the definition of~$\varphi$ and~\eqref{eq:jj} in the fourth, we obtain that 
    \begin{align*}
    	\phi_{\bbZ^2}[|\sfC|\ge R]
    	&= \phi_{\bbZ^2}[|\sfC|\ge R,{\rm rad}(\sfC)> r] + \phi_{\bbZ^2}[|\sfC| \geq R,{\rm rad}(\sfC)\le r]\\
    	&\le \pi_1(r)+\phi_{\bbZ^2}[|\sfC \cap \Lambda_r| \geq R] \\
    	&\le \pi_1(r)+\frac{1}{R}\sum_{u \in \Lambda_r}\phi_{\bbZ^2}[0 \longleftrightarrow u] \\ 
    	&\les \pi_1(r)+\frac{1}{r^2\pi_1(r)}\sum_{u \in \Lambda_r} \pi_1(|u|)^2\les \pi_1(r),
	\end{align*}
	where the last inequality follows from \eqref{eq:LOWER_BOUND_ONE_ARM} via the following computation
$$\frac1k\sum_{k=1}^r \frac{k\pi_1(k)^2}{r\pi_1(r)^2} \les \frac{1}{r}\sum_{k=1}^r(k/r)^{1-2c}\les 1.$$
\end{proof}

\subsection{Scaling relations in the near-critical regime: proof of Theorem~\ref{thm:scaling_rel_nc}}\label{sec:main}\label{sec:13}

By Theorem~\ref{thm:L_equiv_xi} and~\eqref{eq:MULTIPLICATIVITY}, we have that~$L(p)\asymp \xi(p)$ and~$\pi_1(L(p))\asymp\pi_1(\xi(p))$, so we only need to show~\eqref{eq:main_theta}--\eqref{eq:lp2} with~$L(p)$ instead of~$\xi(p)$.

\subsubsection{Proof of~\eqref{eq:main_theta} (scaling relation between~$\beta$,~$\nu$, and~$\xi_1$)}
On the one hand, the stability below the characteristic length (Theorem~\ref{thm:stability}) gives
\begin{align*}
\theta(p)\le \pi_1(p,L(p))\les 
\pi_1(L(p)).
\end{align*}
On the other hand, the FKG inequality~\eqref{eq:FKG},~\eqref{eq:RSWA}, and Corollary~\ref{rmk:1} imply that 
\begin{align*}
	\theta(p) 
	\geq\phi_p[0\longleftrightarrow\partial\Lambda_{2L(p)},A_{L(p)},\Lambda_{L(p)}\longleftrightarrow\infty]\ges  \pi_1(p,2L(p))\ge \pi_1(2L(p))\ges \pi_1(L(p)).
\end{align*}

\subsubsection{Proof of~\eqref{eq:main_chi} (scaling relation between~$\gamma$,~$\nu$ and~$\xi_1$)}

We start with the case~$p<p_c$. First, Proposition~\ref{prop:estimate subcritical} can be improved to 
\begin{equation}\label{eq:fund}
	c\pi_1(R)^2\exp[-C|x|/L(p)]\le \phi_{p}[0\longleftrightarrow x]\le C\pi_1(R)^2\exp[-c|x|/L(p)],
\end{equation}
where~$R:=\min\{|x|,L(p)\}$ and~$c,C > 0$ are uniform constants. 
Indeed, one simply needs to combine the proofs of Proposition~\ref{prop:estimate subcritical} and Lemma~\ref{lem:0} in a standard fashion. 

Now,~\eqref{eq:main_chi} follows by summing~\eqref{eq:fund} over every~$x\in\bbZ^2$ and using that, by~\eqref{eq:LOWER_BOUND_ONE_ARM},
\begin{equation}
	\pi_1(R)\le C(L(p)/R)^{1/2-c}\pi_1(L(p)).
\end{equation}
We now turn to the case~$p>p_c$. 

The only additional difficulty comes from the fact that we need to force 0 and~$x$ not to be connected to infinity.
For the lower bound, sum over every~$x\in\La_{L(p)/2}$ the following inequality 
$$
	\phi_p[0\longleftrightarrow x,0\not\longleftrightarrow \infty]
	\ge \phi_p[\La_{L(p)}\not\longleftrightarrow \infty]\phi_{\La_{L(p)},p}^0[0\longleftrightarrow x]
	\ge c\pi_1(R)^2,
$$ 
where the first inequality is due to the spatial Markov property~\eqref{eq:SMP}, 
and the second inequality follows \eqref{eq:RSWnear}, \eqref{eq:pmon} and an argument similar to Lemma~\ref{lem:0}. 

For the upper bound,  let~$A_x^*$ be the event that there exists a circuit in~$\omega^*$ surrounding 0 and~$x$.
 The FKG inequality in the first inequality, Corollary~\ref{rmk:1} and~\eqref{eq:CBC} in the second, and a reasoning similar to Lemma~\ref{lem:0} in the third (based on~\eqref{eq:mix} and Theorem~\ref{thm:stability}) imply that 
\begin{align*}\phi_p[0\longleftrightarrow x,0\not\longleftrightarrow\infty]&\le  \phi_p[0\longleftrightarrow x]\phi_p[A_x^*]\\
&\les\phi_{\Lambda_R,p}^1[0\longleftrightarrow\partial\Lambda_{R/2}]^2\exp(-c|x|/L)\\
&\les \pi_1(R)^2\exp(-c|x|/L).\end{align*}
Summing this inequality over~$x\in\bbZ^2$ gives the upper bound.

\subsubsection{Proof of~\eqref{eq:L_Delta} (scaling relation between~$\iota$ and~$\nu$)}

Assume~$p > p_c$, the case~$p < p_c$ is identical. Write~$L = L(p)$. Use Corollary~\ref{cor:deriv} and integrate the derivative of 
$g(p):=\phi_p[\calC(\La_{L})]$ between~$p_c$ and~$p$ to get
\begin{align}\label{eq:ak2}
	\int_{p_c}^p L^2\Delta_{u}(L)du \les g(p)-g(p_c)\les 1.
\end{align}
In the other direction, let~$p_0 \in [p_c,p]$ be such that~$L(p_0)=RL(p)$ for some~$R > 1$. 
Theorem~\ref{thm:stability} and the definition of~$L(p)$ implies that 
\begin{equation}\label{eq:aki1}
	g(p)-g(p_0)\ge g(p)-g(p_c)-CR^{-\ep}\ges 1
\end{equation}
provided that~$R$ is large enough. 
For~$u \in [p_0, p]$, Corollary~\ref{cor:deriv}, together with the quasi-multiplicativity property Theorem~\ref{thm:delta}(ii), implies
\begin{align*}
	g'(u)\les L(u)^2\Delta_{u}(L(u))\les L^2\Delta_u(L).
\end{align*}
(Note that the constant in $\les$ depends on $R$.) Integrating the previous displayed equation between~$p_0$ and~$p$ and then using \eqref{eq:aki1} gives
\begin{align*}
	1\les g(p)-g(p_0)
	\les \int_{p_0}^p L^2\Delta_{u}(L)du 
	\le \int_{p_c}^p L^2\Delta_{u}(L)du.
\end{align*}
Together with \eqref{eq:ak2}, the previous displayed equation and the stability of~$\Delta_u(L)$ given by Theorem~\ref{thm:delta}(iii) 
show that 
\begin{align*}
	1 \asymp \int_{p_c}^p L^2\Delta_{u}(L)du \asymp (p - p_c) L^2\Delta_{p_c}(L),
\end{align*}
which concludes the proof.

\subsubsection{Proof of~\eqref{eq:lp2} (scaling relation between~$\iota$ and~$\alpha$)}\label{sec:scaling_iota}

A straightforward computation involving \eqref{eq:Russ} shows that
\begin{align*}
	f''(p)  = 2 \frac{{\rm d}}{{\rm d}p}\phi_p [\omega_e]=2\sum_{f}\Cov_p(\omega_e,\omega_f),
\end{align*}
where~$e$ is a fixed edge of~$\bbZ^2$ and the sum is over all edges~$f$. 
By Lemma~\ref{lem:Cov_Delta} applied to~$\calD$ formed of the single edge~$e$, we find 
\[
\Cov_p(\omega_e;\omega_f)\asymp \Delta_p(\ell \wedge L(p))^2 e^{-c \ell /L(p)}\asymp \Delta_{p_c}(\ell \wedge L(p))^2 e^{-c \ell /L(p)},
\]
where~$\ell$ is the distance between~$e$ and~$f$.
For the second equivalence, we use Theorem~\ref{thm:delta}(iii).
Summing the displayed equation above over all edges~$f$, we conclude that 
\begin{align*}
	f''(p)  \asymp \sum_{\ell = 1}^{L(p)} \ell \Delta_{p_c}(\ell)^2,
\end{align*}
as required.

\begin{remark}\label{rem:alpha_positive}
	The above shows that if the phase transition is of second order (meaning that~$f''(p)$ diverges as~$p$ tends to~$p_c$), 
	then~$\sum_{\ell}\ell\Delta_p(\ell)^2$ diverges, which, using the interpretation of crossing probabilities in terms of~$\Delta_p(\ell)$,
	implies that the crossing probabilities of quads for the infinite-volume measure are not differentiable at~$p_c$.
	\end{remark}

\begin{remark}\label{rem:alpha_negative}
When~$\sum_{\ell  \geq 1}\ell\Delta(\ell)^2$ converges, the computation above simply proves that~$f''(p)$ remains bounded uniformly in~$p$. Nevertheless, it is easy to deduce by differentiating one more time that for~$p\ne p_c$ (we drop~$p$ from the notation), \begin{align*}
	f'''(p)
	&\le \sum_{f,g}\phi[\omega_e\omega_f\omega_g]-\phi[\omega_e\omega_f]\phi[\omega_g]-\phi[\omega_f\omega_g]\phi[\omega_e]-\phi[\omega_e\omega_g]\phi[\omega_f]+2\phi[\omega_e]\phi[\omega_f]\phi[\omega_g]\\
	&\les \sum_{\ell\le L(p)}\ell \Delta(\ell)\Delta(\ell',\ell)\sum_{\ell'\le\ell}\ell'\Delta(\ell')^2
	\les L(p)^4\Delta(L(p))^3
	\les \frac{1}{|p-p_c|^3L(p)^{2}}.
\end{align*}
If one defines~$\alpha$ in this framework by the formula~$f'''(p)=|p-p_c|^{-\alpha-1-o(1)}$, we deduce that~$\alpha\le2-2\nu$. 
The converse bound does not follow by the same computation since the summand on the first line is not of definite sign; 
at the time of writing the matching lower bound on~$f'''(p)$ remains unproved.
\end{remark}

\subsection{Scaling relation with magnetic field: proof of Theorem~\ref{thm:scaling_rel_h}}\label{sec:scaling_rel_h}\label{sec:14}

We insist one last time on the fact that this section is independent of the rest of the paper. Below, we work with the graph $\mathbb Z^2$ with the addition of the ghost.
We drop~$p_c$ and~$q$ but keep~$h$ in the notation except when it is equal to~$0$, in which case we omit it as well. 
We start by a lemma relating certain quantities at~$h=0$ with the corresponding quantities at~$h\ge0$.

Set~$\tilde\pi_1(h,R):=\phi_{\Lambda_R,h}^1[0\lra\partial\Lambda_R]$, 
where connections need to occur in~$\bbZ^2$. 
Additionally, let $A_R^*$ be the event that there exists a dual circuit in~$\Ann(R,2R)$ surrounding~$\La_R$ and
\begin{align*}
	\beta(h,R):=\phi_{\Ann(R,2R),h}^1[A_R^*].
\end{align*}
Finally, for~$C > 1$, define 
\begin{align*}
	h_c(R)=h_c(R,C):=\inf\big\{h>0:\,\exists\ r\le R\text{ such that }  \tilde\pi_1(h,r)> C \tilde\pi_1(r) 
	\text{ or } \beta(h,r)<  C^{-1}\beta(r)\big\}.
\end{align*}

\begin{lemma}\label{lem:ab}
    For any ~$C>1$, there exists~$\varepsilon>0$ such that for every~$R$,
    \begin{align}\label{eq:h_c}
	h_c(R)R^2\pi_1(R) \geq \varepsilon.\end{align} 
\end{lemma}

Due to the definition of~$h_c$, 
crossing estimates as in Theorem~\ref{thm:RSWnear} apply in the regime of~$(r,h)$ with~$r\le R$ and~$h\le h_c(R)$. 
Indeed, the crossing probabilities in the primal model increase with~$h$, which ensures the lower bounds. 
For the upper bounds, observe that~$\beta(h,r)$ involves the boundary conditions that render dual crossings least likely. 
Theorem~\ref{thm:RSWnear} applied at~$p_c$ combined with the definition of~$h_c(R)$ implies the uniform positivity of~$\beta(r,h)$
for~$r \leq R$ and~$h \leq h_c(R)$.  
The FKG inequality and the monotonicity of boundary conditions imply lower bounds for crossing probabilities in the dual model, as claimed. 

As a consequence, a similar proof to that of Lemma~\ref{lem:0} applies for all~$r\le R$ and~$0\le h\le h_c(R)$, and yields 
\begin{align}
	\phi_{\Ann(r,2r),h}^1[y\longleftrightarrow \partial\Lambda_{2r}]&\les \pi_1({\rm dist}(y,\partial \Ann(r,2r))) 
	&& \text{ for ~$y \in \Ann(r,2r)$ and }\nonumber\\
	\phi_{\Lambda_r,h}^1[0\longleftrightarrow y,0\longleftrightarrow \partial\Lambda_r]&\les \pi_1(r)\pi_1(\|y\| \wedge {\rm dist}(y,\partial \La_n)) 
	&&\text{ for ~$y \in \La_r$},\label{eq:yy}
\end{align}
where the constants in~$\les$ depend on~$C$. 

\begin{proof}[Lemma~\ref{lem:ab}]
Fix~$C> 1$ and~$r\le R$. All constants in the signs~$\les$ below are allowed to depend on~$C$, but not on~$r$ or~$R$. 
The differential formula~\cite[Thm.~3.12]{Gri06} reads 
\begin{align}\label{eq:ko}
	\tfrac {d}{dh}\tilde\pi_1(h,r)
	= \tfrac{1}{1-e^{-h}}\sum_{y\in\Lambda_r}\phi_{\Lambda_r,h}^1[0\lra \partial\Lambda_r,\omega_{y\ghost}=1]
	-\phi_{\Lambda_r,h}^1[0\lra \partial\Lambda_r]\phi_{\Lambda_r,h}^1[\omega_{y\ghost}=1].
\end{align}
Let us analyse the right-hand side of the above. Recall that~$\sfC$ is the cluster of the origin for the connectivity in~$\bbZ^2$.
First we show that the vertices~$y \notin \sfC$ have a negative contribution to~\eqref{eq:ko}.
Fix some~$y \in \La_r$.  
For a set~$\calC\subset\Lambda_r$ containing~$0$ and~$y \notin \calC$, 
\begin{align*}
    \phi_{\Lambda_r,h}^1[0\lra\partial\Lambda_r,\omega_{y\ghost}=1,\sfC=\calC]
    &=\phi_{\Lambda_r,h}^1[\phi_{\Lambda_r,h}^1(\omega_{y\ghost}=1\,|\,\sfC=\calC) \,\ind_{0\lra\partial\Lambda_r}\ind_{\sfC=\calC}]\\
    &\le \phi_{\Lambda_r,h}^1[\omega_{y\ghost}=1]\phi_{\Lambda_r,h}^1[0\lra\partial\Lambda_r,\sfC=\calC].
\end{align*}
The first inequality is due~\eqref{eq:SMP} and~\eqref{eq:CBC} since~$\phi_{\Lambda_r,h}^1[ \cdot \,|\,\sfC = \calC]$ is a random-cluster measure with free boundary conditions on~$(\Lambda_r \setminus \calC ) \cup \{\ghost\}$, 
and is stochastically dominated by the restriction of~$\phi_{\Lambda_r,h}^1$ to~$(\Lambda_r \setminus \calC ) \cup \{\ghost\}$.

Summing the above over every~$\calC\subset \Lambda_r$, we find that 
\begin{align*}
	&\phi_{\Lambda_r,h}^1[0\lra \partial\Lambda_r,\omega_{y\ghost}=1]-\phi_{\Lambda_r,h}^1[\omega_{y\ghost}=1]\phi_{\Lambda_r,h}^1[0\lra\partial\Lambda_r]
	 \le \phi_{\Lambda_r,h}^1[0\lra y,0\lra \partial\Lambda_r,\omega_{y\ghost}=1].
\end{align*}
Plugging this inequality in~\eqref{eq:ko} gives
\begin{align*}
	\tfrac {d}{dh}\tilde\pi_{1}(h,r)
	\le \tfrac{1}{1-e^{-h}} \sum_{y\in\Lambda_r}\phi_{\Lambda_r,h}^1[0\lra y,0\lra\partial\Lambda_r,\omega_{y\ghost}=1]
	\le  \sum_{y\in\Lambda_r}\phi_{\Lambda_r,h}^1[0\lra y,0\lra\partial\Lambda_r],
\end{align*}
where the second inequality comes from~\eqref{eq:SMP}, which implies that 
\[
\phi_{\Lambda_r,h}^1[\omega_{v\ghost}=1\,|0\lra y,0\lra\partial\Lambda_r]\le 1-e^{-h}.
\]
Now, assuming that~$h \le h_c(R)$,~\eqref{eq:yy} applies, and we conclude that 
\begin{align*}
    \tfrac {d}{dh}\tilde\pi_1(h,r)
    \les \sum_{y \in \La_r} \pi_1(r)\pi_1(\|y\|\wedge {\rm dist}(y,\partial\La_r))
    \les r\pi_1(r)\sum_{k = 1}^{r/2}   \pi_1(k)
    \les r^2\pi_1(r)^2.
\end{align*}
The second inequality uses the fact that there are at most order~$r$ vertices at distance~$k$ from~$0$ or~$\partial \La_r$; 
the third one is a standard consequence of~\eqref{eq:LOWER_BOUND_ONE_ARM} and \eqref{eq:MULTIPLICATIVITY}.
Keeping in mind that~$\pi_1(r) \le \tilde\pi_1(h,r)$, the above implies
\begin{align}\label{eq:da_up}
    \tfrac {d}{dh}\log\tilde\pi_1(h,r) \les r^2\pi_1(r).
\end{align}
A similar computation, where~$\sfC$ is replaced by the cluster of~$\partial \La_{2r}$ implies that 
\begin{align}\label{eq:betaderiv}
	-\tfrac {d}{dh} \log \beta(h,r)
	\les  \sum_{y\in\Ann(r,2r)}\phi_{\Ann(r,2r),h}^1[y\lra\partial\Lambda_{2r}]
	\les r^2\pi_1(r).
\end{align}

We are now in a position to conclude. 
Let~$c_0$ be a constant larger than the constants involved in the inequalities~$\les$ in~\eqref{eq:da_up} and~\eqref{eq:betaderiv}. 
Then, for~$\ep >0$, integrate these two inequalities for~$h$ between~$0$ and~$h' = \min\{h_c(R),\ep/(R^2\pi_1(R))\}$. 
We find
\begin{align*}
\left.\begin{array}{r}
\log\tilde\pi_1(h,r) - \log\tilde\pi_1(r)\\
\log \beta(r)  - \log \beta(h,r) 
\end{array}\right\} \le  c_0h' r^2\pi_1(r) \le \eps c_0.
\end{align*}
Now, for~$\eps < (\log C) /c_0$, the above shows that~$h' < h_c(R)$, 
which implies~\eqref{eq:h_c}. 
\end{proof}

\begin{proof}[Theorem~\ref{thm:scaling_rel_h}]
    Fix~$h >0$. Again, $\mathsf C$ is the cluster of the origin when considering connections in $\mathbb Z^2$ only.
    We start with the lower bound. We have
    \begin{align*}
    	\phi_h[0\longleftrightarrow\ghost]
		&\ge \phi_{h}[0\longleftrightarrow\ghost,|\sfC|\ge 1/h]\ges\phi_{h}[|\sfC|\ge 1/h ]
		\ge\phi_0[|\sfC|\ge 1/h ]
		\ge\pi_1(\varphi(1/h)),
    \end{align*}
    where the fourth inequality is due to~\eqref{eq:relation3}, the third to monotonicity in~$h$~\eqref{eq:hmon}, and the second to the fact that conditioned on~$|\sfC| \geq 1/h$, there is a positive probability for~$0$ to be connected to~$\ghost$. This last property can be easily deduced from the finite energy property, which states that every edge connecting~$\bbZ$ to~$\ghost$ has a probability larger than~$(1-e^{-h})/(1+(q-1)e^{-h})$ and smaller than~$h$ of being open, regardless of the states of other edges.
    
    Let us now derive the upper bound.
    Let~$\varepsilon$ be the quantity given by Lemma~\ref{lem:ab} for some fixed~$C > 1$. 
    Choose~$R$ to be the largest integer such that~$h R^2\pi_1(R)\leq\varepsilon$. 
    Notice that this implies that~$h\le h_c(R)$. We deduce that
    \begin{align*}
        \phi_{h}[0 \longleftrightarrow\ghost]
        &\le \pi_1(h,R)+\phi_h[0\longleftrightarrow\ghost, \, 0\not\longleftrightarrow \partial\Lambda_R]\\
        &\le \pi_1(h,R)+\sum_{y\in \Lambda_R}\phi_h[0\longleftrightarrow y,\,\omega_{y\ghost}=1]\\
        &\le \pi_1(h,R)+ h \sum_{y\in \Lambda_R}\phi_h[0\longleftrightarrow y],
    \end{align*}
    where the last inequality uses the finite energy property. 
        Now, use that~$h\le h_c(R)$ to deduce that~$ \pi_1(h,R)\les \pi_1(R)$ and that, similarly to Lemma~\ref{lem:0},~$\phi_h[0\longleftrightarrow y]\les \pi_1(|y|)^2$. We deduce from the above that 
        \begin{align*}
     \phi_{h}[0 \longleftrightarrow\ghost]   &\les \pi_1(R)+ h \sum_{y\in \Lambda_R} \pi_1(|y|)^2        \les\pi_1(R)+ hR^2\pi_1(R)^2 \les\pi_1(R),
    \end{align*}
    where in the second inequality we used a computation similar to~\eqref{eq:main_chi}.
    Observe that, with the notation of Theorem~\ref{thm:scaling_rel_crit},~$R=\varphi(\ep/h) \ges \varphi(1/h)$, so
 ~$$
    	\phi_{h}[0 \longleftrightarrow\ghost] \les\pi_1(\varphi(1/h)),
~$$
which concludes the proof. 
\end{proof}

\paragraph{Acknowledgments} 

The first author is supported by the ERC CriBLaM, the NCCR SwissMAP, the Swiss NSF and an IDEX Chair from Paris-Saclay. 
The second author is supported by  the NCCR SwissMAP and the Swiss NSF. 
We thank Matan Harel for useful discussions at an early stage of the project, as well as Jeremy Dubout and Robin Pierre Kaufmann who worked on parts of Sections~\ref{sec:11} and~\ref{sec:boost} during their respective Master thesis with the two authors. Finally, we are extremely thankful to Vincent Tassion, whose discussions on this project have been very inspiring and absolutely crucial.

\newcommand{\etalchar}[1]{$^{#1}$}
\providecommand{\bysame}{\leavevmode\hbox to3em{\hrulefill}\thinspace}
\providecommand{\MR}{\relax\ifhmode\unskip\space\fi MR }
\providecommand{\MRhref}[2]{%
  \href{http://www.ams.org/mathscinet-getitem?mr=#1}{#2}
}
\providecommand{\href}[2]{#2}

\end{document}